\documentclass[10pt,reqno]{amsart}
\usepackage{amsthm, amsmath,amsfonts,amssymb,euscript,hyperref,graphics,color,slashed}
\usepackage{graphicx}
\usepackage{comment}
\usepackage{import}
\usepackage{tikz}
\usepackage{latexsym}
\usepackage{xcolor}

\def\ub{{\underline{u}}}
\def\Lb{\underline{L}}
\def\Cb{{\underline{C}}}
\def\Eb{\underline{E}}
\def\Fb{\underline{F}}
\def\gslash{\mbox{$g \mkern -8mu /$ \!}}

\def\nablaslash{\mbox{$\nabla \mkern -13mu /$ \!}}
\def\laplacianslash{\mbox{$\Delta \mkern -13mu /$ \!}}

\newtheorem*{MainTheorem}{Main Theorem}
\newtheorem{theorem}{Theorem}[section]
\newtheorem{lemma}[theorem]{Lemma}
\newtheorem{proposition}[theorem]{Proposition}
\newtheorem{corollary}[theorem]{Corollary}
\newtheorem{definition}[theorem]{Definition}
\newtheorem{remark}[theorem]{Remark}
\newtheorem*{Main Theorem}{Main Theorem}

\numberwithin{equation}{section}

\def\Et{\widetilde{E}}
\def\Ebt{\widetilde{\underline{E}}}
\def\lot{{\text{l.o.t.}}}
\def\Ntop{{N_{\text{top}}}}
\def\Nmu{N_{\mu}}
\def\Ninfty{N_{\infty}}
\def\gt{{\widetilde{g}}}
\def\ds{{\slashed{d}}}
\def\gslash{{\slashed{g}}}
\def\ub{\underline{u}}
\def\Lb{\underline{L}}
\def\Cb{\underline{C}}
\def\Eb{\underline{E}}
\def\Fb{\underline{F}}
\def\tr{{\textrm{tr}}}
\def\dmug{{d \mu_{\slashed{g}}}}
\def\alphab{\underline{\alpha}}

\def\etab{\underline{\eta}}
\def\zetab{\underline{\zeta}}

\def\O{\mathcal{O}}

\def\slashedL{{\slashed{\mathcal{L}}}}

\def\slashedLRi{{\slashed{\mathcal{L}}_{R_i}}}

\def\chib{\underline{\chi}}
\def\chibh{\widehat{\underline{\chi}}}
\def\chit{\widetilde{\chi}}
\def\chibt{\widetilde{\underline{\chi}}}
\def\gslash{\mbox{$g \mkern -8mu /$ \!}}
\def\divslash{\slashed{\textrm{div}}}

\def\nablaslash{\mbox{$\nabla \mkern -13mu /$ \!}}
\def\laplacianslash{\mbox{$\Delta \mkern -13mu /$ \!}}

\def\Zb{{\underline{Z}}}

\def\Et{{\widetilde{E}}}

\def\Ebt{{\underline{\widetilde{E}}}}

\newcommand{\leftexp}[2]{{\vphantom{#2}}^{#1}{#2}}

\begin{document}
\title[Shock QLW with weak pulse]{On the formation of shock for quasilinear wave equations by pulse with weak intensity}
\author[Shuang Miao]{Shuang Miao}
\thanks{Department of Mathematics, University of Michigan,
Ann Arbor MI, U.S.A. shmiao@umich.edu\\
\emph{Current address}: B\^atiment des Math\'ematiques, EPFL, Lausanne, Switzerland, shuang.miao@epfl.ch}

\begin{abstract}
In this paper we continue to study the shock formation for the $3$-dimensional quasilinear wave equation

\begin{align}\label{main eq}
-(1+3G''(0)(\partial_{t}\phi)^{2})\partial^{2}_{t}\phi+\Delta\phi=0,\tag{\textbf{$\star$}}
\end{align}
with $G''(0)$ being a non-zero constant. Since \eqref{main eq} admits global-in-time solution with small initial data, to present shock formation, we consider 
a class of large data. Moreover, no symmetric assumption is imposed on the data. Compared to our previous work \cite{M-Y}, 
here we pose data on the hypersurface $\{(t,x)|t=-r_{0}\}$ instead of $\{(t,x)|t=-2\}$, with $r_{0}$ being arbitrarily large. We prove an a priori energy estimate independent of $r_{0}$. 
Therefore a complete description of the solution behavior as $r_{0}\rightarrow\infty$ is obtained. This allows us to relax the restriction on the profile of initial 
data which still guarantees shock formation. Since \eqref{main eq} can be viewed as a model equation for describing the propagation of electromagnetic waves in nonlinear
dielectric, the result in this paper reveals the possibility to use wave pulse with weak intensity to form electromagnetic shocks in laboratory. A main new feature in the proof is that all estimates in the present paper do \emph{not} depend on the parameter $r_{0}$, which requires different methods to obtain energy estimates. As a byproduct, we prove the existence of semi-global-in-time solutions which lead to shock formation by showing that the limits of the initial energies exist as $r_{0}\rightarrow\infty$. The proof combines the ideas in \cite{Ch1} where the the formation of shocks for 3-dimensional relativistic compressible Euler equations with small initial data is established, and the short pulse method introduced in \cite{Ch2} and generalized in \cite{K-R}, where the formation of black holes in general relativity is proved.
\end{abstract}

\maketitle

\section{Introduction}
In this paper we study the following \emph{quasilinear} wave equation
\begin{align}\label{Main Equation}
 - \left(1+3G^{\prime\prime}(0) (\partial_t\phi)^2\right)\partial^2_t \phi +\Delta\phi=0,\tag{\textbf{$\star$}}
\end{align}
where $G''(0)$ is a \emph{nonzero} constant and $\phi \in C^\infty(\mathbb{R}_t\times\mathbb{R}_{x}^3;\mathbb{R})$ is a smooth solution. 
The aim of the paper is to present an energy-estimate-based proof for the stable shock formation of smooth solutions to \eqref{Main Equation} generated by data prescribed at $\Sigma_{-r_{0}}:=\{(t,x)|t=-r_{0}\}$ with $r_{0}$ being arbitrarily large. A classical result by Klainerman \cite{K1} says that the equation \eqref{Main Equation} admits global-in-time smooth solution with small data, so in order to prove shock formation, we consider a class of ``large" data which will be specified below. As we shall see, the equation \eqref{Main Equation} can be regarded as a model equation for the Maxwell equations in nonlinear electromagnetic theory, in which the shocks can be observed experimentally. The shock formation in nonlinear electromagnetic theory will be the subject of our forthcoming work. 

\subsection{Lagrangian formulation of the main equation and its relation to nonlinear electromagnetic waves}
We briefly discuss the derivation of the main equation $(\star)$. The linear wave equation in Minkowski spacetime $(\mathbb{R}^{3+1}, m_{\mu\nu})$ can be derived by a variational principle: we take the Lagrangian density $L(\phi)$ to be $\frac{1}{2}\big(-(\partial_t \phi)^2 + |\nabla_x \phi|^2\big)$ and take the action functional $\mathcal{L}(\phi)$ to be $\int_{\mathbb{R}^{3+1}} L(\phi) d \mu_m$ where $d \mu_m$ denotes the volume form of the standard Minkowski metric $m_{\mu\nu}$. The corresponding Euler-Lagrange equation is exactly the linear wave equation  $-\partial^2_t\phi +\Delta \phi = 0$. We observe that the quadratic nature of the Lagrangian density result in the \textit{linearity} of the equation. This simple observation allows one to derive plenty of \textit{nonlinear} wave equations by changing the quadratic nature of the Lagrangian density. In particular, we will change the quadratic term in $\partial_t\phi$ to a quartic term, this will lead to a \textit{quasi-linear} wave equation.

In fact, we consider a perturbation of the Lagrangian density of linear waves:
\begin{equation}\label{Lagrangian density}
 L(\phi)= -\frac{1}{2}G\big((\partial_{t}\phi)^{2}\big)+\frac{1}{2}|\nabla\phi|^{2},
\end{equation}
where $G=G(\rho)$ is a smooth function defined on $\mathbb{R}$ and $\rho = |\partial_t  \phi|^2$. The corresponding Euler-Lagrange equation is
\begin{equation*}
 -\partial_t \big(G^{\prime}(\rho) \partial_t\phi\big)+\Delta\phi=0.
\end{equation*}

 The function $G(\rho)$ is a perturbation of $G_0(\rho)=\rho$ and therefore we can think of the above equation as a perturbation of the linear wave equation. For instance, we can work with a real analytic function $G(\rho)$ with $G(0)=0$ and $G^{\prime}(0)=1$. In particular, we can perturb $G(\rho) = \rho$ in the simplest possible way by adding a quadratic function so that $G(\rho) = \rho +\frac{1}{2}G''(0)\rho^2$. In this situation, we obtain precisely the main equation $(\star)$. It is in this sense that $(\star)$ can be regarded as the simplest quasi-linear wave equation derived from action principle.\vspace{2mm}

The main equation $(\star)$ is also closely tight to electromagnetic waves in a nonlinear dielectric. The Maxwell equations in a homogeneous insulator is derived from a Lagrangian $L$ which is a function of the electric field $E$ and the magnetic field $B$. The corresponding displacements $D$ and $H$ are defined through $L$ by $D=-\frac{\partial L}{\partial E}$ and $H=\frac{\partial L}{\partial B}$ respectively. In the case of an isotropic dielectric, $L$ is of the form
\begin{align}\label{Lagrangian Maxwell}
L=-\frac{1}{2}G(|E|^{2})+\frac{1}{2}|B|^2,
\end{align}
hence $H = B$. The fields $E$ and $B$ are derived from the scalar potential $\psi$ and the vector potential $A$ according to $E=-\nabla\psi-\partial_{t}A$ and $B=\nabla\times A$ respectively. This is equivalent to the first pair of Maxwell equations:
\begin{align}\label{Maxwell 1}
\nabla\times E+\partial_{t}B=0,\quad \nabla\cdot B=0.
\end{align}
The potentials are determined only up to a gauge transformation $\psi\mapsto \psi-\partial_{t}f$ and $A\mapsto A+df$, where $f$ is an arbitrary smooth function. The second pair of Maxwell equations
\begin{align}\label{Maxwell 2}
\nabla\cdot D=0,\quad \nabla\times H-\partial_{t}D=0.
\end{align}
are the Euler-Lagrange equations, the first resulting from the variation of $\psi$ and the second resulting from the variation of $A$. Fixing the gauge by setting $\psi = 0$, we obtain a simplified model if we neglect the vector character of $A$ replacing it by a scalar function $\phi$. Then the above equations for the fields in terms of the potentials simplify to $E=-\partial_{t}\phi$ and $B=\nabla\phi$.
The Lagrangian \eqref{Lagrangian Maxwell} becomes
\begin{align*}
L=-\frac{1}{2}G((\partial_{t}\phi)^{2})+\frac{1}{2}|\nabla\phi|^{2}
\end{align*}
which is exactly \eqref{Lagrangian density}. Therefore, the main equation $(\star)$ provides a good approximation for shock formation in a natural physical model: the shock formations for nonlinear electromagnetic waves. We will discuss the physical motivation to start with a weak intensity pulse in detail after we introduce the initial data in Section 1.3.\vspace{2mm}

\subsection{A geometric perspective for shock formation}
To present a geometric picture of shock formation, we start with the inviscid Burgers Equation

\begin{align}\label{Burgers}
\partial_{t}u+u\partial_{x}u=0.
\end{align}
We assume that $u \in C^\infty(\mathbb{R}_t\times \mathbb{R}_x; \mathbb{R})$ is a smooth solution. Given smooth initial data $u(0,x)$(non-zero everywhere for simplicity), along each \emph{characteristic curve} $u(0,x)t+x$, the solution $u(t,x)$ to \eqref{Burgers} remains constant. We now consider two specific characteristics passing through $x_1$ and $x_2$ ($x_1<x_2$). If we choose datum in such a way that $u(0,x_1) > u(0,x_2) >0$, both the characteristics travel towards the right (see the picture below). Moreover, the characteristic on the left (noted as $C_1$) travels with speed $c_1 = u(0,x_1)$ and the characteristic on the right (noted as $C_2$) travels with speed $c_2 = u(0,x_2)$. Since $C_1$ travels faster than $C_2$, $C_1$ will eventually catch up with $C_2$. The collision of two characteristics causes the breakdown on the smoothness of the solution. In summary, we have a geometric perspective on shock formation: a ``faster" characteristic catches up a ``slower" one so that it causes a collapse of characteristics. 

\includegraphics[width = 4.7 in]{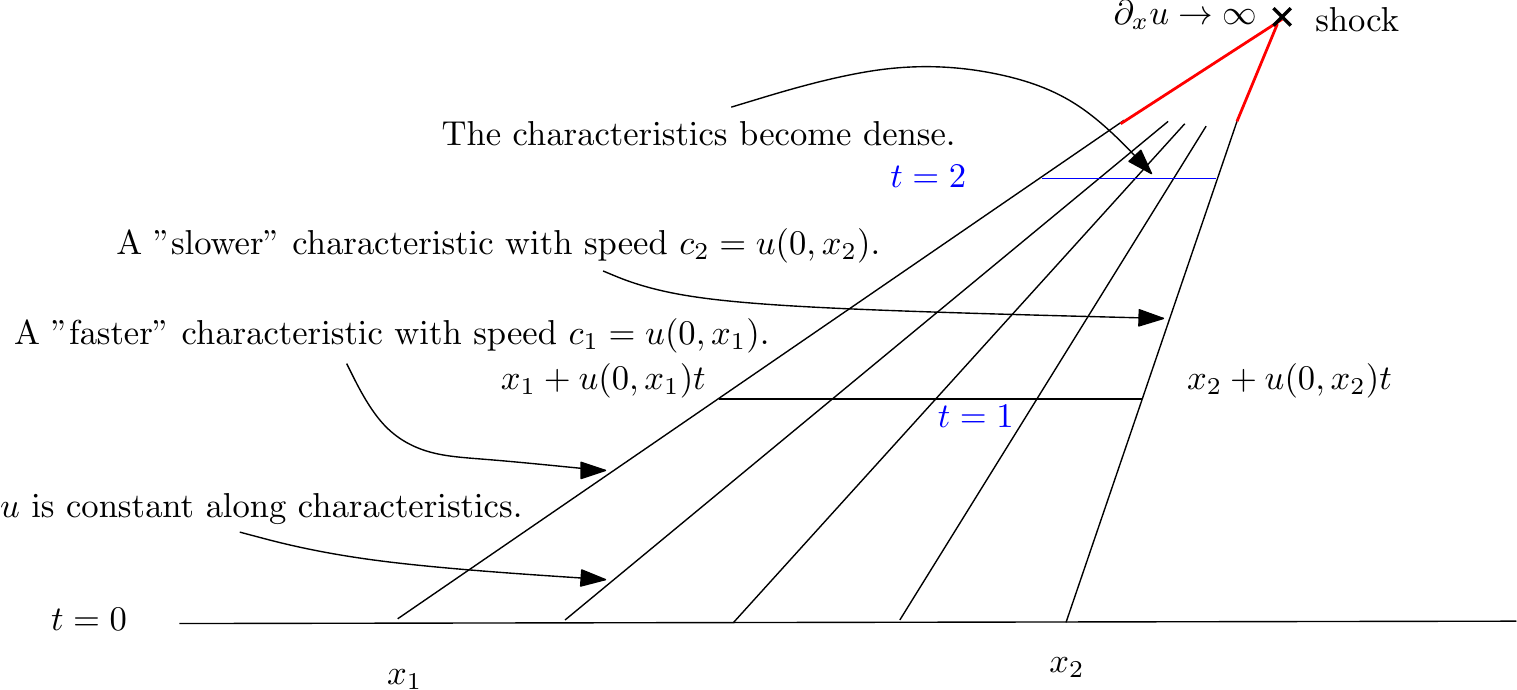}

In reality, instead of showing that characteristics collapse, we show that $|\partial_x u|$  blows up. Instead of being na\"{i}vely a derivative, $|\partial_x u|$ have an important \emph{geometric interpretation}. Recall that the level sets of $u$ are exactly the characteristic curves and the $(t,x)$-plane is foliated by these curves (see the above picture). Therefore, $|\partial_x u|$ is the \emph{density of the foliation by the characteristics}. As a consequence,  we can regard the shock formation as the following geometric picture: \emph{the foliation of characteristic curves becomes infinitely dense}.

Let us recall a standard way to prove the blow-up of $|\partial_x u|$. The remarkable feature of this standard proof is that in three dimensions similar phenomenon happens for the main equation $(\star)$. Let $\underline{L} = \partial_t + u \partial_x$ be the tangent vectorfield of the characteristic curves (for $(\star)$, the corresponding vectorfield are tangent vectorfield of null geodesics on the characteristic hypersurfaces). Therefore, by taking $\partial_x$ derivatives, we obtain
\begin{equation*}
\underline{L}  \partial_x u + (\partial_x u)^2 = 0.
\end{equation*}
This is a Riccati equation for $\partial_x u$ and $\partial_x u$ blows up in finite time if it is negative initially. However, we would like to understand the blow-up in another way (which is tied to the shock formation for $(\star)$). We define the \emph{inverse density function} $\mu = - (\partial_x u)^{-1}$, therefore, along each characteristic curve, $\mu$ satisfies the following equation:
\begin{equation*}
\underline{L} \mu(t,x) = -1,
\end{equation*}
i.e. $\underline{L} \mu$ is constant along each characteristic curve so that it is determined by its initial value. Therefore,if $\mu$ is positive initially it will eventually become $0$ after finite time. $\mu\rightarrow0$ means that the foliation becomes infinitely dense (For $(\star)$, we will also define an inverse density function $\mu$ for the foliation of characteristic hypersurfaces and show that $\underline{L} \mu(t,x)$ is almost a constant along each generating geodesic of the characteristic hypersurfaces).\vspace{2mm}

We return to the main equation $(\star)$, which can be written as the following \emph{geometric form}:

\begin{align}\label{Main equation geometric}
-\frac{1}{c^{2}}\partial^{2}_{t}\phi+\Delta\phi=0
\end{align}
with $c=(1+3G''(0)(\partial_{t}\phi)^{2})^{-1/2}$. If $c$ were a constant, \eqref{Main equation geometric} would describe the propagation of light in Minkowski spacetime. In the present situation, we still regard $c$ as the speed of wave propagation which depends on time and position $(t,x)$ through the unknown $\phi$. We prescribe the initial data $\phi(-r_{0},\theta), (\partial_{t}\phi)(-r_{0},\theta)$ on the hypersurface $\Sigma_{-r_{0}}$. Let $S_{r}$ be the sphere centered at the origin of $\Sigma_{-r_{0}}$ with radius $r$ and $B_{r}$ be the ball with $S_{r}$ as its boundary. We choose the data such that

\begin{enumerate}
\item The data is trivial inside $B_{r_{0}}$. By the Huygens' Principle, the solution $\phi$ is identically zero in the interior of the backward solid light cone with the base $B_{r_{0}}$. If we denote the wave propagation speed along the incoming characteristic hypersurface emanated from the leaf $S_{r_{0}}$ by $c_{1}$, then in view of the formula $c=(1+3G''(0)(\partial_{t}\phi)^{2})^{-1/2}$, $c_{1}=1$.
\item $|\partial_{t}\phi|$ is approximately of size $\delta^{1/2}r_{0}^{-1}$ in the annulus region between $S_{r_{0}}$ and $S_{r_{0}+\delta}$. If we denote the wave propagation speed along the incoming characteristic hypersurface emanated from the leaf $S_{r_{0}+\delta}$ by $c_{2}$, then a Taylor expansion implies that $c_{2}\sim 1-\frac{3}{2}G''(0)\delta r^{-2}_{0}+O(\delta^{2}r_{0}^{-4})$.
\end{enumerate}
A key point of the paper is to identify a set of initial data such that profile $c_{2}(t)\sim 1-\frac{3}{2}G''(0)\delta t^{-2}+O(\delta^{2}t^{-4})$ is preserved in the evolution. Another way to understand this is through energy estimates: we can find a specific set of data so that we can obtain a priori energy estimates. Once we show that the energy (and its higher order analogue) is almost conserved, we can use Sobolev inequality to show that the pointwise profile of $\partial_t \phi$ is almost conserved along the generators of the incoming characteristic hypersurfaces. According to the formula of $c$, this also implies the preservation of the profile for wave speed. We are now in a situation that resembles the Burgers' picture. The initial distance between the inner most characteristic hypersurface, which is emanated from $S_{r_{0}}$, and the outer most characteristic hypersurface, which is emanated from $S_{r_{0}+\delta}$, is $\delta$. We also expect the ``faster"(outer) characteristic hypersurface catching up the ``slower"(inner) one. This catching up process can be described as $``distance"=\int_{-r_{0}}^{``time"}(c_{2}-c_{1})dt$. Since $c_{2}-c_{1}\sim \delta t^{-2}$ and we regard shock formation as the collision of characteristic hypersurfaces, we expect the shock formation around $t=-1$. Finally, we point out that, as in the case of Burgers' equation, instead of showing that characteristic hypersurfaces intersect, we show that the inverse density $\mu$ of the foliation by the characteristic hypersurfaces becomes $0$, i.e. the foliation turns to being infinitely dense. Similarly, this can be done by showing that profile of $\underline{L} \mu(t,x)$ is almost a constant along each generating geodesic of the characteristic hypersurfaces.

\subsection{Main results}
We now state the main result of the paper. Let $t$ be the time function in Minkowski spacetime.  We use $\Sigma_t$ to denote its level set, which is identified as $\mathbb{R}^{3}$ for each $t$. 
We also use $\Sigma_{-r_0}^{\delta}$ to denote the following annulus:
\begin{equation}\label{delta annulus}
 \Sigma_{-r_{0}}^{\delta}:=\left\{x\in\Sigma_{-r_{0}} \big| r_{0}\leq r(x)\leq r_{0}+\delta\right\}.
\end{equation}
With the wave speed $c:=\left(1+3G''(0)(\partial_{t}\phi)^{2}\right)^{-1/2}$, let $\Lb:=\partial_{t}-c\partial_{r}$ and $L:=\partial_{t}+c\partial_{r}$ on $\Sigma_{-r_{0}}^{\delta}$.  We introduce a pair of functions $\big(\phi_1(s,\theta),\phi_2(s,\theta)\big) \in C^\infty\big((0,1]\times \mathbb{S}^2\big)$ and we will call it the \emph{seed data}.

The seed data $\big(\phi_1(s,\theta),\phi_2(s,\theta)\big)$ can be freely prescribed and once it is given once forever. In particular, the choice of the seed data is independent of the small parameter $\delta$. As we will see below, even though we call $(\phi_{1},\phi_{2})$ the seed data, from the proof of the following Lemma 1.1, our initial data depends on $\phi_{2}$ implicitly.

\begin{lemma}\label{lemma constraint}
 Given seed data $(\phi_1,\phi_2)$, there exists a $\delta'>0$ depending only on the seed data, for all $\delta<\delta'$, we can construct another function $\phi_0\in  C^\infty\big((0,1]\times \mathbb{S}^2\big)$ satisfying the following two properties:

(1) For all $k\in \mathbb{Z}_{\geq 0}$, the $C^k$-norm of $\phi_0$ are bounded by a function of the $C^k$-norms of $\phi_1$ and $\phi_2$;

(2) If we pose initial data for $(\star)$ on $\Sigma_{-r_{0}}$ in the following way:

For all $x \in \Sigma_{-r_{0}}$ with $r(x) \leq r_{0}$, we require $\big(\phi(-r_{0},x),  \partial_t \phi(-r_{0},x)\big) = \big(0,0\big)$;\ For $r_{0} \leq r(x) \leq r_{0}+\delta$, we require that
\begin{equation}\label{data profile}
 \phi(-r_{0},x)=\frac{\delta^{3/2}}{r_{0}}\phi_{0}(\frac{r-r_{0}}{\delta},\theta),\ \ (\partial_{t}\phi)(-r_{0},x)=\frac{\delta^{1/2}}{r_{0}}\phi_{1}(\frac{r-r_{0}}{\delta},\theta).
\end{equation}
Then we have
\begin{equation}\label{no outgoing radiation}
\|\Lb\phi\|_{L^\infty\left(\Sigma^{\delta}_{-r_{0}}\right)}\lesssim \frac{\delta^{3/2}}{r_{0}^{2}},\ \ \|\Lb^{2}\phi\|_{L^\infty\left(\Sigma^{\delta}_{-r_{0}}\right)}\lesssim \frac{\delta^{3/2}}{r_{0}^{3}}.
\end{equation}
\end{lemma}
We remark that the condition \eqref{no outgoing radiation} has a clear physical meaning: since $\Lb\phi$ controls the outgoing radiation, initially the waves are actually set to be incoming and the outgoing radiation is very little, because the $L^{\infty}$-norm of the outgoing radiation field $\Lb\phi$ is controlled by $\delta$ and its $L^{2}$-norm on $\Sigma_{-r_{0}}$ decays to $0$ as $r_{0}\rightarrow\infty$.

\begin{proof}
As in \cite{M-Y}, a direct computation gives
\begin{align*}
\begin{split}
 \Lb^2 \phi &= \big(2-3G''(0)c^3\partial_t\phi\,\partial_r \phi\big)\big(c^2 \partial_r^2 \phi\big) + \big(1-3G''(0)c^3\partial_t\phi\,\partial_r \phi\big)\big(\frac{2c^2}{r}\partial_r\phi + \frac{c^2}{r^{2}} \laplacianslash_{\mathbb{S}^{2}} \phi \big)\\
 &\ \ \ +c\partial_r c \, \partial_r \phi-2c \partial_r(\partial_t \phi),
\end{split}
\end{align*}
which is rewritten, in terms of $\phi_{0}$ and $\phi_{1}$, as
\begin{align}\label{Lb2 phi seed}
\begin{split}
 \Lb^2 \phi|_{t=-r_{0}} &= \left(2-3G''(0)c^3\phi_1\,\partial_s \phi_0 \,\frac{\delta}{r_{0}^{2}}\right)\left(\frac{\delta^{-\frac{1}{2}}}{r_{0}}c^2 \partial_s^2 \phi_0 \right)\\
 &\ \ \ + \left(1-3G''(0)c^3\phi_1\,\partial_s \phi_0 \,\frac{\delta}{r_{0}^{2}}\right)\left(\frac{2c^2}{r_{0}^{2}}\partial_s\phi_0 \, \delta^{\frac{1}{2}} + \frac{c^2}{r_{0}^{3}} \laplacianslash_{\mathbb{S}^{2}} \phi_0 \, \delta^\frac{3}{2} \right)\\
 &\ \ \ -3c^{4}G''(0)\phi_{1}\partial_{s}\phi_{1} \partial_{s}\phi_{0}\frac{\delta^\frac{1}{2}}{r_{0}^{3}}-2c \partial_s\phi_1 \frac{\delta^{-\frac{1}{2}}}{r_{0}}.
\end{split}
\end{align}
Claim: We can choose $\phi_0$, which may depend on the choice of $\phi_1$ but is independent of $\delta$, in such a way that $|\Lb^2 \phi| \lesssim \dfrac{\delta^\frac{3}{2}}{r_{0}^{3}}$. 

To see this, we first observe that since $\phi_1$ is given and $\delta$ is small, we have $|c| \lesssim 1$. We make the following ansatz for $\phi_0$:
\begin{equation}\label{ansatz on phi_0}
 |\partial_s \phi_0| +  |\partial^2_\theta \phi_0| \leq C,
\end{equation}
where the constant $C$ may only depend on $\phi_1$ but not on $\delta$ or $r_{0}$. By the ansatz \eqref{ansatz on phi_0} and by looking at the powers in $\delta$ and $r_{0}$, one can ignore all the terms controlled by $\dfrac{\delta^{\frac{3}{2}}}{r_{0}^{3}}$. Therefore, to show $|\Lb^2 \phi| \lesssim \dfrac{\delta^\frac{3}{2}}{r_{0}^{3}}$, it suffices to consider
\begin{equation*}
\begin{split}
& \left(2-3G''(0)c^3\phi_1\,\partial_s \phi_0 \,\frac{\delta}{r_{0}^{2}}\right)\left(\frac{\delta^{-\frac{1}{2}}}{r_{0}}c^2 \partial_s^2 \phi_0 \right) +  \frac{2c^2}{r_{0}^{2}}\partial_s\phi_0 \, \delta^{\frac{1}{2}}    -3c^{4}G''(0)\phi_{1}\partial_{s}\phi_{1} \, \partial_s \phi_0 \cdot \frac{\delta^\frac{1}{2}}{r_{0}^{3}}-2c \partial_s\phi_1 \frac{\delta^{-\frac{1}{2}}}{r_{0}} \\
=& O\left(\frac{\delta^\frac{3}{2}}{r_{0}^{3}}\right),
\end{split}
\end{equation*}
or equivalently,

\begin{align*}
\begin{split}
 \left(2-3G''(0)c^3\phi_1\,\partial_s \phi_0 \,\frac{\delta}{r_{0}^{2}}\right)\left(c^2 \partial_s^2 \phi_0 \right) +  \frac{2c^2}{r_{0}}\partial_s\phi_0 \, \delta -3c^{4}G''(0)\phi_{1}\partial_{s}\phi_{1} \, \partial_s \phi_0 \cdot \frac{\delta}{r_{0}^{2}}-2c \partial_s\phi_1=O\left(\frac{\delta^2}{r_{0}^{2}}\right).
\end{split}
\end{align*}
Since $\left(2-3G''(0)c^3\phi_1\,\partial_s \phi_0 \,\frac{\delta}{r_{0}^{2}}\right)^{-1}=\dfrac{1}{2}+\dfrac{3}{4}G''(0)\phi_{1}\partial_{s}\phi_{0}\dfrac{\delta}{r_{0}^{2}}+O\left(\dfrac{\delta^{2}}{r_{0}^{4}}\right)$, by multiplying both sides of the above identity by $\left(2-3G''(0)c^3\phi_1\,\partial_s \phi_0 \,\dfrac{\delta}{r_{0}^{2}}\right)^{-1}$ and using the fact $c\sim 1$, it suffices to consider 

\begin{align*}
\partial_{s}^{2}\phi_{0}+\left(\frac{\delta}{r_{0}}+\frac{3}{2c}G''(0)\phi_{1}\frac{\delta}{r_{0}}(1+c^{2})\right)\partial_{s}\phi_{0}=\frac{1}{c}\partial_{s}\phi_{1}+O\left(\frac{\delta^{2}}{r_{0}^{4}}\right)
\end{align*}
To solve for $\phi_{0}$, for $s\in[0,1)$, we consider the following family (parametrized by a compact set of parameters $\theta\in\mathbb{S}^{2}$ and the parameters $\delta$ and $\dfrac{1}{r_{0}}$) of linear ordinary differential equations

\begin{align}\label{ODE phi0 1}
\begin{split}
&\partial_{s}^{2}\phi_{0}+\left(\frac{\delta}{r_{0}}+\frac{3}{2c}G''(0)\phi_{1}\frac{\delta}{r_{0}}(1+c^{2})\right)\partial_{s}\phi_{0}-c^{-1}\partial_{s}\phi_{1}=\frac{\delta^{2}}{r_{0}^{4}}\phi_{2},\\
&\phi_{0}(r_{0},\theta)=\partial_{s}\phi_{0}(r_{0},\theta)=0.
\end{split}
\end{align}
Since the $C^k$-norms of the solution depends smoothly on the coefficients and the parameters $\theta, \delta$ and $\frac{1}{r_{0}}$, all $C^k$-norms of $\phi_0$ are of order $O(1)$ and indeed are determined by the solution of 
\begin{equation}\label{ODE phi0 2}
\begin{split}
&\ \ \  \partial_s^2 \phi_0  -c^{-1}\partial_{s}\phi_1  = 0, \\
&\   \phi_0(r_{0},\theta)=0, \ \partial_s\phi_0(r_{0},\theta) = 0.
\end{split}
\end{equation}
In particular, this shows that the ansatz \eqref{ansatz on phi_0} holds if we choose $C$ appropriately large in \eqref{ansatz on phi_0} and $\delta$ sufficiently small. Note that although the ODE \eqref{ODE phi0 1} depends on $\delta$ and $r_{0}^{-1}$ explicitly, if a priorily we choose $\delta$ and $r_{0}^{-1}$ to be less than $\frac{1}{2}$, the $C^{k}$-norms of $\phi_{0}$ in \eqref{ODE phi0 1} depends on $\phi_{1}$ through \eqref{ODE phi0 2} and some other absolute constants, therefore indepedent of the final choice of $\delta$ and $r_{0}^{-1}$. So the above construction shows that 
$$|\Lb^2 \phi| \lesssim \frac{\delta^\frac{3}{2}}{r_{0}^{3}}.$$
\smallskip
We claim that, by the above choice of initial data, on $\Sigma_{-r_{0}}^\delta$, we automatically have
$$|\Lb \phi| \lesssim \frac{\delta^\frac{3}{2}}{r_{0}^{2}}.$$
\smallskip
Indeed, by replacing $\partial_t = \Lb +c\partial_r$ in the main equation, we obtain
\begin{equation*}
 \partial_r \Lb\phi = \frac{1}{2c}\Big(-\Lb^2\phi-Lc \, \partial_r\phi+\frac{2c^2}{r}\partial_r\phi+\frac{c^2}{r^{2}}\laplacianslash_{\mathbb{S}^{2}} \phi\Big).
\end{equation*}
By the construction of the data, it is obvious that all the terms on the right hand side are of size $O\left(\frac{\delta^\frac{1}{2}}{r_{0}^{2}}\right)$. By integrating from $r_{0}$ to $r$ with $r\in [r_{0},r_{0}+\delta)$ and $\Lb \phi(r_{0},\theta)=0$, we have 
$$|\Lb\phi(r,\theta)|\leq \delta \cdot O\left(\frac{\delta^\frac{1}{2}}{r_{0}^{2}}\right)\lesssim \frac{\delta^\frac{3}{2}}{r_{0}^{2}}.$$
\end{proof}

\begin{definition}
 The Cauchy initial data of $(\star)$ constructed in the lemma (satisfying the two properties) are called no-outgoing-radiation short pulse data.
\end{definition}

The main theorem of the paper is as follows:
\begin{MainTheorem}
For a given constant $G''(0)\neq 0$, we consider
\begin{equation*}
 - \big(1+3G^{\prime\prime}(0) (\partial_t\phi)^2\big)\partial^2_t \phi +\Delta\phi=0.
\end{equation*}
Let $(\phi_1,\phi_2)$ be a pair of seed data and the initial data for the equation is taken to be the no-outgoing-radiation initial data.

If the following condition on $\phi_{1}$ holds for at least one $(s,\theta) \in (0,1] \times \mathbb{S}^2$:
\begin{equation}\label{shock condition}
G''(0)\cdot\partial_s \phi_1(s,\theta) \cdot \phi_1(s,\theta) \leq -\frac{1}{3},
\end{equation}
then there exists a constant $\delta_0$ which depends only on the seed data $(\phi_1,\phi_2)$, so that for all $\delta< \delta_0$, shocks form for the corresponding solution $\phi$ before $t=-1$, i.e. $\phi$ will no longer be smooth.
\end{MainTheorem}

\begin{remark}
The choice of $\phi_1$ in the proof of Lemma \ref{lemma constraint} is arbitrary. In particular, this is \textbf{consistent} with the condition \eqref{shock condition} since $\phi_1$ can be freely prescribed.
\end{remark}

\begin{remark}

\begin{enumerate}
\item We do \emph{not} assume spherical symmetry on the initial data. Therefore, the theorem is in nature a higher dimensional result. 

\item The proof can be applied to a large family of equations derived through action principle. We will discuss this point when we consider the Lagrangian formulation of \eqref{Main Equation}.

\item The condition \eqref{shock condition} is \emph{only} needed to create shocks. It is not necessary at all for the a priori energy estimates.
\end{enumerate}
\end{remark}

\begin{remark}
 The smoothness of $\phi$ breaks down in the following sense:
\begin{itemize}
 \item[1)]\ The solution and its first derivative, i.e. $\phi$ and $\partial \phi$, are always bounded. Moreover, $|\partial_t\phi| \lesssim \delta^{\frac{1}{2}}|t|^{-1}$, therefore \eqref{Main Equation} is always of wave type.

 \item[2)]\ The second derivative of the solution blows up. In fact, when one approaches the shocks, $|\partial_t \partial_r \phi|$ blows up.
\end{itemize}
\end{remark}
\subsection{A brief discussion on physical motivation}
To actually create electromagnetic shocks in laboratory, one would have to focus sufficiently strong electromagnetic wave pulses into a suitable nonlinear medium. As a preliminary step for the model equation \eqref{Main Equation}, in \cite{M-Y} we identified a class of large initial data, which can be thought as a strongly focused wave pulse, on the initial hypersurface $\{(t,x)|t=-2\}$, and proved shock formation before the time slice $\{(t,x)|t=-1\}$. This means that pulse has to be strong enough such that shock can form within a time period approximately $1$. However, in reality due to the limitation of lasers, one can only focus weak wave pulse in experiments, namely, the initial electric field has to be small. For our model equation, $\partial_{t}\phi$ has to be small initially. Mathematically, this can be achieved by choosing $\delta$ sufficiently small. However, physically, if $\delta$, which measures the width the initial pulse, is too small, then wave pulse would have high frequency and the dispersion effect would dominate. In this case, the system \eqref{Maxwell 1}-\eqref{Maxwell 2} would be no longer accurate to model the physical situation. (See \cite{LL}.) Therefore the question is whether we can make the initial pulse ($\partial_{t}\phi|_{t=-r_{0}}$) small without increasing its frequency, which is measured by $\delta^{-1}$. Thanks to the profile \eqref{data profile} of the short pulse data, by choosing $r_{0}$ sufficiently large, this problem is addressed.
\subsection{Historical works}
The study of singularity formation for quasilinear wave equations dates back to the work \cite{John1}, in which he obtained upper bounds for the lifespan of the rotationally symmetric solutions to the equation $-\partial^{2}_{t}\phi+\Delta\phi=\partial_{t}\phi\partial^{2}_{t}\phi$ (see also the survey article \cite{John2} and references therein). Later in \cite{Alin1} and \cite{Alin2}, Alinhac removed the symmetric assumption and showed the solution blows up in $2d$. Moreover, he gave a precise description of the solution near the blow-up point. \vspace{2mm}

A major breakthrough in understanding the shock formations in higher dimensional space is made by Christodoulou in his monograph \cite{Ch1}. He considers the relativistic $3d$ Euler equations for a perfect irrotational fluid with an arbitrary equation of state. Given the initial data being a small perturbation from the constant state, he obtained a complete picture of shock formation. A similar result for classical Euler equations is obtained in \cite{Ch-M}. The approaches are based on differential geometric methods originally introduced by Christodoulou and Klainerman in their monumental proof \cite{Ch-K} of the nonlinear stability of the Minkowski spacetime in general relativity. More recently, based on similar ideas, Holzegel, Klainerman, Speck and Wong have obtained remarkable results in understanding the stable mechanism for shock formations for certain types of quasilinear wave equations with small data in three dimensions, see their overview paper \cite{HKSW} and Speck's detailed proof \cite{Sp}. In \cite{Ch1} and \cite{Ch-M} the authors obtained sharp lower and upper bounds for the lifespan of smooth solutions associated to the given data \emph{without} any symmetry conditions. Prior to \cite{Ch1} \cite{Ch-M}, most of works on shock waves in fluid are limited to the simplified case of with some symmetric assumptions, i.e. essentially the one space dimension case. As an example, we note the work \cite{Alin3} by Alihnac in which the singularity formation for the compressible Euler equations on $\mathbb{R}^2$ with rotational symmetry is studied. Let us also note a recent work \cite{Ch-P} in which a complete description of shock formation for genuinely nonlinear one-dimensional hyperbolic system is given. \cite{Ch-P} generalizes an influential result \cite{John3} which states that no genuinely nonlinear strictly hyperbolic quasilinear first order system in one dimensional space has a global smooth solution for small enough initial data.\vspace{2mm}

All the aforementioned  works have the common feature that the initial data are small. However, as we have explained at the beginning of the paper, we need to use a special family of large data--the so called \emph{short pulse data}--to create shock for \eqref{Main Equation}. The short pulse data was first introduced by Christodoulou in a milestone work \cite{Ch2} in understanding the formation of black holes in general relativity. By identifying an open set of initial data without any symmetry assumptions (the short pulse ansatz), he showed that a trapped surface can form, even in vacuum spacetime, from completely dispersed initial configurations and by means of the focusing effect of gravitational waves. Although the data are no longer close to Minkowski data, in other words, the data are no longer small, he is still able to prove a long time existence result for these data. This establishes the first result on the long time dynamics in general relativity and paves the way for many new developments on dynamical problems related to black holes. Shortly after Christodoulou's work, Klainerman and Rodnianski extends and significantly simplifies Christodoulou's work, see \cite{K-R}. From a pure PDE perspective, the data appeared in the above works are carefully chosen large profiles which can be preserved by the Einstein equations along the evolution. 
Later in \cite{M-Y} the short pulse data was applied to prove shock formation for \eqref{Main Equation}. There based on an energy estimate, we showed that the hierarchy with respect to the small parameter $\delta$ of the short pulse data is propagated until the shock formation. Since in \cite{M-Y} the energy estimate is proved for a finite time interval $t\in[-2,-1]$, one only needs to track the behavior of solution with respect to $\delta$. However, in the present work, since the time behavior is crucial of interest, one needs to propagate both the $\delta$ and $t$ hierarchy in the energy estimate. 
Let us conclude this subsection by noting that recently, Speck, Holzegel, Luk and Wong studied shock formation for a class of $2d$ quasilinear wave equation with initial data being a small perturbation of a plane symmetric solution (see \cite{SHLW}, and also its recent generalization \cite{LS} to compressible Euler equation with non-zero vorticity.). Since the size of plane symmetric solution could be large, the derivative of the solution along one certain direction is allowed to be large. On the other hand, due to the smallness of the perturbation, the derivatives along other directions are small, which allows the authors to close the energy estimate. Since plane waves do not disperse, there are no dispersive estimates in \cite{SHLW}.

\subsection{New features of the proof}
Since the equation \eqref{Main Equation} is invariant under the space and time translations, we consider its linearized equation which is a linear wave equation with respect to a Lorentzian metric defined by the solution. Like in \cite{M-Y}, the proof of the energy estimates is based on the study of geometry for the incoming null hypersurfaces with respect to this Lorentizain metric. The shock formation is equivalent to the collapse of the foliation by these null hypersurfaces. In the energy estimates, we need to use descent scheme to eliminate the singularity raised by shock formation. (See more details in the introduction of \cite{M-Y}.) Here we will discuss some new features related to the dispersive estimates but not appearing in \cite{M-Y}.

\begin{enumerate}

\item A modified multiplier vectorfield. In \cite{M-Y} since we are interested in the behavior of solution 
in the time interval $t\in[-2,-1]$, no dispersive estimate is needed. Therefore to prove the energy estimate, 
it suffice to prove that the $L^{2}$-norm $\|(\partial\phi)(t,\cdot)\|_{L^{2}(\mathbb{R}^{3})}$ is bounded by initial data. 
To prove such an energy estimate we only need to use the standard multiplier $\partial_{t}$. (In reality, we use the analog of null vectorfields $\Lb:=\partial_{t}-\partial_{r}, L:=\partial_{t}+\partial_{r}$ in Minkowski spacetime.)  While in the present work, since we solve the equation from $t=-r_{0}$, where $r_{0}$ can be arbitrarily large, to $t=-1$, pointwise decay estimate is needed. To this end, we prove that the weighted $L^{2}$-norm $\|t(\partial\phi)(t,\cdot)\|_{L^{2}(\mathbb{R}^{3})}$ is bounded by initial data. For wave equations in Minkowski spacetime, the standard way to obtain such estimates is to use the conformal Killing vectorfield $K_{0}:=\ub^{2}L+u^{2}\Lb$ as the multiplier in proving energy estimate. Here $u:=\dfrac{1}{2}(t-r), \ub:=\dfrac{1}{2}(t+r)$. In this paper we also choose such a vectorfield as the multiplier and of course the functions $u, \ub$ and the vectorfields $L, \Lb$ are associated to the Lorentzian metric defined by the solution. In reality, since we are interested in the solution defined in a region corresponding to a small range of $\ub$, instead of $\ub^{2}L+u^{2}\Lb$, we use an analogy of $u^{2}\Lb$ as the multiplier. (See Remark \ref{remark: K1} for more discussion)\vspace{2mm}

\item The scattering data. In the proof we will show that all the constants appearing in the estimates do not depend on $r_{0}$. In other words, as $r_{0}\rightarrow\infty$, 
the energy estimates still hold if the initial energies are finite as $r_{0}\rightarrow\infty$. We will prove that the limits of energies as $r_{0}\rightarrow\infty$ indeed exist. 
(See Proposition \ref{prop: scattering}) The \emph{scattering data} is in this sense. The energy estimates imply the semi-global existence of the smooth solution from $t=-\infty$ 
all the way up to shock formation. On the other hand, as we have stated, since $r_{0}$ can be made as large as we wish, the initial intensity of the pulse, $
\left|\partial_{t}\phi\cdot\partial_{r}\partial_{t}\phi\right|(-r_{0}, \theta)$, can be made as small as possible without shrinking the size of $\delta$, which fits the reality.

 \item Resolve of a logarithmic divergence in estimating geometric quantities. To obtain the $L^{2}$-estimates for higher order derivatives, one needs to commute 
 certain vectorfields with the linearized equation. Since the linearized wave equation is with respect to the Lorentzian metric defined by solution, 
 the commutator vectorfields are usually no longer Killing or conformal Killing. So we need to estimate the nonlinear contributions from the deformation tensors of commutators. 
 In particular, the highest order derivatives of these deformation tensors are the most difficult ones to estimate. Not only that we need to modify the propagation equations 
 satisfied by them to avoid loss of regularity (This modification also appears in \cite{Ch1}, \cite{Ch-M}, \cite{HKSW}, \cite{Sp}, \cite{M-Y} and \cite{SHLW}.), 
 but also the top order derivatives of the deformation tensors result in a logarithmic divergence:
 
 \begin{align}\label{log divergence}
E(t)\lesssim E(-r_{0})+\int_{-r_{0}}^{t}(-t')^{-2}\int_{-r_{0}}^{t'}(-t'')^{-1}E(t'')dt''dt'
 \end{align}
 Here $E$ is the $L^{2}$-norm of highest order derivatives for $\phi$.  Since $-t\leq -t'\leq -t''\leq r_{0}$, the estimate for $E(t)$ would grow in $r_{0}$. In fact, let us denote by $X(t)$ the $L^{2}$-norm of the top order derivative of the defomration tensor. Since the contribution of $X(-r_{0})$ is lower order, we omit it in this rough outline. Systematically, we have
 
 \begin{align}\label{log div 1}
  X'(t)\leq (-t)^{-1}E(t)\quad \Rightarrow\quad X(t)\leq X(-r_{0})+\int_{-r_{0}}^{t}(-t')^{-1}E(t')dt'.
 \end{align}
On the other hand, $E(t)$ depends on $X(t)$ (roughly) through the following inequality:

\begin{align}\label{log div 2}
 E(t)\leq E(-r_{0})+\int_{-r_{0}}^{t}(-t')^{-2}X(t')dt'.
\end{align}
Combining \eqref{log div 1} and \eqref{log div 2}, we obtain \eqref{log divergence}. To avoid this logarithmic divergence, we use a modified energy $\widetilde{X}(t)$ for the top order derivative of deformation tensor, such that $\widetilde{X}(t)\sim(-t)^{-1}X(t)$ and $\widetilde{X}'(t)\sim(-t)^{-1}X'(t)$. This modification can be achieved by modifying a commutator vectorfield. Therefore the second inequality in \eqref{log div 1} becomes

\begin{align}\label{log div 3}
 \widetilde{X}(t)\leq \widetilde{X}(-r_{0})+\int_{-r_{0}}^{t}(-t')^{-2}E(t')dt',
\end{align}
and the inequality in \eqref{log div 2} becomes

\begin{align}\label{log div 4}
  E(t)\leq E(-r_{0})+\int_{-r_{0}}^{t}(-t')^{-1}\widetilde{X}(t')dt'.
\end{align}
Combing \eqref{log div 3} and \eqref{log div 4}, we obtain

\begin{align}\label{log resolved}
 E(t)\lesssim E(-r_{0})+\int_{-r_{0}}^{t}(-t')^{-1}\int_{-r_{0}}^{t'}(-t'')^{-2}E(t'')dt''dt',
\end{align}
and the estimate for $E(t)$ will no longer depend on $r_{0}$. The rigorous and detailed derivation of this argument is given in Section 7.
\end{enumerate}

\section{The optical geometry}
The construction of optical geometry is similar to that in \cite{M-Y}. To let the present paper be self-contained, here we repeat some necessary discussions in \cite{M-Y}.
\subsection{The optical metric and linearized equation}
We observe that main equation \eqref{Main Equation} is invariant under space translations, rotations and the time translation.
This can also be seen from the invariance of the Lagrangian $L(\phi)$ under these symmetries. We use $A$ to denote any possible choice from $\{\partial_t, \partial_i, \Omega_{ij}= x^i \partial_j-x^j \partial_i\}$ where $i,j=1,2,3$ and $i<j$. These vectorfields correspond to the infinitesimal generators of the symmetries of \eqref{Main Equation}.

To linearize \eqref{Main Equation}, we apply the symmetry generated by $A$ to a solution $\phi$ of \eqref{Main Equation} to obtain a family of solutions $\{\phi_{\tau}:\tau\in \mathbb{R}\big| \phi_0 = \phi\}$. Therefore, $ -\frac{1}{c(\phi_\tau)^2} \partial^2_t \phi_\tau +\Delta\phi_\tau=0$ for $\tau \in \mathbb{R}$. We then differentiate in $\tau$ and evaluate at $\tau=0$.  We define the so called \emph{variations} $\psi$ as
\begin{equation}\label{definition for variation}
 \psi:= A\phi = \frac{d\phi_{\tau}}{d\tau}|_{\tau=0}
\end{equation}
By regarding $\phi$ as a fixed function, this procedure produces a linear equation for $\psi$, which is \emph{the linearized equation of \eqref{Main Equation} for the solution $\phi$ with respect to the symmetry $A$}.

In the tangent space at each point in $\mathbb{R}^{3+1}$ where the solution $\phi$ is defined, we introduce the following Lorentzian metric $g_{\mu\nu}$
\begin{equation}\label{optical metric}
g = - c^2 dt \otimes d t +d x^1 \otimes d x^1 + d x^2 \otimes d x^2 + d x^3 \otimes d x^3,
\end{equation}
with $(t,x^{1},x^{2},x^{3})$ being the standard rectangular coordinates in Minkowski spacetime. Since $c$ depends on the solution $\phi$, so does $g_{\mu\nu}$. We also introduce a conformal metric $\gt_{\mu\nu}$ with the conformal factor $\Omega=\dfrac{1}{c}$
\begin{align}\label{conformal metric}
 \gt_{\mu\nu}=\Omega \cdot g_{\mu\nu} = \frac{1}{c} g_{\mu\nu}.
\end{align}
We refer $g_{\mu\nu}$ and $\gt_{\mu\nu}$ as the \textit{optical metric} and the \textit{conformal optical metric} respectively.

A direct computation shows
\begin{lemma}
The linearized equation of \eqref{Main Equation} for a solution $\phi$ with respect to $A$ can be written as
\begin{equation}\label{linearization conformal}
 \Box_{\tilde{g}}\psi=0,
\end{equation}
where $\Box_{\gt}$ is the wave operator with respect to $\gt$ and $\psi = A \phi$.
\end{lemma}
This lemma can also be proved using a more natural way which is standard in Lagrangian field theory, e.g. see \cite{Ch3}, by showing the linearized Lagrangian density of \eqref{Lagrangian density} is exactly the Lagrangian density for the linear wave equation associated to $\widetilde{g}$.
\subsection{Lorentzian geometry of the maximal development}
\subsubsection{The maximal development}
We define a function $\ub$ on $\Sigma_{-r_{0}}$ as follows:
\begin{align}\label{ub initial}
\ub:=r-r_{0}.
\end{align}
The level sets of $\ub$ in $\Sigma_{-r_{0}}$ are denoted by $S_{-r_{0},\ub}$ and they are round spheres of radii $\ub+r_{0}$.  The annular region $\Sigma_{-r_{0}}^\delta$ defined in \eqref{delta annulus} is foliated by $S_{-r_{0},\ub}$ as
\begin{align}\label{foliation initial}
\Sigma_{-r_{0}}^{\delta}:=\bigcup_{\ub\in[0,\delta]}S_{-r_{0},\ub}.
\end{align}

Given an initial data set $(\phi,\partial_t \phi)\big |_{t=-r_{0}}$ defined on $B_{-r_{0}}^{r_{0}+\delta}:=\bigcup_{\ub\in[-r_{0},\delta]}S_{-r_{0},\ub}$ to the main equation \eqref{Main Equation} (as we stated in the Main Theorem), we recall the notion of \emph{the maximal development} or maximal solution with respect to the given data. 

\smallskip

By virtue of the local existence theorem (to \eqref{Main Equation} with smooth data), one claims the existence of a \emph{development} of the given initial data set, namely, the existence of 
\begin{itemize}
\item a domain $\mathcal{D}$ in Minkowski spacetime, whose past boundary is $B_{-r_{0}}^{r_{0}+\delta}$;
\item a smooth solution $\phi$ to \eqref{Main Equation} defined on $\mathcal{D}$ with the given data on $B_{-r_{0}}^{r_{0}+\delta}$ with following property: For any point $p\in\mathcal{D}$, if an inextendible curve $\gamma: [0,\tau) \rightarrow \mathcal{D}$ satisfies the property that 

(1) $\gamma(0)=p$,

(2) For any $\tau' \in [0,\tau)$, the tangent vector $\gamma'(\tau')$ is past-pointed and causal (i.e., $g(\gamma'(\tau'),\gamma'(\tau'))\leq 0$) with respect to the optical metric $g_{\alpha\beta}$ at the point $\gamma(\tau')$,

\noindent then the curve $\gamma$ 
must terminate at a point of  $B_{-r_{0}}^{r_{0}+\delta}$.
\end{itemize}

By the standard terminology of Lorentzian geometry, the above simply says that $B_{-r_{0}}^{r_{0}+\delta}$ is a Cauchy hypersurface of $\mathcal{D}$.

\smallskip

The local uniqueness theorem asserts that if $(\mathcal{D}_{1},\phi_{1})$ and $(\mathcal{D}_{2},\phi_{2})$ are two developments of the same initial data sets, then $\phi_{1}=\phi_{2}$ in $\mathcal{D}_{1}\bigcap\mathcal{D}_{2}$. Therefore the union of all developments of a given initial data set is itself a development. This is the so called \emph{maximal development} and its corresponding domain is denoted by $W^*$. The corresponding solution is called the \emph{maximal solution}. Sometimes we also identify the development as its corresponding domain when there is no confusion.
\subsubsection{Geometric set-up}
Given an initial data set, we consider a specific family of incoming null hypersurfaces (with respect to the optical metric $g$) in the maximal development $W^*$. Recall that $\ub$ is defined on $\Sigma_{-r_{0}}$ as $r-r_{0}$. For any $\ub \in [0,\delta]$, we use $\Cb_{\ub}$ to denote the incoming null hypersurface  emanated from the sphere  $S_{-r_{0},\ub}$. By definition, we have $\Cb_{\ub} \subset W^*$ and $\Cb_{\ub}\bigcap\Sigma_{-r_{0}}=S_{-r_{0},\ub}$. 

Let $W_{\delta}$ be the subset of the maximal development of the given initial data foliated by $\Cb_{\ub}$ with $\ub \in [0,\delta]$, i.e.,
\begin{equation}\label{W delta}
W_{\delta}: =\bigcup_{\ub\in[0,\delta]}\Cb_{\ub}.
\end{equation}
Roughly speaking, our main estimates will be carried out only on $W_{\delta}$. The reason is as follows: since we assume that the data set is completely trivial for $\ub\leq 0$ on $\Sigma_{-r_{0}}$, the uniqueness of smooth solutions for quasilinear wave equations implies that the spacetime in the interior of $\Cb_{0}$ 
is indeed determined by the trivial solution. In particular, $\Cb_{0}$ is a flat cone in Minkowski spacetime (with respect to the Minkowski metric).

We extend the function $\ub$ to $W_{\delta}$ by requiring that the hypersurfaces $\Cb_{\ub}$ are precisely the level sets of the function $\ub$.  Since $\Cb_{\ub}$ is null with respect to $g_{\alpha\beta}$, the function $\ub$ is then a solution to the equation
\begin{align}\label{eikonal eq}
(g^{-1})^{\alpha\beta}\partial_{\alpha}\ub\partial_{\beta}\ub=0,
\end{align}
where $(g^{-1})^{\alpha\beta}$ is the inverse of the metric $g_{\alpha\beta}$. 
We call such a function $\ub$ an \emph{optical function}. 

With respect to the affine parameter, the future-directed tangent vectorfield of a null geodesic on $\Cb_{\ub}$ is given by
\begin{align}\label{Lb affine}
\widehat{\underline{L}}:=-(g^{-1})^{\alpha\beta}\partial_{\alpha}\ub\,\partial_{\beta}.
\end{align}
However, for an apparent reason, which will be seen later, instead of using $\widehat{\underline{L}}$, we will work with a renormalized (by the time function $t$) vectorfield $\Lb$ defined through
\begin{align}\label{Lb}
\Lb=\mu\widehat{L},\quad \Lb t=1,
\end{align}
i.e., $\Lb$ is the tangent vectorfield of null geodesics parametrized by $t$. 

The function $\mu$ can be computed as
\begin{align*}
\frac{1}{\mu}=-(g^{-1})^{\alpha\beta}\partial_{\alpha}\ub\partial_{\beta}t.
\end{align*}
We will see later on that the $\mu$ also has a very important geometric meaning: ${\mu}^{-1}$ is the density of the foliation $\bigcup_{\ub\in[0,\delta]}\Cb_{\ub}$.

\medskip

Given $\ub \leq \delta$, to consider the density of null-hypersurface-foliation on $\Sigma_t \bigcap W_\delta$, we define
\begin{align}\label{definition of mu m}
\mu^{\ub}_{m}(t)=\min\left(\inf_{(\ub',\theta)\in[0,\ub]\times \mathbb{S}^{2}}\mu(t,\ub',\theta), 1\right).
\end{align}
For $\ub=\delta$, we  define 
$$s_{*} = \sup \big\{t\big | t\geq -r_{0} \ \text{and} \ \mu_m^\delta(t)>0 \big\}.$$ 
From the PDE perspective, for the given initial data to \eqref{Main Equation} (as constructed in Lemma \ref{lemma constraint}), we also define 

\begin{align*}
t_{*} =  \sup \big\{&\tau \big | \tau \geq -r_{0} \ \text{such that the smooth solution exists for all} \ (t,\ub)\in[-r_{0},\tau)\times[0,\delta] \ \text{and} \ \theta\in\mathbb{S}^{2}\big\}.
\end{align*}

Finally, we define
\begin{align}\label{geodesic upper bound}
s^{*}=\min\{s_{*},-1\},\ \ t^{*}=\min\{t_{*},s^{*}\}.
\end{align}
We remark that we will exhibit data in such a way that the solution breaks down before $t=-1$. This is the reason we take $-1$ in the definition of $s^*$.

\smallskip

In the sequel, we will work in a further confined spacetime domain $W^{*}_{\delta}\subset W_\delta \subset W^*$ to prove a priori energy estimates. By definition, it consists of all the points in $W_{\delta}$ with time coordinate $t\leq t^{*}$, i.e.,
$$W^*_\delta =W_\delta \bigcap\Big( \bigcup_{-r_{0}\leq t \leq t^*} \!\!\!\!\Sigma_t\Big).$$

\smallskip

For the purpose of future use, we introduce more notations to describe various geometric objects.

\smallskip

For each $(t,\ub)\in[-r_{0},t^{*})\times[0,\delta]$, we use $S_{t,\ub}$ to denote the closed two dimensional surface
\begin{align}\label{S t ub}
S_{t,\ub}:=\Sigma_{t}\bigcap\Cb_{\ub}.
\end{align}
In particular, we have
\begin{align}\label{W star}
W^{*}_{\delta}=\bigcup_{(t,\ub)\in[-r_{0},t^{*})\times[0,\delta]}S_{t,\ub}.
\end{align}
For each $(t,\ub) \in [-r_{0},t^*)\times [0,\delta]$, we define
\begin{align}\label{various notations}
\begin{split}
\Sigma_{t}^{\ub}:=&\big\{(t,\ub',\theta) \in \Sigma_t \mid 0\leq\ub' \leq \ub \big\},\\ \Cb_{\ub}^{t}:=&\big\{(\tau,\ub,\theta) \in \Cb_{\ub} \mid  -r_{0}\leq\tau \leq t\big\},\\
 W^{t}_{\ub}:=&\bigcup_{(t',\ub')\in[-r_{0},t)\times [0,\ub]}S_{t',\ub'}.
 \end{split}
\end{align}

In what follows when working in $W^{t}_{\ub}$, we usually omit the superscript $\ub$ to write $\mu^{\ub}_{m}(t)$ as $\mu_{m}(t)$, whenever there is no confusion. 

\smallskip

We define the vectorfield $T$ in $W^{*}_{\delta}$ by the following three conditions:
\begin{itemize}
\item[1)] $T$ is tangential to $\Sigma_{t}$;
\item[2)] $T$ is orthogonal (with respect to $g$) to $S_{t,\ub}$ for each $\ub\in[0,\delta]$;
\item[3)] $T\ub=1$.
\end{itemize}
The letter $T$ stands for ``transversal" since the vectorfield is transversal to the foliation of null hypersurfaces $\Cb_{\ub}$.

In particular, the point $1)$ implies
\begin{align}\label{T t}
Tt=0.
\end{align}
According to \eqref{eikonal eq}-\eqref{Lb}, we have
\begin{align}\label{Lb ub}
\Lb\ub=0,\quad \Lb t=1.
\end{align}
In view of \eqref{Lb}, \eqref{Lb ub}, \eqref{T t} and the fact $T\ub=1$, we see that the commutator 
\begin{align}\label{commutator Lb T}
\Lambda:=[\Lb,T]
\end{align}
is tangential to $S_{t,\ub}$.

In view of \eqref{eikonal eq}-\eqref{Lb} and the fact $T\ub=1$ we have
\begin{align}\label{Lb T}
g(\Lb, T)=-\mu, \quad g(\Lb,\Lb)=0.
\end{align}

Since $T$ is spacelike with respect to $g$ (indeed, $\Sigma_t$ is spacelike and $T$ is tangential to $\Sigma_t$), we denote
\begin{align}\label{T magnitude}
g(T,T)=\kappa^{2},\quad \kappa>0.
\end{align}

\begin{lemma}\label{lemma mu kappa}
As in \cite{M-Y}, we have the following relations for $\Lb$, $T$, $\mu$ and $\kappa$:
\begin{align}\label{mu kappa}
\mu=c\kappa,\ \ \Lb=\partial_{0}-c\kappa^{-1}T,
\end{align}
where $\partial_0$ is the standard time vectorfield in Minkowski spacetime.
\end{lemma}

\begin{proof}
The proof is by direct computations, which can be found in \cite{M-Y}. We omit it here.
\end{proof}

\begin{remark}
On the initial Cauchy surface $\Sigma_{-r_{0}}$, since $\ub=r-r_{0}$, we have $T=\partial_r$ and $\kappa=1$. Therefore, by using the standard rectangular coordinates, we obtain that 
$$\Lb = \partial_t -c\partial_r.$$
This is coherent with the notations and computations in  Lemma \ref{lemma constraint}.
\end{remark}
\subsubsection{The optical coordinates}

We construct a new coordinate system on $W_{\delta}^*$. If shocks form, the new coordinate system is completely different from the standard rectangular coordinates. Indeed, we will show that they define two differentiable structures on $W_\delta^*$ when shocks form.

\smallskip

Given $\ub\in[0,\delta]$, the generators of $\Cb_{\ub}$ define a diffeomorphism between $S_{-r_{0},\ub}$ and $S_{t,\ub}$ for each $t\in[-r_{0},t^{*})$. Since $S_{-r_{0},\ub}$ is diffeomorphic to the standard sphere $\mathbb{S}^{2}\subset\mathbb{R}^{3}$ in a natural way. We obtain a natural diffeomorphism between $S_{t,\ub}$ and $\mathbb{S}^{2}$. If local coordinates $(\theta^{1},\theta^{2})$ are chosen on $\mathbb{S}^{2}$,  then the diffeomorphism induces local coordinates on $S_{t,\ub}$ for every $(t,\ub)\in[-r_{0},t^{*})\times[0,\delta]$. The local coordinates $(\theta^{1},\theta^{2})$, together with the functions $(t,\ub)$ define a complete system of local coordinates $(t,\ub,\theta^{1},\theta^{2})$ for $W^{*}_{\delta}$. This new coordinates are defined as \emph{the optical coordinates}. 

\smallskip
We now express for $\Lb$, $T$ and the optical metric $g$ in the optical coordinates. 

\smallskip
First of all, the integral curves of $\Lb$ are the lines with constant $\ub$ and $\theta$. Since $\Lb t=1$, therefore in optical coordinates we have
\begin{align}\label{Lb op}
\Lb=\frac{\partial}{\partial t}.
\end{align}
Similarly, since $T\ub=1$ and $T$ is tangential to $\Sigma_{t}$, we have
\begin{align}\label{T op}
T=\frac{\partial}{\partial\ub}-\Xi
\end{align}
with $\Xi$ a vectorfield tangential to $S_{t,\ub}$. Locally, we can express $\Xi$ as
\begin{align}\label{Xi}
\Xi=\sum_{A=1,2}\Xi^{A}\frac{\partial}{\partial\theta^{A}},
\end{align}
The metric $g$ then can be written in the optical coordinates $(t,\ub,\theta^{1},\theta^{2})$ as
\begin{align}\label{metric op}
g=-2\mu dtd\ub+\kappa^{2}d\ub^{2}+\slashed{g}_{AB}\left(d\theta^{A}+\Xi^{A}
d\ub\right)\left(d\theta^{B}+\Xi^{B}d\ub\right)
\end{align}
with
\begin{align}\label{gslash}
\slashed{g}_{AB}=g\left(\frac{\partial}{\partial\theta^{A}},\frac{\partial}{\partial\theta^{B}}\right), \ 1\leq A,B\leq 2.
\end{align}

\medskip

To study the differentiable structure defined by the optical coordinates, we study
the Jacobian $\triangle$ of the transformation from the optical coordinates $(t,\ub,\theta^{1},\theta^{2})$ to the rectangular coordinates $(x^0,x^{1},x^{2},x^{3})$. 

First of all, since $x^0=t$, we have 
\begin{align*}
\frac{\partial x^{0}}{\partial t}=&1,\quad \frac{\partial x^{0}}{\partial\ub}=\frac{\partial x^{0}}{\partial\theta^{A}}=0.
\end{align*}
Secondly, by \eqref{T op}, we can express $T =  T^i \partial_{i}$ in the rectangular coordinates $(x^1,x^2,x^3)$ as
\begin{align*}
T^{i}=\frac{\partial x^{i}}{\partial\ub}-\sum_{A=1,2}\Xi^{A}\frac{\partial x^{i}}{\partial\theta^{A}}
\end{align*}
In view of the fact that $T$ is orthogonal to $\frac{\partial}{\partial\theta^{A}}$ with respect to the Euclidean metric (which is the induced metric of $g$ on $\Sigma_{t}$!), we have

\begin{align*}
\triangle=\det\left(
\begin{array}{ccc}
T^{1}&T^{2}&T^{3}\\
\dfrac{\partial x^{1}}{\partial\theta^{1}}&\dfrac{\partial x^{2}}{\partial\theta^{1}}&\dfrac{\partial x^{3}}{\partial\theta^{1}}\\
\dfrac{\partial x^{1}}{\partial\theta^{2}}&\dfrac{\partial x^{2}}{\partial\theta^{2}}&\dfrac{\partial x^{3}}{\partial\theta^{2}}
\end{array}
\right)=\|T\|\left\|\dfrac{\partial}{\partial\theta^{1}}\wedge\dfrac{\partial}{\partial\theta^{2}}\right\|=c^{-1}\mu\sqrt{\det\slashed{g}},
\end{align*}
where $\|\cdot\|$ measures the magnitude of a vectorfield with respect to the Euclidean metric in $\mathbb{R}^{3}$ (defined by the rectangular coordinates $(x^1,x^2,x^3)$). 

\smallskip

We end the discussion by an important remark. 

\begin{remark}[{\bf Geometric meaning of $\mu$}]\label{geometric meaning of mu}
In the sequel, we will show that the wave speed function $c$ will be always approximately equal to $1$ in $W_\delta^*$. Since $\mu = c\kappa$, we may think of $\mu$ being $\kappa$ in a efficient way.

On the other hand, by the definition of $T$, in particular $T\ub=1$,  we know that $\kappa^{-1}$ is indeed the density of the foliation by the $\Cb_{\ub}$'s. This is because $g(T,T)=\kappa^2$. Since the optical metric coincides with the Euclidean metric on each constant time slice $\Sigma_t$, by $\mu \sim \kappa$, we arrive at the following conclusion:
\begin{itemize}
\item  $\mu^{-1}$ measures the foliation of the incoming null hypersurfaces $\Cb_{\ub}$'s.
\end{itemize}
Therefore, by regarding shock formation as the collapsing (i.e. the density blows up) of the characteristics ($\simeq$ the incoming null hypersurfaces), we may say that
\begin{itemize}
\item Shock formation is equivalent to $\mu \rightarrow 0$.
\end{itemize}

By virtue of the formula $\triangle=c^{-1}\mu\sqrt{\det\slashed{g}}$, it is clear (the volume element $\sqrt{\det\slashed{g}}$ will be controlled in the sequel) that if shock forms then the coordinate transformation between the optical coordinates and the rectangular coordinates will fail to be a diffeomorphism. Therefore, we can also say that
\begin{itemize}
\item Shock formation is equivalent to the fact that the optical coordinates on the maximal development defines a different differentiable structure (compared to the usual differentiable structure induced from the Minkowski spacetime).
\end{itemize}
\end{remark}
\subsection{Connection, curvature and structure equations}
 The $2^{\text{nd}}$ fundamental forms of $S_{t,\ub}\hookrightarrow \Cb_{\ub}$ and $S_{t,\ub}\hookrightarrow \Sigma_{t}$ are defined as
\begin{align}\label{2nd fundamental forms}
\chib_{AB}:=g(\nabla_{X_{A}}\Lb,X_{B}),\quad\theta_{AB}:=g(\nabla_{X_{A}}\hat{T},X_{B})
\end{align}
respectively. Here $\nabla$ is the Levi-Civita connection of $g$ and $\widehat{T}:=c\mu^{-1} T$. For $\chib$, sometimes we will treat the trace/traceless parts with respect to $\slashed{g}$ separately, which are defined as $\tr\chib:=(\slashed{g}^{-1})^{AB}\chib_{AB}$ and $\widehat{\chib}_{AB}:=\chib_{AB}-\frac{1}{2}\slashed{g}_{AB}\tr\chib$. At each point $p\in W^{*}$, the vectorfields $\{\Lb, T\}$ form a basis of subspace of $T_{p}W^{*}$, which is orthogonal to $T_{p}S_{t,\ub}$. Since by virtue of \eqref{optical metric} the vectorfield $\dfrac{\partial}{\partial x^{0}}$ is orthogonal to $T_{p}S_{t,\ub}$, $\dfrac{\partial}{\partial x^{0}}$ is a linear combination of $\Lb$ and $T$. Actually $\dfrac{\partial}{\partial x^{0}}=\Lb+c\widehat{T}$ and $\chib=-c\theta$.

The torsion one forms $\etab_A$ and $\zetab_A$ are defined by $\zetab_A =g(\nabla_{X_A}\Lb,T)$ and $\etab_A=-g(\nabla_{X_A}T,\Lb)$. They are related to the inverse density $\mu$ by $\etab_A=\zetab_A+X_A(\mu)$ and $\zetab_A=-c^{-1}\mu \, X_A(c)$. 

An outgoing null vectorfield
\begin{align}\label{outgoing null}
L:=c^{-2}\mu\Lb+2T
\end{align}
is introduced so that $g(\Lb,L)=-2\mu$. The corresponding  $2^{\text{nd}}$ fundamental form is $\chi_{AB} = g(\nabla_{X_A} L, X_B)$. Similarly, we define $\tr\chi = \tr_{\gslash}\chi = \gslash^{AB}\chi_{AB}$ and $\widehat{\chi}_{AB} = \chi_{AB}-\dfrac{1}{2}\tr\chi \,\gslash_{AB}$. 

The covariant derivative $\nabla$ is expressed in the frame $(\Lb,T,X_{1},X_{2})$:
\begin{align*}
 \nabla_{\Lb}\Lb&=\mu^{-1}(\Lb\mu)\Lb,\ \ \nabla_{T}\Lb=\etab^{A}X_{A}-{c}^{-1}\Lb(c^{-1}\mu)\Lb,\\
 \nabla_{X_{A}}\Lb&=-\mu^{-1}\zetab_{A}\Lb+\underline{\chi}_{A}{}^{B} X_{B},\ \  \nabla_{\Lb}T=-\zetab^{A}X_{A}-{c}^{-1} \Lb(c^{-1}\mu)\Lb,\\
 \nabla_{T}T &=c^{-3}\mu\big(Tc + \Lb (c^{-1}\mu)\big)\Lb+\Big(c^{-1}\big(Tc+\Lb(c^{-1}\mu)\big)+T\big(\log(c^{-1}\mu)\big)\Big)T-c^{-1}\mu\slashed{g}^{AB}X_{B}(c^{-1}\mu)X_{A},\\
\nabla_{X_{A}}T&=\mu^{-1}\etab_{A}T+c^{-1}\mu\theta_{AB}\slashed{g}^{BC}X_{C},\ \ \nabla_{\Lb}X_{A}=\nabla_{X_{A}}\Lb, \ \ \nabla_{X_{A}}X_{B}=\nablaslash_{X_{A}}X_{B}+\mu^{-1}\underline{\chi}_{AB}T.
\end{align*}
In terms of null frames $(L,\Lb, X_1,X_2)$, we have
\begin{align*}
\nabla_{L}{\Lb}&=-{\Lb}({c}^{-2}\mu){\Lb}+2\underline{\eta}^{A}X_{A}, \ \ \nabla_{{\Lb}}L=-2\underline{\zeta}^{A}X_{A},\\
\nabla_{L}L&=(\mu^{-1}L\mu + {\Lb}({c}^{-2}\mu))L-2\mu X^{A}({c}^{-2}\mu)X_{A}, \ \ \nabla_{X_{A}}L=\mu^{-1}\underline{\eta}_{A}L+\chi_{A}{}^{B}X_{B},\\
\nabla_{X_{A}}X_{B}&=\nablaslash_{X_{A}}X_{B}+\frac{1}{2}\mu^{-1}\underline{\chi}_{AB}L+\frac{1}{2}\mu^{-1}\chi_{AB}{\Lb}.
\end{align*}
Here $\slashed{\nabla}$ is the induced covariant derivative on $S_{t,\ub}$.

In the Cartesian coordinates, the only non-vanishing curvature components are $R_{0i0j}$'s:
\begin{equation*}
R_{0i0j} =\frac{1}{2}\frac{d(c^{2})}{d\rho}\nabla_{i}\nabla_j\rho +\frac{1}{2}\frac{d^{2}(c^{2})}{d\rho^{2}}\nabla_{i}\rho\nabla_{j}\rho
-\frac{1}{4}c^{-2}|\frac{d(c^{2})}{d\rho}|^2\nabla_{i}\rho \nabla_{j}\rho.
\end{equation*}
In the optical coordinates, the only nonzero curvature components are $\alphab_{AB}=R(X_{A},\Lb,X_{B},\Lb)$:
\begin{equation*}
\alphab_{AB}= \frac{1}{2}\frac{d(c^{2})}{d\rho}\slashed{\nabla}^{2}_{X_A,X_B}\rho -\frac{1}{2}\mu^{-1}\frac{d(c^{2})}{d\rho}T(\rho) \chib_{AB}+\frac{1}{2}\big(\frac{d^{2}(c^{2})}{d\rho^{2}}
-\frac{1}{2}c^{-2}\big|\frac{d(c^{2})}{d\rho}\big|^2\big)X_{A}(\rho) X_{B}(\rho).
\end{equation*}
We define $\alphab'_{AB}= \frac{1}{2}\frac{d(c^{2})}{d\rho}\slashed{\nabla}^{2}_{X_A,X_B}\rho +\frac{1}{2}[\frac{d^{2}(c^{2})}{d\rho^{2}}
-\frac{1}{2}c^{-2}|\frac{d(c^{2})}{d\rho}|^2]X_{A}(\rho) X_{B}(\rho)$ such that 
\begin{align}\label{alpha expansion in mu and alpha prime}
\alphab_{AB} =  -\dfrac{1}{2}\mu^{-1}\dfrac{d(c^{2})}{d\rho}T(\rho) \chib_{AB}+ \alphab'_{AB}.
\end{align}

\begin{remark}
As a convention, we say that the first term on the right hand side of \eqref{alpha expansion in mu and alpha prime}  is \emph{singular} in $\mu$ (since $\mu$ may go to zero). The second term $\alphab'_{AB}$ is \emph{regular} in $\mu$.

Indeed, in the course of the proof, we will see that $\alphab'_{AB}$ are bounded and $\alphab_{AB}$ behaves exactly as ${\mu}^{-1}$ in amplitude. Therefore, in addition to two equivalent descriptions of the shock formation in Remark \ref{geometric meaning of mu}, we have another geometric interpretation: 
\begin{itemize}
\item Shock formation is equivalent to the fact that curvature tensor of the optical metric $g$ becomes unbounded.
\end{itemize}
Compared to the one dimensional picture of shock formations in conservation laws, e.g., for inviscid Burgers equation, this new description of shock formation is purely geometric in the following sense: it does not even depend on the choice of characteristic foliation (because the curvature tensor is tensorial!).
\end{remark}

In the frame $(T,\Lb,\frac{\partial}{\partial\theta^{A}})$, we have the following \emph{structure equations} in optical coordinates:
\begin{equation}\label{Structure Equation Lb chibAB}
\Lb(\underline{\chi}_{AB})=\mu^{-1}(\Lb\mu)\underline{\chi}_{AB} +\underline{\chi}_{A}{}^{C}\underline{\chi}_{BC}-\alphab_{AB},
\end{equation}
\begin{equation}\label{Structure Equation div chib}
\divslash\underline{\chi}-\slashed{d}\textrm{tr}\underline{\chi} = -\mu^{-1}(\zetab\cdot\underline{\chi}-\zetab\textrm{tr}\underline{\chi}),
\end{equation}
\begin{equation}\label{Structure Equation T chib}
\slashed{\mathcal{L}}_{T}\underline{\chi}_{AB}=(\nablaslash\widehat{\otimes}\etab)_{AB}+\mu^{-1}(\zetab \hat{\otimes}\etab)_{AB}-{c}^{-1}\Lb(c^{-1}\mu)\underline{\chi}_{AB}+c^{-1}\mu(\theta \widehat{\otimes}\chib)_{AB},
\end{equation}
which will be used in proving the main estimates. Here $(\zetab\cdot\underline{\chi})_{B}=\slashed{g}^{AC}\zetab_{A}\underline{\chi}_{BC}$, $(\nablaslash\widehat{\otimes}\etab)_{AB} = \dfrac{1}{2}(\nablaslash_{A}\etab_{B}+\nablaslash_{B}\etab_{A})$, $(\zetab \widehat{\otimes}\etab)_{AB}=\frac{1}{2}(\zetab_{A}\etab_{B}+\zetab_{B}\etab_{A})$ and $(\theta \widehat{\otimes}\chib)_{AB}=\frac{1}{2}(\theta_{AC}\underline{\chi}_{B}^{C}+\theta_{BC}\underline{\chi}_{A}^{C})$.
By taking the trace of \eqref{Structure Equation Lb chibAB}, we have
\begin{equation}\label{Structure Equation Lb trchib}
\Lb\textrm{tr}\underline{\chi}=\mu^{-1}(\Lb\mu)\textrm{tr}\underline{\chi}-|\underline{\chi}|^{2}_{\slashed{g}}-\textrm{tr}\alphab.
\end{equation}

The inverse density function $\mu$ satisfies the following transport equation:
\begin{equation}\label{Structure Equation Lb mu}
\Lb\mu=m+\mu e,
\end{equation}
with $m=-\frac{1}{2}\frac{d(c^{2})}{d\rho}T\rho$ and $e=\frac{1}{2c^{2}}\frac{d(c^{2})}{d\rho}\Lb\rho$. With these notations, we have $\alphab_{AB} = \mu^{-1} m \chib_{AB} + \alphab'_{AB}$.

Regarding the regularity in $\mu$, we use \eqref{Structure Equation Lb mu} to replace $\Lb\mu$ in \eqref{Structure Equation Lb chibAB}. This yields
\begin{equation}\label{Structure Equation Lb chibAB nonsingular}
\Lb (\chib_{AB}) = e \chib_{AB} + \chib_{A}{}^{C}\chib_{BC}-\alphab_{AB}'.
\end{equation}
Compared to the original \eqref{Structure Equation Lb chibAB}, the new equation is \textit{regular} $\mu$ in the sense that it has no $\mu^{-1}$ terms. 
\subsection{Rotation Vectorfields}\label{section definition for Ri}
Although $g\big|_{\Sigma_t}$ is flat, the foliation $S_{t,\ub}$ is different from the standard spherical foliation. In the Cartesian coordinates on $\Sigma_t$, let $\Omega_{1} = x^2 \partial_3 -x^3 \partial_2$, $\Omega_{2} = x^3 \partial_1 -x^1 \partial_3$ and $\Omega_{3} = x^1 \partial_2 -x^2 \partial_1$ be the standard rotations. Let $\Pi$ be the orthogonal projection to $S_{t,\ub}$. The rotation vectorfields $R_i \in \Gamma(TS_{t,\ub})$ ($i=1,2,3$) are defined by
\begin{equation}\label{definition for rotation vectorfields}
R_i = \Pi \,\Omega_{i}.
\end{equation}
Let indices $i, j, k \in \{1,2,3\}$. We use the $T^k$, $\Lb^k$ and $X_A^{k}$ to denote the components for $T$, $\Lb$ and $X_A$ in the Cartesian frame $\{\partial_{i}\}$ on $\Sigma_t$ (notice that $\Lb$ has also a $0$th component $\Lb^{0} = 1$). We introduce some functions to measure the difference between the foliations $S_{t,\ub}$ and the standard spherical foliations.

The functions $\lambda_i$'s measure the derivation from $R_i$ to $\Omega_i$:
\begin{equation}\label{definition for lambda_i}
 \lambda_i \widehat{T} = \Omega_i - R_i.
\end{equation}
The functions $y'^k$'s measure the derivation from $\widehat{T}$ to the standard radial vectorfield $\frac{x^i}{r}\partial_i$:
\begin{equation}\label{definition for y'}
 y'^k= \widehat{T}^k-\frac{x^k}{r}.
\end{equation}
We also define (we will show that $|y^k-y'^k|$ is bounded by a negligible small number)
\begin{equation}\label{definition for y}
 y^k= \widehat{T}^k-\frac{x^k}{\ub-t}.
\end{equation}
The functions $z^k$'s measure the derivation of $\Lb$ from $\partial_t - \partial_r$ in Minkowski spacetime:
\begin{equation*}
 z^k= \Lb^k+\frac{x^k}{\ub-t} = -\frac{(c-1)x^k}{\ub-t}-cy^k.
\end{equation*}
Finally, the rotation vectorfields can be expressed as
\begin{equation}\label{explicit expression of R i}
R_i = \Omega_i- \lambda_i \displaystyle\sum_{j=1}^3 \widehat{T}^j \partial_j,\quad \lambda_i= \displaystyle\sum_{j,k,l=1}^3 \varepsilon_{ilk}x^l y^k,
\end{equation}
where $\varepsilon_{ijk}$ is the totally skew-symmetric symbol.
\section{Initial data, bootstrap assumptions and the main estimates}
\subsection{Preliminary estimates on initial data}\label{section short pulse data}
In the Main Theorem, we take the so called short pulse datum for $(\star)$ on $\Sigma_{-r_0}^{\delta}$. Recall that $\phi(-r_0,x)=\frac{\delta^{3/2}}{r_{0}}\phi_{0}(\frac{r-r_{0}}{\delta},\theta)$ and $\partial_{t}\phi(-r_0,x)=\frac{\delta^{1/2}}{r_{0}}\phi_{1}(\frac{r-r_{0}}{\delta},\theta)$, where $\phi_{0}, \phi_{1}\in C^\infty_0\big((0,1]\times \mathbb{S}^2\big)$. The condition $(3)$ in the statement of the Main Theorem reads as
\begin{align*}
\|\Lb\phi\|_{L^\infty\left(\Sigma_{-r_{0}}\right)}\lesssim \frac{\delta^{3/2}}{r_{0}^{2}},\ \ \|\Lb^{2}\phi\|_{L^\infty\left(\Sigma_{-r_{0}}\right)}\lesssim \frac{\delta^{3/2}}{r_{0}^{2}}.
\end{align*}
We now derive estimates for $\phi$ and its derivatives on $\Sigma_{-r_0}$. These estimates also suggest the estimates, e.g. the bootstrap assumptions in next subsection, that one can expect later on.

For $\phi$ and $\psi = A \phi$ where $A \in \{\partial_\alpha\}$, by the form of the data, we clearly have
\begin{equation}\label{initial L infinity estimates for psi}
 \|\phi\|_{L^{\infty}(\Sigma_{-r_{0}})}\lesssim\frac{\delta^{3/2}}{r_{0}}, \  \|\psi\|_{L^{\infty}(\Sigma_{-r_{0}})}\lesssim\frac{\delta^{1/2}}{r_{0}}.
\end{equation}
We will use $Z$ or $Z_j$ to denote any vector from $\{T, R_i, Q\}$ where $Q= t\Lb$. On $\Sigma_{-r_0}$, $Z$ is simply $\partial_r$, $\Omega_{i}$ or $-r_0(\partial_{t}-\partial_{r})$, therefore, we have
\begin{equation}\label{initial L infinity estimates for commutations on psi}
\|Z^m \psi\|_{L^{\infty}(\Sigma_{-r_{0}})} \lesssim \frac{\delta^{1/2-l}}{r_{0}},
\end{equation}
with $l$ is the number of $T$'s and $Z \in \{T,\Omega_i,Q\}$. We remark that throughout the whole argument $Q$ appears at most twice in the string of $Z$'s.

We also consider the incoming energy for $Z^m \psi$ on $\Sigma_{-r_0}$. According to \eqref{initial L infinity estimates for commutations on psi}, we have
\begin{align*}
\|\Lb ( Z^m \psi )\|_{L^{2}(\Sigma_{-r_{0}})} + \|\slashed{d}(Z^m \psi)\|_{L^{2}(\Sigma_{-r_{0}})} \lesssim \frac{\delta^{1-l}}{r_{0}}, \  \|T(Z^m \psi)\|_{L^{2}(\Sigma_{-r_{0}})} \lesssim {\delta^{-l}}
\end{align*}
where $\ds$ denotes for the exterior differential on $S_{t,\ub}$. In terms of $L$, for $m\in \mathbb{Z}_{\geq 0}$, we obtain
\begin{equation}\label{initial energy estimates for commutations on psi}
\|\Lb ( Z^m \psi )\|_{L^{2}(\Sigma_{-r_{0}})} + \|\slashed{d}(Z^m \psi)\|_{L^{2}(\Sigma_{-r_{0}})}  \lesssim \frac{\delta^{1-l}}{r_{0}}, \ \ \ \|L (Z^m \psi)\|_{L^{2}(\Sigma_{-r_{0}})} \lesssim \delta^{-l}
\end{equation}
where $l$ is the number of $T$'s in $Z$'s.

We also consider the estimates on some connection coefficients on $\Sigma_{-r_0}$. For $\mu$, since we have $g(T,T) = c^{-2}\mu^{2}$ and $T=\partial_r$ on $\Sigma_{-r_0}$, we then have $\mu = c$ on $\Sigma_{r_0}$. Since $c =\big(1+ 2 G''(0)(\partial_t \phi)^2\big)^{-\frac{1}{2}}$, according to \eqref{initial L infinity estimates for psi}, for sufficiently small $\delta$, we obtain
\begin{equation}\label{initial estimates on mu}
\|\mu - 1\|_{L^\infty(\Sigma_{-r_0})} \lesssim \frac{\delta}{r_{0}^{2}}.
\end{equation}
For $\chib_{AB}$, since $\chib_{AB} = -c\theta_{AB} =-\dfrac{c}{r_0}\slashed{g}_{AB}$, we have $\chib_{AB} +\dfrac{1}{r_0}\slashed{g}_{AB}  = (1-c)\dfrac{1}{r_0}\slashed{g}_{AB}$.
Hence,
\begin{equation}\label{initial L infinity estimates for tr chib prime}
\|\chib_{AB} +\dfrac{1}{r_0}\slashed{g}_{AB}\|_{L^{\infty}(\Sigma_{-r_0})}  \lesssim \frac{\delta}{r_{0}^{3}}.
\end{equation}
It measures the difference between the $2^{\text{nd}}$ fundamental form with respect to $g_{\alpha\beta}$ and $m_{\alpha\beta}$.

\subsection{Bootstrap assumptions and the main estimates}\label{section Bootstrap Assumptions}
We expect the estimates \eqref{initial L infinity estimates for psi}, \eqref{initial L infinity estimates for commutations on psi} and \eqref{initial energy estimates for commutations on psi} hold not only for $t=-r_0$ but also for later time slice in $W^*_\delta$. For this purpose, we will run a bootstrap argument to derive the \emph{a priori} estimates for the $Z^m \psi$'s.

\subsubsection{Conventions}
We first introduce three large positive integers $N_{\text{top}}$, $N_{\mu}$ and $\Ninfty$. They will be determined later on.  We require that $\Nmu = \lfloor \frac{3}{4} \Ntop\rfloor$ and $\Ninfty = \lfloor \frac{1}{2} \Ntop\rfloor + 1$. $\Ntop$ will eventually be the total number of derivatives applied to the linearized equation $\Box_{\gt} \psi = 0$.

As in \cite{M-Y}, to count the number of derivatives, we define the \emph{order} of an object. The solution $\phi$ is considered as an order $-1$ object. The variations $\psi = A \phi$ are of order $0$. The metric $g$ depends only on $\psi$, so it is of order $0$. The inverse density function $\mu$ is of order $0$. The connection coefficients are 1st order derivatives on $g$, hence, of order $1$. In particular, $\chib_{AB}$ is of order $1$. Let $\alpha =(i_1,\cdots,i_{k-1})$ be a multi-index with $i_j$'s from $\{1,2,3\}$. We use $Z^{\alpha} \psi$ as a schematic expression of $Z_{i_1}Z_{i_2}\cdots Z_{i_{k-1}} \psi$. The order of $Z^\alpha\psi$ is $|\alpha|$, where $|\alpha| = k-1$. Similarly, for any tensor of order $|\alpha|$, after taking $m$ derivatives, its order becomes $|\alpha|+m$. The highest order objects in this paper will be of order $\Ntop +1$.

Let $l \in \mathbb{Z}_{\geq 0}$ and $k \in \mathbb{Z}$. We use $\O^l_{k,j}$ or $\O^{\leq l}_{k,j}$ to denote any term of order $l$ or at most $l$ with estimates
\begin{equation*}
 \|\O^l_{k,j}\|_{L^\infty(\Sigma_t)}\lesssim \frac{\delta^{\frac{1}{2}k}}{(-t)^{j}}, \ \|\O^{\leq l}_{k,j}\|_{L^\infty(\Sigma_t)}\lesssim \frac{\delta^{\frac{1}{2}k}}{(-t)^{j}}.
\end{equation*}
 Similarly, we use $\Psi^l_{k,j}$ or $\Psi^{\leq l}_{k,j}$ to denote any term of order $l$ or at most $l$ with estimates
\begin{equation*}
 \|\Psi^l_{k,j}\|_{L^\infty(\Sigma_t)}\lesssim \frac{\delta^{\frac{1}{2}k}}{(-t)^{j}}, \ \|\Psi^{\leq l}_{k,j}\|_{L^\infty(\Sigma_t)}\lesssim \frac{\delta^{\frac{1}{2}k}}{(-t)^{j}},
\end{equation*}
and moreover, it can be explicitly expressed a function of the variations $\psi$. For example,
$\partial_t \phi \cdot \partial_i\phi \in \Psi^0_{1,1}$; A term of the form $\prod_{i=1}^{n} Z^{\alpha_i} \psi$
so that $\max|\alpha_i|\leq m$ is $\Psi^{\leq m}_{n-2l,1}$, where $l$ is the number of $T$ appearing in the derivatives. Note that $\chib$ and $\mu$ can not be expressed explicitly in terms of $\psi$. The $\O^l_{k,j}$ terms (or similarly the $\Psi^l_{k,j}$ terms) obey the following algebraic rules:
\begin{align*}
 \O^{\leq l}_{k,j} + \O^{\leq l'}_{k',j'} &= \O^{\leq \max(l,l')}_{\min(k,k'),\min(j,j')},\ \   \O^{\leq l}_{k,j}  \O^{\leq l'}_{k',j'} = \O^{\leq \max(l,l')}_{k+k',j+j'}.
\end{align*}

\subsubsection{Bootstrap assumptions on $L^\infty$ norms}

Motivated by \eqref{initial L infinity estimates for commutations on psi}, we make the following bootstrap assumptions (B.1) on $W^*_\delta$: For all $t$ and $2\leq |\alpha|\leq \Ninfty$, \footnote{For a multi-index $\alpha$, the symbol $\alpha-1$ means another multi-index $\beta$ with degree $|\beta| = |\alpha|-1$.}
\begin{align*}\tag{B.1}
\|\psi\|_{L^{\infty}(\Sigma_{t})}+(-t)\|\Lb \psi\|_{L^\infty(\Sigma_t)}+(-t)\|\ds\psi\|_{L^\infty(\Sigma_t)} + \delta \|T \psi\|_{L^\infty(\Sigma_t)} + \delta^l\|Z^\alpha \psi\|_{L^\infty(\Sigma_t)}&\lesssim \delta^{1/2}M(-t)^{-1}.
\end{align*}
where $l$ is the number of $T$'s appearing in $Z^\alpha$ and $M$ is a large positive constant depending on $\phi$. We will show that if $\delta$ is sufficiently small which may depend on $M$, then we can choose $M$ in such a way that it depends only on the initial datum.
\subsubsection{Energy norms}
For $(t,\ub) \in [-r_0,t^*)\times [0,\delta]$, let $\dmug$ be the volume form of $\slashed{g}$. For a function $f(t,\ub,\theta)$, we define
\begin{equation*}
 \int_{\Sigma_{t}^{\ub}} f = \int_{0}^{\ub} \Big(\int_{S_{t,\ub'}}f(t,\ub',\theta)\dmug\Big) d \ub', \ \ \int_{\Cb_{\ub}^t} f= \int_{-r_0}^{t} \Big(\int_{S_{\tau,\ub}}f(\tau,\ub,\theta)\dmug\Big) d \tau.
\end{equation*}

For a function $\Psi(t,\ub,\theta)$, we define the energy flux through the hypersurfaces $\Sigma_{t}^{\ub}$ and $\Cb_{\ub}^{t}$ as
\begin{equation}\label{definition for energy and flux}
 \begin{split}
 {E}(\Psi)(t,\ub)&=\int_{\Sigma_{t}^{\ub}}(L\Psi)^2 + \mu\left((\Lb\psi)^{2}+(1+\mu)|\slashed{d} \Psi|^2\right), \ \  {F}(\Psi)(t,\ub)=\int_{\Cb_{\ub}^{t}} (\Lb\Psi)^{2}+\mu |\slashed{d} \Psi|^2,\\
{\Eb}(\Psi)(t,\ub)&= (-t)^{2}\int_{\Sigma_{t}^{\ub}} \mu \left(\left(\Lb \Psi + \frac{1}{2}\widetilde{\text{tr}}\chib\Psi\right)^{2}+ |\slashed{d}\Psi|^2\right),\ \ {\Fb}(\Psi)(t,\ub)=\int_{\Cb_{\ub}^{t}}(-t')^{2} \left(\Lb \Psi+\frac{1}{2}\widetilde{\text{tr}}\chib\Psi\right)^2.
 \end{split}
\end{equation}
Here $\widetilde{\text{tr}}\chib$ is the trace of $\chib$ with respect to the induced conformal metric $\widetilde{\slashed{g}}$.

For each integer $0\leq k\leq \Ntop$, we define
\begin{equation}\label{definition for kth energy}
 \begin{split}
 {E}_{k+1}(t,\ub)=\sum_{\psi}\sum_{|\alpha| = k-1} \delta^{2l}E(Z^{\alpha+1} \psi)(t,\ub),\quad F_{k+1}(t,\ub)=  \sum_{\psi}\sum_{|\alpha| = k-1} \delta^{2l}F(Z^{\alpha+1} \psi)(t,\ub),\\
{\Eb}_{k+1}(t,\ub)=\sum_{\psi}\sum_{|\alpha| = k-1} \delta^{2l}\Eb(Z^{\alpha+1} \psi)(t,\ub),\quad \Fb_{k+1}(t,\ub)= \sum_{\psi}\sum_{|\alpha| = k-1}\delta^{2l}\Fb(Z^{\alpha+1} \psi)(t,\ub),
 \end{split}
\end{equation}
and

\begin{align}\label{sum of kth energy}
\begin{split}
E_{\leq k+1}(t,\ub):=\sum_{j\leq k}E_{j+1}(t,\ub),\quad \Eb_{\leq k+1}(t,\ub):=\sum_{j\leq k}\Eb_{j+1}(t,\ub)\\
F_{\leq k+1}(t,\ub):=\sum_{j\leq k}F_{j+1}(t,\ub),\quad \Fb_{\leq k+1}(t,\ub):=\sum_{j\leq k}\Fb_{j+1}(t,\ub)
\end{split}
\end{align}
where $l$ is the number of $T$'s appearing in $Z^\alpha$.  The symbol $\displaystyle\sum_{\psi}$ means to sum over all the first order variations $A\phi$ of $\psi$. For the sake of simplicity, we shall omit this sum symbol in the sequel.

For each integer $0\leq k \leq \Ntop$, we assign a nonnegative integer $b_k$ to $k$ in such a way that
\begin{equation}
b_0 = b_1 =\cdots = b_{\Nmu} =0, \ \ b_{\Nmu +1} <  b_{{\Nmu}+2}< \cdots <  b_{\Ntop}.
\end{equation}
We call $b_k$'s \emph{the blow-up indices}. The sequence $(b_k)_{0\leq k \leq \Ntop}$ will be determined later on.

For each integer $0\leq k\leq \Ntop$, we also define the modified energy ${\Et}_k(t,\ub)$ and ${\Ebt}_k(t,\ub)$  as
\begin{equation}\label{definition for modified energy}
\begin{split}
 &{\Et}_{k+1}(t,\ub)= \sup_{\tau \in[-r_{0},t]}\{\mu^{\ub}_{m}(\tau)^{2b_{k+1}}E_{k+1}(\tau,\ub)\},
 \quad {\Ebt}_{k+1}(t,\ub)=\sup_{\tau \in[-r_{0},t]}\{\mu^{\ub}_{m}(\tau)^{2b_{k+1}}\Eb_{k+1}(\tau,\ub)\}\\
&{\widetilde{F}}_{k+1}(t,\ub)= \sup_{\tau \in[-r_{0},t]}\{\mu^{\ub}_{m}(\tau)^{2b_{k+1}}F_{k+1}(\tau,\ub)\},\quad {\widetilde{\Fb}}_{k+1}(t,\ub)=\sup_{\tau \in[-r_{0},t]}\{\mu^{\ub}_{m}(\tau)^{2b_{k+1}}\Fb_{k+1}(\tau,\ub)\}\\
&\Et_{\leq k+1}(t,\ub):=\sum_{j\leq k}\Et_{j+1}(t,\ub),\quad \Ebt_{\leq k+1}(t,\ub):=\sum_{j\leq k}\Ebt_{j+1}(t,\ub),\\
&\widetilde{F}_{\leq k+1}(t,\ub):=\sum_{j\leq k}\widetilde{F}_{j+1}(t,\ub),\quad \widetilde{\Fb}_{\leq k+1}(t,\ub):=\sum_{j\leq k}\widetilde{\Fb}_{j+1}(t,\ub),
 \end{split}
 \end{equation}
where $\mu^{\ub}_{m}(t)$ is defined as
\begin{align}\label{definition  of mu m}
\mu^{\ub}_{m}(t)=\min_{\tau\in[-r_0,t]}\left\{\inf_{(\ub',\theta)\in[0,\ub]\times \mathbb{S}^{2}}\mu(\tau,\ub',\theta), 1\right\}
\end{align}
For simplicity, we omit the parameter $\ub$ to write the weight as $\mu_{m}(t)$. When we work on the estimates on $W^{t}_{\ub}$, the weight $\mu_{m}(t)$ means $\mu^{\ub}_{m}(t)$.

We now state the main estimates of the paper.
\begin{theorem}\label{theorem main estimates}
There exists a constant $\delta_0$ depending only on the seed data $\phi_0$ and $\phi_1$, so that for all $\delta< \delta_0$, there exist constants $M_0$, $\Ntop$ and $(b_k)_{0 \leq k \leq \Ntop}$ with the following properties
\begin{itemize}
\item $M_0$, $\Ntop$ and $(b_k)_{0 \leq k \leq \Ntop}$ depend only on the initial datum.
\item The inequalities (B.1) holds for all $t < t^*$ with $M=M_0$.
\item Either $t^{*}=-1$ and we have a smooth solution in the time slab $[-r_{0},-1]$; or $t^{*}<-1$ and then $\psi_{\alpha}$'s as well as the rectangular coordinates $x^{i}$'s extend smoothly as functions of the coordinates $(t,\ub,\theta)$ to $t=t^{*}$ and there is at least one point on $\Sigma_{t^{*}}^{\delta}$ where $\mu$ vanishes, thus we have shock formation.
\item If, moreover, the initial data satisfies the largeness condition \eqref{shock condition}, then in fact $t^{*}<-1$.
\end{itemize}
\end{theorem}
\subsection{Preliminary estimates based on (B.1)}
In this subsection, we list the pointwise estimates for components of optical metric, connection coefficients as well as deformation tensors associated to various vectorfields based on (B.1). The proofs are similar to those in \cite{M-Y}.
\subsubsection{Estimates on metric and connection}
\begin{lemma} For sufficiently small\footnote{This sentence always means that, there exists $\varepsilon = \varepsilon(M)$ so that for all $\delta \leq \varepsilon$, we have ...} $\delta$, we have
\begin{equation}\label{bound on c}
\begin{split}
\|c-1\|_{L^{\infty}(\Sigma_{t})}&\lesssim \delta M^{2}(-t)^{-2},
\frac{1}{2} \leq c \leq 2,\\
(-t)\|\Lb c\|_{L^\infty(\Sigma_t)} + (-t)\|\ds c\|_{L^\infty(\Sigma_t)} &\lesssim \delta M^2(-t)^{-2},\ \|T c\|_{L^\infty(\Sigma_t)} \lesssim M^2(-t)^{-2}.
\end{split}
\end{equation}
\end{lemma}
\begin{lemma}
 For sufficiently small $\delta$, we have
\begin{equation}\label{bound on m}
 \|m\|_{L^\infty(\Sigma_t)} +(-t)\|\ds m\|_{L^\infty(\Sigma_t)} \lesssim M^2(-t)^{-2},\quad  \|T m\|_{L^\infty(\Sigma_t)} \lesssim \delta^{-1} M^2(-t)^{-2},
\end{equation}
\begin{equation}\label{bound on e}
 \|e\|_{L^\infty(\Sigma_t)}  +  (-t)\|\ds e\|_{L^\infty(\Sigma_t)}  \lesssim \delta M^2(-t)^{-2},\quad  \|T e\|_{L^\infty(\Sigma_t)}  \lesssim  M^2(-t)^{-2},
\end{equation}
\begin{equation}\label{bound on mu}
 \|\mu-1\|_{L^\infty(\Sigma_t)} +(-t)\|\Lb \mu\|_{L^\infty(\Sigma_t)}\lesssim  M^2(-t)^{-1}.
\end{equation}
\end{lemma}
\begin{lemma}
For sufficiently small $\delta$, we have
\begin{equation}\label{bound on T mu and slashed d mu}
 \|T\mu\|_{L^\infty(\Sigma_t)} \lesssim \delta^{-1} M^2(-t)^{-1}, \quad (-t)\|\ds \mu\|_{L^\infty(\Sigma_t)} \lesssim M^2(-t)^{-1}.
\end{equation}
\end{lemma}
As a corollary of these lemmas, we can show
\begin{corollary}
For sufficiently small $\delta$, we have
\begin{equation}\label{bound on zetab and etab}
\|\zetab\|_{L^\infty(\Sigma_t)} \lesssim \delta M^2(-t)^{-2}, \ \ \|\etab\|_{L^\infty(\Sigma_t)} \lesssim M^2(-t)^{-2}.
\end{equation}
\end{corollary}
We now estimate $\chib_{AB}$. For this purpose, we introduce
\begin{equation}\label{definition for chib prime}
\chib'_{AB} = \chib_{AB} + \frac{\slashed{g}_{AB}}{\ub-t}.
\end{equation}
which measures the deviation of $\chib_{AB}$ from the null $2^{\text{nd}}$ fundamental form in Minkowski space. We have
\begin{lemma}
For sufficiently small $\delta$, we have
\begin{equation}\label{precise bound on chib'}
 \| \chib'_{AB}\|_{L^\infty(\Sigma_t)} \lesssim  \delta M^2(-t)^{-3}.
\end{equation}
\end{lemma}
\begin{proof}
According to \eqref{Structure Equation Lb chibAB nonsingular}, we have
\begin{equation}\label{transport equation for chib prime}
\Lb (\chib_{AB}') = e \chib'_{AB} + \chib'_{A}{}^{C}\chib'_{BC}- \frac{e}{\ub - t}\slashed{g}_{AB} - \alphab'_{AB}.
\end{equation}
Hence, $\Lb|\chib'|^2_{\slashed{g}} =2e|\chib'|^2 -2\chib'_{A}{}^{B} \chib'_{B}{}^{C} \chib'_{C}{}^{A}+\frac{2|\chib'|^{2}}{\ub-t}-\frac{2e}{\ub-t}\tr\chib' - 2\chib'^{AB}\alphab_{AB}$.
Therefore, we obtain
\begin{align}\label{a 1}
\Lb((t-\ub)^{2}|\chib'|)\lesssim (t-\ub)^{2}\left((|e|+|\chib'|)|\chib'|+\frac{|e|}{\ub-t}+|\alphab'|\right)
\end{align}
Let $\mathcal{P}(t)$ be the property that $\|\chib'\|_{L^{\infty}(\Sigma_{t'})}\leq C_{0}\delta M^{2}(-t')^{-3}$ for all $t'\in[-r_{0},t]$. By choosing $C_{0}$ suitably large, according to the assumptions on initial data, we have $\|\chib'\|_{L^{\infty}(\Sigma_{-r_{0}})}<C_{0}\delta r^{-3}_{0}$. It follows by continuity that $\mathcal{P}(t)$ is true for $t$ sufficiently close to $-r_{0}$. Let $t_{0}$ be the upper
bound of $t\in [-r_{0}, t_0]$ for which $\mathcal{P}(t)$ holds. By continuity, $\mathcal{P}(t_0)$ is true. Therefore, for $t\leq t_0$, we have $|\chib'(t)|+|e(t)|\leq (C_{0}+C_{1})\delta M^{2}(-t)^{-3}$ for a universal constant $C_{1}$. According to the explicit formula of $\alphab'$ and (B.1), for sufficiently small $\delta$, there is a universal constant $C_3$ so that $\|\alphab'\|_{L^\infty(\Sigma_t)} \leq C_3\delta M^{2}(-t)^{-4}$. In view of \eqref{a 1}, there is a universal constant $C_{4}$ so that
\begin{align}\label{chib transport}
\Lb((t-\ub)^{2}|\chib'|)\leq C_{4}(t-\ub)^{2}\left((C_{0}+C_{1})\delta M^{2}(-t)^{-3}|\chib'|+(C_{1}+C_{3})\delta M^{2}(-t)^{-4}\right).
\end{align}
If we define $x(t)=(t-\ub)^{2}|\chib'|$ along the integral curve of $\Lb$, then we can rewrite \eqref{chib transport} as $\frac{dx}{dt}\leq fx+g$, where $f(t)=C_{4}(C_{0}+C_{1})\delta M^{2}(-t)^{-3}$ and $g(t)=C_{4}(C_{1}+C_{3})\delta M^{2}(-t)^{-2}$. By integrating from $-r_{0}$ to $t$, we obtain
\begin{align*}
x(t)\leq e^{\int_{-r_{0}}^{t}f(t')dt'}\big(x(-r_{0})+\int_{-r_{0}}^{t}e^{-\int_{-r_{0}}^{t'}f(t'')dt''}g(t')dt'\big).
\end{align*}
Taking into account the facts that $\int_{-r_{0}}^{t}f(t')dt'\leq C_{4}(C_{0}+C_{1})\delta M^{2}(-t)^{-1}$ and $\int_{-r_{0}}^{t}g(t')dt'\leq C_{4}(C_{1}+C_{3})\delta M^{2}(-t)^{-1}$, since $(t-\ub)^{2}\sim t^{2}$ on the support of $\chib'$, for some universal constant $C_{5}$, we have
\begin{align}\label{chib infty 1}
\|\chib'\|_{L^{\infty}(\Sigma_{t})}\leq &C_{5}e^{C_{4}(C_{0}+C_{1})\delta M^{2}}(-t)^{-2}\left(r_{0}^{2}\|\chib'\|_{L^{\infty}(\Sigma_{-r_{0}})}+C_{4}(C_{1}+C_{3})\delta M^{2}(-t)^{-1}\right)\\\notag
\leq &C_{5}e^{C_{4}(C_{0}+C_{1})\delta M^{2}}(-t)^{-3}\left(r_{0}^{3}\|\chib'\|_{L^{\infty}(\Sigma_{-r_{0}})}+C_{4}(C_{1}+C_{3})\delta M^{2}\right)
\end{align}
We then fix $C_{0}$ in such a way that $C_{0}>2C_{5}C_{4}(C_{1}+C_{3})$ and $C_{0}>\frac{r^{3}_{0}\|\chib'\|_{L^{\infty}(\Sigma_{t})}}{4C_{5}}$. Provided $\delta$ satisfying $\delta<\frac{\log 2}{C_{4}(C_{0}+C_{1})M^{2}}$, the estimate \eqref{chib infty 1} implies $\|\chib'\|_{L^{\infty}(\Sigma_{t})}<C_{0}\delta M^{2}(-t)^{-3}$ for all $t\in[-r_{0},t_{0}]$. By continuity, $\mathcal{P}(t)$ holds for some $t>t_{0}$. Hence the lemma follows.
\end{proof}

\begin{remark}[\textbf{Estimates related to the conformal optical metric $\gt$}]
 As in \cite{M-Y}, we shall use $\ \widetilde{ } \ $ to indicate the quantities defined with respect to $\gt$.

We expect the quantities (with $\ \widetilde{} \ $) defined with respect to $\gt$ have the similar estimates as the counterparts (without $\ \widetilde{} \ $) defined with respect to $g$. This is clear: the difference can be explicitly computed in terms of $c$ and hence controlled by the estimates on $c$. For example, the difference between $\chibt'_{AB}$ and $\chib'_{AB}$ is $\chibt_{AB}- \chib_{AB}= -\frac{1}{2c^2} \Lb(c)\slashed{g}_{AB}$, based on \eqref{bound on c} and \eqref{precise bound on chib'}, we have
\begin{equation}\label{precise bound on chibt}
\left\| \chibt_{AB} + \frac{\slashed{g}_{AB}}{\ub-t}\right\|_{L^\infty(\Sigma_t)} \lesssim  \delta M^2(-t)^{-3}.
\end{equation}
\end{remark}
\subsubsection{Estimates on deformation tensors}
Now we consider the deformation tensors of the following five commutation vectorfields: $Z_1 = T, Z_2 = R_1, Z_3 = R_2, Z_4 =R_3, Z_5 = Q$. The notation $Z^\alpha$ for a multi-index
$\alpha =(i_1,\cdots,i_m)$ means $Z_{i_1} Z_{i_2} \cdots Z_{i_m}$ with $i_j \in \{1,2,3,4,5\}$. For a vectorfield $Z$, the deformation tensor $^{(Z)}{}\pi$ or $^{(Z)}{}\widetilde{\pi}$ with respect to $g$ and $\gt$ is defined by
${}^{(Z)}{}{\pi}_{\alpha\beta} = \nabla_\alpha Z_\beta + \nabla_\beta Z_\alpha$ or ${}^{(Z)}{}\widetilde{\pi}_{\alpha\beta} = \frac{1}{c}{}^{(Z)}{}{\pi}_{\alpha\beta} + Z(\frac{1}{c})g_{\alpha\beta}$. As in \cite{M-Y}, a direct calculation gives the expressions for deformation tensors of $T$
\begin{equation}\label{deformation tensor of T}
 \begin{split}
  ^{(T)}{}\widetilde{\pi}_{\Lb\Lb}&=0,\ \ ^{(T)}{}\widetilde{\pi}_{LL}= 4 c^{-1}\mu T(c^{-2}\mu),\ \ ^{(T)}{}\widetilde{\pi}_{L\Lb}= -2c^{-1}(T\mu -c^{-1} \mu T(c)),\\
 ^{(T)}{}\widetilde{\pi}_{L A}&=-c^{-3}\mu(\zetab_A+\etab_A),\ \ ^{(T)}{}\widetilde{\pi}_{\Lb A}=-c^{-1}(\zetab_A+\etab_A),\\
^{(T)}{}\widehat{\widetilde{\slashed\pi}}_{A B}&=-2c^{-3}\mu\widehat{\chib}_{AB},\ \ \text{tr}{}^{(T)}\widetilde{\slashed\pi}=-2c^{-3}\mu \text{tr}_{\tilde{\slashed{g}}}\widetilde{\chib}.
 \end{split}
\end{equation}
and those of $Q$:
\begin{equation}\label{deformation tensor of Q}
\begin{split}
\leftexp{(Q)}{\tilde{\pi}}_{\Lb\Lb}&=0,\ \leftexp{(Q)}{\tilde{\pi}}_{LL}=4t\Lb(c^{-2}\mu)(c^{-1}\mu)-4c^{-3}\mu^{2},\  \leftexp{(Q)}{\tilde{\pi}}_{L\Lb}=-2t\Lb(c^{-1}\mu)
-2c^{-1}\mu\\
\leftexp{(Q)}{\tilde{\pi}}_{\Lb A}&=0,\  \leftexp{(Q)}{\tilde{\pi}}_{LA}=2tc^{-1}(\zetab_{A}+\etab_{A}), \  \leftexp{(Q)}{\widehat{\tilde{\slashed{\pi}}}}_{AB}
=2tc^{-1}\widehat{\chib}_{AB},\  \text{tr}\leftexp{(Q)}{\tilde{\slashed{\pi}}}=2c^{-1}t\text{tr}_{\tilde{\slashed{g}}}\widetilde{\chib}.
\end{split}
\end{equation}
as well as their estimates:
\begin{align}\label{L infinity etimates on the deformation tensor of T}
 \|\mu^{-1}{}^{(T)}{}\widetilde{\pi}_{LL}\|_{L^\infty(\Sigma_t)}&\lesssim \delta^{-1}M^2(-t)^{-1},\quad \|{}^{(T)}{}\widetilde{\pi}_{L\Lb}\|_{L^\infty(\Sigma_t)}\lesssim \delta^{-1}M^2(-t)^{-1},\\\notag
  \|\mu^{-1}{} ^{(T)}{}\widetilde{\pi}_{L A}\|_{L^\infty(\Sigma_t)}&\lesssim M^2(-t)^{-2}, \quad
 \|{}^{(T)}{}\widetilde{\pi}_{\Lb A}\|_{L^\infty(\Sigma_t)}\lesssim M^2(-t)^{-2}, \\\notag \|\mu^{-1} \leftexp{(T)}{\widehat{\widetilde{\slashed\pi}}}_{A B}\|_{L^\infty(\Sigma_t)}&\lesssim \delta M^2(-t)^{-3},\quad \|\mu^{-1}\text{tr}{}^{(T)}\widetilde{\slashed\pi}\|_{L^\infty(\Sigma_t)}\lesssim (-t)^{-1}.
\end{align}
\begin{equation}\label{Linfty of deformation tensor of Q}
\begin{split}
\|\mu^{-1}\leftexp{(Q)}{\tilde{\pi}}_{LL}\|_{L^{\infty}(\Sigma_{t})}&\lesssim 1+M^{2}(-t)^{-1},\ \|\leftexp{(Q)}{\tilde{\pi}}_{L\Lb}\|_{L^{\infty}(\Sigma_{t})}\lesssim 1+M^{2}(-t)^{-1},\ \|\leftexp{(Q)}{\tilde{\pi}}_{LA}\|_{L^{\infty}(\Sigma_{t})}\lesssim M^{2}(-t)^{-1},\\
\|\leftexp{(Q)}{\widehat{\tilde{\slashed{\pi}}}}_{AB}\|_{L^{\infty}(\Sigma_{t})}&\lesssim\delta M^{2}(-t)^{-2},\ \|\text{tr}\leftexp{(Q)}{\tilde{\slashed{\pi}}}\|_{L^{\infty}(\Sigma_{t})}\lesssim 1.
\end{split}
\end{equation}
Actually the estimate for $\tr\leftexp{(Q)}{\widetilde{\slashed{\pi}}}$ can be improved more precisely. Let us rewrite the following component of deformation tensor of $Q$:
\begin{align*}
\text{tr}\leftexp{(Q)}{\tilde{\slashed{\pi}}}=2c^{-1}t
\text{tr}_{\tilde{\slashed{g}}}
\widetilde{\chib}=2c^{-1}t
\text{tr}_{\tilde{\slashed{g}}}
\widetilde{\chib}'+4\frac{c^{-1}\ub}{\ub-t}-4(c^{-1}-1)-4
\end{align*}
This tells us:
\begin{align}\label{Improved estimate for trQ}
\|\text{tr}\leftexp{(Q)}{\tilde{\slashed{\pi}}}+4\|_{L^{\infty}(\Sigma_{t}^{\ub})}\lesssim\delta(-t)^{-1}.
\end{align}

To derive the pointwise estimates for the deformation tensors of $R_{i}$ requires more work. First, in view of \eqref{explicit expression of R i}, we have the following expressions:
\begin{equation}\label{deformation tensor of Rotational R_i in T,Lb,X_A}
\begin{split}
  ^{(R_i)}{}{\pi}_{\Lb\Lb}&=0,\  ^{(R_i)}{}{\pi}_{TT}= 2c^{-1}\mu \cdot R_i(c^{-1}\mu), \  ^{(R_i)}{}{\pi}_{\Lb T}= -R_i (\mu), \ \ ^{(R_i)}{}{\pi}_{AB}= -2\lambda_i \theta_{AB},\\
^{(R_i)}{}{\pi}_{T A}
&=-c^{-1}\mu (\theta_{AB}-\frac{\slashed{g}_{AB}}{\ub-t})R_i{}^B + c^{-1} \mu \varepsilon_{ikj} y^k X_A{}^{j} + \lambda_i X_A (c^{-1}\mu),\\
^{(R_i)}{}{\pi}_{\Lb A}&=-\chib_{AB}R_i{}^B + \Lb^k \varepsilon_{ikj}X_A{}^{j} + c\mu^{-1}\lambda_i \zetab_A=-(\chib_{AB}+\frac{\slashed{g}_{AB}}{\ub-t})R_i{}^B +  \varepsilon_{ikj}z^k X_A{}^{j} + c\mu^{-1}\lambda_i \zetab_A.
\end{split}
\end{equation}
The Latin indices $i,j,k$ are defined with respect to the Cartesian coordinates on $\Sigma_t$. To bound deformation of $R_{i}$, it suffices to control the $\lambda_i$'s, $y^i$'s and $z^i$'s.

First of all, we have
\begin{equation}\label{estimates on T hat i and Lb i}
|\widehat{T}^i|+|\Lb^i|\lesssim 1.
\end{equation}
The proof is straightforward: $g\big|_{\Sigma_t}$ is flat and $\widehat{T}$ is the unit normal of $S_{t,\ub}$ in $\Sigma_t$, so $|\widehat{T}^i|\leq 1$. In the Cartesian coordinates $(t,x^1,x^2,x^3)$, $\Lb = \partial_t - c\widehat{T}^i \partial_i$, so $|\Lb^i|\lesssim 1$.

Let $r=(\sum_{i=1}^3 x^i)^{\frac{1}{2}}$. Since $Tr =  c^{-1}\mu \sum_{i=1}^3 \frac{x^i \widehat{T}^i}{r}$,
\eqref{estimates on T hat i and Lb i} implies that $|T r| \lesssim 1+ M^2(-t)^{-1}$. We then integrate from $0$ to $\ub$,
since $r=-t$ when $\ub = 0$ and $|\ub| \leq \delta$, we obtain $|r+t|\lesssim \delta M^2$. In application, for sufficiently small $\delta$,
we often use $r \sim |t|$. The estimate can also be written as
\begin{equation}\label{estimates on r precise}
\left|\frac{1}{r} - \frac{1}{\ub+|t|}\right|\lesssim \delta M^{2}(-t)^{-2}.
\end{equation}
To control $\lambda_i$, we consider its $\Lb$ derivative. By definition $\lambda_i = g(\Omega_i, \widehat{T})$,
we can write its derivative along $\Lb$ as $\Lb \lambda_i = \sum_{k=1}^3 (\Omega_i){}^k \Lb \widehat{T}^k=\sum_{k=1}^3 (\Omega_i){}^k X^A(c) X_A{}^k$.
As $|t| \sim r$, we have  $|\Omega_i| \lesssim |t|$, this implies
\begin{equation}\label{estimates on Lb lambda_i}
 \|\Lb \lambda_i\|_{L^\infty(\Sigma_t)} \lesssim \delta M^2(-t)^{-2}.
\end{equation}
Since $\lambda_i = 0$ on $\Sigma_{-r_0}$, we have
\begin{equation}\label{estimates on lambda_i}
 \|\lambda_i\|_{L^\infty(\Sigma_t)} \lesssim \delta M^2(-t)^{-1}.
\end{equation}
To control $y^i$'s and $z^i$'s, let $\overline{y}=(y^1,y^2,y^3)$ and $\overline{x}=(x^1,x^2,x^3)$, we then have
\begin{equation}\label{y estimates 1}
 |\overline{y}-(\frac{1}{r}-\frac{1}{\ub-t})\overline{x}|^2=|(g(\widehat{T},\partial_r)-1)\partial_r|^2+\frac{1}{r^{2}}\sum_{i=1}^{3}\lambda_{i}^{2}
\end{equation}
On the other hand, we have $1-|g(\widehat{T},\partial_{r})|^{2}=\frac{1}{r^{2}}\sum_{i=1}^{3}\lambda_{i}^{2}\lesssim \delta M^{2}(-t)^{-4}$.
While on $S_{t,0}$,
Since $g(\partial_{r},\widehat{T})=1$ on $S_{t,0}$, for sufficiently small $\delta$, the angle between $\partial_{r}$ and $\widehat{T}$ is less than $\frac{\pi}{2}$, which implies $1+g(\widehat{T},\partial_{r})\geq 1$. Therefore,
\begin{align}\label{angle estimate}
|1-g(\widehat{T},\partial_{r})|\lesssim \delta M^{2}(-t)^{-4}.
\end{align}
Together with
\eqref{estimates on lambda_i} and \eqref{y estimates 1}, this implies
\begin{equation}\label{estimates on y^i}
 |y^i|\lesssim \delta M^2(-t)^{-2},\quad |y'^{i}|\lesssim \delta M^{2}(-t)^{-2}.
\end{equation}
We then control $z^i$ from its definition
\begin{equation}\label{estimates on z^i}
 |z^i|\lesssim \delta M^2(-t)^{-2}.
\end{equation}
The derivatives of $\lambda_{i}$ on $\Sigma_{t}$ are given by $X_{A}(\lambda_{i})=\left(\theta_{AB}-\frac{\slashed{g}_{AB}}{\ub-t}\right)R^{B}_{i}-\varepsilon_{ikj}y^{k}X^{j}_{A}$ and $T(\lambda_{i})=-R_{i}(c^{-1}\mu)$. Hence,
\begin{align*}
\|\ds\lambda_{i}\|_{L^{\infty}(\Sigma_{t})}\lesssim \delta M^{2}(-t)^{-2},\quad \|T\lambda_{i}\|_{L^{\infty}(\Sigma_{t})}\lesssim M^{2}(-t)^{-1}
\end{align*}
Finally, we obtain the following estimates for the deformation tensor of $R_i$:
\begin{equation*}
 \begin{split}
  ^{(R_i)}{}{\pi}_{\Lb \Lb}=0, \ \ \|{} ^{(R_i)}&{}{\widetilde{\pi}}_{\Lb T}\|_{L^\infty(\Sigma_t)}\lesssim M^2(-t)^{-1}, \ \  \|\mu^{-1}{}^{(R_i)}{}{{\pi}}_{TT}\|_{L^\infty(\Sigma_t)} \lesssim M^2(-t)^{-1},\\
\|{}^{(R_i)}{}{{\pi}}_{\Lb A}\|_{L^\infty(\Sigma_t)} &\lesssim \delta M^2(-t)^{-2},\ \ \|{}^{(R_i)}{}{\pi}_{T A}\|_{L^\infty(\Sigma_t)} \lesssim \delta M^4(-t)^{-2},\\
\|{}^{(R_i)}{}\widehat{{\slashed\pi}}_{AB}\|_{L^\infty(\Sigma_t)} &\lesssim \delta^2 M^4(-t)^{-4},\ \ \|\text{tr}{}^{(R_i)}{}{{\slashed\pi}}\|_{L^\infty(\Sigma_t)} \lesssim \delta M^4(-t)^{-2}.
  \end{split}
\end{equation*}
We use the relation $L = c^{-2}\mu \Lb + 2T$ to rewrite the above estimates in null frame as follows:
\begin{equation}\label{estimates on the deformation tensor of Rotational R_i in g metric}
 \begin{split}
 \|{} ^{(R_i)}{}{\pi}_{\Lb L}\|_{L^\infty(\Sigma_t)}&\lesssim M^2(-t)^{-1}, \ \  \|\mu^{-1}{}^{(R_i)}{}{\pi}_{LL}\|_{L^\infty(\Sigma_t)} \lesssim M^2(-t)^{-1},\\
\|{}^{(R_i)}{}{\pi}_{\Lb A}\|_{L^\infty(\Sigma_t)} &\lesssim  \delta M^2(-t)^{-2},\ \ \|{}^{(R_i)}{}{\pi}_{L A}\|_{L^\infty(\Sigma_t)} \lesssim \delta M^4(-t)^{-2},\\
\|{}^{(R_i)}{}\widehat{\slashed\pi}_{AB}\|_{L^\infty(\Sigma_t)} &\lesssim \delta^2 M^4(-t)^{-4},\ \ \|\text{tr}{}^{(R_i)}{}{\slashed\pi}\|_{L^\infty(\Sigma_t)} \lesssim \delta M^4(-t)^{-2}.
  \end{split}
\end{equation}
The deformation tensors of $R_i$'s with respect to $\widetilde{g}$ are estimated by
\begin{equation}\label{estimates on the deformation tensor of Rotational R_i in g tilde}
 \begin{split}
  ^{(R_i)}{}\widetilde{\pi}_{\Lb \Lb}=0, \ \ \|{} ^{(R_i)}&{}{\widetilde{\pi}}_{\Lb L}\|_{L^\infty(\Sigma_t)}\lesssim M^2(-t)^{-1}, \ \  \|\mu^{-1}{}^{(R_i)}{}{\widetilde{\pi}}_{LL}\|_{L^\infty(\Sigma_t)} \lesssim M^2(-t)^{-1},\\
\|{}^{(R_i)}{}{\widetilde{\pi}}_{\Lb A}\|_{L^\infty(\Sigma_t)} &\lesssim \delta M^2(-t)^{-2},\ \ \|{}^{(R_i)}{}\widetilde{\pi}_{L A}\|_{L^\infty(\Sigma_t)} \lesssim \delta M^4(-t)^{-2},\\
\|{}^{(R_i)}{}\widehat{\widetilde{\slashed\pi}}_{AB}\|_{L^\infty(\Sigma_t)} &\lesssim \delta^2 M^4(-t)^{-4},\ \ \|\text{tr}{}^{(R_i)}{}{\widetilde{\slashed\pi}}\|_{L^\infty(\Sigma_t)} \lesssim \delta M^4(-t)^{-2}.
  \end{split}
\end{equation}
\subsubsection{Applications of the estimates on $\lambda_i$ and $y^i$} As in \cite{M-Y}, based on calculations in [Ch-Shocks] and [Ch-Miao], we are able to show that the $R_i$ derivatives are equivalent to the $(-t)\slashed{d}$ and $(-t)\nablaslash$ derivative.
For a 1-form $\xi$ on $S_{t,\ub}$, we have $\sum_{i=1}^3 \xi(R_i)^2 = r^2\Big(|\xi|^2-\big(\xi(y')\big)^2\Big)$.
This is indeed can be derived from the formula $\sum_{i=1}^{3}(R_{i})^{a}(R_{i})^{b}=r^{2}(\delta_{cd}-y^{\prime c}y^{\prime d})\Pi^{a}_{c}\Pi^{b}_{d}$, where $a,b,c,d\in\{1, 2, 3\}$.
In view of \eqref{estimates on r precise}, \eqref{estimates on y^i} and the definition of $y^{\prime i}$, for sufficiently small $\delta$,
we have $\sum_{i=1}^3 \xi(R_i)^2 \sim  r^2 |\xi|^2$. Since $r$ is bounded below and above by $(-t)$,
we obtain $\sum_{i=1}^3 \xi(R_i)^2 \sim   t^{2}|\xi|^2$. Similarly, for a $k$-covariant tensor $\xi$ on $S_{t,\ub}$,
we have $\sum_{i_1,i_2,\cdots, i_k=1}^3 \xi(R_{i_1},R_{i_2},\cdots,R_{i_k})^2 \sim   t^{2}|\xi|^2$.
In particular, we can take $\xi = \slashed{d}\psi$, therefore, $\sum_{i=1}^3 (R_i\psi)^2 \sim   t^{2}|\slashed{d}\psi|^2$.
Henceforth, we omit the summation and write schematically as $|R_i\psi| \sim  (-t) |\slashed{d}\psi|$.

We can also compare the $R_i$-derivatives with the $\nablaslash$-derivatives for tensors. For $S_{t,\ub}$-tangential 1-form $\xi$ and vectorfield $X$,  let $\slashedLRi \xi$ be the orthogonal projection of the Lie derivative $\mathcal{L}_{R_i} \xi$ onto the surface $S_{t,\ub}$. Since $(\slashedLRi \xi)(X) = (\nablaslash_{R_i}\xi)(X) + \xi(\nablaslash_X R_i)$, we obtain
\begin{equation*}
\sum_{i=1}^3|\slashedLRi \xi|^2 =\sum_{i=1}^3|\xi(R_{i})|^2 + 2 \sum_{i=1}^3 \xi^k (\nablaslash_{R_i} \xi)_a (\nablaslash R_i)^a{}_k + \sum_{i=1}^3 \xi^k \xi^l (\nablaslash R_i)_a {}_k (\nablaslash R_i)^a{}_l.
\end{equation*}
We also have $\sum_{i=1}^3|\nablaslash_{R_i}\xi|^2 = r^2(\delta^{cd}-y'^c y'^d)(\nablaslash \xi)_{ca}(\nablaslash \xi)_d{}^a$. In view of the estimates on $y^{\prime i}$, for sufficiently small $\delta$, we obtain
\begin{equation*}
\sum_{i=1}^3|\nablaslash_{R_i}\xi|^2  \gtrsim r^{2}|\nablaslash \xi|^2.
\end{equation*}
Let $\varepsilon_{ijk}$ be the volume form on $\Sigma_t$ and $v_i$ be a $S_{t,\ub}$ 1-form with
rectangular components $(v_{i})_{a}=\Pi_{a}^{b}\varepsilon_{ibk}\xi_{k}$. By virtue of the formula $
(\slashed{\nabla}R_{i})^{k}_{l}=\Pi^{m}_{k}\Pi^{n}_{l}
\varepsilon_{inm}-\lambda_{i}\theta_{kl}$, we have
\begin{equation*}
\sum_{i=1}^3 \xi^k \xi^l (\nablaslash R_i)_a {}_k (\nablaslash R_i)^a{}_l =|\xi|^2+ 2c \sum_{i=1}^3 \lambda_i \xi\cdot\chib\cdot v_i + c^2 |\chib \cdot \xi|^2 \sum_{i=1}^3 \lambda_i^2.
\end{equation*}
In view of the estimates on $\lambda_i$, for sufficiently small $\delta$, we have $\sum_{i=1}^3 \xi^k \xi^l (\nablaslash R_i)_a {}_k (\nablaslash R_i)^a{}_l  \sim |\xi|^2$.
Similarly, we have $
\big|\sum_{i=1}^3 \xi^k (\nablaslash_{R_i} \xi)_a (\nablaslash {R_i})^b{}_k \big| \lesssim
r|\xi||\nablaslash \xi |$.
Finally, we conclude that
\begin{equation*}
|\xi|^2 + r^{2}|\nablaslash \xi|^2 \lesssim \sum_{i=1}^3 |\slashedLRi\xi|^2 \lesssim |\xi|^2 + r^{2}|\nablaslash \xi|^2 .
\end{equation*}
 Henceforth, we omit the summation and write schematically as  $|\slashedLRi\xi| \sim   |\xi|+(-t)|\nablaslash\xi|$. Similarly, for a tracefree symmetric 2-tensors $\theta_{AB}$ tangential to $S_{t,\ub}$, we have
 $|\theta|+|\nablaslash \theta| \lesssim |\slashedLRi\theta| \lesssim  |\theta|+(-t)|\nablaslash\theta|$. This will be applied to $\theta = \widehat{\chib}_{AB}$ later on.

\subsubsection{Sobolev inequalities and elliptic estimates}
To obtain the Sobolev inequalities on $S_{t,\ub}$, we introduce
$$I(t,\ub) = \displaystyle \sup_{U \in S_{t,\ub}, \atop \partial U \text{ is } C^1} \frac{\min\big(|U|, |S_{t,\ub}-U|\big)}{|\partial U|^2}$$
the isoperimetric constant on $S_{t,\ub}$, where $|U|$, $|S_{t,\ub}-U|$ and $|\partial U|$ are the measures of the corresponding sets with respect to $\slashed{g}$ on $S_{t,\ub}$.
Therefore, in view of the fact that $R_i \sim r\nablaslash$, for sufficiently small $\delta$, we have the following Sobolev inequalities:
\begin{align}\label{isoperimetric}
\begin{split}
\|f\|_{W^{1,4}(S_{t,\ub})} &\lesssim |I(t,\ub)|^{\frac{1}{4}}|S_{t,\ub}|^{-\frac{3}{4}}\big(\|f\|_{L^2(S_{t,\ub})}+ \|R_i f\|_{L^2(S_{t,\ub})} + \|R_i R_j f\|_{L^2(S_{t,\ub})}\big),\\
\|f\|_{L^{\infty}(S_{t,\ub})} &\lesssim |I(t,\ub)|^{\frac{3}{4}}|S_{t,\ub}|^{-\frac{1}{2}}\big(\|f\|_{L^2(S_{t,\ub})}+ \|R_i f\|_{L^2(S_{t,\ub})} + \|R_i R_j f\|_{L^2(S_{t,\ub})}\big).
\end{split}
\end{align}
where $\|f\|_{W^{1,4}(S_{t,\ub})}$ is defined as $
\|f\|_{W^{1,4}(S_{t,\ub})}=|S_{t,\ub}|^{-1/2}\|f\|_{L^{4}(S_{t,\ub})}+\|\slashed{d}f\|_{L^{4}(S_{t,\ub})}$.
It remains to control the isoperimetric constant $I(t,\ub)$.

We use $T$ to generate a diffeomorphism of $S_{t,\ub}$ to $S_{t,0}$ which maps $U$, $S_{t,\ub}-U$ and $\partial U$ to corresponding sets $U_{\ub}$, $S_{t,0}-U_{\ub}$ and $\partial U_{\ub}$ on $S_{t,0}$. Let $U_{\ub'}, S_{t,\ub'}-U_{\ub'}, \partial U_{\ub'}$ be the inverse images of these on each $S_{t,\ub'}$ for $\ub'\in[0,\ub]$. Since $\slashed{\mathcal{L}}_T \gslash_{AB} = 2c^{-1}\mu\theta = -2c^{-2}\mu\chib_{AB}$, for $\ub' \in [0,\ub]$, we obtain
\begin{align*}
\frac{d}{d\ub'}\big(|U_{\ub'}|\big) = -\int_{U_{\ub'}}c^{-2}\mu \tr\chib d\mu_{\gslash}, \ \ \frac{d}{d\ub'}\big(|\partial U_{\ub'}|\big) = -\int_{U_{\ub'}}c^{-2}\mu \chib(\nu,\nu) ds,
\end{align*}
where $\nu$ is the unit normal of $\partial U_{\ub'}$ in $S_{t,0}$ and $ds$ the element of arc length of $\partial U_{\ub'}$.  In view of the estimates on $\chib$ and $\mu$ derived before,  for sufficiently small $\delta$, we have
\begin{align*}
\frac{d}{d \ub'}\big(|U_{\ub'}|\big) \lesssim |U_{\ub'}|, \ \ \frac{d}{d\ub'}\big(|\partial U_{\ub'}|\big) \lesssim |\partial U_{\ub'}|.
\end{align*}
Therefore, by integrating from $0$ to $\ub$, we have
\begin{align*}
|U| \sim |U_{\ub}|, \ \ |\partial U| \sim |\partial U_{\ub}|.
\end{align*}
Hence, $I(t,\ub) \sim I(t,0) \sim 1$. Finally, since $|S_{t,\ub}|\sim t^{2}$, we conclude that
\begin{equation}\label{Sobolev on S_t ub}
\begin{split}
(-t)^{3/2}\|f\|_{W^{1,4}(S_{t,\ub})} + (-t)\|f\|_{L^{\infty}(S_{t,\ub})} &\lesssim  \|f\|_{L^2(S_{t,\ub})}+ \|R_i f\|_{L^2(S_{t,\ub})} + \|R_i R_j f\|_{L^2(S_{t,\ub})}.
\end{split}
\end{equation}
We remark that, similarly, we have
\begin{equation}\label{Sobolev on S_t ub L4}
\|f\|_{L^{4}(S_{t,\ub})} \lesssim  \|f\|_{L^2(S_{t,\ub})}+ \|R_i f\|_{L^2(S_{t,\ub})}.
\end{equation}

As in \cite{M-Y}, we also have the following elliptic estimates for traceless two-tensors.
\begin{lemma}
If $\delta$ is sufficiently small, for any traceless 2-covariant symmetric tensor $\theta_{AB}$ on $S_{t,\ub}$, we have
\begin{equation}\label{elliptic estimates}
\int_{S_{t,\ub}} \mu^2 \big(|\nablaslash \theta|^2 + |\theta|^2\big)d\mu_{\slashed{g}}\lesssim \int_{S_{t,\ub}} \mu^2 |\divslash \theta|^2 + |\ds\mu|^2 |\theta|^2 \slashed{d}\mu_{\slashed{g}}
\end{equation}
\end{lemma}
\section{The behavior of the inverse density function}
As in \cite{M-Y}, the behavior of the inverse density function $\mu$ also plays an dominant r\^{o}le in this paper. The method of obtaining estimates on $\mu$ is to relate $\mu$ to its initial value on $\Sigma_{t=-r_0}$. Besides the behavior with respect to $\delta$, in this paper we will also take into account the behavior with respect to $t$. Since the metric $g$ depends only on $\psi_0 = \partial_t \phi$, $\mu$ is also determined by $\psi_0$. This leads naturally to the study of the wave equation $\Box_{\widetilde{g}} \psi_0 =0$. We rewrite it in the null frame as
\begin{equation}\label{box psi_0 in null frame}
\Lb(L\psi_0) + \frac{1}{2}\tr\chibt\cdot L\psi_0 +\left(-\mu\laplacianslash \psi_0 +\frac{1}{2}\tr\widetilde{\chi}\cdot \Lb\psi_0 + 2 \zetab \cdot \slashed{d}\psi_0 + \mu \slashed{d}\log(c)\cdot \slashed{d}\psi_0\right)=0.
\end{equation}
\subsection{The asymptotic expansion for $\mu$}
\begin{lemma}
For sufficiently small $\delta$, we have
\begin{equation}\label{expansion of L psi_0}
\big|(-t) L\psi_0(t,\ub, \theta)-r_0 L \psi_0(-r_0,\ub, \theta)\big| \lesssim \delta^{\frac{1}{2}} M^3(-t)^{-1}.
\end{equation}
\end{lemma}
\begin{proof}
We regard \eqref{box psi_0 in null frame} as a transport equation for $L\psi_0$. According to (B.1) and the estimates from previous sections,
the $L^\infty$ norm of the terms in the big parenthesis in \eqref{box psi_0 in null frame} is bounded by $\delta^{\frac{1}{2}}M^{3}$. Hence,
\begin{equation*}
|\Lb(L\psi_0)(t,\ub, \theta) + \frac{1}{2}\tr\chibt(t,\ub, \theta)\cdot L\psi_0(t,\ub,\theta) | \lesssim  \delta^\frac{1}{2}M^{3}(-t)^{-3}.
\end{equation*}
By virtue of \eqref{precise bound on chibt}, this implies $|\Lb\big(L\psi_0\big)(t,\ub, \theta) -\frac{1}{\ub-t}  L\psi_0(t,\ub,\theta) | \lesssim \delta^\frac{1}{2}M^{3}(-t)^{-3}$. Therefore, we obtain
\begin{equation*}
|\Lb\big((\ub-t)L\psi_0\big)(t,\ub, \theta)| \lesssim \delta^\frac{1}{2}M^{3}(-t)^{-2}.
\end{equation*}
Since $|\ub| \leq \delta$, we integrate from $-r_0$ to $t$ and this yields the desired estimates.
\end{proof}
\begin{remark}
The estimates \ref{expansion of L psi_0} also hold for $R_{i}L\psi_{0}$ or $R_{i}R_{j}\psi_{0}$, e.g., see \eqref{expansion of L R_i psi_0}. To derive these estimates, we commute $R_{i}$'s with \eqref{box psi_0 in null frame} and follow the same way as in the above proof.
\end{remark}
Since $L = c^{-2}\mu \Lb + 2 T$, as a corollary, we have
\begin{corollary}
For sufficiently small $\delta$, we have
\begin{equation}\label{expansion of T psi_0}
|(-t) T\psi_0(t,\ub, \theta)-r_0 T \psi_0(-r_0,\ub, \theta)| \lesssim \delta^{\frac{1}{2}} M^3(-t)^{-1},
\end{equation}
\begin{equation}\label{expansion of psi_0}
|(-t) \psi_0(t,\ub, \theta)-r_0 \psi_0(-r_0,\ub, \theta)| \lesssim \delta^{\frac{3}{2}} M^3(-t)^{-1}.
\end{equation}
\end{corollary}
We turn to the behavior of $\Lb \mu$.
\begin{lemma}\label{lemma on the expansion of Lb mu}
For sufficiently small $\delta$, we have
\begin{equation}\label{expansion of Lb mu}
|t^2\Lb \mu(t,\ub, \theta)-r_0^2\Lb \mu(-r_0,\ub, \theta)|\lesssim \delta M^4(-t)^{-1}.
\end{equation}
\end{lemma}
\begin{proof}
According to \eqref{Structure Equation Lb mu}, we write $t^2\Lb \mu(t,\ub, \theta)-r_0^2\Lb \mu(-r_0,\ub, \theta)$ as
\begin{align*}
\left(t^2 m(t,\ub, \theta)-r_0^2 m(-r_0,\ub, \theta)\right) +\left[ t^2\ (\mu \cdot e)(t,\ub, \theta)-r_0^2(\mu \cdot e)(-r_0,\ub, \theta)\right].
\end{align*}
In view of \eqref{bound on e}, we bound the terms in the bracket by $\delta M^4(-t)^{-1}$ up to a universal constant. Therefore,
\begin{align*}
t^2\Lb \mu(t,\ub, \theta)-r_0^2\Lb \mu(-r_0,\ub, \theta) &=\left(t^2 m(t,\ub, \theta)-r_0^2 m(-r_0,\ub, \theta)\right) +O\left(\frac{\delta M^4}{-t}\right).
\end{align*}
Since
\begin{equation*}
t^2 m(t,\ub, \theta)-r_0^2 m(-r_0,\ub, \theta)=\tau^2 \frac{3G''(0)(\psi_0(\tau,\ub, \theta) \cdot T \psi_0(\tau,\ub, \theta))}{(1+3G''(0)\psi_0^2(\tau,\ub, \theta))^{2}}\bigg|_{\tau=-r_{0}}^{\tau=t}.
\end{equation*}
It is clear that the estimates follow immediately after \eqref{expansion of T psi_0} and \eqref{expansion of psi_0}.
\end{proof}
We now are able to prove an accurate estimate on $\mu$.
\begin{proposition}\label{proposition on expansion for mu}
For sufficiently small $\delta$, we have
\begin{equation}\label{expansion of mu}
\left|\mu(t,\ub, \theta)-1 + r_0^2(\frac{1}{t}+\frac{1}{r_0})\Lb\mu(-r_0,\ub, \theta) \right|\lesssim \delta M^4(-t)^{-2}.
\end{equation}
In particular, we have $\mu \leq C_0$ where $C_0$ is a universal constant depending only on the initial data.
\end{proposition}
\begin{proof}
According to the previous lemma, we integrate $\Lb \mu$:
\begin{align*}
\mu(t,\ub,\theta)-\mu(-r_0,\ub,\theta)&=\int_{-r_0}^t \Lb \mu(\tau,\ub, \theta)d\tau = \int_{-r_0}^t \frac{\tau^2\Lb \mu(\tau,\ub, \theta)}{\tau^2}d\tau\\
&\stackrel{\eqref{expansion of Lb mu}}{=}\int_{-r_0}^t \frac{r_0^2\Lb \mu(-r_0,\ub, \theta)}{\tau^2} + \frac{O (\delta M^4)}{-\tau^{3}}d\tau\\
&=-(\frac{1}{t}+\frac{1}{r_0})r_0^2\Lb\mu(-r_0,\ub, \theta) +O\left(\frac{\delta M^4}{t^{2}}\right).
\end{align*}
Therefore, we can use \eqref{initial estimates on mu} to conclude.
\end{proof}
We are ready to derive two key properties of the inverse density function $\mu$. The first asserts that the shock wave region is
trapping for $\mu$.
\begin{proposition}\label{Proposition C3}
For sufficiently small $\delta$ and for all $(t,\ub, \theta) \in W_{shock}$, we have
\begin{equation}\label{C3}
\Lb \mu (t,\ub,\theta) \leq -\frac{1}{4|t|^{2}}.
\end{equation}
\end{proposition}
\begin{proof}
For $(t,\ub, \theta) \in W_{shock}$, we have $\mu(t,\ub, \theta) < \frac{1}{10}$. In view of \eqref{expansion of mu},
we claim that $r_0^2\Lb\mu (-r_0,\ub, \theta)<0$. Otherwise, since $\frac{1}{t}+\frac{1}{r_0}<0$, we would have $\mu(t,\ub, \theta) \geq 1 + O(\delta M^2)>\frac{1}{10}$,
provided $\delta$ is sufficiently small. This contradicts the fact that $\mu(t,\ub, \theta) < \frac{1}{10}$.

We can also use this argument to show that $(\frac{1}{t}+\frac{1}{r_0})r_0^2\Lb\mu(-r_0,\ub, \theta) \geq \frac{1}{2}$. Otherwise, for sufficiently small $\delta$, we would have $\mu(t,\ub, \theta) \geq \frac{1}{2} + O(\delta M^2)>\frac{1}{10}$.

Therefore, we obtain $r_0^2\Lb\mu(-r_0,\ub, \theta) \leq \frac{1}{2}\frac{r_0 t}{r_0+t}$. In view of \eqref{expansion of Lb mu}, we have
\begin{equation*}
t^2 \Lb\mu(t,\ub,\theta)\leq  \frac{1}{2}\frac{r_0 t}{r_0+t}+O\left(\frac{\delta M^2}{-t}\right).
\end{equation*}
By taking a sufficiently small $\delta$ and noticing that $\frac{r_{0}t}{r_{0}+t}$ is bounded from above by a negative number, this yields the desired estimates.
\end{proof}

\begin{remark}
As in \cite{M-Y}, compared to the estimates \eqref{bound on mu} $|\Lb \mu| \lesssim M^2(-t)^{-2}$, \eqref{expansion of Lb mu} $|\Lb \mu| \leq C_0(-t)^{-2} + \delta M^2(-t)^{-3}$, where $C_0$ depends only on the initial data, gives us a more precise estimate. The improvement comes from integrating the wave equation $\Box_{\widetilde{g}}\psi_0 =0$
or equivalently \eqref{box psi_0 in null frame}.
\end{remark}
\subsection{The asymptotic expansion for derivatives of $\mu$}
We start with an estimate on derivatives of $\tr\chibt$.
\begin{lemma}
For sufficiently small $\delta \leq \varepsilon$, we have
\begin{equation}\label{bound on L trchibt}
 \|L \tr_{\widetilde{g}}\widetilde{\chib} \|_{L^{\infty}} \lesssim  M^2(-t)^{-1},
\end{equation}
\begin{equation}\label{bound on d trchibt}
 \|\ds \tr \widetilde{\chib} \|_{L^{\infty}} \lesssim \delta  M^2(-t)^{-4}.
\end{equation}
\end{lemma}

\begin{proof}
We derive a transport equation for $R_i \chib'_{AB}$ by commuting $R_i$ with \eqref{transport equation for chib prime}:
\begin{equation*}
\Lb (R_i \chib'_{AB}) = [\Lb ,R_i] \chib'_{AB} + e R_i \chib'_{AB} + 2 \chib'_{A}{}^{C}R_i \chib'_{BC}+(R_i e) \cdot \chib'_{AB} - \frac{e}{\ub-t}R_i \slashed{g}_{AB} - R_i\big(\frac{e}{\ub-t}\big) \slashed{g}_{AB} -R_i\alphab'_{AB}.
\end{equation*}
Since $[\Lb,R_i]^A = {}^{(R_i)}\pi_{\Lb}{}^A$, the commutator term $[\Lb ,R_i] \chib'_{AB}$ can be bounded by the estimates
on the deformation tensors. We then multiply both sides by $R_{i} \chib'_{AB}$ and repeat the procedure that we used to
 derive \eqref{precise bound on chib'}. Since it is routine, we omit the details and only give the final result
\begin{equation}\label{estimate for L chib}
 \|R_i\left(\chib_{AB}+\frac{\slashed{g}_{AB}}{\ub-t}\right)\|_{L^{\infty}} \lesssim \delta M^2(-t)^{-3}.
\end{equation}
In particular, this yields $\|R_i \tr\chib \|_{L^{\infty}} \lesssim   \delta M^2(-t)^{-3}$ which is equivalent to \eqref{bound on d trchibt}.

We derive a transport equation for $L\chib'_{AB}$ by commuting $L$ with \eqref{transport equation for chib prime}:
\begin{align*}
\Lb (L \chib'_{AB}) &= e L\chib'_{AB} + 2 \chib'_{A}{}^{C}L \chib'_{BC}+ L e \chib'_{AB} - \frac{e}{\ub-t}L \slashed{g}_{AB}\\
 &\ \ + L\big(\frac{e}{\ub-t}\big) \slashed{g}_{AB} -L \alphab'_{AB}+\Lb(c^{-2}\mu)\Lb\chib'_{AB}+\slashed{g}^{CD}(\zetab_{C}+\etab_{D})X_{C}(\chib'_{AB}).
\end{align*}
We then use Gronwall to derive
\begin{equation}\label{expansion for L chib}
 \|L\left(\chib_{AB}+\frac{\slashed{g}_{AB}}{\ub-t}\right)\|_{L^{\infty}} \lesssim M^2(-t)^{-3}.
\end{equation}
In particular, this yields $\|L \tr\chib \|_{L^{\infty}} \lesssim  M^2(-t)^{-1}$. This is equivalent to \eqref{bound on L trchibt}.
\end{proof}
We now derive estimates for $R_i\psi_{0}$.
\begin{lemma}
For sufficiently small $\delta$, we have
\begin{equation}\label{expansion of L R_i psi_0}
\big|(-t) L R_i \psi_0(t,\ub, \theta)-r_0 LR_i \psi_0(-r_0,\ub, \theta)\big| \lesssim \delta^{\frac{1}{2}} M^3(-t)^{-1}.
\end{equation}
\end{lemma}
\begin{proof}
We commute $R_i$ with \eqref{box psi_0 in null frame} and we obtain that $\Lb(R_i L  \psi_0) + \frac{1}{2}\tr\chibt\cdot R_i L \psi_0 = N$ with
\begin{align*}
N=  R_i\left(\mu\laplacianslash \psi_0 -\frac{1}{2}\tr\widetilde{\chi}\cdot \Lb \psi_0 - 2 \zetab \cdot \slashed{d} \psi_0 - \mu \slashed{d}\log(c)\cdot \slashed{d}\psi_0\right)  - \frac{1}{2}R_i \tr\chibt \cdot  L \psi_0+ [\Lb,R_i] L  \psi_0.
\end{align*}
 According to  (B.1) and the previous lemma, $N$ is bounded by $\delta^{\frac{1}{2}} M^3$. Hence, $\big|\Lb(R_i L  \psi_0) + \frac{1}{2}\tr\chibt\cdot R_i L  \psi_0\big| \lesssim  \delta^{\frac{1}{2}}M^3(-t)^{-3}$.
 We then integrate to derive
\begin{equation*}
\big||t| R_i L \psi_0(t,\ub, \theta)-r_0 R_i L \psi_0(-r_0,\ub, \theta)\big| \lesssim \delta^{\frac{1}{2}} M^3(-t)^{-1}.
\end{equation*}
The commutator $[R_i, \Lb] \psi_0$ is bounded by $\delta M^3(-t)^{-2}$ thanks to the estimates on deformation tensors. This completes the proof.
\end{proof}
Using this lemma, we can obtain a more accurate estimate for $\ds\mu$. 
\begin{lemma}
For sufficiently small $\delta$, we have
\begin{equation}\label{precise bound on d mu}
\|\ds\mu\|_{L^\infty(\Sigma_t)} \lesssim \frac{\left(1 + \delta M^4\right)}{t^{2}}.
\end{equation}
\end{lemma}
\begin{proof}
We commute $R_i$ with $\Lb \mu= m+ e\mu$ to derive
\begin{equation*}
\Lb(R_i \mu) = R_i m + \left(e R_i \mu+\mu R_i e +[\Lb,R_i]\mu\right).
\end{equation*}
According to (B.1) and the estimates on $L R_i \psi_{0}$ (needed to bound $R_i m$) from the previous lemma,
it is straightforward to bound the terms in the parenthesis by $\delta M^2$. Similar to \eqref{expansion of Lb mu}, we obtain
\begin{equation}\label{expansion of Lb R_i mu}
|(-t)^2\Lb\left(R_i\mu \right)(t,\ub,\theta)-r_0^2\Lb\left(R_i\mu \right)(-r_0,\ub,\theta)| \lesssim \delta M^4(-t)^{-1}.
\end{equation}
Since $\|[\Lb,R_{i}]\|_{L^{\infty}}\lesssim\delta M^{2}(-t)^{-2}$, we bound $R_i \mu$ as
\begin{align*}
\quad R_i\mu(t,\ub,\theta)-R_i\mu(-r_0,\ub,\theta)&\stackrel{\eqref{expansion of Lb R_i mu}}{=}\int_{-r_0}^t \frac{r_0^2\Lb\big( R_i\mu(\-r_0,\ub, \theta)\big)}{\tau^2} + \frac{O ( \delta M^4)}{-\tau^{3}}d\tau\\
&=-(\frac{1}{t}+\frac{1}{r_0})r_0^2\Lb\big(R_i\mu(-r_0,\ub, \theta)\big) +\frac{O(\delta M^4)}{t^{2}}.
\end{align*}
By using the relation between $R_{i}$ and $\ds$, this inequality yields \eqref{precise bound on d mu} for sufficiently small $\delta$.
\end{proof}
We can also obtain a better estimate for $L\mu$.
\begin{lemma}
For sufficiently small $\delta$, we have
\begin{equation}\label{precise bound on L mu}
\|L\mu\|_{L^\infty(\Sigma_t)} \lesssim \frac{\delta^{-1} + M^4}{-t}.
\end{equation}
\end{lemma}
\begin{proof}
By commuting $L$ with $\Lb \mu= m+ e\mu$, we obtain
\begin{equation*}
\Lb(L\mu) =Lm + \big[-2(\zetab^A+\etab^A)X_A(\mu)+\Lb(c^{-2}\mu)\Lb\mu + eL\mu + \mu L e\big].
\end{equation*}
According to (B.1), we can bound the terms in the bracket by $M^4(-t)^{-3}$. Hence,
\begin{align*}
&\quad |t|^2\Lb\big(L\mu \big)(t,\ub,\theta)-|r_0|^2\Lb\big(L\mu \big)(-r_0,\ub,\theta) =|t|^2 L m(t,\ub,\theta)-|r_0|^2L m(-r_0,\ub,\theta) + \frac{O(M^4)}{-t}.
\end{align*}
By the explicit formula of $m$, we can proceed exactly as in Lemma \ref{lemma on the expansion of Lb mu} and we obtain
\begin{equation}\label{expansion of Lb L mu}
||t|^2\Lb\big(L\mu \big)(t,\ub,\theta)-|r_0|^2\Lb\big(L\mu \big)(-r_0,\ub,\theta)| \lesssim M^4(-t)^{-1}.
\end{equation}
We then integrate along $\Lb$ and we have
\begin{align*}
L\mu(t,\ub,\theta)-L\mu(-r_0,\ub,\theta)&= \int_{-r_0}^t \frac{\tau^2\Lb\big( L\mu(\tau,\ub, \theta)\big)}{\tau^2}d\tau\stackrel{\eqref{expansion of Lb L mu}}{=}\int_{-r_0}^t \frac{r_0^2\Lb\big( L\mu(\-r_0,\ub, \theta)\big)}{\tau^2} + \frac{O ( M^4)}{-\tau^{3}}d\tau\\
&=-(\frac{1}{t}+\frac{1}{r_0})r_0^2\Lb\big(L\mu(-r_0,\ub, \theta)\big) +\frac{O(M^4)}{t^{2}}.
\end{align*}
For sufficiently small $\delta$, this implies \eqref{precise bound on L mu}.
\end{proof}
We now relate $L^2 \psi_0(t,\ub,\theta)$ to its initial value.
\begin{lemma}
For sufficiently small $\delta$, we have
\begin{equation}\label{expansion of L L psi_0}
\left||t| L^2\psi_0(t,\ub, \theta)-r_0 L^2 \psi_0(-r_0,\ub, \theta)\right| \lesssim \delta^{-\frac{1}{2}} M^3(-t)^{-1}.
\end{equation}
\end{lemma}
\begin{proof}
We commute $L$ with \eqref{box psi_0 in null frame} and we obtain the following transport equation for $L^2 \psi_0$:
\begin{equation}\label{box L psi_0 in null frame}
\begin{split}
 \Lb(L^2\psi_0) + \frac{1}{2}\tr\chibt\cdot L^2\psi_0  &=-\frac{1}{2}L\big(\tr\chibt\big) \cdot L\psi_0 + [\Lb,L]L\psi_0 \\
&+ L\Big(\mu\laplacianslash \psi_0 -\frac{1}{2}\tr\widetilde{\chi}\cdot \Lb\psi_0 - 2 \zetab \cdot \slashed{d}\psi_0 -\mu \slashed{d}\log(c)\cdot \slashed{d}\psi_0 \Big).
\end{split}
\end{equation}
The righthand side of the above equation can be expanded as
\begin{equation*}
\begin{split}
&L\tr\chibt L\psi_0 + \Lb(c^{-2}\mu)\Lb L\psi_0 + (\etab+\zetab)\slashed{d}L\psi_0 +\Big( L\mu \laplacianslash \psi_0 +\mu L \laplacianslash \psi_0 +L\tr\chibt \Lb \psi_0 + \tr\chibt L\Lb \psi_0 \\
& + L\zetab \slashed{d}\psi_0 + \zetab L\slashed{d}\psi_0 + L\mu \cdot \slashed{d}\log(c)\cdot \slashed{d}\psi_0 + \mu L\big(\slashed{d}\log(c)\big)\cdot \slashed{d}\psi_0 + \mu \slashed{d}\log(c)\cdot L\slashed{d}\psi_0\Big).
\end{split}
\end{equation*}
Since the exact numeric constants and signs for the coefficients are irrelevant for estimates, we replace all of them by $1$ in the above expressions.

Since $\zetab_A = -c^{-1}\mu X_A(c)$, by applying $L$ and using \eqref{precise bound on L mu}, we obtain $|L\zetab| \lesssim M^2$. Therefore, according \eqref{bound on L trchibt}, (B.1) and the estimates derived previously in this section, we can bound all the terms on the right hand side and we obtain
\begin{equation*}
|\Lb(L^2\psi_0) + \frac{1}{2}\tr\chibt\cdot L^2\psi_0| \lesssim \delta^{-\frac{1}{2}}M^3(-t)^{-3}.
\end{equation*}
We then integrate from $-r_0$ to $t$ to obtain \eqref{expansion of L L psi_0}.
\end{proof}
Following the same procedure, we have:
\begin{lemma}
For sufficiently small $\delta$, we have
\begin{equation}\label{bound on L L trchibt}
 \|L^2 \tr_{\widetilde{g}}\widetilde{\chib} \| \lesssim  M^2\delta^{-1}(-t)^{-1},
\end{equation}
\begin{equation}\label{expansion of L L L psi_0}
||t| L^3 \psi_0(t,\ub, \theta)-r_0 L^3 \psi_0(-r_0,\ub, \theta)| \lesssim \delta^{-\frac{3}{2}} M^3(-t)^{-1}.
\end{equation}
\end{lemma}
We omit the proof since it is routine. Similarly, we commute $L$ twice with $\Lb \mu= m+ e\mu$, we can use \eqref{bound on L L trchibt} and \eqref{expansion of L L L psi_0} to obtain
\begin{lemma}
There exists $\varepsilon = \varepsilon(M)$ so that for all $\delta \leq \varepsilon$, we have
\begin{equation}\label{precise bound on L L mu}
\|L^2\mu\|_{L^\infty(\Sigma_t)} \lesssim \left(\delta^{-2} + \delta^{-1}M^2\right)\left(-t\right)^{-1},
\end{equation}
\begin{equation}\label{precise bound on T T mu}
\|T^2\mu\|_{L^\infty(\Sigma_t)} \lesssim \left(\delta^{-2} + \delta^{-1}M^2\right)\left(-t\right)^{-1}.
\end{equation}
\end{lemma}

We turn to the improved estimate for $\mu^{-1}T\mu$. 
\begin{proposition}\label{Proposition C2}
Let $s$ be such that $t<s<t^{*}$. For $p=(t,\ub, \theta) \in W_\delta$, let $\big(\mu^{-1}T \mu\big)_+$ be the nonnegative part of $\mu^{-1}T \mu$. For sufficiently small $\delta $ and for all $p \in W_{shock}$, we have
\begin{equation}\label{C2}
\big(\mu^{-1}T \mu\big)_+ (t,\ub,\theta) \lesssim \frac{1}{|t-s|^{\frac{1}{2}}}\delta^{-1}(-t)^{1/2}.
\end{equation}
\begin{proof}
As in \cite{M-Y}, by a maximum principle argument, we have:
\begin{equation}\label{eq 1}
\|\mu^{-1}(T\mu)_+ \|_{L^\infty([0,\delta])}\leq \sqrt{\frac{\|T^2\mu\|_{L^\infty([0,\delta])}}{\displaystyle \inf_{\ub \in[0,\delta]}\mu(\ub)}}.
\end{equation}
From the previous lemma, we have:
\begin{align}\label{eq 2}
\|T^2\mu\|_{L^\infty([-\delta,\delta])} \lesssim \frac{1}{|t|}\delta^{-2}.
\end{align}
For $\inf \mu$, we assume $(t,\ub,\theta)\in W_{shock}$. According to \eqref{C3}, the condition $(t,\ub,\theta)\in W_{shock}$ implies $|t|^{2}(\Lb\mu)(t,\ub,\theta)\leq -\frac{1}{4}$. Therefore 
we have
\begin{align*}
\mu(t,\ub, \theta) &= \mu (s, \ub,\theta)- \int_{t}^{s}\Lb \mu(\tau,\ub, \theta)\geq - \int_{t}^{s}\Lb \mu(\tau,\ub, \theta)\\
\geq &\int_{t}^{s}\frac{1}{4\tau^{2}}d\tau\geq \frac{s-t}{4t^{2}}.
\end{align*}
Together with \eqref{eq 1} and \eqref{eq 2}, this completes the proof.
\end{proof}
\end{proposition}

\section{Energy estimates for linear equation}
In this section we establish the energy estimates for the following inhomogeneous equation
\begin{align}\label{model linear wave equation} 
\Box_{\widetilde{g}}\psi=\rho
\end{align}
Since $\widetilde{g}$ depends on $\phi$, we will need to handle the error terms contributed by the deformation tensors of the multiplier vectorfields. In addition, since we need to deal with the time decay, we have to modify the multipliers $K_{0}$ and $K_{1}$ in \cite{M-Y}, which will be specified later in this section. 
\subsection{Energy and Flux}
As usual, we introduce the energy-momentum tensor, which is the same with respect to $g$ and $\gt$:
\begin{align}\label{energymomentum tensor}
T_{\mu\nu}(\psi):=\partial_{\mu}\psi\partial_{\nu}\psi-\frac{1}{2}g_{\mu\nu}g^{\alpha\beta}\partial_{\alpha}\psi\partial_{\beta}\psi,
\end{align}
and we have the decomposition for $T_{\mu\nu}$ with respect to the null frame $(L,\Lb, X_{1}, X_{2})$:
\begin{equation}\label{energy momentum tensor in null frame}
 \begin{split}
  T_{LL} &= (L\psi)^2,\ T_{\Lb\Lb} = (\Lb\psi)^2, \ T_{\Lb L} = \mu |\slashed{d}\psi|^2, \ T_{LA} = L\psi \cdot X_A(\psi),\\
T_{\Lb A} &= \Lb \psi \cdot X_A(\psi),\ T_{AB} = X_A(\psi)X_B(\psi)-\frac{1}{2}\slashed{g}_{AB}(-\mu^{-1}L\psi \Lb \psi + |\slashed{d} \psi|^2).
 \end{split}
\end{equation}  
We use two \emph{multiplier vectorfields} $K_0 = L+\Lb$ and $K_1 = \left(\dfrac{2(-t)}{\widetilde{\tr}\chib}\right)\Lb$. Following a similar argument as in \cite{M-Y}, the associated energy and flux for $K_{0}$ are given by:
\begin{align}\label{energy flux K0}
E^0(t,\ub)=\int_{\Sigma_{t}^{\ub}}\frac{1}{2c}\left((L\psi)^2 +c^{-2}\mu(\Lb\psi)+ (\mu+c^{-2}\mu^2) |\slashed{d} \psi|^2\right),\ \  F^0(t,\ub)=\int_{\Cb_{\ub}{}^{t}}\frac{1}{c}\left((\Lb\psi)^{2}+\mu |\slashed{d} \psi|^2\right),
\end{align}
which satisfy
\begin{align*}
E^0(t,\ub)\sim {E}(\psi)(t,\ub),\ \  F^0(t,\ub)\sim {F}(\psi)(t,\ub)
\end{align*}
The energy estimates will be based on the following identity:
\begin{align}\label{energy identity E0}
E^{0}(t,\ub)-E^{0}(-r_0,\ub)+F^{0}(t,\ub) &= \int_{W_t^{\ub}} c^{-2} \widetilde{Q_0}
\end{align}
where
\begin{align*}
\widetilde{Q}_0 := -\rho\cdot K_0 \psi-\frac{1}{2}\widetilde{T}^{\mu\nu}\widetilde{\pi}_{0,\mu\nu},
\end{align*}
with $\widetilde{\pi}_{0,\mu\nu}=(\mathcal{L}_{K_{0}}\gt)_{\mu\nu}$ being the deformation tensor of $K_{0}$. In the above identity, the spacetime integral is defined as follows:
\begin{align}\label{spacetime integral definition}
\int_{W_t^{\ub}}f  = \int_{-r_0}^{t} \int_{0}^{\ub} \left(\int_{S_{\tau,\ub'}} \mu \cdot f(\tau,\ub',\theta)\dmug\right) d\ub' d\tau.
\end{align}

The discussion for $K_{1}$ is much more complicated. First, a standard calculation as in \cite{M-Y} implies the following energy flux associated to $K_{1}$:
\begin{align}\label{energy flux K1}
E^{1}(t,\ub)&=\int_{\Sigma_{t}^{\ub}}\frac{1}{2c}\left(\frac{2(-t)}{\widetilde{\tr\chib}}\left(c^{-2}\mu(\Lb\psi)^{2}+\mu|\ds\psi|^{2}\right)+t\psi\left(c^{-2}\mu(\Lb\psi)+L\psi\right)
+2c^{-2}\mu\psi^{2}\right)\\\notag
F^{1}(t,\ub)&=\int_{\Cb^{t}_{\ub}}\frac{1}{c}\left(\frac{2(-t')}{\widetilde{\tr\chib}}(\Lb\psi)^{2}+t\psi(\Lb\psi)+\frac{1}{2}\psi^{2}\right)
\end{align}
Again, the energy estimates will be based on the following identity:
\begin{align}\label{energy identity E1}
E^{1}(t,\ub)-E^{1}(-r_0,\ub)+F^{1}(t,\ub) &= \int_{W_t^{\ub}} c^{-2} \widetilde{Q_1}
\end{align}
where
\begin{align*}
\widetilde{Q}_1 := -\rho\cdot \left(K_1 \psi-t\psi\right)-\frac{1}{2}\widetilde{T}^{\mu\nu}\widetilde{\pi}_{1,\mu\nu}+t\gt^{\mu\nu}\partial_{\mu}\psi\partial_{\nu}\psi-\frac{1}{2}\psi^{2}\Box_{\gt}(t),
\end{align*}
which can be written as
\begin{align*}
\widetilde{Q}_1 := -\rho\cdot \left(K_1 \psi-t\psi\right)-\frac{1}{2}\widetilde{T}^{\mu\nu}\widetilde{\pi}^{\prime}_{1,\mu\nu}-\frac{1}{2}\psi^{2}\Box_{\gt}(t),
\end{align*}
with $\widetilde{\pi}_{1,\mu\nu}=(\mathcal{L}_{K_{1}}\gt)_{\mu\nu}$ being the deformation tensor of $K_{1}$ and $\widetilde{\pi}^{\prime}_{1,\mu\nu}=\widetilde{\pi}_{1,\mu\nu}+2t\gt_{\mu\nu}$. However, unlike the case of $K_{0}$, it is not straightforward to show that $E^{1}(t,\ub),F^{1}(t,\ub)$ are equivalent to $\Eb(\psi)(t,\ub),\Fb(\psi)(t,\ub)$. So instead of $E^{1}(t,\ub),F^{1}(t,\ub)$, we work with 
\begin{align}\label{modified E1 F1}
E^{\prime1}(t,\ub)=\int_{\Sigma_{t}^{\ub}}\frac{1}{2c}\frac{4(-t)}{\widetilde{\tr\chib}}\left(c^{-2}\mu\left(\Lb\psi+\frac{1}{2}\widetilde{\tr\chib}\psi\right)^{2}+\mu|\ds\psi|^{2}\right),\quad
F^{\prime1}(t,\ub)=\int_{C^{t}_{\ub}}\frac{1}{c}\frac{4(-t')}{\widetilde{\tr}\chib}\left(\Lb\psi+\frac{1}{2}\widetilde{\tr\chib}\psi\right)^{2}
\end{align}
Now it is straightforward to show
\begin{align*}
E^{\prime1}(t,\ub)\sim\Eb(\psi)(t,\ub),\quad F^{\prime1}(t,\ub)\sim\Fb(\psi)(t,\ub)
\end{align*}
So we need to establish an identity for $E^{\prime1}(t,\ub)$ and $F^{\prime1}(t,\ub)$ similar to \eqref{energy identity E1}. A direct calculation implies
\begin{align}\label{difference E1prime E1 preliminary}
E^{1}(t,\ub)-E^{\prime1}(t,\ub)=\int_{\Sigma_{t}^{\ub}}\frac{1}{2c}\left(2(-t)\psi(T\psi)+\left(c^{-2}\mu+\frac{1}{2}c^{-2}\mu t\widetilde{\tr\chib}\right)\psi^{2}\right)
\end{align}
By using the relations
\begin{align*}
\slashed{\mathcal{L}}_{T}\dmug=c^{-1}\mu\tr\theta\dmug,\quad \tr\chib=-c\tr\theta,\quad \tr\chi=c^{-1}\mu\tr\theta,
\end{align*}
the difference can be rewritten as:
\begin{align}\label{difference E1prime E1 final}
E^{1}(t,\ub)-E^{\prime1}(t,\ub)=\int_{S_{t,\ub}}\frac{1}{2c}(-t)\psi^{2}-I
\end{align}
where
\begin{align}\label{I}
I=\int_{\Sigma_{t}^{\ub}}\frac{1}{2c}\left(\frac{1}{2}\widetilde{\tr\chi}(-t)-c^{-2}\mu\right)\psi^{2}
\end{align}
We also compute the difference between $F^{1}(t,\ub)$ and $F^{\prime 1}(t,\ub)$:
\begin{align}\label{difference F1prime F1 preliminary}
F^{1}(t,\ub)-F^{\prime 1}(t,\ub)=-\int_{\Cb^{t}_{\ub}}\frac{1}{2c}\left(\Lb\left((-t)\psi^{2}\right)+\widetilde{\tr\chib}(-t)\psi^{2}\right)
\end{align}
Using the identity:
\begin{align*}
\slashed{\mathcal{L}}_{\Lb}\dmug=\tr\chib\dmug
\end{align*}
This difference is rewritten as:
\begin{align}\label{difference F1prime F1 final}
F^{1}(t,\ub)-F^{\prime 1}(t,\ub)=-\int_{S_{t,\ub}}\frac{1}{2c}(-t)\psi^{2}+\int_{S_{-r_{0},\ub}}\frac{1}{2c}r_{0}\psi^{2}
\end{align}
Substituting \eqref{difference E1prime E1 final} and \eqref{difference F1prime F1 final} into \eqref{energy identity E1}, we have:
\begin{align}\label{energy identity E1 modified}
&E^{\prime1}(t,\ub)+F^{\prime 1}(t,\ub)+\int_{\Sigma_{t}^{\ub}}\frac{1}{2c}\left(\frac{1}{2}\widetilde{\tr\chi}(-t)-c^{-2}\mu\right)\psi^{2}\\\notag=&E^{\prime1}(-r_{0},\ub)+\int_{\Sigma_{-r_{0}}^{\ub}}\frac{1}{2c}\left(\frac{1}{2}\widetilde{\tr\chi}(-t)-c^{-2}\mu\right)\psi^{2}+ \int_{W_t^{\ub}} c^{-2} \widetilde{Q_1}
\end{align}
To estimate the term involving $\psi^{2}$, one needs the following lemma:
\begin{lemma}
For a $\psi$ which vanishes on $\Cb_{0}$, we have
\begin{equation}\label{Calculus Inequality}
\begin{split}
\int_{S_{t,\ub}} \psi^2 &\lesssim \delta \int_{\Sigma_{t}^{\ub}} (L\psi)^2 +\mu(\Lb \psi)^2\lesssim\delta E(\psi)(t,\ub), \ \ \int_{\Sigma_{t}^{\ub}} \psi^2 \lesssim\delta^{2}E(\psi)(t,\ub) 
\end{split}
\end{equation}
\end{lemma}
The goal of this section is to bound $E(t,\ub), F(t,\ub)$ and $\Eb(t,\ub)$ and $\Fb(t,\ub)$ in terms of their corresponding initial data and $\rho$. 
\begin{proposition}\label{prop: scattering}
The limits $\lim_{r_{0}\rightarrow\infty}E^{0}(-r_{0},\ub)$ and $\lim_{r_{0}\rightarrow\infty}E^{\prime 1}(-r_{0},\ub)$ both exist and we have

\begin{align}\label{scattering estimates}
\lim_{r_{0}\rightarrow\infty}E^{0}(-r_{0},\ub)\lesssim 1,\quad \lim_{r_{0}\rightarrow\infty}E^{\prime 1}(-r_{0},\ub)\lesssim \delta^{2}.
\end{align}
\end{proposition}
\begin{proof}
For the proof we choose $\ub=\delta$ and the proof for any $\ub\in[0,\delta]$ is similar. In view of the definitions of $E^{0}(-r_{0},\ub)$ and $E^{\prime1}(-r_{0},\ub)$ and the fact that

\begin{align*}
\mu|_{t=-r_{0}}=c,\quad \widetilde{\textrm{tr}\chib}|_{t=-r_{0}}=2\Lb\left(\frac{1}{c}\right)-\frac{2}{r_{0}},
\end{align*}
it suffices to prove that the limits $\lim_{r_{0}\rightarrow\infty}\|T\psi\|_{L^{2}(\Sigma_{-r_{0}}^{\delta})}$, $\lim_{r_{0}\rightarrow\infty}\|r_{0}\Lb\psi\|_{L^{2}(\Sigma_{-r_{0}}^{\delta})}$, $\lim_{r_{0}\rightarrow\infty}\|r_{0}\ds\psi\|_{L^{2}(\Sigma_{-r_{0}}^{\delta})}$, $\lim_{r_{0}\rightarrow\infty}\|r_{0}\psi\|_{L^{2}(\Sigma_{-r_{0}}^{\delta})}$ exist and satisfy

\begin{align*}
\lim_{r_{0}\rightarrow\infty}\|T\psi\|_{L^{2}(\Sigma_{-r_{0}}^{\delta})}\lesssim 1,\quad \lim_{r_{0}\rightarrow\infty}\|r_{0}\Lb\psi\|_{L^{2}(\Sigma_{-r_{0}}^{\delta})},\quad \lim_{r_{0}\rightarrow\infty}\|r_{0}\ds\psi\|_{L^{2}(\Sigma_{-r_{0}}^{\delta})},\quad \lim_{r_{0}\rightarrow\infty}\|r_{0}\psi\|_{L^{2}(\Sigma_{-r_{0}}^{\delta})} \lesssim \delta.
\end{align*}
Here we only give a detailed proof for $\lim_{r_{0}\rightarrow\infty}\|T\psi\|_{L^{2}(\Sigma_{-r_{0}}^{\delta})}$ and for $\psi=\partial_{t}\phi$. The proof for $\psi=\partial_{i}\phi$ and for $\Lb, \ds$ derivatives are similar. According to the initial data constructed in Lemma \ref{lemma constraint}, we have

\begin{align*}
T\psi(-r_{0},\theta)=\frac{\delta^{-1/2}}{r_{0}}(\partial_{s}\phi_{1})\left(\frac{r-r_{0}}{\delta},\theta\right)
\end{align*}
where $\partial_{s}$ is derivative of $\phi_{1}(s,\theta)$ with respect to its first argument. A direct computation shows

\begin{align*}
\lim_{r_{0}\rightarrow\infty}\|T\psi\|^{2}_{L^{2}(\Sigma_{-r_{0}}^{\delta})}=&\lim_{r_{0}\rightarrow\infty}\int_{r_{0}}^{r_{0}+\delta}\int_{\mathbb{S}^{2}}\frac{\delta^{-1}r^{2}}{r^{2}_{0}}(\partial_{s}\phi_{1})^{2}\left(\frac{r-r_{0}}{\delta},\theta\right)d\mu_{\mathbb{S}^{2}}dr\\
=&\lim_{r_{0}\rightarrow\infty}\int_{0}^{1}\int_{\mathbb{S}^{2}}\frac{(r_{0}+\delta s)^{2}}{r^{2}_{0}}(\partial_{s}\phi_{1})^{2}(s,\theta)d\mu_{\mathbb{S}^{2}}ds\\=&\int_{0}^{1}\int_{\mathbb{S}^{2}}(\partial_{s}\phi_{1})^{2}(s,\theta)d\mu_{\mathbb{S}^{2}}ds.
\end{align*}
Applying the argument in the proof of Lemma \ref{lemma constraint} to $\partial_{t}\phi$ instead of $\phi$, one can see that $\Lb\psi$ can also be written as

\begin{align*}
(\Lb\psi)(-r_{0},\theta)=\frac{\delta^{1/2}}{r_{0}^{2}}\phi_{3}\left(\frac{r-r_{0}}{\delta},\theta\right) 
\end{align*}
for some smooth function $\phi_{3}(s,\theta)$ which vanishes for $s\leq 0$. Therefore the above argument applies to $\Lb\psi, \ds\psi$ and $\psi$.
\end{proof}

\begin{remark}\label{remark: scattering}
From the above proof, one can see that Proposition \ref{prop: scattering} is also valid if the commutators $R_{i}, T, t\Lb$ are applied on $\psi:=\partial_{t}\phi,\partial_{i}\phi$. Therefore the higher order initial energies are also finite at $t=-\infty$, and our energy estimates independent of $r_{0}$ implies the existence of semi-global-in-time solutions which lead to shock formation.
\end{remark}
\subsection{Error terms}
Now we study the error terms $\widetilde{Q}_{0}$ and $\widetilde{Q}_{1}$. The deformation tensor $\widetilde{\pi}_{0,\mu\nu}$ is given by:
\begin{align*}
 \widetilde{\pi}_{0,LL}&=\frac{4}{c}\mu\underline{L}(c^{-2}\mu),\quad \widetilde{\pi}_{0,\underline{L}\underline{L}}=0 \\\notag
\widetilde{\pi}_{0,L\underline{L}}&=-\frac{2}{c}\mu\left(\mu^{-1}(K_{0}\mu)-(K_{0}\log c)+2\underline{L}(c^{-2}\mu)\right)\\\notag
\widetilde{\pi}_{0,LA}&=\frac{2}{c}\left((\zetab_{A}+\etab_{A})
-\mu\slashed{\nabla}_{A}(c^{-2}\mu)\right)\\\notag
\widetilde{\pi}_{0,\underline{L}A}&=-\frac{2}{c}(\zetab_{A}+\etab_{A})\\\notag
\hat{\widetilde{\slashed{\pi}}}_{0,AB}&
=\frac{2}{c}(\hat{\underline{\chi}}_{AB}+\hat{\chi}_{AB})\\\notag
\textrm{tr}\tilde{\slashed{\pi}}_{0}&=2\left(\widetilde{\tr\chi}+\widetilde{\tr\chib}\right)
\end{align*}
The modified deformation tensor $\widetilde{\pi}^{\prime}_{1,\mu\nu}$ is given by:
\begin{align*}
\widetilde{\pi}^{\prime}_{1,LL}&=\frac{4}{c}\mu\left(\left(\frac{2(-t)}{\widetilde{\tr\chib}}\right)\underline{L}(c^{-2}\mu)-L\left(\frac{2(-t)}{\widetilde{\tr\chib}}\right)\right),\quad \widetilde{\pi}^{\prime}_{1,\underline{L}\underline{L}}=0\\\notag
\widetilde{\pi}^{\prime}_{1,L\underline{L}}&=-\frac{2}{c}\mu\left(\mu^{-1}K_{1}\mu+\underline{L}\left(\frac{2(-t)}{\widetilde{\tr\chib}}\right)+2t-K_{1}\log c\right)\\\notag
\widetilde{\pi}^{\prime}_{1,LA}&=\frac{2}{c}\left(\frac{2(-t)}{\widetilde{\tr\chib}}\right)(\zetab_{A}+\etab_{A})-\mu\slashed{\nabla}_{A}\left(\frac{2(-t)}{\widetilde{\tr\chib}}\right),\quad
\widetilde{\pi}^{\prime}_{1,\underline{L}A}=0\\\notag
\widehat{\tilde{\slashed{\pi}}}^{\prime}_{1,AB}&=\frac{2}{c}\left(\frac{2(-t)}{\widetilde{\tr\chib}}\right)\widehat{\underline{\chi}}_{AB},\quad
\textrm{tr}\widetilde{\slashed{\pi}}^{\prime}_{1}=0
\end{align*}

\begin{remark}\label{remark: K1}
As we stated in the introduction, the choice of $K_{1}=\left(\dfrac{2(-t)}{\widetilde{\textrm{tr}\chib}}\right)\Lb$ is such that it behaves like $u^{2}\Lb$ (up to a multiplication by a constant) when $|t|$ is large. On the other hand, the specific choice of the coefficient $\left(\dfrac{2(-t)}{\widetilde{\textrm{tr}\chib}}\right)$ is to guarantee that $\textrm{tr}\widetilde{\slashed{\pi}}^{\prime}_{1}$ vanishes, which would cause a divergence in time if it is non-zero.
\end{remark}
To calculate $\widetilde{Q}_{0}$ and $\widetilde{Q}_{1}$, we need to raise the indices for $T_{\mu\nu}$:
\begin{align}\label{energy momentum tensor raised}
\begin{split}
  T^{LL} &= \frac{(\Lb\psi)^2}{4\mu^2},\ T^{\Lb\Lb} = \frac{(L\psi)^2}{4\mu^2}, \ T^{\Lb L} = \frac{(\slashed{d}\psi)^2}{4\mu}, \ T^{LA} =-\frac{ \Lb\psi X_A(\psi)}{2\mu},\\
T^{\Lb A} &= -\frac{L \psi X_A(\psi)}{2\mu},\ T^{AB} = \slashed{g}^{AC}\slashed{g}^{BD}X_C(\psi)X_D(\psi)-\frac{1}{2}\slashed{g}^{AB}(-\frac{L\psi \Lb \psi}{\mu} + |\slashed{d} \psi|^2).
 \end{split}
\end{align}
Now we can compute the integrands $\widetilde{Q_0}$ and $\widetilde{Q_1}$ explicitly. For $\widetilde{Q_0}$, we have
\begin{equation}\label{Q_0}
\begin{split}
 {c}^{-2} \widetilde{Q}_0 &= -{c}^{-2}\rho\cdot K_0 \psi-\frac{1}{2}{T}^{\mu\nu}\widetilde{\pi}_{0,\mu\nu}= \widetilde{Q}_{0,0} + \widetilde{Q}_{0,1}+\widetilde{Q}_{0,2}+\widetilde{Q}_{0,3}+ \widetilde{Q}_{0,4}+\widetilde{Q}_{0,5}\\
&=-{c}^{-2}\rho \cdot K_0\psi  -T^{L\Lb}\widetilde{\pi}_{0,L\Lb}-T^{LA}\widetilde{\pi}_{0,LA}-T^{\Lb A}\widetilde{\pi}_{0,\Lb A}-\frac{1}{2}T^{AB}\widetilde{\pi}_{0,AB}+\frac{1}{2}T^{LL}\widetilde{\pi}_{0,LL}.
\end{split}
\end{equation}
The $\widetilde{Q}_{0,i}$s are given by
\begin{align}\label{Q_0,i}
\begin{split}
  \widetilde{Q}_{0,1}& =   \frac{1}{2c}\Big(\mu^{-1}K_{0}\mu+K_{0}\log(c^{-1}) + \Lb(c^{-2}\mu)\Big)|\slashed{d}\psi|^2, \\
\widetilde{Q}_{0,2}&=-{c}^{-1}\left(X_{A}(c^{-2}\mu)-\mu^{-1}(\zetab_{A}+\etab_{A})\right) \Lb\psi \cdot X_A(\psi),\\ \widetilde{Q}_{0,3}&=-c^{-1} \mu^{-1}\left(\zetab_A+\etab_A\right)L\psi X_A(\psi),\\
\widetilde{Q}_{0,4}&=-\frac{1}{2}\left(\widehat{\chit}_{AB}X^A \psi X^B \psi +\frac{c}{2} \mu^{-1}\left(\tr_{\widetilde{g}}\chit+\tr_{\widetilde{g}}\widetilde{\chib}\right) {L\psi \cdot \Lb \psi}\right),\\
\widetilde{Q}_{0,5}&=-\frac{1}{2}\frac{1}{c\mu}\Lb(c^{-2}\mu)\left(\Lb\psi\right)^{2}.
\end{split}
\end{align}
For $\widetilde{Q}_{1}$ we have:
\begin{equation}\label{Q_1}
\begin{split}
  {c}^{-2} \widetilde{Q}_1 &= -{c}^{-2}\rho\cdot \left(K_1 \psi-t\psi\right)-\frac{1}{2}{T}^{\mu\nu}\widetilde{\pi}^{\prime}_{1,\mu\nu}= \widetilde{Q}_{1,0} + \widetilde{Q}_{1,1}+\widetilde{Q}_{1,2}+\widetilde{Q}_{1,3}+ \widetilde{Q}_{1,4}+\widetilde{Q}_{1,5}\\
&=-{c}^{-2}\rho \cdot\left( K_1 \psi-t\psi\right)-\frac{1}{2}T^{LL}\widetilde{\pi}^{\prime}_{1,LL}-T^{L\Lb}\widetilde{\pi}^{\prime}_{1,L\Lb} -T^{LA}\widetilde{\pi}^{\prime}_{1,LA} -\frac{1}{2}T^{AB}\widetilde{\pi}^{\prime}_{1,AB}-\frac{1}{2}\psi^{2}\Box_{\gt}(t).
\end{split}
\end{equation}
The $\widetilde{Q}_{1,i}$ are given by
\begin{align}\label{Q_1,i}
\begin{split}
\widetilde{Q}_{1,1}&=-\frac{1}{2c\mu}\left(\left(\frac{2(-t)}{\widetilde{\tr\chib}}\right)\underline{L}(c^{-2}\mu)-L\left(\frac{2(-t)}{\widetilde{\tr\chib}}\right)\right)\left(\Lb\psi\right)^{2},\\
\widetilde{Q}_{1,2}&=\frac{1}{2c}\left(\mu^{-1}K_{1}\mu+\underline{L}\left(\frac{2(-t)}{\widetilde{\tr\chib}}\right)+2t-K_{1}\log c\right)|\ds\psi|^{2},\\
\widetilde{Q}_{1,3}&=2\mu^{-1}\left(\frac{2}{c}\left(\frac{2(-t)}{\widetilde{\tr\chib}}\right)(\zetab_{A}+\etab_{A})-\mu\slashed{\nabla}_{A}\left(\frac{2(-t)}{\widetilde{\tr\chib}}\right)\right)(\Lb\psi)(X_{A}\psi),\\
\widetilde{Q}_{1,4}&=-\frac{1}{c}\left(\frac{2(-t)}{\widetilde{\tr\chib}}\right)\widehat{\underline{\chi}}_{AB}(X^{A}\psi)(X^{B}\psi)
\end{split}
\end{align}

The rest of this section is devoted to the estimates for $\int_{W^{t}_{\ub}}\widetilde{Q}_{0}$ and $\int_{W^{t}_{\ub}}\widetilde{Q}_{1}$. We will make use of \eqref{spacetime integral definition}.
\subsection{Estimates on $\widetilde{Q}_{1,2}$}
We separate the estimates on $\widetilde{Q}_{1,2}$ from others, which are of lower order. We analyze the contribution of each term in the parenthesis in the expression of $\widetilde{Q}_{1,2}$. By the definition of $K_{1}$ and \eqref{box psi_0 in null frame} \eqref{expansion of L psi_0} \eqref{expansion of psi_0}, as well as the fact $\frac{(-t)}{\tr\chib}\sim t^{2}$ we have
\begin{align*}
\|K_{1}\log c\|_{L^{\infty}}\lesssim\delta
\end{align*}
So the contribution of this term to the spacetime integral is bounded by:
\begin{align}\label{Q_1,2 1}
\delta\int_{-r_{0}}^{t}(-t')^{-2}\Eb(t',\ub)dt'
\end{align}
We move to the contribution of the term $\Lb\left(\dfrac{2(-t)}{\widetilde{\tr\chib}}\right)+2t$, a direct calculation implies
\begin{align*}
\Lb\left(\dfrac{2(-t)}{\widetilde{\tr\chib}}\right)=-\frac{2}{\widetilde{\tr\chib}}-\frac{2(-t)}{(\widetilde{\tr\chib})^{2}}\Lb(\widetilde{\tr\chib})
\end{align*}
In view of the expression of $\Eb(t,\ub)$, the term on the right hand side of above equation behaving like $(-t)$ gives us the borderline contribution. However, we will show these actually cancel. 
In view of the propagation equation \eqref{Structure Equation Lb chibAB nonsingular}, the second term on the right hand side of above equation can be written as:
\begin{align*}
\frac{2t}{(\widetilde{\tr\chib})^{2}}\left(e\tr\chib-\frac{1}{2}(\tr\chib)^{2}-|\widehat{\chib}|^{2}
-\tr\alpha'+O\left(\delta M^{2}(-t)^{-3}\right)\right)
\end{align*}
The borderline term is $(-t)\dfrac{(\tr\chib)^{2}}{(\widetilde{\tr\chib})^{2}}$, which can be rewritten as 
\begin{align*}
(-t)\frac{(\widetilde{\tr\chib})^{2}+O(\delta M^{2})(-t)^{-4}}{(\widetilde{\tr\chib})^{2}},
\end{align*}
so the borderline term is $(-t)$. On the other hand, in view of \eqref{precise bound on chibt}, the borderline term in $-\frac{2}{\widetilde{\tr\chib}}$ is also $(-t)$. Therefore there is no borderline term in $\Lb\left(\dfrac{2(-t)}{\widetilde{\tr\chib}}\right)+2t$. And the contribution of this term is bounded as
\begin{align}\label{Q_1,2 2}
\delta M^{2}\int_{-r_{0}}^{t}(-t')^{-2}\Eb(t',\ub)dt'.
\end{align}
Finally the contribution of the term $\mu^{-1}K_{1}\mu$ is:
In the non-shock region we have, using the fact $\mu\gtrsim\dfrac{1}{10}$ and $|\Lb\mu|\lesssim(-t)^{-2}$,
\begin{align}\label{Q_1,2 3}
\int_{W^{t}_{\ub}\bigcap W_{rare}}\frac{1}{2c}\left(\frac{2(-t)}{\widetilde{\tr\chib}}\right)\mu^{-1}\Lb\mu\cdot |\ds\psi|^{2}\lesssim\int_{-r_{0}}^{t}(-t')^{-2}\Eb(t',\ub)dt'.
\end{align}
In the shock region, by Proposition \ref{Proposition C3}, the spacetime integral
\begin{align}\label{Q_1,2 4}
\int_{W^{t}_{\ub}\bigcap W_{shock}}\frac{1}{2c}\left(\frac{2(-t)}{\widetilde{\tr\chib}}\right)\mu^{-1}\Lb\mu\cdot|\ds\psi|^{2}:=-K(t,\ub)
\end{align}
is negative. Combining \eqref{Q_1,2 1}-\eqref{Q_1,2 4}, the spacetime integral involving $\widetilde{Q}_{1,2}$ is bounded by:
\begin{align}\label{Q_1,2 final}
\int_{-r_{0}}^{t}(-t')^{-2}\Eb(t',\ub)dt'-K(t,\ub).
\end{align}
\subsection{Estimates on $\widetilde{Q}_{0,1}$}
Another term which also needs to be treated separately is $\widetilde{Q}_{0,1}$. Again, we will estimate the contribution of each term in parenthesis. In view of \eqref{expansion of L psi_0}, \eqref{expansion of psi_0}, \eqref{expansion of Lb mu} and \eqref{box psi_0 in null frame}, we have:
\begin{align*}
\|K_{0}\log c\|_{L^{\infty}(\Sigma_{t}^{\ub})}\lesssim (-t)^{-2},\quad \|\Lb(c^{-2}\mu)\|_{L^{\infty}(\Sigma_{t}^{\ub})}\lesssim (-t)^{-2}.
\end{align*}
Therefore the corresponding contributions are bounded by:
\begin{align}\label{Q_0,1 1}
\int_{-r_{0}}^{t}(-t')^{-2}\Eb(t',\ub)dt'
\end{align}
There are two different terms in $\mu^{-1}K_{0}\mu$, $\mu^{-1}\Lb\mu$ and $\mu^{-1}L\mu$. We split the contribution from $\mu^{-1}\Lb\mu$ as:
\begin{align*}
\int_{W^{t}_{\ub}\bigcap W_{shock}}\frac{1}{2c}\mu^{-1}\Lb\mu|\ds\psi|^{2}+\int_{W^{t}_{\ub}\bigcap W_{rare}}\frac{1}{2c}\mu^{-1}\Lb\mu|\ds\psi|^{2}.
\end{align*}
The integral in the non-shock region is bounded through \eqref{expansion of Lb mu} by:
\begin{align}\label{Q_0,1 2}
\int_{-r_{0}}^{t}(-t')^{-2}\Eb(t',\ub)dt'
\end{align}
While by Proposition \ref{Proposition C3}, the integrand in the shock region is actually negative, so it does not contribute. For the contribution from $\mu^{-1}L\mu$, we only need to consider its positive part, namely, $\mu^{-1}\left(L\mu\right)_{+}$, which is bounded by $\mu^{-1}\left(T\mu\right)_{+}+c^{-2}\left(\Lb\mu\right)_{+}$. The contribution from the second term is bounded similarly as \eqref{Q_0,1 2} through \eqref{expansion of Lb mu}. In view of Proposition \ref{Proposition C2} the contribution of $\mu^{-1}\left(T\mu\right)_{+}$ is bounded as:
\begin{align}\label{Q_0,1 3}
\int_{-r_{0}}^{t}\|\mu^{-1}\left(T\mu\right)_{+}\|_{L^{\infty}(\Sigma_{t}^{\ub})}(-t')^{-2}\Eb(t',\ub)dt'\lesssim\delta^{-1}\int_{-r_{0}}^{t}\frac{1}{\left|t'-s\right|^{1/2}}(-t')^{-3/2}\Eb(t',\ub)dt'.
\end{align}

\subsection{Estimates for $\widetilde{Q}_{1,1}$}
In order to estimate the contribution of $\widetilde{Q}_{1,1}$ to the spacetime error integral, we use the trace of the structure equation \eqref{Structure Equation T chib} and \eqref{Structure Equation Lb chibAB nonsingular}:

\begin{align}\label{Structure Equation trace}
\begin{split}
T(\tr\chib)+\frac{1}{2}
c^{-1}\mu(\theta^{B}_{C}\chib_{B}^{C}+\theta^{A}_{C}\chib_{A}^{C})
&=\slashed{\text{div}}\etab+\mu^{-1}\slashed{g}^{AB}
\zetab_{A}\etab_{B}-c^{-1}\Lb(c^{-1}\mu)\tr\chib,\\
\Lb(\tr\chib)+|\chib|^{2}_{\slashed{g}}&=e\tr\chib-\tr\alpha'.
\end{split}
\end{align}
In view of the formula for $\Lb\mu$ and the pointwise estimate for $\widetilde{\tr}\chib$, the first term in the parenthesis of the first formula in \eqref{Q_1,i} is bounded by an absolute constant. For the second term, recall that $L=c^{-2}\mu\Lb+2T$. Therefore

\begin{align}\label{Q_1,1 detail}
\begin{split}
L\left(\frac{2(-t)}{\widetilde{\tr}\chib}\right)&=c^{-2}\mu\Lb\left(\frac{2(-t)}{\widetilde{\tr}\chib}\right)+2T\left(\frac{2(-t)}{\widetilde{\tr}\chib}\right)\\
&=-\frac{2c^{-2}\mu}{\widetilde{\tr}\chib}-\frac{2(-t)c^{-2}\mu\Lb(\widetilde{\tr}\chib)}{(\widetilde{\tr}\chib)^{2}}-\frac{4(-t)T(\widetilde{\tr}\chib)}{(\widetilde{\tr}\chib)^{2}}\\
&=:I+II+III.
\end{split}
\end{align}
First, the contribution of $I$ to $\widetilde{Q}_{1,1}$ is negative, so we ignore it. For $II$ and $III$, their contributions from the right hand side of the two equations in \eqref{Structure Equation trace} are bounded by an absolute constant. We focus on the contributions from the second term in each of \eqref{Structure Equation trace}. Since the difference between $\widetilde{\tr}\chib$ and $\tr\chib$ is lower order, we work the original metric $\slashed{g}$. In $II$ this contribution is 

\begin{align}\label{Q_1,1 principal I}
\frac{2(-t)c^{-2}\mu|\chib|^{2}_{\slashed{g}}}{(\widetilde{\tr}\chib)^{2}}.
\end{align}
In view of the fact $\chib=-c\theta$, this contribution from $III$ is 

\begin{align}\label{Q_1,1 principal II}
-\frac{4(-t)c^{-2}\mu|\chib|^{2}_{\slashed{g}}}{(\widetilde{\tr}\chib)^{2}}.
\end{align}
\eqref{Q_1,1 principal I} and \eqref{Q_1,1 principal II} together give a negative contribution. Therefore this contribution is ignored.
Therefore the spacetime error integral contributed by $Q_{1,1}$ is bounded by (up to a constant)

\begin{align}\label{Q_1,1 temp 3}
\begin{split}
&\int_{-r_{0}}^{t}\int_{0}^{\ub}\int_{S_{t',\ub'}}\left(\Lb\psi+\frac{1}{2}\widetilde{\tr\chib}\psi\right)^{2}\dmug d\ub'dt'+\int_{-r_{0}}^{t}\int_{0}^{\ub}\int_{S_{t',\ub'}}\left(\frac{1}{2}\widetilde{\tr\chib}\psi\right)^{2}\dmug d\ub'dt'.
\end{split}
\end{align}
The first term above is bounded by 

\begin{align}\label{Q_1, 1 1}
\int_{0}^{\ub}\int_{-r_{0}}^{t}(-t')^{2}\int_{S_{t',\ub'}}\left(\Lb\psi+\frac{1}{2}\widetilde{\tr\chib}\psi\right)^{2}\dmug d\ub'dt'\lesssim\int_{0}^{\ub}\Fb(t,\ub')d\ub'.
\end{align}
In view of \eqref{Calculus Inequality}, the second term in \eqref{Q_1,1 temp 3} is bounded by (up to a constant)

\begin{align}\label{Q_1,1 temp 4}
\begin{split}
&\int_{-r_{0}}^{t}\int_{0}^{\ub}\int_{S_{t',\ub'}}\left(\frac{1}{2}\widetilde{\tr\chib}\psi\right)^{2}\dmug d\ub'dt'\\
\lesssim&\delta^{2}\int_{-r_{0}}^{t}(-t')^{-2}E(t',\ub)dt'+\delta\int_{-r_{0}}^{t}(-t')^{-2}\int_{0}^{\ub}\int_{S_{t',\ub'}}(\Lb\psi)^{2}\dmug d\ub'dt'\\
\lesssim&\delta^{2}\int_{-r_{0}}^{t}(-t')^{-2}E(t',\ub)dt'+\delta\int_{0}^{\ub}\Fb(t,\ub')d\ub'+\delta\int_{-r_{0}}^{t}(-t')^{-2}\int_{0}^{\ub}\int_{S_{t',\ub'}}\left(\frac{1}{2}\widetilde{\tr}\chib\psi\right)^{2}\dmug d\ub'dt'.
\end{split}
\end{align}
If $\delta$ is appropriately small, we have

\begin{align}\label{Q_1, 1 2}
\begin{split}
&\int_{-r_{0}}^{t}\int_{0}^{\ub}\int_{S_{t',\ub'}}\left(\frac{1}{2}\widetilde{\tr\chib}\psi\right)^{2}\dmug d\ub'dt'\\
\lesssim&\delta^{2}\int_{-r_{0}}^{t}(-t')^{-2}E(t',\ub)dt'+\delta\int_{0}^{\ub}\Fb(t,\ub')d\ub'.
\end{split}
\end{align}
\subsection{Estimates for the other error terms}
\subsubsection{Estimates for $\widetilde{Q}_{0,2}$}
The terms in the parenthesis can be written as
\begin{align*}
\mu^{-1}\ds\mu+O((-t)^{-2})
\end{align*}
The contribution from the second term is bounded as:
\begin{align*}
\int_{-r_{0}}^{t}(-t')^{-2}\int_{0}^{\ub}\int_{S_{t',\ub'}}\mu(\Lb\psi)^{2}d\mu_{\slashed{g}}d\ub'dt'+\int_{-r_{0}}^{t}(-t')^{-2}\int_{0}^{\ub}\int_{S_{t',\ub'}}\mu|\ds\psi|^{2}d\mu_{\slashed{g}}d\ub'dt':=I+II
\end{align*}
In view of the definition of $\Eb(t,\ub)$, 
\begin{align}\label{Q_0, 2 1}
|II|\lesssim\int_{-r_{0}}^{t}(-t')^{-2}\Eb(t',\ub)dt'.
\end{align}
While $I$ is bounded as 
\begin{align}\label{Q_0, 2 2}
\begin{split}
|I|\lesssim&\int_{-r_{0}}^{t}(-t')^{-2}\int_{0}^{\ub}\int_{S_{t',\ub'}}\mu\left((\Lb\psi+\frac{1}{2}\widetilde{\tr\chib}\psi\right)^{2}d\mu_{\slashed{g}}d\ub'dt'\\
+&\int_{-r_{0}}^{t}(-t')^{-2}\int_{0}^{\ub}\int_{S_{t',\ub'}}\mu\left(\frac{1}{2}\widetilde{\tr\chib}\psi\right)^{2}d\mu_{\slashed{g}}d\ub'dt'
\\
\lesssim&\int_{-r_{0}}^{t}(-t')^{-2}\Eb(t',\ub)dt'+\int_{-r_{0}}^{t}(-t')^{-4}E(t',\ub)dt'.
\end{split}
\end{align}
where in the last step we used \eqref{Calculus Inequality}. The contribution of $\mu^{-1}\ds\mu$ is bounded as:
\begin{align*}
\int_{-r_{0}}^{t}(-t')^{-2}\int_{0}^{\ub}\int_{S_{t',\ub'}}(\Lb\psi)^{2}d\mu_{\slashed{g}}d\ub'dt'+\int_{W^{t}_{\ub}}(-t')^{-2}\mu^{-1}|\ds\psi|^{2}d\mu_{g}
\end{align*}
The first term above is bounded as:
\begin{align}\label{Q_0, 2 3}
\int_{0}^{\ub}\int_{-r_{0}}^{t}(-t')^{-2}\int_{S_{t',\ub'}}(\Lb\psi)^{2}d\mu_{\slashed{g}}dt'd\ub'\lesssim \int_{0}^{\ub}F(t,\ub')d\ub'.
\end{align}
While the second term is split as:
\begin{align*}
\int_{W^{t}_{\ub}\bigcap W_{rare}}(-t')^{-2}\mu^{-1}|\ds\psi|^{2}d\mu_{g}+\int_{W^{t}_{\ub}\bigcap W_{shock}}(-t')^{-2}\mu^{-1}|\ds\psi|^{2}d\mu_{g},
\end{align*}
The contribution in non-shock region is bounded as:
\begin{align}\label{Q_0, 2 4}
\int_{-r_{0}}^{t}(-t')^{-2}\int_{0}^{\ub}\int_{S_{t',\ub'}}\mu|\ds\psi|^{2}\dmug d\ub'dt'\lesssim \int_{-r_{0}}^{t}(-t')^{-2}\Eb(t',\ub)dt'.
\end{align}
The contribution in shock region is bounded as:
\begin{align}\label{Q_0, 2 5}
\int_{W^{t}_{\ub}\bigcap W_{shock}}\mu^{-1}(-t')^{-2}|\ds\psi|d\mu_{g}\lesssim  K(t,\ub).
\end{align}
Here we used the fact that in the shock region $\Lb\mu\lesssim -(-t)^{-2}$. This completes the estimates for $\widetilde{Q}_{0,2}$.
\subsubsection{Estimates for $\widetilde{Q}_{0,3}$}
The estimates for $\widetilde{Q}_{0,3}$ is similar to $\widetilde{Q}_{0,2}$. Its contribution is bounded by:
\begin{align}\label{Q_0, 3}
\int_{-r_{0}}^{t}(-t')^{-2}E(t',\ub)dt'+\int_{-r_{0}}^{t}(-t')^{-2}\Eb(t',\ub)dt'+ K(t,\ub).
\end{align}
\subsubsection{Estimates for $\widetilde{Q}_{0,4}$}
In view of \eqref{precise bound on chibt}, the contribution of the first term is bounded as:
\begin{align}\label{Q_0, 4 1}
\delta M^{2}\int_{-r_{0}}^{t}(-t')^{-3}\int_{0}^{\ub}\int_{S_{t',\ub'}}\mu|\ds\psi|^{2}\dmug d\ub'dt'\lesssim\delta M^{2}\int_{-r_{0}}^{t}(-t')^{-3}\Eb(t',\ub)dt'.
\end{align}
For the contribution of the second term, in view of \eqref{outgoing null}, we have $\widetilde{\chi}=c^{-2}\mu\widetilde{\chib}+2c^{-1}\mu\theta
=c^{-2}\mu\widetilde{\chib}-2c^{-2}\mu\widetilde{\chib}=-c^{-2}\mu\widetilde{\chib}$. Therefore in view of \eqref{precise bound on chibt} and \eqref{expansion of mu}, we have:
\begin{align*}
\|\widetilde{\tr}\widetilde{\chi}+\widetilde{\tr}\widetilde{\chib}\|_{L^{\infty}(\Sigma_{t}^{\ub})}=\|\widetilde{\tr}\widetilde{\chib}(1-c^{-2}\mu)\|_{L^{\infty}(\Sigma_{t}^{\ub})}\lesssim(-t)^{-2}
\end{align*}
So the contribution of the second term in $\widetilde{Q}_{0,4}$ is bounded as:
\begin{align}\label{Q_0, 4 2}
\begin{split}
&\int_{-r_{0}}^{t}(-t')^{-2}\int_{0}^{\ub}\int_{S_{t',\ub'}}(L\psi)^{2}\dmug d\ub'dt'\\
+&\int_{0}^{\ub}\int_{-r_{0}}^{t}(-t')^{-2}\int_{S_{t',\ub'}}(\Lb\psi)^{2}\dmug dt'd\ub'\\
\lesssim&\int_{-r_{0}}^{t}(-t')^{-2}E(t',\ub)dt'+\int_{0}^{\ub}F(t,\ub')d\ub'.
\end{split}
\end{align}
\subsubsection{Estimates for $\widetilde{Q}_{0,5}$}
The estimates for $\widetilde{Q}_{0,5}$ is straightforward. In view of \eqref{expansion of Lb mu}, we have 
\begin{align*}
\|\Lb(c^{-2}\mu)\|_{L^{\infty}(\Sigma_{t}^{\ub})}\lesssim(-t)^{-2}.
\end{align*}
 So this contribution is bounded by:
\begin{align}\label{Q_0, 5}
\int_{0}^{\ub}\int_{-r_{0}}^{t}(-t')^{-2}\int_{S_{t',\ub}}(\Lb\psi)^{2}\dmug dt'd\ub'\lesssim\int_{0}^{\ub}F(t,\ub')d\ub'.
\end{align}
\subsubsection{Estimates for $\widetilde{Q}_{1,3}$}
The first term in the parenthesis together with the factor $\mu^{-1}$ can be written as
\begin{align*}
\mu^{-1}\ds\mu O((-t)^{2})+O(\delta M^{2}(-t)^{-1}),
\end{align*}
The contribution of the second factor above is bounded as:
\begin{align*}
&\delta M^{2}\int_{-r_{0}}^{t}(-t')^{-1}\int_{0}^{\ub}\int_{S_{t',\ub'}}\mu|\ds\psi|^{2}\dmug d\ub'dt'\\
+&\delta M^{2}\int_{-r_{0}}^{t}(-t')^{-1}\int_{0}^{\ub}\int_{S_{t',\ub'}}\mu\left(\Lb\psi+\frac{1}{2}\widetilde{\tr\chib}\psi\right)^{2}\dmug d\ub'dt'\\
+&\delta M^{2}\int_{-r_{0}}^{t}(-t')^{-3}\int_{0}^{\ub}\int_{S_{t',\ub'}}\psi^{2}\dmug d\ub'dt'.
\end{align*}
The first two terms above can be bounded as 
\begin{align}\label{Q_1, 3 1}
\delta M^{2}\int_{-r_{0}}^{t}(-t')^{-3}\Eb(t',\ub)dt'.
\end{align}
While in view of \eqref{Calculus Inequality}, the last term is bounded as 
\begin{align}\label{Q_1, 3 2}
\delta^{3}M^{2}\int_{-r_{0}}^{t}(-t')^{-3}E(t',\ub)dt'
\end{align}
For the contribution of the factor $\mu^{-1}\ds\mu O((-t)^{2})$, we bound the term $(\Lb\psi)(X_{A}\psi)$ as:
\begin{align*}
|(\Lb\psi)(X_{A}\psi)|\leq \epsilon|\ds\psi|^{2}+C_{\epsilon}(\Lb\psi)^{2},
\end{align*}
where $\epsilon$ is a small absolute positive constant which will be determined later.
The contribution of the first term on the right hand side above is bounded as:
\begin{align}\label{Q_1, 3 3}
\begin{split}
&\epsilon\int_{W^{t}_{\ub}\bigcap W_{rare}}|\ds\psi|^{2}+\epsilon\int_{W^{t}_{\ub}\bigcap W_{shock}}\mu^{-1}|\ds\psi|^{2}\\
\leq &\epsilon\int_{-r_{0}}^{t}(-t')^{-2}\Eb(t',\ub)dt'+\epsilon K(t,\ub).
\end{split}
\end{align}
The contribution from $(\Lb\psi)^{2}$ can be bounded as:
\begin{align}\label{Q_1, 3 4}
\begin{split}
&C_{\epsilon}\int_{-r_{0}}^{t}\int_{0}^{\ub}\int_{S_{t',\ub'}}\left(\Lb\psi+\frac{1}{2}\widetilde{\tr\chib}\psi\right)^{2}\dmug d\ub'dt'+\int_{-r_{0}}^{t}(-t')^{-2}\int_{0}^{\ub}\int_{S_{t',\ub'}}\psi^{2}\dmug d\ub'dt'\\
\leq &C_{\epsilon}\int_{-r_{0}}^{t}(-t')^{-2}\Eb(t',\ub)dt'+C_{\epsilon}\delta^{2}\int_{-r_{0}}^{t}(-t')^{-2}E(t',\ub)dt'.
\end{split}
\end{align}
In view of \eqref{bound on d trchibt}, the second term in the parenthesis of $\widetilde{Q}_{1,3}$ together with the factor $\mu^{-1}$ is bounded by $(-t)^{-1}$. Therefore the contribution of this term is bounded, in the same process as we derive \eqref{Q_1, 3 4}, by:
\begin{align}\label{Q_1, 3 5}
\int_{-r_{0}}^{t}(-t')^{-3}\Eb(t',\ub)dt'+\delta^{2}\int_{-r_{0}}^{t}(-t')^{-3}E(t',\ub)dt'.
\end{align}
\subsubsection{Estimates for $\widetilde{Q}_{1,4}$}
In view of \eqref{precise bound on chibt}, this term is bounded by:
\begin{align}\label{Q_1, 4}
\delta M^{2}\int_{-r_{0}}^{t}(-t')^{-3}\Eb(t',\ub)dt'.
\end{align}
\subsubsection{Estimates for $\widetilde{Q}_{1,5}$}
Finally we consider the contribution from $\widetilde{Q}_{1,5}$. Writing the operator $\Box_{\gt}$ in the null frame (see \eqref{box psi_0 in null frame}), we have:
\begin{align*}
\Box_{\gt}(t)=-\frac{1}{2}\tr\widetilde{\chib}Lt-\frac{1}{2}\tr\widetilde{\chi}\Lb(t)=-\frac{1}{2}\tr\widetilde{\chib}\left(c^{-2}\mu\right)-\frac{1}{2}\tr\widetilde{\chi}=-\left(c^{-2}\mu-c^{-1}\right)\tr\chib
\end{align*} 
In view of \eqref{expansion of mu}, 
\begin{align*}
\left|\left(c^{-2}\mu-c^{-1}\right)\tr\chib\right|\lesssim(-t)^{-2}
\end{align*}
Therefore the contribution of $\widetilde{Q}_{1,5}$ is bounded by:
\begin{align}\label{Q_1, 5}
\delta^{2}\int_{-r_{0}}^{t}(-t')^{-2}E(t',\ub)dt'.
\end{align}
\subsection{Conclusion}
In view of \eqref{energy identity E0} and the estimates for spacetime integral of $\widetilde{Q}_{0}$ \eqref{Q_0,1 1}-\eqref{Q_0,1 3}, \eqref{Q_0, 2 1}-\eqref{Q_0, 2 5}, \eqref{Q_0, 3}, \eqref{Q_0, 4 1}-\eqref{Q_0, 4 2}, \eqref{Q_0, 5}, we have:
\begin{align}\label{K0 linear preliminary}
\begin{split}
E(t,\ub)+F(t,\ub)\leq &E(-r_{0},\ub)+C\int_{0}^{\ub}F(t,\ub')d\ub'+C\int_{-r_{0}}^{t}(-t')^{-2}\Eb(t',\ub)dt'+\int_{W^{t}_{\ub}}\left|\rho\cdot K_{0}\psi\right|\\+&CK(t,\ub)
+C\int_{-r_{0}}^{t}(-t')^{-2}E(t',\ub)dt'+C\delta^{-1}\int_{-r_{0}}^{t}\frac{1}{|t'-s|^{1/2}}(-t')^{-3/2}\Eb(t',\ub)dt'.
\end{split}
\end{align} 
Here $s\in[-2,-1]$ and $-r_{0}\leq t'\leq t<s$.
The integral in the last term on the right hand side above can be bounded as follows:

\begin{align}\label{Tmu integral}
\begin{split}
&\int_{-r_{0}}^{t}\frac{(-t')^{-3/2}}{|t'-s|^{1/2}}\Eb(t',\ub)dt'=\int_{-r_{0}}^{2s}\frac{(-t')^{-3/2}}{|t'-s|^{1/2}}\Eb(t',\ub)dt'+\int_{2s}^{t}\frac{(-t')^{-3/2}}{|t'-s|^{1/2}}\Eb(t',\ub)dt'\\
\lesssim&\int_{-r_{0}}^{t}(-t')^{-3/2}\Eb(t',\ub)dt'+\int_{2s}^{t}\frac{1}{|t'-s|^{1/2}}\Eb(t',\ub)dt'.
\end{split}
\end{align}
Using Gronwall we have

\begin{align}\label{K0 linear temp 1}
\begin{split}
E(t,\ub)+F(t,\ub)\leq &E(-r_{0},\ub)+C\delta^{-1}\int_{-r_{0}}^{t}(-t')^{-3/2}\Eb(t',\ub)dt'+\left|\int_{W^{t}_{\ub}}\rho\cdot K_{0}\psi\right|\\+&CK(t,\ub)
+C\delta^{-1}\int_{2s}^{t}\frac{1}{|t'-s|^{1/2}}\Eb(t',\ub)dt'.
\end{split}
\end{align}
On the other hand, in view of \eqref{energy identity E1 modified} and the estimates for spacetime integral of $\widetilde{Q}_{1}$ \eqref{Q_1,2 final},\eqref{Q_1, 1 1}-\eqref{Q_1, 1 2}, \eqref{Q_1, 3 1}-\eqref{Q_1, 3 5}, \eqref{Q_1, 4}, \eqref{Q_1, 5} as well as \eqref{Calculus Inequality}, we have:
\begin{align}\label{K1 linear preliminary}
\begin{split}
\Eb(t,\ub)+\Fb(t,\ub)+K(t,\ub)\leq &\Eb(-r_{0},\ub)+\left|\int_{W^{t}_{\ub}}\rho\cdot(K_{1}\psi-t\psi)\right|+C\int_{0}^{\ub}\Fb(t,\ub')d\ub'\\+&\epsilon  K(t,\ub)
+C\int_{-r_{0}}^{t}(-t')^{-3/2}\Eb(t',\ub)dt'+C\delta^{2}\int_{-r_{0}}^{t}(-t')^{-2}E(t',\ub)dt'.
\end{split}
\end{align}
Using Gronwall we obtain

\begin{align}\label{K1 linear temp 1}
\Eb(t,\ub)+\Fb(t,\ub)+K(t,\ub)\leq &\Eb(-r_{0},\ub)+\left|\int_{W^{t}_{\ub}}\rho\cdot(K_{1}\psi-t\psi)\right|+C\delta^{2}\int_{-r_{0}}^{t}(-t')^{-2}E(t',\ub)dt'.
\end{align}
Substituting \eqref{K1 linear temp 1} into \eqref{K0 linear temp 1} we obtain

\begin{align}\label{K0 linear temp 2'}
\begin{split}
E(t,\ub)+F(t,\ub)\lesssim& E(-r_{0},\ub)+\delta^{-1}\Eb(-r_{0},\ub)+\delta\int_{-r_{0}}^{t}(-t')^{-2}E(t',\ub)dt'\\
&+\sup_{t'\in[-r_{0},t]}\left|\int_{W^{t}_{\ub}}\rho\cdot(K_{1}\psi-t\psi)\right|+\left|\int_{W^{t}_{\ub}}\rho\cdot K_{0}\psi\right|,
\end{split}
\end{align}
which implies

\begin{align}\label{K0 linear temp 2}
\begin{split}
E(t,\ub)+F(t,\ub)\lesssim& E(-r_{0},\ub)+\delta^{-1}\Eb(-r_{0},\ub)\\
&+\sup_{t'\in[-r_{0},t]}\left|\int_{W^{t'}_{\ub}}\rho\cdot(K_{1}\psi-t\psi)\right|+\left|\int_{W^{t}_{\ub}}\rho\cdot K_{0}\psi\right|.
\end{split}
\end{align}
Substituting \eqref{K0 linear temp 2} in \eqref{K1 linear temp 1} we obtain

\begin{align}\label{K1 linear temp 2}
\begin{split}
\Eb(t,\ub)+\Fb(t,\ub)+K(t,\ub)\lesssim &\Eb(-r_{0},\ub)+\delta^{2}E(-r_{0},\ub)\\
&+\sup_{t'\in[-r_{0},t]}\left|\int_{W^{t'}_{\ub}}\rho\cdot(K_{1}\psi-t\psi)\right|+\delta^{2}\sup_{t'\in[-r_{0},t]}\left|\int_{W^{t'}_{\ub}}\rho\cdot K_{0}\psi\right|.
\end{split}
\end{align}

\begin{remark}\label{multiplier interaction}
Here we remark that the spacetime integral involving $K_{1}$ in \eqref{K0 linear temp 2} and the spacetime integral involving $K_{0}$ in \eqref{K1 linear temp 2} are lower order compared to the other spacetime integral in the corresponding inequality. More specifically, in \eqref{K0 linear temp 2}, if we disregard the behavior with respect to $t$, $K_{1}\psi-t\psi\sim \delta^{1/2}, K_{0}\psi\sim \delta^{-1/2}$, so the term involving $K_{0}$ is lower order. Similar argument applies to \eqref{K1 linear temp 2} in which the factor $\delta^{2}$ makes the term involving $K_{0}$ lower order. 
\end{remark}
\section{Estimates for rectangular coordinates as functions of optical coordinates, Estimates for non-top order terms}
Our energy estimates for $(\star)$ is with respect to the optical coordinates $(t,\ub,\theta)$. In order to go back to the rectangular coordinates $(t,x^{i})$, one needs to investigate the relation between the rectangular coordinates $(t,x^{i})$ and the optical coordinates $(t,\ub,\theta)$. In what follows, we will consider $x^{i}$s as functions of $(t,\ub,\theta)$ and estimate their derivatives with respect to optical coordinates. In the meantime, we also estimate the quantities $y^{k}, z^{k}, \lambda_{i}$ and their derivatives in optical coordinates. As a by product, we will also obtain estimates for the lower order objects, i.e., with order $< \Ntop+1$.

Given a vectorfield $V$, we define the null components of its deformation tensor as
\begin{equation}\label{definition for Zb and pislash}
{}^{(V)}\Zb_A = {}^{(V)}\pi (\Lb, X_A), \ \ {}^{(V)}Z_A = {}^{(V)}\pi (L, X_A), \ \  {}^{(V)}\slashed{\pi}_{AB} =  {}^{(V)}\pi (X_A, X_B).
\end{equation}
The projection of Lie derivative $\mathcal{L}_{V}$ to $S_{t,\ub}$ is denoted as $\slashedL_{V}$.
The shorthand notation $\slashedL_{Z_i}^{\, \alpha}$ to denote $\slashedL_{Z_{i_1}} \slashedL_{Z_{i_2}}\cdots \slashedL_{Z_{i_k}} $
for a multi-index $\alpha =(i_1,\cdots,i_k)$. We will show that,
for all $|\alpha| \leq \Ninfty$, we have $(-t)\slashedL_{Z_i}^{\, \alpha} \chib', \slashedL_{Z_i}^{\, \alpha}  {}^{(Z_j)}\Zb, \slashedL_{Z_i}^{\, \alpha}  {}^{(Z_j)}\slashed{\pi}, (-t)^{-1}\slashedL_{Z_i}^{\, \alpha}  {}^{(Q)}\left(\slashed{\pi}+4\right) \in \O^{|\alpha|+1}_{2-2l,2}$,
where $l$ is the number of $T$'s in $Z_i$'s, $|\alpha|\geq 1$ and $Z_{j}\slashed{=}T$. If $Z_{j}=T$, then we have $(-t)\slashedL_{Z_i}^{\, \alpha}  {}^{(T)}\Zb, \slashedL_{Z_i}^{\, \alpha}  {}^{(T)}\slashed{\pi} \in \O^{|\alpha|+1}_{-2l,1}$. Similarly, we will derive $L^2$-estimates on objects of order $\leq \Nmu$. The $L^2$ estimates depend on the $L^\infty$ estimates up to order $\Ninfty+2$.
In the course of the proof, it will be clear why $\Ninfty$ is chosen to be approximately $\frac{1}{2} N_{top}$.
\subsection{$L^\infty$ estimates}
We assume (B.1): for all $\ |\alpha| \leq N_{\infty}$, $Z_i^{\alpha+2} \psi \in \Psi^{|\alpha|+2}_{1-2l,1}$. 
\begin{proposition}\label{L infity estimates on lot}
For sufficiently small $\delta$, for all $|\alpha| \leq \Ninfty$ and $t\in [-r_0, s^*]$, we can bound $\big\{(-t)\slashedL_{Z_i}^{\, \alpha} \chib'$, $\slashedL_{Z_i}^{\, \alpha}  {}^{(Z_j)}\Zb$, $\slashedL_{Z_i}^{\, \alpha}  {}^{(Z_j)}\slashed{\pi}, (-t)^{-1}\slashedL_{Z_i}^{\, \alpha}  {}^{(Q)}\left(\slashed{\pi}+4\right), Z_i^{\,\alpha+1} y^j, (-t)^{-1}Z_i^{\,\alpha+1} \lambda_j\big\} \subset \O_{2-2l,2}^{|\alpha|+1}$ and\\ $\big\{
(-t)^{-2}Z_i^{\,\alpha+2} x^j,  (-t)^{-1}Z_i^{\,\alpha+1} \widehat{T}^j, (-t)^{-2}\slashedL_{Z_i}^{\, \alpha+1} Z_j, \slashedL_{Z_i}^{\, \alpha}  {}^{(T)}\Zb, \slashedL_{Z_i}^{\, \alpha}  {}^{(T)}\slashed{\pi} \big\} \subset \O_{-2l,1}^{|\alpha|+1}$ in terms of $Z_{i}^{\alpha+1}\psi$ which belongs to $\Psi^{|\alpha|+2}_{1-2l,1}$. Here $l$ is the number of $T$'s in $Z_i$'s.
\end{proposition}
\begin{proof}
We do induction on the order. When $|\alpha|=0$, the estimates are treated in Section 3. Here we only treat the estimates when $l=0$, when $l\geq1$, we can use the structure equation \eqref{Structure Equation T chib} to reduce the problem to the estimates for $\mu$, which is treated in Proposition \ref{prop mu L infty estimates}. Given $|\alpha| \leq \Ninfty$, 
we assume that estimates hold for terms of order $\leq |\alpha|$. In particular, 
we have $\slashedL_{R_i}^{\, \beta-1} \chib', R_i^{\,\beta} y^j \in \O^{|\beta|+1}_{2,2}$ for all $|\beta| \leq |\alpha|$. We prove the proposition for $|\alpha|+1$.

\underline{Step 1 \ Bounds $R_i^{\,\alpha+1} x^j$}. \ \  Let $\delta_{\alpha+1,i}^j = \Omega_i^{\alpha+1} x^j - R_i^{\,\alpha+1} x^j$ where $\Omega_i$'s are the standard rotational vectorfields on Euclidean space. It is obvious that $\Omega_i^\alpha x^j$ is equal to some $x^i$, therefore, bounded by $r$. Since $R_i = \Omega_i -\lambda_i \widehat{T}^j \partial_j$, $R_{i}x^{j}\in\O^{0}_{0,-1}$ and by ignoring all the numerical constants, we have
\begin{equation*}
 \delta_{\alpha+1,i}^j = R_{i}^{\alpha}\left(\lambda_i \widehat{T}^j\right)= R_i^{\alpha} \left(\lambda_i \left(\frac{x^j}{\ub-t}+y^j\right)\right).
\end{equation*}
Here the index $i$ is not a single index. It means we apply a string of different $R_{i}$s. This notation applies in the following when a string of $R_{i}$s are considered. 

By the induction hypothesis, the right hand side above is in $\O^{|\alpha|}_{2,1}$. Therefore $R^{\alpha+1}_{i}x^{j}\in \O^{|\alpha|}_{0,-1}$.

\underline{Step 2 \ Bounds on $\slashedL_{R_i}^{\, \alpha} \chib'$}. \ \ We commute $\slashedL_{R_i}^{\, \alpha}$ with \eqref{transport equation for chib prime} to derive
\begin{equation}\label{b2}
\begin{split}
\slashedL_{\Lb} \slashedL_{R_i}^{\, \alpha}  \chib' &= \left[ \slashedL_{\Lb}, \slashedL_{R_i}^{\, \alpha} \right]\chib' + e \cdot \slashedL_{R_i}^{\, \alpha} \chib' + \chib' \cdot  \slashedL_{R_i}^{\, \alpha} \chib' + \sum_{|\beta_1|+|\beta_2| = |\alpha| \atop |\beta_2|<|\alpha|}  {R_i}^{\beta_1} e \cdot \slashedL_{R_i}^{\beta_2}\chib' \\
&\quad +   \sum_{|\beta_1|+|\beta_2| +|\beta_3| = |\alpha| \atop |\beta_1|<|\alpha|, |\beta_3|<|\alpha|}  \slashedL_{R_i}^{\beta_1} \chib' \cdot \slashedL_{R_i}^{\beta_2}\slashed{g}^{-1} \cdot \slashedL_{R_i}^{\beta_3}\chib'+ \slashedL_{R_i}^{\, \alpha}\big(\frac{e\slashed{g}_{AB}}{t-\ub} -\alphab'_{AB}\big).
\end{split}
\end{equation}
Since $e =c^{-1}\dfrac{d c}{d\rho}\Lb\rho$, ${R_i}^{\alpha} e$ is of order $\leq |\alpha|+1$. By (B.1), we have $R_{i}^\beta e \in \Psi_{2,3}^{|\alpha|+1}$.
Similarly, by the explicit formula of $\alphab'_{AB}$, we have $\slashedL_{R_i}^\beta \alphab'_{AB} \in \Psi_{2,4}^{|\alpha|+2}$. Since ${}^{(R_i)}\pi_{AB} = 2c^{-1}\lambda_i \chib_{AB}$ and $\slashedL_{R_i}^{\beta_2}\slashed{g} = \slashedL_{R_i}^{\beta_2-1}\,{}^{(R_i)}\slashed{\pi}_{AB}$, by the estimates derived in previous sections and by the induction hypothesis, we can rewrite \eqref{b2} as
\begin{equation*}
\slashedL_{\Lb} \slashedL_{R_i}^{\, \alpha}  \chib' = \big[ \slashedL_{\Lb}, \slashedL_{R_i}^{\, \alpha} \big]\chib' +\O_{2,3}^{1}\cdot \slashedL_{R_i}^{\, \alpha} \chib'+\Psi_{2,4}^{\leq |\alpha|+2}.
\end{equation*}
The commutator can be computed as $\left[ \slashedL_{\Lb}, \slashedL_{R_i}^{\, \alpha} \right]\chib' = \sum_{|\beta_1|+|\beta_2| = |\alpha|-1} \slashedL_{R_i}^{\beta_1}\slashedL_{{}^{(R_i)}\Zb}\slashedL_{R_i}^{\beta_2}\chib'$.
Since ${}^{(R_i)}\Zb_A = -\chib'_{AB}R_i{}^B + \varepsilon_{ijk}z^{j}X_A{}^k + \lambda_i \slashed{d}_A (c)$ and $z^j=-\frac{(c-1)x^j}{\ub-t}-cy^j$, the commutator term is of type $\mathcal{O}^{1}_{2,2}\cdot\slashed{\mathcal{L}}^{\alpha-1}_{R_{i}}\chib'$. Therefore, we have
\begin{equation*}
\slashedL_{\Lb} \slashedL_{R_i}^{\, \alpha}  \chib' = \O_{2,3}^{1}\cdot \slashedL_{R_i}^{\, \alpha} \chib'+\Psi_{2,4}^{\leq|\alpha|+2}.
\end{equation*}
By integrating this equation from $-r_0$ to $t$, the Gronwall's inequality yields $\|\slashedL_{R_i}^{\, \alpha} \chib'\|_{L^\infty(\Sigma_{t}^{\ub})} \lesssim_M \delta(-t)^{-3}$.

\underline{Step 3 \ Bounds on $R_i^{\,\alpha+1} y^j$ and $R_i^{\,\alpha+1}\lambda_{j}$}.  Since $R_i y^j = (-c^{-1}\chib^{A}_{B}-\frac{\delta^{A}_{B}}{\ub-t}) R_{i}^A \slashed{d}_B x^j$, schematically we have
\begin{equation*}
 R_i^{\,\alpha+1} y^j = R_{i}^{\alpha}R_k y^j = R_i^{\alpha}\left( \left(c^{-1}\chib+\frac{\delta}{\ub-t}\right)\cdot R_k \cdot \slashed{d} x^j \right).
\end{equation*}
 We distribute $R_i^\alpha$ inside the parenthesis by Leibniz rule. (Here again, the index $i$ is not a single index, so we use index $k$ to distinct the last rotation vectorfield.) Therefore, a typical term would be either $\slashedLRi^{\beta_1}\chib\cdot \slashedLRi^{\beta_2}\slashed{g}^{-1}\cdot \slashedLRi^{\beta_3} R_k \cdot \slashed{d} R_i^{\beta_4} x^j$ or $\slashedLRi^{\beta_1}\slashed{g}\cdot \slashedLRi^{\beta_2}\slashed{g}^{-1}\cdot \slashedLRi^{\beta_3} R_k \cdot \slashed{d} R_i^{\beta_4} x^j$ with $|\beta_1|+|\beta_2|+|\beta_3|+|\beta_4| = |\alpha|$.
There are only two terms where are not included in the induction hypothesis: $\slashedLRi^{\beta_2} \slashed{g}_{AB}$ and $\slashedLRi^{\beta_3} R_j$. The first term is in fact easy to handle by induction hypothesis and estimates derived in Step 1 and Step 2, since $\slashedLRi \slashed{g}_{AB} = {}^{(R_i)}\slashed{\pi}_{AB} = 2\lambda_i c^{-1}\chib_{AB}$. For the second one, we use the following expression:
\begin{align*}
\slashedLRi R_j = -\sum_{k=1}^3 \varepsilon_{ijk} R_k + \lambda_i\big( \frac{c^{-1}-1}{\ub-t}\slashed{g}&-c^{-1}(\chib+\frac{\slashed{g}}{\ub-t})\big)R_j-\lambda_j\big( \frac{c^{-1}-1}{\ub-t}\slashed{g}-c^{-1}(\chib+\frac{\slashed{g}}{\ub-t})\big)R_i\\
 & -\lambda_i \varepsilon_{jkl}y^k\slashed{d}x^l \cdot \slashed{g}^{-1}+ \lambda_j \varepsilon_{ikl}y^k\slashed{d}x^l\cdot \slashed{g}^{-1}.
\end{align*}
Therefore, $\slashedLRi^{\beta_3} R_k =\slashedLRi^{\beta_3-1} \slashedLRi R_k = \O_{0,-1}^{|\beta_3|-1} +\O_{\geq2,\geq2}^{|\beta_3|-1} R_i^{\beta_3}x^j $. Finally, we obtain that
\begin{equation*}
 R_i^{\,\alpha+1} y^j = \O^{\leq |\alpha|+1}_{2,2} + \O^{\leq |\alpha|}_{2,\geq2} \cdot \slashed{d} R_i^{\beta_4} x^j.
\end{equation*}
Although $ \slashed{d} R_i^{\beta_4} x^j$ and $\slashedLRi^{\beta_1}\chib$ may have order $|\alpha|+1$, they have been controlled from previous steps. This gives the bounds on $R_i^{\,\alpha+1} y^j$. Then by the fact that $\lambda_{j}=\epsilon_{jkl}x^{k}y^{l}$, the estimate for $R^{\,\alpha+1}_{i}\lambda_{j}$ follows. This completes the proof of the proposition.
\end{proof}
\begin{proposition}\label{prop mu L infty estimates}
For sufficiently small $\delta$, for all $|\alpha|\leq \Ninfty, t \in [-r_{0},s^{*}]$, we can bound $Z_i^{\alpha+1} \mu \in \O_{-2l,1}^{|\alpha|+1}$ in terms of $Z^{\alpha+2}_{i}\psi\in\Psi^{|\alpha|+2}_{1-2l,1}$.
\end{proposition}
\begin{proof}
We do induction on order. The base case $|\alpha|=0$ has be treated in Section 3 and Section 4.  We assume the proposition holds with order of derivatives on $\mu$ at most $|\alpha|$. For $|\alpha|+1$, by commuting $Z^{\alpha+1}_{i}$ with $\Lb\mu=m+\mu e$, we have
\begin{align*}
\Lb \delta^{l} Z^{\alpha+1}_{i}\mu = (e + \leftexp{(Z_{i})}{\Zb}) \delta^{l}Z^{\alpha+1}_{i}\mu + \delta^{l}Z^{\alpha+1}_{i}m +\sum_{|\beta_1|+|\beta_2|\leq |\alpha|+1 \atop|\beta_{1}|<|\alpha|}\delta^{l_{1}}Z_{i}^{\beta_{1}}\mu \delta^{l_{2}}Z_{i}^{\beta_{2}}e
\end{align*}
where $l_{a}, a=1,2$ is the number of $T$'s in $Z^{\beta_{a}}$'s and $l$ is the number of $T$'s in $Z^{\alpha}$'s.
By the induction hypothesis, the above equation can be written as:
\begin{align*}
\Lb \delta^{l}Z^{\alpha+1}_{i}\mu=\O^{\leq1}_{2,2}\cdot \delta^{l} Z^{\alpha+1}_{i}\mu+\Psi^{\leq |\alpha|+2}_{0,2}
\end{align*}
Similar to the estimates derived in the Step 3 in previous section, we can use induction hypothesis and Gronwall's inequality to conclude that $\|Z^{\alpha+1}_{i}\mu\|_{L^{\infty}(\Sigma_{t})}\lesssim_M \delta^{-l}(-t)^{-1}$. 
\end{proof}
\subsection{$L^2$ estimates}
$\Ntop$ will be the total number of derivatives commuted with $\Box_{\widetilde{g}} \psi = 0$. The highest order objects will be of order $\Ntop +1$. In this subsection, based on (B.1) and (B.2),  we will derive $L^2$ estimates on the objects of order $\leq \Ntop$ in terms of the $L^{2}$ norms of $Z^{\alpha+2}_{i}\psi\in\Psi^{|\alpha|+2}_{1-2l}$ with $|\alpha|\leq N_{top}-1$.
\begin{proposition}\label{Proposition lower order L2}
 For sufficiently small $\delta$, for all $\alpha$ with $|\alpha| \leq \Ntop-1$ and $t\in [-r_0, s^*]$, the $L^2(\Sigma_t^{\ub})$ norms of all the quantities listed below
\begin{equation*}
\slashedL_{Z_i}^{\, \alpha} \chib', (-t)^{-1}\slashedL_{Z_i}^{\, \alpha}  {}^{(Z_j)}\Zb, (-t)^{-1}\slashedL_{Z_i}^{\, \alpha}  {}^{(Z_j)}\slashed{\pi}, (-t)^{-1}Z_i^{\,\alpha+1} y^j, (-t)^{-2}Z_i^{\,\alpha+1} \lambda_j, 
\end{equation*}
are bounded \footnote{The inequlaity is up to a constant depending only on the bootstrap constant $M$.} by $\delta^{1/2-l}\int_{-r_{0}}^{t}(-t')^{-3}\mu_{m}^{-1/2}(t')\sqrt{\Eb_{\leq|\alpha|+2}(t',\ub)}dt'$, where $l$ is the number of $T$'s in $Z_i$'s.
\end{proposition}

\begin{proof}
We use an induction argument on the order of derivatives. When $|\alpha|=0$, the result follows from the estimates in Section 3 and 4. Again, here we only treat the case $l=0$. The case $\l\geq 1$ can be treated using \eqref{Structure Equation T chib}. By assuming the proposition holds for terms with order $\leq |\alpha|$, 
we show it holds for $|\alpha|+1$.

\underline{Step 1 \ Bounds on $\slashedL_{R_i}^{\, \alpha} \chib'$}. \ \  By affording a $\Lb$-derivative,  we have
\begin{equation}\label{b3}
\|\slashedL_{R_i}^{\, \alpha}  \chib'\|_{L^2(\Sigma_t^{\ub})} \lesssim \|\slashedL_{R_i}^{\, \alpha}  \chib'\|_{L^2(\Sigma_{-r_0}^{\ub})} + \int_{-r_0}^t \||\chib'||\slashedL_{R_i}^{\, \alpha}  \chib'| + |\slashedL_{\Lb} \slashedL_{R_i}^{\, \alpha}  \chib'|\|_{L^2(\Sigma_\tau^{\ub})} d\tau.
\end{equation}
We use formula \eqref{b2} to replace $\slashedL_{\Lb} \slashedL_{R_i}^{\, \alpha}  \chib'$ by the  terms with lower orders. Each nonlinear term has at most one factor with order $>\Ninfty$. We bound this factor in $L^2(\Sigma_t)$ and the rest in $L^\infty$. We now indicate briefly how the estimates on the factors involving $e$ and $\alphab'$ work.

For $\alphab'$, since $\alphab'_{AB} = c\frac{d c}{d\rho}\slashed{D}^2_{A,B}\rho + \frac{1}{2}\big[\dfrac{d^2(c^2)}{d\rho^2}-\dfrac{1}{2c^2}\big(\dfrac{d c^2}{d\rho}\big)^2 \big]X_A(\rho)X_B(\rho)$, in view of the definition of $\Eb(t,\ub)$, for sufficiently small $\delta$, we have 
\begin{align*}
\|\slashedL_{R_i}^{\alpha} \alphab'_{AB}\|_{L^2(\Sigma_t^{\ub})} \lesssim\sum_{|\alpha|\leq k}(-t)^{-2}\|\ds R^{\alpha+1}_{i}\psi\|_{L^{2}(\Sigma_{t}^{\ub})}\lesssim_M \delta^\frac{1}{2}(-t)^{-3}\mu_{m}^{-1/2}(t) \displaystyle\sqrt{\Eb_{\leq |\alpha|+2}(t,\ub)}.
\end{align*}

For $e$, since $e =c^{-1}\dfrac{d c}{d\rho}\Lb\rho$,  we have
\begin{align*}
\|R_{i}^{\alpha} e\|_{L^2(\Sigma_t^{\ub})} &\lesssim_M (-t)^{-2}\sum_{|\beta|\leq |\alpha|} \left(\delta^{1/2} \|\ds R^{|\beta|-1}_{i}\psi\|_{L^{2}(\Sigma_{t}^{\ub})}+\delta^{1/2} \|\ds R^{|\beta|-1}_{i}Q \psi\|_{L^{2}(\Sigma_{t}^{\ub})}\right)\\
&\lesssim_M \delta^{1/2}(-t)^{-3}\mu_{m}^{-1/2}(t) \displaystyle\sqrt{\Eb_{\leq |\alpha|+1}(t,\ub)}.
\end{align*}

By applying Gronwall's inequality to \eqref{b3}, we obtain immediately that
\begin{align*}
\|\slashedL_{R_i}^{\, \alpha} \chib'\|_{L^2(\Sigma_t^{\ub})} \leq\|\slashedL_{R_i}^{\, \alpha} \chib'\|_{L^2(\Sigma_{-r_{0}}^{\ub})} +C_{M} \delta^{1/2}\int_{-r_{0}}^{t}  \displaystyle(-\tau)^{-3}\mu_{m}^{-1/2}(\tau)\sqrt{\Eb_{\leq |\alpha|+2}(\tau,\ub)}d\tau
\end{align*}

\underline{Step 2 \ Bounds on $R_i^{\,\alpha+1} y^j$}. \ \ By the computations in the Step 2 of the proof of Proposition 
\ref{L infity estimates on lot}, $R_i^{\,\alpha+1} y$ is a linear combination of the terms such as 
$\slashedLRi^{\beta_1}\chib\cdot \slashedLRi^{\beta_2}\slashed{g}^{-1}\cdot \slashedLRi^{\beta_3} R_j \cdot \slashed{d} R_i^{\beta_4} x^j$ 
or as $\slashedLRi^{\beta_1}\slashed{g}\cdot \slashedLRi^{\beta_2}\slashed{g}^{-1}\cdot \slashedLRi^{\beta_3} R_j \cdot \slashed{d} R_i^{\beta_4} x^j$,
 where $|\beta_1|+|\beta_2|+|\beta_3|+|\beta_4| = |\alpha|$. Similarly, we bound all factors with order $\leq \Ninfty$ by 
 the $L^\infty$ estimates in Proposition \ref{L infity estimates on lot}. By the induction hypothesis, this yields the bound on $R_i^{\,\alpha+1} y^j$ immediately. The estimates for other quantities follow from the estimates of $\chib'$ and $y^{j}$. Again, as in \cite{M-Y}, In this process, the terms like $R_{i}^{\beta}x^{j}$ and the leading term in $\slashed{\mathcal{L}}_{R_{i}}R_{j}$, which can be bounded by a constant $C(-t)$ disregarding the order of the derivatives, are bounded in $L^{\infty}$. The rest terms in $\slashed{\mathcal{L}}_{R_{i}}R_{j}$, which depend on $\chib'$ and $y^{j}$ as well as their derivatives, are bounded in $L^{2}$ based on the $L^{2}$ estimates for $\slashed{\mathcal{L}}_{R_{i}}^{\beta}\chib'$.
\end{proof}
We also have $L^2$ estimates for derivatives of $\mu$.
\begin{proposition}\label{Proposition lower order L2 mu}
For sufficiently small $\delta$, for all $\alpha$ with $|\alpha| \leq \Ntop-1$ and $t\in [-r_0, s^*]$, we have
\begin{align*}
\delta^{l}(-t)^{-1}\|Z_{i}^{\alpha+1}\mu\|_{L^{2}(\Sigma^{\ub}_{t})}
\lesssim_M &\delta^{l}(-r_{0})^{-1}\|Z_{i}^{\alpha+1}\mu\|_{L^{2}(\Sigma^{\ub}_{-r_{0}})}\\
+&\delta^{1/2}\int_{-r_{0}}^{t}(-\tau)^{-2}\left(
\sqrt{E_{\leq |\alpha|+2}(\tau,\ub)}+\mu_{m}^{-1/2}(\tau)(-\tau)^{-1}\sqrt{\Eb_{\leq |\alpha|+2}(\tau,\ub)}\right)d\tau.
\end{align*}
\end{proposition}
\begin{proof}
According to the proof of Proposition \ref{prop mu L infty estimates}, we have
\begin{align*}
\delta^{l}\Lb\left((-t)^{-1}|Z^{\alpha+1}_{i}\mu|\right)\lesssim &\delta^{l}(-t)^{-1}|Z^{\alpha+1}_{i}m|+\delta^{l}(-t)^{-2}|Z_{i}^{\alpha+1}\mu|
\\+&\delta^{l}(|e|+|\leftexp{(R_{i})}{\Zb}|) (-t)^{-1}|Z^{\alpha+1}_{i}\mu|\\
+&\sum_{|\beta_{1}+\beta_2|\leq|\alpha|}\delta^{l_{1}}(-t)^{-1}|Z_{i}^{\beta_{1}}\mu| \delta^{l_{2}}|R_{i}^{\beta_{2}}e|.
\end{align*}
Then the result follows by using Gronwall and an induction argument, which are similar as in the proof of Proposition 6.2.
\end{proof}
\section{Estimates for the top order optical terms}
As we already stated, the highest possible order of an object in the paper will be $\Ntop+1$. The current section is devoted to the $L^2$ estimates of $\ds X_{A} R_i^{\,\alpha-1}\tr\chib$ and $X_{A}Z_i^{\,\alpha+1} \mu$ with $|\alpha| = \Ntop-1$. Here we choose $X_{A}$ as the first member of the string of commutators because this avoids a logarithmic divergence in the estimates. When there is no confusion, we use the notation $\mathcal{R}_{i}^{\alpha}:=X_{A}R^{\alpha-1}_{i}$.
\subsection{Estimates for the contribution from $\tr\chib$}
Since we deal with top order terms, we can not use the transport equation \eqref{Structure Equation Lb trchib} directly as in the previous section, which causes a loss of derivative. Instead, we derive an elliptic system coupled with a transport equation for $\widehat{\chib}$ and $ \slashed{d}\tr\chib$:
\begin{equation*}
\Lb (\slashed{d}X_{A}R_i^{\alpha-1}\tr\chib) = \widehat{\chib}\cdot \slashed{\nabla}\mathcal{L}_{X_{A}}\mathcal{L}_{R_i}^{\alpha-1} \widehat{\chib}+ \cdots, \ \ \slashed{\text{div}}\mathcal{L}_{X_{A}} \mathcal{L}_{R_i}^{\alpha-1} \widehat{\chib} = \slashed{d}X_{A}R_i^{\alpha-1}\tr\chib + \cdots.
\end{equation*}
The new idea is using elliptic estimates and rewriting the right hand side of the transport equations to avoid the loss of derivatives. To rewrite the equation, we need to use the inhomogeneous wave equation satisfied by $\rho:=\psi_{0}^{2}$:
\begin{equation}\label{wave equation for rho}
\Box_g \rho = \frac{d(\log(c))}{d \rho}g^{\mu\nu}\partial_\mu \rho \partial_\nu \rho + 2g^{\mu\nu}\partial_\mu \psi_0 \partial_\nu \psi_0
\end{equation}
Then following the same procedure as in \cite{M-Y}, we obtain:
\begin{equation}\label{renormalized equation for tr chib}
\Lb \big(\mu \tr\chib -\check{f} \big) = 2 \Lb\mu \tr\chib -\frac{1}{2}\mu (\tr\chib)^2-\mu |\widehat{\chib}|^2+\check{g},
\end{equation}
with
\begin{align*}
\check{f} &= -\dfrac{1}{2}\dfrac{d(c^2)}{d\rho}L\rho,\\
\check{g} &= \left( 2 \left(\frac{d(c)}{d\rho}\right)^2 + c\frac{d^2(c)}{d \rho^2} \right)\left(\Lb\rho L\rho-\mu|\slashed{d}\rho|^2\right)\\
&+ 2 c\frac{d(c)}{d\rho}\left(\left(L\psi_0 \Lb
\psi_0 -\mu |\slashed{d}\psi_0|^2\right)+ \left(\frac{1}{4}\frac{\mu|\slashed{d}\rho|^2}{c^2}-\zetab^A\slashed{d}_A\rho\right)\right).
\end{align*}
being of order at most $1$ and regular as $\mu\rightarrow0$. So there will be no loss of derivatives by integrating the equation \eqref{renormalized equation for tr chib}. Let us introduce the notation $F_{\alpha}$:
\begin{align*}
F_{\alpha}=  \mu\slashed{d}\left(X_{A}R_i{}^{\alpha-1}\tr\chib\big)- \slashed{d}\big(X_{A}R_i{}^{\alpha-1} \check{f}\right).
\end{align*}
Therefore $F_{\alpha}$ satisfies the equation:
\begin{equation}\label{transport equation for F alpha for chib}
\slashedL_{\Lb} F_\alpha + \left(\tr\chib-2\mu^{-1}\Lb \mu\right)F_\alpha =\left(-\frac{1}{2}\tr\chib + 2\mu^{-1}\Lb\mu\right)\slashed{d}\left(X_{A}R_i{}^{\alpha-1}\check{f}\right)
-\mu\slashed{d}\left(X_{A}R_i{}^{\alpha-1}\left(|\widehat{\chib}|^2\right)\right) + g_\alpha,
\end{equation}
with $g_\alpha$ in the following schematic expression (by setting all the numerical constants to be $1$):
\begin{align*}
g_\alpha &= \slashedL_{X_{A}}\slashedL_{R_i}{}^{\alpha-1} g_0 + \sum_{|\beta_1|+|\beta_2| = |\alpha|-1}\!\!\!\!\!\!\!\!\!\slashedL_{\mathcal{R}_i}^{\beta_1}\slashedL_{^{(\mathcal{R}_i)}\Zb}F_{\beta_2} + \sum_{|\beta_1|+|\beta_2| = |\alpha|-1} \!\!\!\!\!\!\!\!\!\slashedL_{\mathcal{R}_i}^{\beta_1} \Big(\big(\mu \mathcal{R}_i\tr\chib + \mathcal{R}_i\Lb\mu + {}^{(\mathcal{R}_i)}\Zb \mu\big)\slashed{d}\big(\mathcal{R}_i{}^{\beta_2}\tr\chib\big)\Big)\\
&\quad + \sum_{|\beta_1|+|\beta_2| = |\alpha|-1}\!\!\!\!\!\!\!\!\!\slashedL_{\mathcal{R}_i}^{\beta_1} \Big(\mathcal{R}_i \mu \big[\slashedL_{\Lb}\slashed{d}\big(\mathcal{R}_i{}^{\beta_2}\tr\chib\big)
+\tr\chib\slashed{d}\big(\mathcal{R}_i{}^{\beta_2}\tr\chib \big) + \slashed{d}\big(\mathcal{R}_i{}^{\beta_2}(|\widehat{\chib}|^2)\big)\big] + \mathcal{R}_i\tr\chib \slashed{d}\big(\mathcal{R}_{i}{}^{\beta_2}\check{f}\big)\Big).
\end{align*}
Here
\begin{align*}
g_0 = \slashed{d}\check{g}-\frac{1}{2}\tr\chib \,\slashed{d}\big(\check{f}-2\Lb\mu\big)-(\slashed{d}\mu)\big(\Lb \tr\chib+|\widehat{\chib}|^2\big)
\end{align*}
and $\mathcal{R}_{i}$ is either $R_{i}$ or $X_{A}$. Moreover, in the string of $\mathcal{R}_{i}$'s, $X_{A}$ appears once exactly and in front of $R_{i}$'s. 
For any form $\xi$, since $|\xi| \Lb |\xi| = (\xi,\slashedL_{\Lb}\xi) -\xi \cdot\widehat{\chib}\cdot \xi-\dfrac{1}{2}\tr\chib|\xi|^2$, we have $\Lb |\xi| \leq  |\slashedL_{\Lb}\xi| +|\widehat{\chib}||\xi|-\dfrac{1}{2}\tr\chib|\xi|$. Applying these to \eqref{transport equation for F alpha for chib}, we obtain
\begin{equation}\label{eq 11}
\Lb |F_\alpha| \leq (\mu^{-1}\Lb \mu - \frac{3}{2}\tr\chib +|\widehat{\chib}|)|F_\alpha| + \big(2\mu^{-1}|\Lb\mu| -\tr\chib\big)|\slashed{d} X_{A} R_i{}^{\alpha-1}\check{f}|+|\mu\slashed{d} X_{A} R_i{}^{\alpha-1}\left(|\widehat{\chib}|^2\right)| + |g_\alpha|.
\end{equation}
or 
\begin{align}\label{inequality F alpha}
\Lb |F_\alpha| \leq (\mu^{-1}\Lb \mu - \frac{3}{2}\tr\chib +|\widehat{\chib}|)|F_\alpha| + \left(2\mu^{-1}|\Lb\mu| -\tr\chib\right)|\slashed{d} X_{A}R_i{}^{\alpha-1}\check{f}|+\widetilde{g}_{\alpha}.
\end{align}
with 
\begin{align*}
\widetilde{g}_{\alpha}:=\left|\mu\slashed{d} X_{A}R_i{}^{\alpha-1}\left(|\widehat{\chib}|^2\right)\right| + |g_\alpha|
\end{align*}
In terms of $\chib'_{AB}=\chib_{AB}+\frac{\slashed{g}_{AB}}{\ub-t}$, we write the above inequality in optical coordinates as:
\begin{align}\label{revised ineq}
&\frac{\partial|F_{\alpha}|}{\partial t}+\frac{3}{t-\ub}|F_{\alpha}|+\left(-\mu^{-1}\frac{\partial\mu}{\partial t}+\frac{3}{2}\textrm{tr}\underline{\chi'}-|\widehat{\underline{\chi}}|\right)|F_{\alpha}|\\\notag
\leq&\left(2\mu^{-1}|\partial_{t}\mu|-\tr\chib\right)|\ds X_{A}R_{i}^{\alpha-1}\check{f}|+\widetilde{g}_{\alpha}.
\end{align}
Multiplying both sides by $-(t-\ub)^{3}$, we obtain:
\begin{align*}
&\partial_{t}\left(\left(-t+\ub\right)^{3}|F_{\alpha}|\right)+\left(-\mu^{-1}\frac{\partial\mu}{\partial t}+\frac{3}{2}\textrm{tr}\underline{\chi'}-|\widehat{\underline{\chi}}|\right)\left(-t+\ub\right)^{3}|F_{\alpha}|\\
\leq &(-t+\ub)^{3}\left(2\mu^{-1}|\partial_{t}\mu|-\tr\chib\right)|\ds X_{A}R_{i}^{\alpha-1}\check{f}|+(-t+\ub)^{3}\widetilde{g}_{\alpha}.
\end{align*}
For the term involving $-\mu^{-1}\partial_{t}\mu$ on the left hand side, if $\mu<1/10$, then $\partial_{t}\mu<0$, and this term can be dropped. Otherwise, since $\mu\geq 1/10$, $-\mu\partial_{t}\mu$ can be bounded in absolute value by $C(-t)^{-2}$, with $C$ being an absolute constant. Therefore, in view of the pointwise bounds for $\chib'$ and $\widehat{\chib}$, the second term on the left hand side can be bounded, after it is moved to the right hand side, as:
\begin{align*}
C(-t)^{-2}\left(-t+\ub\right)^{3}|F_{\alpha}|.
\end{align*}
So this can be treated by Gronwall and we obtain:
\begin{align*}
(-t+\ub)^{3}|F_{\alpha}(t,\ub)|\leq& C\Big(\int_{-r_{0}}^{t}\left((-t'+\ub)^{3}(2\mu^{-1}|\partial_{t}\mu|-\tr\chib)(t,\ub)|\ds X_{A}R_{i}^{\alpha-1}\check{f}(t',\ub)|\right)dt'\\\notag
&+\int_{-r_{0}}^{t}\left((-t'+\ub)^{3}\widetilde{g}_{\alpha}(t',\ub)\right)dt'\Big)\\
&:=I+II 
\end{align*}
 Let us now investigate the $L^{2}$ norms of the quantities appearing in the above inequality. For a smooth function $\phi$, the $L^{2}$ norm of it on $[0,\delta]\times \mathbb{S}^{2}$ is defined by:
\begin{align*}
\|\phi(t)\|_{L^{2}([0,\delta]\times\mathbb{S}^{2})}:=\sqrt{\int_{[0,\delta]\times\mathbb{S}^{2}}\phi(t,\ub)^{2}r_{0}^{-2}d\mu_{\slashed{g}(-r_{0},0)}d\ub}
\end{align*}
On $\Sigma_{t}^{\delta}$, the $L^{2}$ norm is defined by:
\begin{align*}
\|\phi(t)\|_{L^{2}(\Sigma_{t}^{\delta})}:=\sqrt{\int_{[0,\delta]\times\mathbb{S}^{2}}\phi(t,\ub)^{2}d\mu_{\slashed{g}(t,\ub)}d\ub}
\end{align*}
The relation between these two norms is:
\begin{align*}
(-t)\|\phi(t)\|_{L^{2}([0,\delta]\times\mathbb{S}^{2})}\lesssim\|\phi(t)\|_{L^{2}(\Sigma_{t}^{\delta})}\lesssim(-t)\|\phi(t)\|_{L^{2}([0,\delta]\times\mathbb{S}^{2})}
\end{align*}
This discussion also applies to tensors. 

Now we give an estimate for $(-t)^{2}\|F_{\alpha}\|_{L^{2}(\Sigma_{t}^{\ub})}$. In view of 
\begin{align*}
(-t)^{3}\|F_{\alpha}(t)\|_{l^{2}([0,\ub]\times\mathbb{S}^{2})}\lesssim(-t)^{2}\|F_{\alpha}\|_{L^{2}(\Sigma_{t}^{\ub})}\lesssim(-t)^{3}\|F_{\alpha}(t)\|_{l^{2}([0,\ub]\times\mathbb{S}^{2})},
\end{align*}
we only need to give an estimate for $(-t)^{3}\|F_{\alpha}(t)\|_{l^{2}([0,\ub]\times\mathbb{S}^{2})}$. We first estimate $I$:
\begin{align*}
\|I\|_{L^{2}([0,\ub]\times\mathbb{S}^{2})}\lesssim\int_{-r_{0}}^{t}(-t')^{3}\|(2\mu^{-1}|\partial_{t}\mu|-\tr\chib)|\ds X_{A}R_{i}^{\alpha-1}\check{f}|(t',\ub')\|_{L^{2}([0,\ub]\times\mathbb{S}^{2})}dt'
\end{align*}
In view of the definition of $\ds X_{A}R_{i}^{\alpha-1}\check{f}$ and \eqref{Calculus Inequality}, we have:
\begin{align*}
&\||\ds X_{A}R_{i}^{\alpha-1}\check{f}|(t',\ub)\|_{L^{2}([0,\ub]\times\mathbb{S}^{2})}\\
\lesssim&\delta^{1/2}(-t')^{-3}\sum_{|\beta|\leq |\alpha|+1}\|TR_{i}^{\beta}\psi_{0}\|_{L^{2}([0,\ub]\times\mathbb{S}^{2})}\\
+&\delta^{-1/2}(-t')^{-3}\sum_{|\beta|\leq |\alpha|}\|R^{\beta}_{i}\psi_{0}\|_{L^{2}([0,\ub]\times\mathbb{S}^{2})}\\
\lesssim&\delta^{1/2}(-t')^{-3}\sum_{|\beta|\leq |\alpha|+1}\|TR_{i}^{\beta}\psi_{0}\|_{L^{2}([0,\ub]\times\mathbb{S}^{2})}
\end{align*}
Let $t_{0}$ be the first point in evolution for which $\mu_{m}(t_{0})=\dfrac{1}{10}$. When $t'\in[-r_{0},t_{0}]$, we have $\mu_{m}(t')\geq 1/10$. The contribution to $\|I\|_{L^{2}([0,\ub]\times\mathbb{S}^{2})}$ is bounded by:
\begin{align*}
&C\delta^{1/2}\int_{-r_{0}}^{t_{0}}(-t')^{-1}\sum_{|\beta|\leq |\alpha|+1}\|TR_{i}^{\beta}\psi_{0}\|_{L^{2}([0,\ub]\times\mathbb{S}^{2})}dt'\\
\leq& C\delta^{1/2}\int_{-r_{0}}^{t}(-t')^{-2}\sqrt{E_{\leq |\alpha|+2}(t',\ub)}dt'\\
\leq &C\delta^{1/2}\int_{-r_{0}}^{t}\mu_{m}^{-b_{|\alpha|+2}}(t')(-t')^{-2}\sqrt{\widetilde{E}_{\leq|\alpha|+2}(t',\ub)}dt'
\end{align*}
In order to treat the case when $t'\in[t_{0},t]$, we need an analogy to the Lemma 8.1 in \cite{M-Y}, concerning the behavior of $\left(\Lb(\log\mu)\right)_{-}$. 
Let us define the following quantities:
\begin{equation*}
\begin{split}
M(t)&:=\max_{(\ub,\theta), (t,\ub,\theta)\in W_{shock}}\big|\big(\Lb(\log\mu)\big)_{-}(t,\ub,\theta)\big|,\\ I_{a}(t)&:=\int_{t_{0}}^{t}\mu_{m}^{-a}(\tau)M(\tau)d\tau,
\end{split}
\end{equation*}
where $a \in \mathbb{R}_{>0}$ is a constant.
\begin{lemma}\label{lemma on mu to power a} (1) Given a constant $a\geq4$, for all $t\in[t_{0},t^{*})$, we have
\begin{align}\label{key lemma}
I_{a}(t) \lesssim a^{-1}\mu^{-a}_{m}(t).
\end{align}

(2) For $a\geq 4$ and $\delta$ sufficiently small, there is an absolute constant $C_0$ independent of $a$, so that for all $\tau\in[-r_{0},t]$, we have
\begin{align}\label{mu is decreasing}
\mu^{a}_{m}(t)\leq C_0 \mu^{a}_{m}(\tau)
\end{align}
\end{lemma}
\begin{proof}
(1) By Proposition \ref{Proposition C3}, for $t\geq t_{0}$, the minimum of $r_{0}^{2}(\Lb\mu)(-r_{0},\ub,\theta)$ on $[0,\delta]\times\mathbb{S}^{2}$ is negative and we denote it by
\begin{align}\label{min Lbmu initial}
-\eta_{m}=\min_{(\ub,\theta)\in[0,\delta]\times \mathbb{S}^{2}}\{r_{0}^{2}(\Lb\mu)(-r_{0},\ub,\theta)\}.
\end{align}
We notice that $1\leq \eta_{m}\leq C_{m}$ where $C_{m}$ is a constant depending on the initial data. In view of the asymptotic expansion for $(\Lb\mu)(t,\ub,\theta)$ in Lemma \ref{lemma on the expansion of Lb mu}, we have
\begin{align}\label{mu expansion in use}
\begin{split}
\mu(t,\ub,\theta)=&1-\left(\frac{1}{t}+\frac{1}{r_{0}}\right)r_{0}^{2}(\Lb\mu)(-r_{0},\ub,\theta)+O\left(\delta M^{4}\right)\left(\frac{1}{t^{2}}-\frac{1}{r_{0}^{2}}\right).
\end{split}
\end{align}
We fix an $s\in (t_{0}, t^{*})$ in such a way that $t_{0}\leq t<s<t^{*}$. There exists $(\ub_{s},\theta_{s})\in[0,\delta]\times \mathbb{S}^{2}$ and $(\ub_{m},\theta_{m})\in[0,\delta]\times\mathbb{S}^{2}$ so that 
\begin{align}\label{achieve min}
\mu(s,\ub_{s},\theta_{s})=\mu_{m}(s),\quad r_{0}^{2}(\Lb\mu)(-r_{0},\ub_{m},\theta_{m})=-\eta_{m}.
\end{align}
We claim that
\begin{align}\label{mins are close}
\left|\eta_{m}+r_{0}^{2}(\Lb\mu)(-r_{0},\ub_{s},\theta_{s})\right|\leq O(\delta M^{4}).
\end{align}
Indeed, one can apply \eqref{mu expansion in use} to $\mu(s,\ub_{m},\theta_{m})$ and $\mu(s,\ub_{s},\theta_{s})$ to derive
\begin{align}
\begin{split}
\mu(s,\ub_{s},\theta_{s})=&1-\left(\frac{1}{s}+\frac{1}{r_{0}}\right)(-\eta_{m}+d_{ms})+O(\delta M^{4})\left(\frac{1}{t^{2}}-\frac{1}{r_{0}^{2}}\right)\\
\mu(s,\ub_{m},\theta_{m})=&1-\left(\frac{1}{s}+\frac{1}{r_{0}}\right)(-\eta_{m})+O(\delta M^{4})\left(\frac{1}{t^{2}}-\frac{1}{r_{0}^{2}}\right),
\end{split}
\end{align}
where the quantity $d_{ms}>0$ is defined as 
\begin{align}\label{dms}
d_{ms}:=\eta_{m}+r_{0}^{2}(\Lb\mu)(-r_{0},\ub_{s},\theta_{s}).
\end{align}
Since $\mu(s,\ub_{s},\theta_{s})\leq \mu(s,\ub_{m},\theta_{m})$, we have
\begin{align}\label{dms esti pre}
0<-\left(\frac{1}{s}+\frac{1}{r_{0}}\right)d_{ms}\leq O(\delta M^{4})\left(\frac{1}{t^{2}}-\frac{1}{r_{0}^{2}}\right).
\end{align}
Hence,
\begin{align*}
-\left(\frac{1}{t}+\frac{1}{r_{0}}\right)d_{ms}\leq -\left(\frac{1}{s}+\frac{1}{r_{0}}\right)d_{ms}\leq O(\delta M^{4})\left(\frac{1}{t^{2}}-\frac{1}{r_{0}^{2}}\right),
\end{align*}
which implies

\begin{align}\label{dms esti}
d_{ms}\leq O(\delta M^{4})\left(\frac{1}{r_{0}}-\frac{1}{t}\right)\leq O(\delta M^{4}).
\end{align}
With this preparation, one can derive precise upper and lower bounds for $\mu_{m}(t)$. 

We pick up a $(\ub'_{m},\theta'_{m})\in[0,\delta]\times\mathbb{S}^{2}$ in such a way that $\mu(t,\ub'_{m},\theta'_{m})=\mu_{m}(t)$.  For the lower bound, by virtue of Lemma \ref{lemma on the expansion of Lb mu}, we have
\begin{equation}\label{mu m lower bound temp 1}
\begin{split}
\mu_{m}(t)=&\mu(t,\ub'_{m},\theta'_{m})=\mu(s,\ub'_{m},\theta'_{m})+\int_{s}^{t}(\Lb\mu)(t',\ub'_{m},\theta'_{m})dt'\\
\geq &\mu_{m}(s)+\int_{s}^{t}\frac{\eta_{m}}{-t^{\prime2}}+\frac{O(\delta M^{4})}{(-t')^{3}}dt'\\
\geq &\mu_{m}(s)+\left(\eta_{m}-\frac{1}{2a}\right)\left(\frac{1}{t}-\frac{1}{s}\right).
\end{split}
\end{equation}
In the last step, we take sufficiently small $\delta$ so that $O(\delta M^{4})\leq \frac{1}{2a}$. 

For the upper bound, in view of Lemma \ref{lemma on the expansion of Lb mu} and \eqref{dms esti}, we have
\begin{align}\label{mu m upper bound temp 1}
\begin{split}
\mu_{m}(t)\leq& \mu(t,\ub_{s},\theta_{s})=\mu_{m}(s)+\int_{s}^{t}(\Lb\mu)(t',\ub_{s},\theta_{s})dt'\\
=& \mu_{m}(s)+\int_{s}^{t}\frac{\eta_{m}-d_{ms}}{-t^{\prime2}}+\frac{O(\delta M^{4})}{(-t')^{3}}dt'\\
\leq &\mu_{m}(s)+\int_{s}^{t}\frac{\eta_{m}}{-t^{\prime2}}+\frac{O(\delta M^{4})}{t^{\prime2}}dt'\\
\leq &\mu_{m}(s)+\left(\eta_{m}+\frac{1}{2a}\right)\left(\frac{1}{t}-\frac{1}{s}\right)
\end{split}.
\end{align}
In the last step, we also take sufficiently small $\delta$ so that $O(\delta M^{4})\leq \frac{1}{2a}$. 

For $I_{a}(t)$, first of all, we have
\begin{align*}
I_{a}(t)\lesssim& \int_{t_{0}}^{t}\left(\mu_{m}(s)+\left(\eta_{m}-\frac{1}{2a}\right)\left(\frac{1}{t}-\frac{1}{s}\right)\right)^{-a-1}t^{\prime-2}dt'\\
=&\int_{\tau}^{\tau_{0}}\left(\mu_{m}(s)+\left(\eta_{m}-\frac{1}{2a}\right)\left(\tau-\tau_{s}\right)\right)^{-a-1}d\tau'\\
\leq &\frac{1}{\eta_{m}-\frac{1}{2a}}\frac{1}{a}\left(\mu_{m}(s)+\left(\eta_{m}-\frac{1}{2a}\right)\left(\tau-\tau_{s}\right)\right)^{-a}.
\end{align*}
Hence,
\begin{align}\label{la temp 1}
\begin{split}
I_{a}(t) \lesssim &\frac{1}{a}\left(\mu_{m}(s)+\left(\eta_{m}-\frac{1}{2a}\right)\left(\frac{1}{t}-\frac{1}{s}\right)\right)^{-a}\\
\leq&\frac{1}{a}\frac{\left(\mu_{m}(s)+\left(\eta_{m}-\frac{1}{2a}\right)\left(\frac{1}{t}-\frac{1}{s}\right)\right)^{-a}}{\left(\mu_{m}(s)+\left(\eta_{m}+\frac{1}{2a}\right)\left(\frac{1}{t}-\frac{1}{s}\right)\right)^{-a}}\mu^{-a}_{m}(t)\\
\leq &\frac{1}{a}\frac{\left(\left(\eta_{m}-\frac{1}{2a}\right)\left(\frac{1}{t}-\frac{1}{s}\right)\right)^{-a}}{\left(\left(\eta_{m}+\frac{1}{2a}\right)\left(\frac{1}{t}-\frac{1}{s}\right)\right)^{-a}}\mu^{-a}_{m}(t).
\end{split}
\end{align}
Since as $a\rightarrow\infty$, one has
\begin{align*}
\frac{\left(\eta_{m}-\frac{1}{2a}\right)^{-a}}{\left(\eta_{m}+\frac{1}{2a}\right)^{-a}}\rightarrow e^{\frac{1}{\eta_{m}}}.
\end{align*}
The limit is an absolute constant. Therefore, \eqref{la temp 1} yields the proof for part (1) of the lemma. The proof for part (1') is exactly the same.

(2) We start with an easy observation: if $\mu(t,\ub,\theta)\leq 1-\dfrac{1}{a}$, then $\Lb \mu(t,\ub,\theta)\lesssim -a^{-1}$. In fact, we claim that $(\frac{1}{t}+\frac{1}{r_{0}})r_{0}^{2}\Lb\mu(-r_{0},\ub,\theta)\geq \frac{1}{2}a^{-1}$.
Otherwise, for sufficiently small $\delta$ (say $\delta^{1/4}\leq a^{-1}$), according to the expansion for $\mu(t,\ub,\theta)$, i.e. $\mu(t,\ub,\theta)=\mu(-r_{0},\ub,\theta)-(\frac{1}{t}+\frac{1}{r_{0}})r^{2}_{0}\Lb\mu(-r_{0},\ub,\theta)+\frac{O(\delta)}{t}$, we have
\begin{align*}
\mu(t,\ub,\theta)> 1-\frac{1}{2a}-C\frac{\delta}{t}\geq 1-\frac{1}{a}.
\end{align*}
which is a contradiction. So in view of the fact that $\dfrac{tr_{0}}{t+r_{0}}$ is bounded above by a negative absolute constant  when $t\in(-r_{0},-1]$, we have $r_{0}^{2}\Lb\mu(-r_{0},\ub,\theta)\lesssim-a^{-1}$. Therefore the expansion of $\Lb\mu$ implies $\Lb \mu(t,\ub,\theta)\lesssim -a^{-1}t^{-2}$. In particular, this observation implies that, if there is a $t' \in[-r_{0},s^{*}_{m}]$, so that $\mu_{m}(t')\leq 1-a^{-1}$, then for all $t\geq t'$, we have $\mu_{m}(t)\leq 1-a^{-1}$. This allows us to define a time $t_1$, such that it is the minimum of all such $t'$ with $\mu_{m}(t')\leq 1-a^{-1}$.

We now prove the lemma. If $\tau \leq t_{1}$, since $\mu_m(t) \leq 1$, we have
\begin{align*}
\mu_{m}^{-a}(\tau)\leq (1-\frac{1}{a})^{-a}\leq C_0\leq C_0\mu^{-a}_{m}(t).
\end{align*}
If $\tau \geq t_{1}$, then $\mu_{m}(\tau)\leq 1-\dfrac{1}{a}$. let $\mu_{m}(\tau)=\mu(\tau,\ub_{\tau},\theta_{\tau})$. We know that $\mu(t,\ub_\tau,\theta_\tau)$ is decreasing in $t$ for $t\geq \tau$. Therefore, we have
\begin{align*}
\mu_{m}(t)\leq \mu(t,\ub_{\tau},\theta_{\tau})\leq \mu(\tau,\ub_{\tau},\theta_{\tau})=\mu_{m}(\tau).
\end{align*}
The proof now is complete.
\end{proof}

Let us continue to estimate $\|I\|_{L^{2}([0,\ub]\times\mathbb{S}^{2})}$. When $t'\in[t_{0},t]$, the contribution of the term $\mu^{-1}\partial_{t}\mu$ is bounded through Lemma \ref{lemma on mu to power a} by:
\begin{align*}
C\delta^{1/2}\int_{t_{0}}^{t}\mu_{m}^{-b_{|\alpha|+2}}(t')(-t')^{-1}M(t')\sqrt{\widetilde{E}_{\leq|\alpha|+2}(t',\ub)}dt'\leq \frac{C\delta^{1/2}(-t)^{-1}}{b_{|\alpha|+2}}\mu^{-b_{|\alpha|+2}}_{m}(t)\sqrt{\widetilde{E}_{\leq|\alpha|+2}(t,\ub)}
\end{align*}
Therefore we have the following estimates for $I$:
\begin{align}\label{L2 chib I}
\begin{split}
\|I\|_{|L^{2}([0,\ub]\times\mathbb{S}^{2})}\leq &\frac{C\delta^{1/2}(-t)^{-1}}{b_{|\alpha|+2}}\mu^{-b_{|\alpha|+2}}_{m}(t)\sqrt{\widetilde{E}_{\leq|\alpha|+2}(t,\ub)}\\
&+C\delta^{1/2}\int_{-r_{0}}^{t}(-t')^{-2}\mu_{m}^{-b_{|\alpha|+2}}(t')\sqrt{\widetilde{E}_{\leq|\alpha|+2}(t'\ub)}dt'.
\end{split}
\end{align}

Next we move to the estimate for $\|II\|_{l^{2}([0,\ub]\times\mathbb{S}^{2})}$. We first investigate the structure of $g_{\alpha}$, which is the sum of four terms. In order to estimate the first term, $\slashed{\mathcal{L}}_{R_{i}}^{\alpha}g_{0}$, we rewrite $g_{0}$ in another way. Taking the trace of the equation \eqref{Structure Equation Lb chibAB nonsingular}, we can rewrite the last term in $g_{0}$ as
\begin{align*}
-(\ds\mu)\left(\Lb\tr\chib+|\chib|^{2}\right)=-(\ds\mu)\left(e\tr\chib-\tr\underline{\alpha'}\right).
\end{align*}
Therefore the contribution of last term in $g_{0}$ to $\|\slashed{\mathcal{L}}^{\alpha}_{R_{i}}g_{0}\|_{L^{2}([0,\delta]\times\mathbb{S}^{2})}$ is $\|\slashed{\mathcal{L}}^{\alpha}_{R_{i}}(\ds\mu)\left(e\tr\chib-\tr\underline{\alpha'}\right)\|_{L^{2}([0,\delta]\times\mathbb{S}^{2})}$. Since the top order is $|\alpha|+2$, there is no top order optical terms in this contribution. In regard to the $L^{2}$ norms of lower order optical terms, we introduce the following notations:
\begin{align}\label{lower order optical norms}
\begin{split}
\mathcal{A}^{\ub}_{l}(t):&=\max_{|\alpha|=l}\|\slashed{\mathcal{L}}_{R_{i}}^{\alpha}\chib\|_{L^{2}(\Sigma_{t}^{\ub})},\\
\mathcal{A}^{\prime\ub}_{l}(t):&=\max_{|\alpha|=l}\|\slashed{\mathcal{L}}_{R_{i}}^{\alpha}\chib'\|_{L^{2}(\Sigma_{t}^{\ub})},\\
\mathcal{B}^{\ub}_{m,l+1}(t):&=\max_{|\beta|+|\gamma|= l}\|R_{i}^{\beta+1}T^{\gamma}\mu\|_{L^{2}(\Sigma_{t}^{\ub})},\\
\mathcal{A}^{\ub}_{\leq l}(t):&=\sum_{k=1}^{l}\mathcal{A}_{k}^{\ub}(t),\\
\mathcal{A}^{\prime\ub}_{\leq l}(t):&=\sum_{k=1}^{l}\mathcal{A}_{k}^{\prime\ub}(t),\\
\mathcal{B}^{\ub}_{\leq m, \leq l+1}(t):&=\sum_{m'\leq m, l'\leq l}\mathcal{B}^{\ub}_{m', l'+1}(t).
\end{split}
\end{align}
It's obvious that 
\begin{align*}
\|\mathcal{R}_{i}^{\alpha}\tr\chib\|_{L^{2}([0,\delta]\times\mathbb{S}^{2})}=\|\mathcal{R}_{i}^{\alpha}\tr\chib'\|_{L^{2}([0,\delta]\times\mathbb{S}^{2})}\quad\text{for}\quad |\alpha|\geq1.
\end{align*}
To estimate  the factor $\|\slashed{\mathcal{L}}_{\mathcal{R}_{i}}^{\alpha}\left((\ds\mu)(e\tr\chib)\right)\|_{L^{2}(\Sigma_{t}^{\ub})}$, we need to use the Leibniz rule. If more than half of the derivatives hit $e$, then the corresponding contribution is bounded by:
\begin{align*}
&C\delta^{1/2}(-t)^{-5}\sum_{|\beta|\leq |\alpha|}\left(\|R^{\beta}_{i}Q\psi_{0}\|_{L^{2}(\Sigma_{t}^{\ub})}+\|R^{\beta}_{i}\psi_{0}\|_{L^{2}(\Sigma_{t}^{\ub})}\right)\\
\leq & C\delta^{3/2}(-t)^{-5}\sqrt{E_{\leq |\alpha|+2}(t,\ub)}\\
\leq & C\delta^{3/2}(-t)^{-5}\mu_{m}^{-b_{|\alpha|+2}}(t)\sqrt{\widetilde{E}_{\leq|\alpha|+2}(t,\ub)}
\end{align*}
in view of the pointwise estimates of $\ds R^{\beta}_{i}\mu$, $R^{\beta}_{i}\tr\chib$ and $R^{\beta}_{i}\psi$, $R_{i}^{\beta}Q\psi$ when $|\beta|\leq N_{\infty}$. Similarly, if more than half of the derivatives hit $\tr\chib$ and $\ds\mu$, their corresponding contribution are bounded by:
\begin{align*}
\delta(-t)^{-5}\mathcal{A}'_{\leq |\alpha|}(t)\quad\text{and}\quad \delta(-t)^{-5}\mathcal{B}_{0,\leq|\alpha|+1}(t)
\end{align*}
respectively. 

On the other hand, in view of Proposition \ref{Proposition lower order L2} and \ref{Proposition lower order L2 mu}, we have the following estimates for $\mathcal{A}'_{\leq|\alpha|}(t)$ and $\mathcal{B}_{0,\leq|\alpha|+1}(t)$:
\begin{align*}
\mathcal{A}'_{\leq|\alpha|}(t)\lesssim&\delta^{1/2}\int_{-r_{0}}^{t}(-t')^{-3}\mu_{m}^{-1/2}(t')\sqrt{\Eb_{\leq|\alpha|+2}(t',\ub)}dt',\\
\lesssim &\delta^{1/2}(-t)^{-2}\mu_{m}^{-b_{|\alpha|+2}-1/2}(t)\sqrt{\widetilde{\Eb}_{\leq|\alpha|+2}(t,\ub)};\\
(-t)^{-1}\mathcal{B}_{0,\leq|\alpha|+1}(t)\lesssim & (-r_{0})^{-1}\mathcal{B}_{0,\leq|\alpha|+1}(-r_{0})\\
+&\delta^{1/2}\int_{-r_{0}}^{t}(-\tau)^{-2}\left(\sqrt{E_{\leq|\alpha|+2}(\tau,\ub)}+\mu_{m}^{-1/2}(\tau)(-\tau)^{-1}\sqrt{\Eb_{\leq|\alpha|+2}(\tau,\ub)}\right)d\tau,\\
\lesssim & (-r_{0})^{-1}\mathcal{B}_{0,\leq|\alpha|+1}(-r_{0})\\
+&(-t)^{-1}\left(\mu_{m}^{-b_{|\alpha|+2}}(t)\sqrt{\widetilde{E}_{\leq|\alpha|+2}(t,\ub)}+\mu_{m}^{-b_{|\alpha|+2}-1/2}(t)(-t)^{-1}\sqrt{\widetilde{\Eb}_{\leq|\alpha|+2}(t,\ub)}\right).
\end{align*}
Combining the above estimates, we have:
\begin{align}\label{3rd term in g0 1}
\begin{split}
\|\slashed{\mathcal{L}}_{\mathcal{R}_{i}}^{\alpha}\left((\ds\mu)(e\tr\chib)\right)\|_{L^{2}(\Sigma_{t}^{\ub})}\lesssim &\delta(-t)^{-5}\mathcal{B}_{0,\leq|\alpha|+1}(-r_{0})\\
+&C\delta^{3/2}(-t)^{-5}\mu_{m}^{-b_{|\alpha|+2}}(t)\sqrt{\widetilde{E}_{\leq|\alpha|+2}(t,\ub)}\\
+&C\delta^{3/2}(-t)^{-6}\mu_{m}^{-b_{|\alpha|+2}+1/2}(t)\sqrt{\widetilde{\Eb}_{\leq|\alpha|+2}(t,\ub)}.
\end{split}
\end{align}
For the last term in $\slashed{\mathcal{L}}_{\mathcal{R}_{i}}^{\alpha}g_{0}$, we are still left with the term $\|\slashed{\mathcal{L}}_{\mathcal{R}_{i}}^{\alpha}\left(\left(\ds\mu\right)\tr\alpha'\right)\|_{L^{2}(\Sigma_{t}^{\ub})}$. If more than half of the derivatives hit $\ds\mu$, the estimates are exactly the same as the previous case. If more than half of the derivatives hit $\tr\alpha'$, then the contribution is bounded by:
\begin{align}\label{3rd term in g0 2}
\delta^{1/2}(-t)^{-4}\|\ds R^{\alpha+1}_{i}\psi_{0}\|_{L^{2}(\Sigma_{t}^{\ub})}\lesssim\delta^{1/2}(-t)^{-5}\mu^{-1/2}_{m}(t)\sqrt{\Eb_{\leq|\alpha|+2}(t,\ub)}
\end{align}
Combining \eqref{3rd term in g0 1} and \eqref{3rd term in g0 2}, we have
\begin{align}\label{3rd term in g0}
\begin{split}
\|\slashed{\mathcal{L}}^{\alpha}_{\mathcal{R}_{i}}\left((\ds\mu)(e\tr\chib-\tr\alpha')\right)\|_{L^{2}(\Sigma_{t}^{\ub})}\lesssim&\delta(-t)^{-4}\mathcal{B}_{0,\leq|\alpha|+1}(-r_{0})\\
+&C\delta^{3/2}(-t)^{-4}\mu_{m}^{-b_{|\alpha|+2}}(t)\sqrt{\widetilde{E}_{\leq|\alpha|+2}(t,\ub)}\\
+&C\delta^{1/2}(-t)^{-5}\mu_{m}^{-b_{|\alpha|+2}-1/2}(t)\sqrt{\widetilde{\Eb}_{\leq|\alpha|+2}(t,\ub)}.
\end{split}
\end{align}
Next we consider the contribution from the second term in the expression of $g_{0}$, whose $L^{2}$ norm is given by $\|\slashed{\mathcal{L}}_{\mathcal{R}_{i}}^{\alpha}
\left(\tr\chib\cdot\ds\left(\check{f}-2\Lb\mu\right)\right)\|_{L^{2}(\Sigma_{t}^{\ub})}$. For the first term $\|\slashed{\mathcal{L}}_{\mathcal{R}_{i}}^{\alpha}
\left(\tr\chib\cdot\ds\check{f}\right)\|_{L^{2}(\Sigma_{t}^{\ub})}$, if more than half of the derivatives hit $\tr\chib$, then it is bounded by
\begin{align}\label{2nd term in g0 1}
(-t)^{-3}\mathcal{A}'_{\leq|\alpha|}(t)\lesssim\delta^{1/2}(-t)^{-5}\mu_{m}^{-b_{|\alpha|+2}-1/2}(t)\sqrt{\widetilde{\Eb}_{\leq|\alpha|+2}(t,\ub)}.
\end{align}
If more than half of the derivatives hit $\ds\check{f}$, then this term is bounded by:
\begin{align}\label{2nd term in g0 2}
\delta^{1/2}(-t)^{-3}\sqrt{\widetilde{E}_{\leq|\alpha|+2}(t,\ub)}\lesssim\delta^{1/2}(-t)^{-3}\mu_{m}^{-b_{|\alpha|+2}}(t)\sqrt{\widetilde{E}_{\leq|\alpha|+2}(t,\ub)}.
\end{align}
Here we also used \eqref{Calculus Inequality} if most of the derivatives hit $\psi_{0}$ instead of $L\psi_{0}$. The treatment for the second term $\|\slashed{\mathcal{L}}_{\mathcal{R}_{i}}^{\alpha}\left(\tr\chib
\cdot\ds\Lb\mu\right)\|_{L^{2}(\Sigma_{t}^{\ub})}$ is similar. If more than half of the derivatives hit $\tr\chib$, the estimate is exactly the same as \eqref{2nd term in g0 1}. If more than half of the derivatives hit $\ds\Lb\mu$, then this term is bounded by
\begin{align}\label{2nd term in g0 3}
\delta^{1/2}(-t)^{-3}\mu_{m}^{-b_{|\alpha|+2}}(t)\sqrt{\widetilde{E}_{\leq|\alpha|+2}(t,\ub)}+\delta^{1/2}(-t)^{-4}\mu_{m}^{-b_{|\alpha|+2}-1/2}(t)\sqrt{\Eb_{\leq|\alpha|+2}(t,\ub)}.
\end{align}
Combining \eqref{2nd term in g0 1}, \eqref{2nd term in g0 2} and \eqref{2nd term in g0 3}, we have:
\begin{align}\label{2nd term in g0}
\begin{split}
&\|\slashed{\mathcal{L}}^{\alpha}_{\mathcal{R}_{i}}\left(\tr\chib
\cdot\ds\left(\check{f}-2\Lb\mu\right)\right)\|_{L^{2}(\Sigma_{t}^{\ub})}\\\lesssim
&\delta^{1/2}(-t)^{-3}\mu_{m}^{-b_{|\alpha|+2}}(t)\sqrt{\widetilde{E}_{\leq|\alpha|+2}(t,\ub)}+\delta^{1/2}(-t)^{-4}\mu_{m}^{-b_{|\alpha|+2}-1/2}(t)\sqrt{\Eb_{\leq|\alpha|+2}(t,\ub)}.
\end{split}
\end{align}

To complete the estimates for the contribution from $\slashed{\mathcal{L}}_{R_{i}}^{\alpha}g_{0}$, we finally estimate the contribution from $\check{g}$: $\|\slashed{\mathcal{L}}_{R_{i}}^{\alpha}\left(\ds\check{g}\right)\|_{L^{2}(\Sigma_{t}^{\ub})}$. Recall the expression of $\check{g}$:
\begin{align*}
\check{g}=&\frac{1}{2}\frac{d^{2}(c^{2})}{d\rho^{2}}\left[\underline{L}\rho L\rho-\mu|\slashed{d}\rho|^{2}\right]-\frac{1}{2}\frac{d(c^{2})}{d\rho}\Omega^{-1}
\frac{d\Omega}{d\rho}\left[\underline{L}\rho L\rho-\mu|\slashed{d}\rho|^{2}\right]\\
+&\frac{d(c^{2})}{d\rho}\left[L\psi_{0}\underline{L}\psi_{0}-\mu|\slashed{d}\psi_{0}|^{2}\right]
+\frac{d(c^{2})}{d\rho}\left[\frac{1}{4}c^{-2}\mu|\slashed{d}\rho|^{2}-\zetab^{A}\cdot\slashed{d}_{A}\rho\right].
\end{align*}
Here we will only treat the top order terms, namely, when all the derivatives hit one factor. The estimates for other lower order terms follow in the same manner.
We first consider the contribution from the term $\Lb\rho L\rho$. When all of $\slashed{\mathcal{L}}_{\mathcal{R}_{i}}^{\alpha}$ hit $L\rho$, the contribution is bounded by:
\begin{align}\label{1st term in g0 1}
\delta^{3/2}(-t)^{-4}\sqrt{E_{\leq|\alpha|+2}(t,\ub)}\leq \delta^{3/2}(-t)^{-4}\mu_{m}^{-b_{|\alpha|+2}}(t)\sqrt{\widetilde{E}_{\leq|\alpha|+2}(t,\ub)}.
\end{align}
Here we also used \eqref{Calculus Inequality} when $\slashed{\mathcal{L}}_{\mathcal{R}_{i}}^{\alpha}$ hit $\psi_{0}$.

When all of $\slashed{\mathcal{L}}_{R_{i}}^{\alpha}$ hit $\Lb\rho$,  the contribution is bounded by:
\begin{align}\label{1st term in g0 2}
\begin{split}
\delta^{1/2}(-t)^{-3}\|\Lb \mathcal{R}^{\alpha+1}_{i}\psi_{0}\|_{L^{2}(\Sigma_{t}^{\ub})}\lesssim &\delta^{1/2}(-t)^{-4}\|\Lb R^{\alpha+1}_{i}\psi_{0}+\underline{\nu}R^{\alpha+1}_{i}\psi_{0}\|_{L^{2}(\Sigma_{t}^{\ub})}\\
+ &\delta^{3/2}(-t)^{-5}\sqrt{E_{\leq|\alpha|+2}(t,\ub)}\\
\lesssim & \delta^{1/2}(-t)^{-5}\mu_{m}^{-1/2}(t)\sqrt{\Eb_{\leq|\alpha|+2}(t,\ub)}\\
+&\delta^{3/2}(-t)^{-5}\sqrt{E_{\leq|\alpha|+2}(t,\ub)}\\
\lesssim & \delta^{1/2}(-t)^{-5}\mu_{m}^{-1/2-b_{|\alpha|+2}}(t)\sqrt{\widetilde{\Eb}_{\leq|\alpha|+2}(t,\ub)}\\
+&\delta^{3/2}(-t)^{-5}\mu_{m}^{-b_{|\alpha|+2}}(t)\sqrt{\widetilde{E}_{\leq|\alpha|+2}(t,\ub)}.
\end{split}
\end{align}
 Here \eqref{Calculus Inequality} is used when $\slashed{\mathcal{L}}_{\mathcal{R}_{i}}^{\alpha}$ hit $\psi_{0}$.

Next we consider the contribution of $\mu|\ds\rho|^{2}$. This is bounded by
\begin{align}\label{1st term in g0 3}
\delta^{3/2}(-t)^{-4}\sqrt{\Eb_{\leq|\alpha|+2}(t,\ub)}\lesssim\delta^{3/2}(-t)^{-4}\mu_{m}^{-b_{|\alpha|+2}}(t)\sqrt{\widetilde{\Eb}_{\leq|\alpha|+2}(t,\ub)}.
\end{align}
By the definition of $\zetab$, the contribution from $\zetab^{A}\cdot\slashed{d}_{A}\rho$ enjoys the same bound. 

Now we move to the contribution from $\mu|\ds\psi_{0}|^{2}$. This is bounded by:
\begin{align}\label{1st term in g0 4}
\delta^{1/2}(-t)^{-4}\sqrt{\Eb_{\leq|\alpha|+2}(t,\ub)}\lesssim\delta^{1/2}(-t)^{-4}\mu_{m}^{-b_{|\alpha|+2}}(t)\sqrt{\widetilde{\Eb}_{\leq|\alpha|+2}(t,\ub)}.
\end{align}
Finally we estimate the most difficult term: $\ds\left(\Lb\psi_{0} L\psi_{0}\right)$. When all of $\slashed{\mathcal{L}}_{R_{i}}^{\alpha}$ and $\ds$ hit $L\psi_{0}$, this contribution is bounded by:
\begin{align}\label{1st term in g0 5}
\delta^{1/2}(-t)^{-3}\sqrt{E_{\leq|\alpha|+2}(t,\ub)}\lesssim\delta^{1/2}(-t)^{-3}\mu_{m}^{-b_{|\alpha|+2}}(t)\sqrt{\widetilde{E}_{\leq|\alpha|+2}(t,\ub)}.
\end{align}
When $\slashed{\mathcal{L}}_{\mathcal{R}_{i}}^{\alpha}$ and $\ds$ hit $\Lb\psi_{0}$, the treatment is different. The spacetime integral we want to estimate is (Remember $R_{i}\sim(-t)\ds, R_{i}\sim(-t)X_{A}$.):
\begin{align*}
\int_{-r_{0}}^{t}(-t')\|L\psi_{0}\Lb \mathcal{R}^{\alpha+1}_{i}\psi_{0}\|_{L^{2}(\Sigma_{t'}^{\ub})}dt'\leq C\delta^{-1/2}\int_{-r_{0}}^{t}(-t')^{-1}\|\Lb R^{\alpha+1}_{i}\psi_{0}\|_{L^{2}(\Sigma_{t'}^{\ub})}dt'.
\end{align*}
In order to use the flux $\Fb(t,\ub)$, the right hand side of the above is bounded through \eqref{Calculus Inequality} by:
\begin{align*}
\delta^{-1/2}\int_{-r_{0}}^{t}(-t')^{-1}\|\Lb R^{\alpha+1}_{i}\psi_{0}+\underline{\nu}R^{\alpha+1}_{i}\psi_{0}\|_{L^{2}(\Sigma_{t'}^{\ub})}dt'+\delta^{1/2}\int_{-r_{0}}^{t}(-t')^{-2}\sqrt{E_{\leq|\alpha|+2}(t',\ub)}dt'.
\end{align*}
The second term above is bounded in the same way as the contribution of \eqref{1st term in g0 5} to the spacetime integral. We focus on the first term above. By using H\"older inequality, this term is bounded by:
\begin{align}\label{1st term in g0 6}
\begin{split}
&\delta^{-1/2}\left(\int_{-r_{0}}^{t}(-t')^{-3}dt'\right)^{1/2}\left(\int_{-r_{0}}^{t}(-t')^{2}\int_{0}^{\ub}\int_{\mathbb{S}^{2}}\left(\Lb R^{\alpha+1}_{i}\psi_{0}+\underline{\nu}R^{\alpha+1}_{i}\psi_{0}\right)^{2}d\mu_{\slashed{g}}d\ub'dt'\right)^{1/2}\\
\lesssim & \delta^{-1/2}(-t)^{-1}\left(\int_{0}^{\ub}\Fb_{\leq|\alpha|+2}(t,\ub')d\ub'\right)^{1/2}.
\end{split}
\end{align}
Now we summarize our estimates for the contribution of the first term in $\slashed{\mathcal{L}}_{\mathcal{R}_{i}}^{\alpha}g_{0}$. The $L^{2}$ norm of $II$ on $\Sigma_{t}^{\ub}$ is bounded as:
\begin{align}\label{L2 chib II preliminary}
\|II\|_{L^{2}([0,\ub]\times\mathbb{S}^{2})}\leq\int_{-r_{0}}^{t}(-t')^{3}\|\widetilde{g}_{\alpha}(t',\ub)\|_{L^{2}([0,\ub]\times\mathbb{S}^{2})}dt'\lesssim\int_{-r_{0}}^{t}(-t')^{2}\|\widetilde{g}_{\alpha}\|_{L^{2}(\Sigma_{t'}^{\ub})}dt'.
\end{align}
Combining \eqref{1st term in g0 1}-\eqref{1st term in g0 6} and \eqref{2nd term in g0}, \eqref{3rd term in g0}, the contribution of $\slashed{\mathcal{L}}_{\mathcal{R}_{i}}^{\alpha}g_{0}$ to right hand side of above is bounded by:
\begin{align}\label{1st term in galpha}
\begin{split}
&\delta\int_{-r_{0}}^{t}(-t')^{-2}\mathcal{B}_{0,\leq|\alpha|+1}(-r_{0})dt'+
 \delta^{-1/2}(-t)^{-1}\left(\int_{0}^{\ub}\Fb_{\leq|\alpha|+2}(t,\ub')d\ub'\right)^{1/2}+\\
&\delta^{1/2}\int_{-r_{0}}^{t}(-t')^{-2}\left(\mu^{-b_{|\alpha|+2}}_{m}(t')\sqrt{\widetilde{E}_{\leq|\alpha|+2}(t',\ub)}+\mu_{m}^{-b_{|\alpha|+2}-1/2}(t')\sqrt{\widetilde{\Eb}_{\leq|\alpha|+2}(t',\ub)}\right)dt'.
\end{split}
\end{align}
Next we consider the second term in $g_{\alpha}$, which is a sum of the terms as follows:
\begin{align*}
\slashed{\mathcal{L}}^{\beta}_{\mathcal{R}_{i}}
\slashed{\mathcal{L}}_{\leftexp{(\mathcal{R}_{i})}{\Zb}}
F_{\alpha-\beta-1},
\end{align*} 
which can be systematically rewritten as:
\begin{align*}
\slashed{\mathcal{L}}_{\leftexp{(\mathcal{R}_{i})}{\Zb}}F_{\alpha-1}+\sum_{\beta}\slashed{\mathcal{L}}
_{[\slashed{\mathcal{L}}_{\mathcal{R}_{i}}^{\beta},\leftexp{(\mathcal{R}_{i})}{\Zb}]} F_{\alpha-\beta-1}.
\end{align*}
Here the terms in the sum on the right hand side of the above, which are of lower order, come from the commutator $[\slashed{\mathcal{L}}_{\mathcal{R}_{j}},\slashed{\mathcal{L}}_{\leftexp{(\mathcal{R}_{i})}{\Zb}}]$. The first term, which is of principal order, can be bounded as follows (see \cite{Ch1} and \cite{Ch-M})
\begin{align*}
|\slashed{\mathcal{L}}_{\leftexp{(R_{i})}{\Zb}}F_{\alpha-1}|\leq &C(-t)^{-1}\Big(|\leftexp{(R_{i})}{\Zb}||F_{\alpha}|+|\ds X_{A}R^{\alpha-1}_{k}\check{f}|+|R_{i}\mu||\ds(X_{A}R_{k}^{\alpha-2}\tr\chib)|\\
+&|\slashed{\mathcal{L}}_{R_{j}}\leftexp{(R_{i})}{\Zb}|(|F_{\alpha-1}|+|\ds X_{A}R^{\alpha-2}_{k}\check{f}|\Big).
\end{align*}
The first term on the right hand side is in terms of the unknown $F_{\alpha}$. In view of the estimate for $\leftexp{(R_{i})}{\Zb}$:
\begin{align*}
|\leftexp{(R_{i})}{\Zb}|\lesssim\delta(-t)^{-2},
\end{align*}
the contribution of this term can be bounded by Gronwall. While the contribution of the second term is bounded by:
\begin{align}\label{2nd term in galpha 1}
\delta^{1/2}\int_{-r_{0}}^{t}(-t')^{-2}\sqrt{E_{\leq|\alpha|+2}(t',\ub)}dt'\lesssim\delta^{1/2}\int_{-r_{0}}^{t}(-t')^{-2}\mu_{m}^{-b_{|\alpha|+2}}(t')\sqrt{\widetilde{E}_{\leq|\alpha|+2}(t',\ub)}dt'.
\end{align}
For the third term of the right hand side, we employ the estimates for $\mathcal{A}'_{\leq|\alpha|}(t)$ to obtain the contribution of this term is bounded by:
\begin{align}\label{2nd term in galpha 2}
\delta^{1/2}\int_{-r_{0}}^{t}(-t')^{-3}\mu_{m}^{-b_{|\alpha|+2}-1/2}(t')\sqrt{\widetilde{\Eb}_{\leq|\alpha|+2}(t',\ub)}dt'.
\end{align}
By Proposition \ref{L infity estimates on lot}, we have
$|\slashed{\mathcal{L}}_{R_{j}}\leftexp{(R_{i})}{\Zb}|\lesssim\delta(-t)^{-2}, |\slashed{\mathcal{L}}_{X_{A}}\leftexp{(R_{i})}{\Zb}|\lesssim\delta(-t)^{-3}, |\slashed{\mathcal{L}}_{R_{j}}\leftexp{(X_{A})}{\Zb}|\lesssim\delta(-t)^{-3}$. Therefore the other two terms in the pointwise estimate for $|\slashed{\mathcal{L}}_{\leftexp{(\mathcal{R}_{i})}{\Zb}}F_{\alpha-1}|$ are of lower order. So the contribution of the second term in $g_{\alpha}$ to $\|II\|_{L^{2}([0,\ub]\times\mathbb{S}^{2}}$ is bounded by:
\begin{align}\label{2nd term in galpha}
\delta^{1/2}\int_{-r_{0}}^{t}(-t')^{-2}\mu_{m}^{-b_{|\alpha|+2}}(t')\sqrt{\widetilde{E}_{\leq|\alpha|+2}(t',\ub)}dt'+\delta^{1/2}\int_{-r_{0}}^{t}(-t')^{-3}\mu_{m}^{-b_{|\alpha|+2}-1/2}(t')\sqrt{\widetilde{\Eb}_{\leq|\alpha|+2}(t',\ub)}dt'.
\end{align} 

Now we move to the third term in $g_{\alpha}$. Again, using Leibniz Rule, the $L^{2}$-norm of this term on $\Sigma_{t}^{\ub}$ is bounded by:
\begin{align*}
(-t)^{-3}\mathcal{A}'_{\leq|\alpha|}(t)+\delta(-t)^{-5}\sqrt{E_{\leq|\alpha|+2}(t,\ub)}+\delta(-t)^{-6}\mathcal{B}_{0,\leq|\alpha|+1}(t).
\end{align*}
The first term is bounded in the same manner as \eqref{2nd term in galpha 2} and the rest terms are of lower order, in view of Proposition \ref{Proposition lower order L2} and \ref{Proposition lower order L2 mu}. The last term in $g_{\alpha}$ has a similar structure and can treated in the same way. 

Finally we need to consider the contribution from the term $\mu\ds(R_{i}^{\alpha}|\hat{\chib}|^{2})$. If not all $\slashed{\mathcal{L}}^{\alpha}_{R_{i}}$ hit one single $\hat{\chib}$, then the estimates for $\mathcal{A}'_{\leq|\alpha|}(t)$ implies the estimates of the contribution to $II$, which is similar to those of the second term in $g_{\alpha}$. If all $\slashed{\mathcal{L}}^{\alpha}_{R_{i}}$ hit one single $\widehat{\chib}$, we need to use the elliptic system for $\widehat{\chib}$:
\begin{align*}
\slashed{\textrm{div}}\widehat{\chib}=\frac{1}{2}\ds(\tr\chib)+\mu^{-1}\left(\zetab\cdot\chib-\frac{1}{2}\tr\chib \zetab\right).
\end{align*}
The elliptic estimate \eqref{elliptic estimates} implies:
\begin{align}\label{elliptic app}
\|\mu|\slashed{D}\widehat{\slashed{\mathcal{L}}}^{\alpha}_{\mathcal{R}_{i}}\hat{\chib}|\|_{L^{2}(\Sigma_{t}^{\ub})}\leq C\|\mu\ds(X_{A}R^{\alpha-1}_{i}\tr\chib)\|_{L^{2}(\Sigma_{t}^{\ub})}+\|\mu |H_{\alpha}|\|_{L^{2}(\Sigma_{t}^{\ub})}+\|\ds\mu\|_{L^{\infty}(\Sigma_{t}^{\ub})}\|\widehat{\slashed{\mathcal{L}}}^{\alpha}_{\mathcal{R}_{i}}\hat{\chib}\|_{L^{2}(\Sigma_{t}^{\ub})}.
\end{align}
 Here $\widehat{\slashed{\mathcal{L}}}_{\mathcal{R}_{i}}$ is the traceless part of $\slashed{\mathcal{L}}_{\mathcal{R}_{i}}$. $H_{\alpha}$ is given by
\begin{align*}
 H_\alpha &=\big(\slashed{\mathcal{L}}_{\mathcal{R}_i}+\frac{1}{2}\tr {}^{(\mathcal{R}_i)}\slashed{\pi}\big)^\alpha \big(\mu^{-1}\zetab\cdot \widehat{\chib}-\frac{1}{2}\mu^{-1}\zetab \, \tr\chib\big)\\ +&\!\!\!\!\!\!\!\!\! \sum_{|\beta_1|+|\beta_2| = |\alpha|-1}\!\!\!\!\!\!\!\!\!\big(\slashed{\mathcal{L}}_{\mathcal{R}_i}+\frac{1}{2}\tr {}^{(\mathcal{R}_i)}\slashed{\pi}\big)^{\beta_1}
 \Big(\tr{}^{(\mathcal{R}_i)} \slashed{\pi}\cdot \slashed{d}(\mathcal{R}_i{}^{\beta_2}\tr\chib) + \big( \divslash {}^{(\mathcal{R}_i)}\widehat{\slashed{\pi}}\big)\cdot \widehat{\slashed{\mathcal{L}}}_{\mathcal{R}_i}{}^{\beta_2} \widehat{\chib}\Big)\\
 & +\!\!\!\!\!\!\!\!\! \sum_{|\beta_1|+|\beta_2| = |\alpha|-1}\!\!\!\!\!\!\!\!\!\big(\slashed{\mathcal{L}}_{\mathcal{R}_i}+\frac{1}{2}\tr {}^{(\mathcal{R}_i)}\slashed{\pi}\big)^{\beta_1} \Big({}^{(\mathcal{R}_i)}\widehat{\slashed{\pi}}\cdot \nablaslash \widehat{\slashed{\mathcal{L}}}_{\mathcal{R}_i}{}^{\beta_2} \widehat{\chib} \Big). 
\end{align*}
Using the estimates:
\begin{align*}
|\mu^{-1}\zetab|\lesssim\delta(-t)^{-3},\quad |\tr\leftexp{(R_{i})}{\slashed{\pi}}|\lesssim \delta(-t)^{-2},\quad |\chib'|\lesssim\delta(-t)^{-3},
\end{align*}
 the right hand side of \eqref{elliptic app} is bounded by:
 \begin{align*}
\|F_{\alpha}\|_{L^{2}(\Sigma_{t}^{\ub})}+\delta^{1/2}(-t)^{-2}\mu_{m}^{-b_{|\alpha|+2}}(t)\sqrt{\widetilde{E}_{\leq|\alpha|+2}(t,\ub)}
+\delta^{3/2}(-t)^{-3}\mu_{m}^{-b_{|\alpha|+2}-1/2}(t)\sqrt{\widetilde{\Eb}_{\leq|\alpha|+2}(t,\ub)}.
 \end{align*}
Therefore we have:
\begin{align*}
\|\mu\slashed{d}(X_{A}R_{i}^{\alpha-1}|\widehat{\chib}|^{2})\|_{L^{2}(\Sigma_{t}^{\ub})}\lesssim &\delta(-t)^{-3}\|F_{\alpha}\|_{L^{2}(\Sigma_{t}^{\ub})}\\ 
+&\delta^{3/2}(-t)^{-5}\mu_{m}^{-b_{|\alpha|+2}}(t)\sqrt{\widetilde{E}_{\leq|\alpha|+2}(t,\ub)}\\
+&\delta^{3/2}(-t)^{-6}\mu_{m}^{-b_{|\alpha|+2}+1/2}(t)\sqrt{\widetilde{\Eb}_{\leq|\alpha|+2}(t,\ub)}
\end{align*}
The contribution to $II$ of the first term on the right hand side is bounded by using Gronwall. The contributions from the rest two terms are bounded as:
\begin{align}\label{extra in tilde galpha}
\begin{split}
&\delta^{3/2}\int_{-r_{0}}^{t}\mu^{-b_{|\alpha|+2}}_{m}(t')(-t')^{-3}\sqrt{\widetilde{E}_{\leq|\alpha|+2}(t',\ub)}dt'\\
+&\delta^{3/2}\int_{-r_{0}}^{t}\mu^{-b_{|\alpha|+2}-1/2}_{m}(t')(-t')^{-4}\sqrt{\widetilde{\Eb}_{\leq|\alpha|+2}(t',\ub)}dt'.
\end{split}
\end{align} 

Combining \eqref{1st term in galpha}, \eqref{2nd term in galpha} and \eqref{extra in tilde galpha}, we have:
\begin{align}\label{estimate for II}
\begin{split}
\|II\|_{L^{2}([0,\ub]\times\mathbb{S}^{2})}\leq &\int_{-r_{0}}^{t}(-t')^{3}\|\widetilde{g}_{\alpha}(t',\ub)\|_{L^{2}([0,\ub]\times \mathbb{S}^{2})}dt'\leq\int_{-r_{0}}^{t}(-t')^{2}\|\widetilde{g}_{\alpha}\|_{L^{2}(\Sigma_{t'}^{\ub})}dt'\\
\lesssim &\delta\int_{-r_{0}}^{t}(-t')^{-2}\mathcal{B}_{0,\leq|\alpha|+1}(-r_{0})dt'\\+
&\delta^{1/2}\int_{-r_{0}}^{t}(-t')^{-2}\mu_{m}^{-b_{|\alpha|+2}-1/2}(t')\sqrt{\widetilde{\Eb}_{\leq|\alpha|+2}(t',\ub)}dt'\\
+&\delta^{1/2}\int_{-r_{0}}^{t}(-t')^{-2}\mu_{m}^{-b_{|\alpha|+2}}(t')\sqrt{\widetilde{E}_{\leq|\alpha|+2}(t',\ub)}dt'\\
+&\delta^{-1/2}(-t)^{-1}\mu_{m}^{-b_{|\alpha|+2}}(t)\left(\int_{0}^{\ub}\widetilde{\Fb}_{\leq|\alpha|+2}(t,\ub')d\ub'\right)^{1/2}.
\end{split}
\end{align}
Now we are ready to give the final estimate for $(-t)^{2}\|F_{\alpha}(t,\ub)\|_{L^{2}(\Sigma_{t}^{\ub})}\sim(-t)^{3}\|F_{\alpha}(t,\ub)\|_{L^{2}([0,\ub]\times\mathbb{S}^{2})}$. Combining \eqref{L2 chib I} and \eqref{estimate for II}, we have:
\begin{align}\label{estimate for Falpha}
\begin{split}
(-t)^{3}\|F_{\alpha}(t,\ub)\|_{L^{2}([0,\ub]\times \mathbb{S}^{2})}\leq &C\delta\int_{-r_{0}}^{t}(-t')^{-2}\mathcal{B}_{0,\leq|\alpha|+1}(-r_{0})dt'\\
+&\frac{C\delta^{1/2}(-t)^{-1}}{b_{|\alpha|+2}}\mu^{-b_{|\alpha|+2}}_{m}(t)\sqrt{\widetilde{E}_{\leq|\alpha|+2}(t,\ub)}\\
+&C\delta^{1/2}\int_{-r_{0}}^{t}(-t')^{-2}\mu_{m}^{-b_{|\alpha|+2}-1/2}(t')\sqrt{\widetilde{\Eb}_{\leq|\alpha|+2}(t',\ub)}dt'\\
+&C\delta^{1/2}\int_{-r_{0}}^{t}(-t')^{-2}\mu_{m}^{-b_{|\alpha|+2}}(t')\sqrt{\widetilde{E}_{\leq|\alpha|+2}(t',\ub)}dt'\\
+&C\delta^{-1/2}\mu_{m}^{-b_{|\alpha|+2}}(t)(-t)^{-1}\left(\int_{0}^{\ub}\widetilde{\Fb}_{\leq|\alpha|+2}(t,\ub')d\ub'\right)^{1/2}.
\end{split}
\end{align}
This also implies an estimate for $\ds X_{A}R^{\alpha-1}_{i}\tr\chib$:
\begin{align}\label{estimate for top order chib}
\begin{split}
(-t)^{2}\|\ds X_{A}R^{\alpha-1}_{i}\tr\chib\|_{L^{2}(\Sigma_{t}^{\ub})}\leq &C\delta\int_{-r_{0}}^{t}(-t')^{-2}\mathcal{B}_{0,\leq|\alpha|+1}(-r_{0})dt'\\
+&C\delta^{1/2}(-t)^{-1}\mu^{-b_{|\alpha|+2}}_{m}(t)\sqrt{\widetilde{E}_{\leq|\alpha|+2}(t,\ub)}\\
+&C\delta^{1/2}\int_{-r_{0}}^{t}(-t')^{-2}\mu_{m}^{-b_{|\alpha|+2}-1/2}(t')\sqrt{\widetilde{\Eb}_{\leq|\alpha|+2}(t',\ub)}dt'\\
+&C\delta^{1/2}\int_{-r_{0}}^{t}(-t')^{-2}\mu_{m}^{-b_{|\alpha|+2}}(t')\sqrt{\widetilde{E}_{\leq|\alpha|+2}(t',\ub)}dt'\\
+&C\delta^{-1/2}\mu_{m}^{-b_{|\alpha|+2}}(t)(-t)^{-1}\left(\int_{0}^{\ub}\widetilde{\Fb}_{\leq|\alpha|+2}(t,\ub')d\ub'\right)^{1/2}.
\end{split}
\end{align}
Here we assume $b_{|\alpha|+2}\geq 4$.
\subsection{Estimates for the contribution from $\mu$}
Here as in \cite{M-Y}, to avoid the loss of derivatives when we estimate the top order spatial derivatives of $\mu$, we need to commute $\slashed{\Delta}$ with the propagation equation of $\mu$.

Thanks to the following commutation formulas,
\begin{equation}\label{commutator slashed Laplace}
 \begin{split}
[\Lb,\slashed{\Delta}]\phi+\text{tr}\chib\slashed{\Delta}\phi &=-2\widehat{\chib}\cdot\widehat{\slashed{D}}^{2}\phi
-2\slashed{\text{div}}\widehat{\chib}\cdot\slashed{d}\phi,\\
[T,\slashed{\Delta}]\phi+c^{-1}\mu\text{tr}\theta
\slashed{\Delta}\phi&=-2c^{-1}\mu\widehat{\theta}
\cdot\widehat{\slashed{D}}^{2}\phi-2\slashed{\text{div}}(c^{-1}\mu\widehat{\theta})\cdot\slashed{d}\phi,
\end{split}
\end{equation}
we have
\begin{equation}\label{equaiton for Lb Laplacian mu}
\begin{split}
\Lb \slashed{\Delta} \mu = &-\dfrac{1}{2}\dfrac{d c^2}{d\rho}\slashed{\Delta} T\rho + \mu \slashed{\Delta} e + e \slashed{\Delta} \mu\\+ &\ds\mu\cdot\ds e-\tr\chib\slashed{\Delta}\mu-2\widehat{\chib}
\cdot\widehat{\slashed{D}}^{2}\mu
-2\slashed{\text{div}}\widehat{\chib}\cdot\ds\mu.
\end{split}
\end{equation}
According to \eqref{wave equation for rho}, 
\begin{align*}
\Box_{g}\rho=\frac{d\log(c)}{d\rho}\left(\mu^{-1}\Lb\rho L\rho+\ds\rho\cdot\ds\rho\right)+2\mu^{-1}L\psi_{0}\Lb\psi_{0}
+2\ds\psi_{0}\cdot\ds\psi_{0}.
\end{align*}
 Therefore, by multiplying $\mu$, we have 
\begin{align*}
\mu \slashed{\Delta} \rho = \Lb(L \rho) + \frac{1}{2}\Lb\rho\tr\chi+\frac{1}{2}L\rho\tr\chib +\frac{d\log(c)}{d\rho}\left(\Lb\rho L\rho+\mu\ds\rho\cdot\ds\rho\right)+2L\psi_{0}\Lb\psi_{0}
+2\mu\ds\psi_{0}\cdot\ds\psi_{0}.
\end{align*} 
  We commute $T$ and obtain
  \begin{align}\label{Laplacian Trho}
  \begin{split}
\mu \slashed{\Delta} T \rho = &\Lb(T L \rho) + \frac{1}{2}\Lb\rho(T\tr\chi)+\frac{1}{2}L\rho(T\tr\chib)\\+&[T,\Lb]L\rho+\frac{1}{2}\tr\chi T\Lb\rho+\frac{1}{2}TL\rho\tr\chib\\
+&T\left(\frac{d\log(c)}{d\rho}\left(\Lb\rho L\rho+\mu\ds\rho\cdot\ds\rho\right)
+2L\psi_{0}\Lb\psi_{0}
+2\mu\ds\psi_{0}\cdot\ds\psi_{0}\right)\\
+&c^{-1}\mu\tr\theta\slashed{\Delta}\rho
+2c^{-1}\mu\widehat{\theta}\cdot\widehat{\slashed{D}}^{2}\rho
+2\slashed{\text{div}}(c^{-1}\mu\widehat{\theta})\cdot\ds\rho
-(T\mu)\slashed{\Delta}\rho.
\end{split}
\end{align}
Therefore
\begin{align}\label{Laplacian m}
\begin{split}
-\frac{1}{2}\frac{dc^{2}}{d\rho}\mu\slashed{\Delta}T\rho=&\Lb\left(-\frac{1}{2}\frac{dc^{2}}{d\rho}TL\rho\right)-\frac{1}{2}\frac{dc^{2}}{d\rho}\left(\frac{1}{2}\Lb\rho(T\tr\chi)+\frac{1}{2}L\rho(T\tr\chib)\right)\\
-&\frac{1}{2}\frac{dc^{2}}{d\rho}\left([T,\Lb]L\rho+\frac{1}{2}\tr\chi T\Lb\rho+\frac{1}{2}TL\rho\tr\chib\right)\\
-&\frac{1}{2}\frac{dc^{2}}{d\rho}T\left(\frac{d\log(c)}{d\rho}\left(\Lb\rho L\rho+\mu\ds\rho\cdot\ds\rho\right)
+2L\psi_{0}\Lb\psi_{0}
+2\mu\ds\psi_{0}\cdot\ds\psi_{0}\right)\\
-&\frac{1}{2}\frac{dc^{2}}{d\rho}\left(c^{-1}\mu\tr\theta\slashed{\Delta}\rho
+2c^{-1}\mu\widehat{\theta}\cdot\widehat{\slashed{D}}^{2}\rho
+2\slashed{\text{div}}(c^{-1}\mu\widehat{\theta})\cdot\ds\rho
-(T\mu)\slashed{\Delta}\rho\right)\\+&\Lb\left(\frac{1}{2}\frac{dc^{2}}{d\rho}\right)TL\rho.
\end{split}
\end{align}
In view of the commutator formula, we also have 
\begin{align}\label{Laplacian e}
\begin{split}
\mu^{2}\slashed{\Delta}e=&\Lb\left(\frac{\mu^{2}}{2c^{2}}\frac{dc^{2}}{d\rho}\slashed{\Delta}\rho\right)+\frac{\mu^{2}}{c^{2}}\frac{dc^{2}}{d\rho}\left(\underline{\chi}\cdot\slashed{D}^{2}\rho
+\slashed{\textrm{div}}\hat{\underline{\chi}}
\cdot\slashed{d}\rho\right)-\Lb\left(\frac{\mu^{2}}{2c^{2}}\frac{dc^{2}}{d\rho}\right)\slashed{\Delta}\rho\\
+&\mu^{2}\frac{d}{d\rho}\left(\frac{1}{c^{2}}\frac{dc^{2}}{d\rho}\right)\slashed{d}\rho\cdot\slashed{d}\underline{L}\rho+
\mu^{2}\left(\frac{d}{d\rho}\left(\frac{1}{2c^{2}}\frac{dc^{2}}{d\rho}\right)\slashed{\Delta}\rho+\frac{d^{2}}{d\rho^{2}}\left(\frac{1}{2c^{2}}\frac{dc^{2}}{d\rho}\right)|\slashed{d}\rho|^{2}\right)\underline{L}\rho.
\end{split}
\end{align} 
With the same notation as in \cite{M-Y}, we define:
\begin{align}\label{modified top mu}
\check{f}':=-\frac{1}{2}\frac{dc^{2}}{d\rho}TL\rho
+\frac{\mu^{2}}{2c^{2}}\frac{dc^{2}}{d\rho}\slashed{\Delta}\rho,\quad F':=\mu\slashed{\Delta}\mu-\check{f}'.
\end{align}
and we have the following propagation equation for $F'$:
\begin{align}\label{transport equation for Lb laplaican mu}
\underline{L}F^{\prime}
+\left(\textrm{tr}\underline{\chi}
-2\mu^{-1}\underline{L}\mu\right)F^{\prime}=-\left(\frac{1}{2}\textrm{tr}\underline{\chi}
-2\mu^{-1}\underline{L}\mu\right)\check{f}^{\prime}
-2\mu\hat{\underline{\chi}}\cdot\hat{\slashed{D}}^{2}\mu
+\check{g}^{\prime}.
\end{align}
where
\begin{align}\label{g prime check}
\check{g}^{\prime}=\left(-\ds\mu+\frac{\mu}{c^{2}}\frac{dc^{2}}{d\rho}\ds\rho\right)\cdot\left(\mu\ds\tr\chib\right)
+\Psi^{\leq2}_{\geq-2,3}+\O^{\leq1}_{0,1}\Psi^{\leq2}_{\geq-2,2}
+\Psi^{\leq2}_{\geq0,4}. 
\end{align}
Similarly as in \cite{M-Y}, we have used the structure equation \eqref{Structure Equation T chib} to cancel the contribution from the term $\dfrac{1}{2}\Lb\rho(T\tr\chi)+\dfrac{1}{2}L\rho(T\tr\chib)$ in \eqref{Laplacian Trho} and the term $(\Lb\mu)\slashed{\Delta}\mu$ when we write $\Lb(\mu\slashed{\Delta}\mu)
=\mu\Lb(\slashed{\Delta}\mu)+\Lb(\mu)\slashed{\Delta}\mu$. We also remark that the $L^2$ norm of all derivatives on $\slashed{\textrm{div}}\chibh$ has been estimated from previous subsection. In such a sense, it can also be considered as a $\Psi_{2,4}^{\leq 2}$ term and we use \eqref{Structure Equation div chib} to replace $\divslash \chib$ by $\slashed{d}\tr\chib + \cdots$. The term $\Psi_{\geq-2,3}^{\leq2}$ comes from the contribution of $L\psi_{0}\Lb\psi_{0}$ and $\O^{\leq1}_{0,1}\Psi^{\leq2}_{\geq-2,2}$ comes from $-\dfrac{1}{4}\dfrac{dc^{2}}{d\rho}TL\rho\tr\chib$  in \eqref{Laplacian m}. Since we already applied $T$ to $L\psi_{0}\Lb\psi_{0}$ once in \eqref{Laplacian m}, instead of using flux $\Fb(t,\ub)$ as we did in the last subsection, we only need to use the energy $E(t,\ub)$ to control the contribution of this term. For the higher order derivatives $F'_{\alpha,l} = \mu X_{A}R_i{}^{\alpha'-1} T^{l}\slashed{\Delta} \mu - X_{A}R_i{}^{\alpha'-1} T^{l}\check{f'}$ with $|\alpha'|+l=|\alpha|$, we have:
\begin{equation}\label{transport equation for F alpha}
\Lb F'_{\alpha,l} + \left(\tr \chib - 2\mu^{-1}\Lb \mu\right)F'_{\alpha,l} = -\left(\frac{1}{2}\tr\chib - 2\mu^{-1}\Lb\mu\right)X_{A}R_i{}^{\alpha'-1}T^{l} \check{f}'-2\mu \widehat{\chib}\cdot \slashed{\mathcal{L}}_{X_{A}}\slashed{\mathcal{L}}_{R_{i}}^{\alpha'-1}
\slashed{\mathcal{L}}_{T}^{l}\widehat{\slashed{D}^2} \mu + \check{g}'_{\alpha',l}.
\end{equation}
where $\check{g}'_{\alpha',l}$ is given by
\begin{equation*}
\begin{split}
\check{g}'_{\alpha',l} = &\left(-\slashed{d}\mu + \frac{\mu}{c^2}\frac{d c^2}{d\rho} \slashed{d}\rho\right)\cdot\mu \slashed{d}\left(X_{A}R_i{}^{\alpha'-1}T^{l}\tr\chib\right) + l\cdot\Lambda F'_{\alpha,l-1} +\,^{(R_i)} \Zb F'_{\alpha-1,l}\\+&\O^{\leq|\beta|+1}_{-k,\geq1}\O^{\leq|\alpha|-|\beta|+1}_{\geq2-2l+k,\geq4}+\Psi_{\geq -2l-2,4}^{\leq |\alpha|+2}+\O^{\leq|\beta|+1}_{-k,1}\Psi^{\leq|\alpha|+2-|\beta|}_{\geq-2l-2+k,3}+\Psi_{\geq-2l,4}^{\leq|\alpha|+2}.
\end{split}
\end{equation*}
Note that $F'_{\alpha,l}$ is a scalar function instead of a $1$-form, so the inequality for $|F'_{\alpha,l}(t,\ub)|$ is slightly different from that of $|F_{\alpha}|$ as in the last subsection. We have:
\begin{align}\label{inequality for absolute value F'alpha}
\Lb|F'_{\alpha,l}|+\left(\tr \chib - 2\mu^{-1}\Lb \mu\right)|F'_{\alpha,l}|\leq\left(-\frac{1}{2}\tr\chib + 2\mu^{-1}|\Lb\mu|\right)\left|X_{A}R_i{}^{\alpha'-1}T^{l} \check{f}'\right|+\widetilde{\check{g}'}_{\alpha',l},
\end{align}
with
\begin{align*}
\widetilde{\check{g}'}_{\alpha',l}:=2\left|\mu\widehat{\chib}
\cdot\slashed{\mathcal{L}}_{X_{A}}\slashed{\mathcal{L}}^{\alpha'-1}_{R_{i}}
\slashed{\mathcal{L}}^{l}_{T}\widehat{\slashed{D}}^{2}\mu\right|+|\check{g}'_{\alpha,l}|
\end{align*}
\begin{align}\label{modified inequality for higher order F'}
\begin{split}
&\partial_{t}\left(\left(-t+\ub\right)^{2}|F'_{\alpha,l}|\right)+\left(\tr \chib - 2\mu^{-1}\Lb \mu\right)\left(-t+\ub\right)^{2}|F'_{\alpha,l}|\\
\leq &(-t+\ub)^{2}\left(2\mu^{-1}|\partial_{t}\mu|-\tr\chib\right)| X_{A}R_{i}^{\alpha'-1}T^{l}\check{f}'|+(-t+\ub)^{2}\widetilde{\check{g}'}_{\alpha',l},
\end{split}
\end{align}

Integrating \eqref{modified inequality for higher order F'} and taking the $L^{2}$ norm on $[0,\ub]\times\mathbb{S}^{2}$, we obtain:
\begin{align}\label{L2 F'alpha preliminary}
\begin{split}
(-t)^{2}\delta^{l+1}\|F'_{\alpha,l}(t)\|_{L^{2}([0,\ub]\times\mathbb{S}^{2})}\lesssim &r_{0}^{2}\delta^{l+1}\|F'_{\alpha,l}(-r_{0})\|_{L^{2}([0,\ub]\times\mathbb{S}^{2})}\\
+&\delta^{l+1}\int_{-r_{0}}^{t}(-t'+\ub)^{2}\left(2\mu^{-1}|\partial_{t}\mu|-\tr\chib\right)\| X_{A}R_{i}^{\alpha'-1}T^{l}\check{f}'(t')\|_{L^{2}([0,\ub]\times\mathbb{S}^{2})}dt'\\
+&\delta^{l+1}\int_{-r_{0}}^{t}(-t'+\ub)^{2}\|\widetilde{\check{g}'}_{\alpha',l}(t')\|_{L^{2}([0,\ub]\times\mathbb{S}^{2})}dt':=I'+II'+III'.
\end{split}
\end{align}
Since $I'$ depends only on initial data, here we will focus on $II'$ and $III'$. In view of the definition of $\check{f}'$ \eqref{modified top mu}, 
\begin{align*}
\delta^{l+1}\| X_{A}R_{i}^{\alpha'-1}T^{l}\check{f}'\|_{L^{2}(\Sigma_{t}^{\ub})}\lesssim\delta^{1/2}(-t)^{-2}\sqrt{E_{\leq|\alpha|+2}(t,\ub)}+\delta^{3/2}(-t)^{-3}\sqrt{\Eb_{\leq|\alpha|+2}(t,\ub)}.
\end{align*}
This together with Lemma \ref{lemma on mu to power a} and the pointwise estimates for $\tr\chib$ as well as $\Lb\mu$, we have:
\begin{align}\label{II'}
\begin{split}
II'
\lesssim&\delta^{1/2}\mu_{m}^{-b_{|\alpha|+2}}(t)(-t)^{-1}\sqrt{\widetilde{E}_{\leq|\alpha|+2}(t,\ub)}+\delta^{1/2}\int_{-r_{0}}^{t}(-t')^{-2}\mu_{m}^{-b_{|\alpha|+2}}(t')\sqrt{\widetilde{E}_{\leq|\alpha|+2}(t',\ub)}dt'\\
+&\delta^{3/2}\int_{-r_{0}}^{t}(-t')^{-3}\sqrt{\Eb_{\leq|\alpha|+2}(t',\ub)}dt'.
\end{split}
\end{align}
Next we move to the $L^{2}$ norm of $\check{g}'_{\alpha',l}$. First, in view of the pointwise estimates $|\Lambda|\lesssim(-t)^{-2}$, $|\leftexp{(R_{i})}{\Zb}|\lesssim\delta(-t)^{-2}$, the two terms $l\cdot\Lambda F'_{\alpha,l-1}$, $\leftexp{(R_{i})}{\Zb}F'_{\alpha-1,l}$ are controlled by Gronwall. If $l=0$, in view of \eqref{estimate for top order chib} and the pointwise estimates for $\ds\mu$ and $\ds\rho$, the contribution of the first term in the expression of $\check{g}'_{\alpha',l}$ to $III'$ can be bounded as
\begin{align}\label{III' 1}
\begin{split}
&\delta^{3/2}\mu^{-b_{|\alpha|+2}}_{m}(t)(-t)^{-1}\sqrt{\widetilde{E}_{\leq|\alpha|+2}(t,\ub)}\\
+&\delta^{3/2}\int_{-r_{0}}^{t}(-t')^{-2}\mu_{m}^{-b_{|\alpha|+2}-1/2}(t')\sqrt{\widetilde{\Eb}_{\leq|\alpha|+2}(t',\ub)}dt'\\
+&\delta^{3/2}\int_{-r_{0}}^{t}(-t')^{-2}\mu_{m}^{-b_{|\alpha|+2}}(t')\sqrt{\widetilde{E}_{\leq|\alpha|+2}(t',\ub)}dt'\\
+&\delta^{1/2}\mu_{m}^{-b_{|\alpha|+2}}(t)(-t)^{-1}\left(\int_{0}^{\ub}\widetilde{\Fb}_{\leq|\alpha|+2}(t,\ub')d\ub'\right)^{1/2}.
\end{split}
\end{align}
If $l\geq1$, then in view of Proposition \ref{Proposition lower order L2 mu}, the pointwise estimates for $\ds\mu$, $\ds\rho$ and the structure equation \eqref{Structure Equation T chib}, the contribution of this term is bounded as:
\begin{align}\label{III' 2}
\begin{split}
&\delta^{2}\int_{-r_{0}}^{t}(-t')^{-3}\left(\mathcal{B}_{l-1,\leq|\alpha|+2}(t')+\sqrt{E_{\leq|\alpha|+2}(t',\ub)}+\mu_{m}^{-1/2}(t')\sqrt{\widetilde{\Eb}_{\leq|\alpha|+2}(t',\ub)}\right)dt'\\
\lesssim &\delta^{2}\int_{-r_{0}}^{t}(-t')^{-3}\left(\mathcal{B}_{l-1,\leq|\alpha|+2}(-r_{0})+\sqrt{E_{\leq|\alpha|+2}(t',\ub)}+\mu_{m}^{-1/2}(t')\sqrt{\widetilde{\Eb}_{\leq|\alpha|+2}(t',\ub)}\right)dt'.
\end{split}
\end{align}
Now we move to the second line of the expression of $\check{g}'_{\alpha',l}$. Again, the term $\O^{\leq|\beta|+1}_{-k,\geq1}\O^{\leq|\alpha|-|\beta|+1}_{\geq2-2l+k,\geq4}$ can be absorbed by $\O^{\leq|\beta|+1}_{-k,1}\Psi^{\leq|\alpha|+2-|\beta|}_{\geq-2l-2+k,3}$. So we only need to consider the last three terms in this expression. In view of Proposition \ref{Proposition lower order L2} and Proposition \ref{Proposition lower order L2 mu}, the contributions of these terms to $III'$ are bounded as follows:
\begin{align}\label{III' 3}
\begin{split}
\delta^{l+1}\int_{-r_{0}}^{t}(-t')\|\Psi^{\leq|\alpha|+2}_{\geq-2l-2,4}\|_{L^{2}(\Sigma_{t'}^{\ub})}dt'
\lesssim\delta^{1/2}\int_{-r_{0}}^{t}(-t')^{-2}\sqrt{E_{\leq|\alpha|+2}(t',\ub)}dt'
\end{split}
\end{align}
\begin{align}\label{III' 4}
\begin{split}
&\delta^{l+1}\int_{-r_{0}}^{t}(-t')\|\O^{\leq|\beta|+1}_{-k,1}\Psi^{\leq|\alpha|+2-|\beta|}_{\geq-2l-2+k,3}\|_{L^{2}(\Sigma_{t'}^{\ub})}dt'\\
\lesssim&\delta^{1/2}\int_{-r_{0}}^{t}(-t')^{-2}\sqrt{E_{\leq|\alpha|+2}(t',\ub)}dt'+\delta^{1/2}\int_{-r_{0}}^{t}(-t')^{-2}\mu_{m}^{-1/2}(t')\sqrt{\Eb_{\leq|\alpha|+2}(t',\ub)}dt'.
\end{split}
\end{align}
and
\begin{align}\label{III' 5}
\begin{split}
\delta^{l+1}\int_{-r_{0}}^{t}(-t')\|\Psi^{\leq|\alpha|+2}_{\geq-2l,4}\|_{L^{2}(\Sigma_{t'}^{\ub})}dt'\lesssim&\delta^{1/2}\int_{-r_{0}}^{t}(-t')^{-2}\mu_{m}^{-1/2}(t')\sqrt{\Eb_{\leq|\alpha|+2}(t',\ub)}dt'\\
+&\delta^{3/2}\int_{-r_{0}}^{t}(-t')^{-2}\sqrt{E_{\leq|\alpha|+2}(t',\ub)}dt'.
\end{split}
\end{align}
Now combining \eqref{L2 F'alpha preliminary} and \eqref{II'}-\eqref{III' 5}, we obtain the following estimates:
\begin{align}\label{estimates for F'alpha final}
\begin{split}
\delta^{l+1}(-t)\|F'_{\alpha,l}(t)\|_{L^{2}(\Sigma_{t}^{\ub})}\lesssim&\delta^{l+1}r_{0}\|F'_{\alpha,l}(-r_{0})\|_{L^{2}(\Sigma_{-r_{0}}^{\ub})}\\
+&\delta^{1/2}\mu^{-b_{|\alpha|+2}}_{m}(t)(-t)^{-1}\sqrt{\Et_{\leq|\alpha|+2}(t,\ub)}\\
+&\delta^{1/2}\int_{-r_{0}}^{t}(-t')^{-2}\mu_{m}^{-b_{|\alpha|+2}}(t')\sqrt{\Et_{\leq|\alpha|+2}(t',\ub)}dt'\\
+&\delta^{1/2}\int_{-r_{0}}^{t}(-t')^{-2}\mu^{-b_{|\alpha|+2}-1/2}_{m}(t')\sqrt{\Ebt_{\leq|\alpha|+2}(t',\ub)}dt'\\
+&\delta^{1/2}\mu_{m}^{-b_{|\alpha|+2}}(t)(-t)^{-1}\left(\int_{0}^{\ub}\widetilde{\Fb}_{\leq|\alpha|+2}(t,\ub')d\ub'\right)^{1/2}.
\end{split}
\end{align}
These in terms imply the estimates for $\delta^{l+1}\|\mu R^{\alpha'}T^{l}\slashed{\Delta}\mu\|_{L^{2}(\Sigma_{t}^{\ub})}$:
\begin{align}\label{estimates for top order mu}
\begin{split}
\delta^{l+1}\|\mu R^{\alpha'}T^{l}\slashed{\Delta}\mu\|_{L^{2}(\Sigma_{t}^{\ub})}\lesssim&\delta^{l+1}r_{0}\|F'_{\alpha,l}(-r_{0})\|_{L^{2}(\Sigma_{-r_{0}}^{\ub})}\\
+&\delta^{1/2}\mu^{-b_{|\alpha|+2}}_{m}(t)(-t)^{-1}\sqrt{\Et_{\leq|\alpha|+2}(t,\ub)}\\
+&\delta^{3/2}(-t)^{-2}\mu^{-b_{|\alpha|+2}}_{m}(t)\sqrt{\Ebt_{\leq|\alpha|+2}(t,\ub)}\\
+&\delta^{1/2}\int_{-r_{0}}^{t}(-t')^{-2}\mu_{m}^{-b_{|\alpha|+2}}(t')\sqrt{\Et_{\leq|\alpha|+2}(t',\ub)}dt'\\
+&\delta^{1/2}\int_{-r_{0}}^{t}(-t')^{-2}\mu^{-b_{|\alpha|+2}-1/2}_{m}(t')\sqrt{\Ebt_{\leq|\alpha|+2}(t',\ub)}dt'\\
+&\delta^{1/2}\mu_{m}^{-b_{|\alpha|+2}}(t)(-t)^{-1}\left(\int_{0}^{\ub}\widetilde{\Fb}_{\leq|\alpha|+2}(t,\ub')d\ub'\right)^{1/2}.
\end{split}
\end{align}
\section{Commutator estimates}
In this section, we shall estimate the error spacetime integrals contributed by commutators.

Let $\psi$ be a solution of the inhomogeneous wave equation $\Box_{\gt} \psi = \rho$ and $Z$ be a vector field, one can commute $Z$ with the equation to derive
\begin{equation}\label{commute an arbitary vector field with inhomogeneous wave equation}
\Box_{\gt} \left( Z \psi \right) = Z \rho + \frac{1}{2}\tr_{\gt}{}^{(Z)}\widetilde{\pi} \cdot \rho + c^2 \text{div}_{g} \, {}^{(Z)} J
\end{equation}
where the vector field ${}^{(Z)}J$ is defined by
\begin{equation*}
{}^{(Z)}J^{\mu} = \left( {}^{(Z)}\widetilde{\pi}^{\mu\nu} - \frac{1}{2}g^{\mu\nu} \tr_{g}{}^{(Z)}\widetilde{\pi} \right)\partial_\nu \psi.
\end{equation*}
We remark that the raising indices for ${}^{(Z)}\widetilde{\pi}^{\mu\nu}$ are with respect to the optic metric $g$.

In applications, we use the above formulas for homogeneous wave equations $\Box_{\gt} \psi =0$ and commute some commutation vector fields $Z_i$'s several times. Therefore, we need the following recursion formulas:
\begin{equation}\label{commute vector fields with equations}
\begin{split}
\Box_{\gt} \psi_n &= \rho_n. \ \ \psi_n = Z \psi_{n-1}, \ \ \rho_1 = 0,\\
\rho_n &= Z \rho_{n-1} + \frac{1}{2}\tr_{\gt}{}^{(Z)}\widetilde{\pi} \cdot \rho_{n-1} + c^2 \text{div}_{g} \, {}^{(Z)} J_{n-1},\\
{}^{(Z)}J_{n-1}^{\mu}&=\left( {}^{(Z)}\widetilde{\pi}^{\mu\nu} - \frac{1}{2}g^{\mu\nu} \tr_{g}{}^{(Z)}\widetilde{\pi} \right)\partial_\nu \psi_{n-1}.
\end{split}
\end{equation}

\begin{remark}\label{modified error conformal}
When we derive energy estimates for $\Box_{\gt} \psi_n = \rho_n$, due to the volume form of the conformal optic metric $\gt$, the integrands $\widetilde{\rho}_n$ appearing in the error terms is slightly different from $\rho_n$. The rescaled source terms $\widetilde{\rho}_n$ are defined as follows:
\begin{equation}\label{rescaled source terms}
\begin{split}
\widetilde{\rho}_n &= \frac{1}{c^2}\mu \rho_n = Z \widetilde{\rho}_{n-1} +{}^{(Z)} \delta \cdot \widetilde{\rho}_{n-1} + {}^{(Z)} \sigma_{n-1}, \\
 \widetilde{\rho}_1 &=0, \ \ {}^{(Z)} \sigma_{n-1} = \mu\cdot \text{div}_{g} \, {}^{(Z)} J_{n-1}, \ \  {}^{(Z)} \delta =\frac{1}{2}\tr_{\gt}{}^{(Z)}\widetilde{\pi}-\mu^{-1}Z\mu +2Z\big(\log(c)\big).
\end{split}
\end{equation}
\end{remark}
Then the error spacetime integrals corresponding to $K_{0}=L+\Lb$ and $K_{1}=\dfrac{2(-t)}{\widetilde{\tr}\chib}\Lb$ containing $\rho_{n}$ are as follows:
\begin{align*}
-\int_{W^{t}_{\ub}}\frac{1}{c^{2}}\rho_{n}K_{0}\psi_{n}d\mu_{g}
=&-\int_{W^{t}_{\ub}}\widetilde{\rho}_{n}\left(L\psi_{n}
+\Lb\psi_{n}\right)dtd\ub d\mu_{\slashed{g}}\\
-\int_{W^{t}_{\ub}}\frac{1}{c^{2}}\rho_{n}\left(K_{1}\psi_{n}-t\psi_{n}\right)d\mu_{g}
=&-\int_{W^{t}_{\ub}}\widetilde{\rho}_{n}\left(K_{1}\psi_{n}-t\psi_{n}\right)dtd\ub d\mu_{\slashed{g}}
\end{align*}
We first consider the contribution of $\leftexp{(Z)}{\sigma}_{n-1}$ in $\widetilde{\rho}_{n}$. We write $\leftexp{(Z)}{\sigma}_{n-1}$ in null frame $(\Lb,L,\dfrac{\partial}{\partial\theta^{A}})$:
\begin{align*}
\leftexp{(Z)}{\sigma}_{n-1}=&-\frac{1}{2}L(\leftexp{(Z)}{J}_{n-1,\underline{L}})-\frac{1}{2}\underline{L}
(\leftexp{(Z)}{J}_{n-1,L})+\slashed{\textrm{div}}(\mu\leftexp{(Z)}{\slashed{J}}_{n-1})\\\notag
&-\frac{1}{2}\underline{L}(c^{-2}\mu)\leftexp{(Z)}{J}_{n-1,\underline{L}}-\frac{1}{2}\textrm{tr}\chi
\leftexp{(Z)}{J}_{n-1,\underline{L}}-\frac{1}{2}\textrm{tr}\underline{\chi}\leftexp{(Z)}{J}_{n-1,L},
\end{align*}
Then with the following expressions for  the components of $\leftexp{(Z)}{J}_{n-1}$ in the null frame:
\begin{align*}
\leftexp{(Z)}{J}_{n-1,\underline{L}}=&-\frac{1}{2}\textrm{tr}\leftexp{(Z)}{\tilde{\slashed{\pi}}}
(\underline{L}\psi_{n-1})+\leftexp{(Z)}{\tilde{\underline{Z}}}\cdot\slashed{d}\psi_{n-1}\\
\leftexp{(Z)}{J}_{n-1,L}=&-\frac{1}{2}\textrm{tr}\leftexp{(Z)}{\tilde{\slashed{\pi}}}(L\psi_{n-1})
+\leftexp{(Z)}{\tilde{Z}}\cdot\slashed{d}\psi_{n-1}-\frac{1}{2\mu}\leftexp{(Z)}{\tilde{\pi}}_{LL}
(\underline{L}\psi_{n-1})\\
\mu\leftexp{(Z)}{\slashed{J}}^{A}_{n-1}=&-\frac{1}{2}\leftexp{(Z)}{\tilde{Z}}^{A}(\underline{L}\psi_{n-1})
-\frac{1}{2}\leftexp{(Z)}{\tilde{\underline{Z}}}^{A}(L\psi_{n-1})\\\notag
&+\frac{1}{2}(\leftexp{(Z)}{\tilde{\pi}}_{L\underline{L}}-\mu\textrm{tr}\leftexp{(Z)}{\tilde{\slashed{\pi}}})
\slashed{d}^{A}\psi_{n-1}+\mu\leftexp{(Y)}{\tilde{\slashed{\pi}}}^{A}_{B}\slashed{d}^{B}\psi_{n-1}
\end{align*}
Based on the above expressions, we decompose:
\begin{align*}
\leftexp{(Z)}{\sigma}_{n-1}=\leftexp{(Z)}{\sigma}_{1,n-1}+\leftexp{(Z)}{\sigma}_{2,n-1}+\leftexp{(Z)}{\sigma}_{3,n-1}
\end{align*}
where $\leftexp{(Z)}{\sigma}_{1,n-1}$ contains the products of components of $\leftexp{(Z)}{\widetilde{\pi}}$ with the 2nd derivatives of $\psi_{n-1}$, $\leftexp{(Z)}{\sigma}_{2,n-1}$ contains the products of the 1st derivatives of $\leftexp{(Z)}{\widetilde{\pi}}$ with the 1st derivatives of
$\psi_{n-1}$, and $\leftexp{(Z)}{\sigma}_{3,n-1}$ contains the other lower order terms. More specifically, we have:
\begin{align}\label{sigma1}
\begin{split}
\leftexp{(Z)}{\sigma}_{1,n-1}=&\frac{1}{2}\text{tr}\leftexp{(Z)}{\tilde{\slashed{\pi}}}(\Lb L\psi_{n-1}+\frac{1}{2}\text{tr}_{\tilde{\slashed{g}}}\widetilde{\chib}L\psi_{n-1})\\
&+\frac{1}{4}(\mu^{-1}\leftexp{(Z)}{\tilde{\pi}}_{LL})\Lb^{2}\psi_{n-1}\\
&-\leftexp{(Z)}{\tilde{Z}}\cdot\ds\Lb\psi_{n-1}-\leftexp{(Z)}{\tilde{\Zb}}\cdot\ds L\psi_{n-1}\\
&+\frac{1}{2}\leftexp{(Z)}{\tilde{\pi}}_{L\Lb}\slashed{\Delta}\psi_{n-1}+\mu\leftexp{(Z)}{\widehat{\tilde{\pi}}}\cdot\slashed{D}^{2}\psi_{n-1}
\end{split}
\end{align}
\begin{align}\label{sigma2}
\begin{split}
\leftexp{(Z)}{\sigma}_{2,n-1}=&\frac{1}{4}\Lb(\text{tr}\leftexp{(Z)}{\tilde{\slashed{\pi}}})L\psi_{n-1}+\frac{1}{4}L(\text{tr}\leftexp{(Z)}{\tilde{\slashed{\pi}}})\Lb\psi_{n-1}\\
&+\frac{1}{4}\Lb(\mu^{-1}\leftexp{(Z)}{\tilde{\pi}}_{LL})\Lb\psi_{n-1}\\
&-\frac{1}{2}\slashed{\mathcal{L}}_{\Lb}\leftexp{(Z)}{\tilde{Z}}\cdot\ds\psi_{n-1}-\frac{1}{2}\slashed{\mathcal{L}}_{L}\leftexp{(Z)}{\tilde{\Zb}}\cdot\ds\psi_{n-1}\\
&-\frac{1}{2}\slashed{\text{div}}\leftexp{(Z)}{\tilde{Z}}\Lb\psi_{n-1}-\frac{1}{2}\slashed{\text{div}}\leftexp{(Z)}{\tilde{\Zb}}L\psi_{n-1}\\
&+\frac{1}{2}\ds\leftexp{(Z)}{\tilde{\pi}}_{L\Lb}\ds\psi_{n-1}+\slashed{\text{div}}(\mu\leftexp{(Z)}{\widehat{\tilde{\slashed{\pi}}}})\cdot\ds\psi_{n-1}
\end{split}
\end{align}
and
\begin{align}\label{sigma3}
\begin{split}
\leftexp{(Z)}{\sigma}_{3,n-1}=&\big(\frac{1}{4}\text{tr}\chi\text{tr}\leftexp{(Z)}{\tilde{\slashed{\pi}}}+\frac{1}{4}\text{tr}\chib(\mu^{-1}\leftexp{(Z)}{\tilde{\pi}}_{LL})\\&+\frac{1}{2}\leftexp{(Z)}{\tilde{\Zb}}\cdot\ds(c^{-2}\mu)\big)\Lb\psi_{n-1}-\frac{1}{4}(\Lb\log(c^{-1}))\text{tr}\leftexp{(Z)}{\tilde{\slashed{\pi}}}L\psi_{n-1}\\
&-\big((\frac{1}{2}\text{tr}\chi+\Lb(c^{-2}\mu))\leftexp{(Z)}{\tilde{\Zb}}+\frac{1}{2}\text{tr}\chib\leftexp{(Z)}{\tilde{Z}}\big)\ds\psi_{n-1}
\end{split}
\end{align}
With these expressions for $\leftexp{(Z)}{\sigma}_{n-1}$, we are able to investigate the structure of $\widetilde{\rho}_{n}$. Basically, we want to use the recursion formulas in \eqref{rescaled source terms} to obtain a relatively explicit expression for $\widetilde{\rho}_{n}$.

On the other hand, for the energy estimates, we consider the following possible $\psi_{n}$:
\begin{align*}
\psi_{n}=R_{i}^{\alpha+1}\psi, \quad \psi_{n}=R^{\alpha'}_{i}T^{l+1}\psi,\quad \psi_{n}=QR^{\alpha'}_{i}T^{l}\psi
\end{align*}
Here $\psi_{n}$ is the $n$th order variation and $n=|\alpha|+1=|\alpha'|+l+1$ and $\psi$ is any first order variation. The reason that we can always first apply $T$, then $R_{i}$, and finally a possible $Q$ is that the commutators $[R_{i},T]$, $[R_{i},Q]$, and $[T,Q]$ are one order lower than $R_{i}T, TR_{i}$; $Q R_{i}, R_{i}Q$; and $QT,TQ$ respectively. Moreover, all these commutators are tangent to $S_{t,\ub}$. Since we let $Q$ be the last possible commutator, there will be no $Q$'s in $\psi_{n-1}$ in the second term on the right hand side of \eqref{sigma1}. Therefore we only need to commute $Q$ once.

Now suppose that we consider the variations of order $n=|\alpha|+2$ in the following form:
\begin{align*}
\psi_{|\alpha|+2}:=Z_{|\alpha|+1}...Z_{1}\psi.
\end{align*}
We have the inhomogeneous wave equation:
\begin{align*}
\Box_{\tilde{g}}\psi_{|\alpha|+2}=\rho_{|\alpha|+2}
\end{align*}
As we pointed out in Remark 8.1, we define:
\begin{align*}
\widetilde{\rho}_{|\alpha|+2}=\frac{\mu}{c^{2}}\rho_{|\alpha|+2}
\end{align*}
Then by a induction argument, the corresponding inhomogeneous term $\widetilde{\rho}_{|\alpha|+2}$ is given by:
\begin{align}\label{rho tilde n}
\widetilde{\rho}_{|\alpha|+2}=\sum_{k=0}^{|\alpha|}\big(Z_{|\alpha|+1}+\leftexp{(Z_{|\alpha|+1})}{\delta}\big)...\big(Z_{|\alpha|-k+2}+\leftexp{(Z_{|\alpha|-k+2})}{\delta}\big)\leftexp{(Z_{|\alpha|-k+1})}{\sigma}_{|\alpha|-1+k}
\end{align}
\subsection{Error Estimates for the lower order terms}
Consider an arbitrary term in the sum \eqref{rho tilde n}. There is a total of $k$ derivatives with respect to the commutators acting on $\leftexp{(Z)}{\sigma}_{|\alpha|-1+k}$. In view of the fact that $\leftexp{(Z)}{\sigma}_{|\alpha|-1+k}$ has the structure described in \eqref{sigma1}, \eqref{sigma2} and \eqref{sigma3}, in considering the partial contribution of each term in $\leftexp{(Z)}{\sigma}_{1,|\alpha|-1+k}$, if the factor which is a component of $\leftexp{(Z)}{\tilde{\pi}}$ receives more than $\big[\dfrac{|\alpha|+1}{2}\big]$ derivatives with respect to the commutators, then the factor which is a $2$nd order derivative of $\psi_{|\alpha|+1-k}$ receives at most $k-\big[\dfrac{|\alpha|+1}{2}\big]-1$ order derivatives of commutators, thus corresponds to a derivative of the $\psi$ of order at most: $k-\big[\dfrac{|\alpha|+1}{2}\big]+1+|\alpha|-k=\big[\dfrac{|\alpha|}{2}\big]+1$, therefore this factor is bounded in $L^{\infty}(\Sigma_{t}^{\ub})$ by the bootstrap assumption. Also, in considering the partial contribution of each term in $\leftexp{(Z)}{\sigma}_{2,|\alpha|+1-k}$, if the factor which is a 1st derivative of $\leftexp{(Z)}{\tilde{\pi}}$ receives more than $\big[\dfrac{|\alpha|+1}{2}\big]-1$ derivatives with respect to the commutators, then the factor which is a 1st derivative of $\psi_{|\alpha|+1-k}$ receives at most $k-\big[\dfrac{|\alpha|+1}{2}\big]$ derivatives with respect to the commutators, thus corresponds to a derivative of the $\psi_{\alpha}$ of order at most $ k-\big[\dfrac{|\alpha|+1}{2}\big]+1+|\alpha|-k=\big[\dfrac{|\alpha|}{2}\big]+1$, therefore this factor is again bounded in $L^{\infty}(\Sigma_{t}^{\ub})$ by the bootstrap assumption. Similar considerations apply to $\leftexp{(Z)}{\sigma}_{3,|\alpha|+1-k}$. We conclude that for all the terms in the sum in \eqref{rho tilde n} of which one factor is a derivative of the $\leftexp{(Z)}{\tilde{\pi}}$ of order more than $\big[\dfrac{|\alpha|+1}{2}\big]$, the other factor is then a derivative of the $\psi_{\alpha}$ of order at most $\big[\dfrac{|\alpha|}{2}\big]+1$ and is thus bounded in $L^{\infty}(\Sigma_{t}^{\ub})$ by the bootstrap assumption. Of these terms we shall estimate the contribution of those containing the top order spatial derivatives of the optical entities in the next subsection. Before we give the estimates for the contribution of the lower order optical terms to the spacetime integrals:

\begin{align}\label{spacetime error lower order}
-\delta^{2k}\int_{W^{t}_{\ub}}\widetilde{\rho}_{\leq|\alpha|+2}\left(L\psi_{\leq|\alpha|+2}+\Lb\psi_{\leq|\alpha|+2}\right)dtd\ub d\mu_{\slashed{g}},\quad -\delta^{2k}\int_{W^{t}_{\ub}}\widetilde{\rho}_{\leq|\alpha|+2}\left(\frac{2(-t')}{\widetilde{\tr}\chib}\Lb\psi_{\leq|\alpha|+2}-t'\psi_{\leq|\alpha|+2}\right)dtd\ub d\mu_{\slashed{g}},
\end{align}
we investigate the behavior of these integrals with respect to $\delta$. Here $k$ is the number of $T$s in string of commutators.
For the multiplier $K_{1}=\dfrac{2(-t)}{\widetilde{\tr}\chib}\Lb$, the associated energy inequality is 
\begin{align}\label{energy inequality delta behavior K1}
\Eb_{\leq|\alpha|+2}(t,\ub)+\Fb_{\leq|\alpha|+2}(t,\ub)+K_{\leq|\alpha|+2}(t,\ub)\lesssim\Eb_{\leq|\alpha|+2}(-r_{0},\ub)+\int_{W^{t}_{\ub}}\widetilde{Q}_{1,\leq|\alpha|+2}.
\end{align}
The quantities $K_{\leq|\alpha|+2}(t,\ub)$ are defined similar as $K(t,\ub)$:
\begin{align*}
K_{\leq|\alpha|+2}(t,\ub):=\sum_{|\alpha'|\leq |\alpha|+1}\delta^{l'}K(t,\ub)[Z^{\alpha'}\psi],
\end{align*}
Again, $l'$ is the number of $T$'s in $Z^{\alpha'}$.

In $\widetilde{Q}_{1,\leq|\alpha|+2}$, there are contributions from the deformation tensors of two multipliers, which has been treated in Section 5. There are also contributions from the deformation tensors of commutators, which are given by \eqref{rho tilde n}. Now we investigate the terms which are not top order optical terms, namely, the terms containing $\chib$ and $\mu$ of order less than $|\alpha|+2$. In view of the discussion in Section 5, the left hand side of \eqref{energy inequality delta behavior K1} is of order $\delta$, so we expect these lower order terms in the second integral of \eqref{spacetime error lower order} is at least of order $\delta$. In fact the integration on $W^{t}_{\ub}$ gives us a $\delta$ and the contribution from the variation $\delta^{k}\left(\dfrac{2(-t')}{\widetilde{\tr}\chib}\Lb\psi_{\leq|\alpha|+2}-t\psi_{\leq|\alpha|+2}\right)$ is of order $\delta^{1/2}$. For the behavior of $\delta^{k}\sigma$, we take a look at $\sigma_{1}$ as an example. Let $k'$ be the number of $T$s applied to $\left(\Lb L\psi_{l}+\frac{1}{2}\widetilde{\tr}\chib L\psi_{l}\right)$. \eqref{box psi_0 in null frame} implies 
\begin{align*}
\delta^{k'}\left(\Lb L\psi_{l}+\frac{1}{2}\widetilde{\tr}\widetilde{\chib}L\psi_{l}\right)\sim\delta^{1/2}
+\delta^{k'}\rho_{l}.
\end{align*}
Since $\rho_{1}=0$, an induction argument implies that 
 \begin{align*}
 \delta^{k'}\left(\Lb L\psi_{l}+\frac{1}{2}\widetilde{\tr}\widetilde{\chib}L\psi_{l}\right)\sim\delta^{1/2}.
 \end{align*}
 Then in view of \eqref{deformation tensor of T}, \eqref{deformation tensor of Q} and \eqref{deformation tensor of Rotational R_i in T,Lb,X_A}, the first term in $\sigma_{1}$ behaves like $\delta^{1/2}$. Following the same procedure, one sees straightforwardly that all the other terms in $\sigma_{1}, \sigma_{2}$ and $\sigma_{3}$ behave like $\delta^{1/2}$ (one keeps in mind that if $Z=T$, then we multiplier a $\delta$ with the corresponding deformation tensor.) except
 the term $\Lb\left(\tr{}^{(Z)}\widetilde{\slashed{\pi}}\right)L\psi_{l}$. For this term we use the argument deriving \eqref{Improved estimate for trQ} and Proposition \ref{L infity estimates on lot} to see actually we have:
\begin{align*}
\|\Lb\big(\text{tr}\leftexp{(Q)}{\tilde{\slashed{\pi}}}\big)\|_{L^{\infty}(\Sigma_{t}^{\ub})}\lesssim\delta.
\end{align*}
This completes the discussions for $\sigma$ associated to $K_{1}$.

The same argument applies to the energy inequality associated to $K_{0}$: 
\begin{align}\label{energy inequality delta behavior K0}
E_{\leq|\alpha|+2}(t,\ub)+F_{\leq|\alpha|+2}(t,\ub)\lesssim E_{\leq|\alpha|+2}(-r_{0},\ub)+\int_{W^{t}_{\ub}}\widetilde{Q}_{0,\leq|\alpha|+2}.
\end{align}
and we conclude that the lower order optical terms in the error spacetime integrals have one more power in $\delta$ than the energies on the left hand side. 

Now we summarize the spacetime error estimates for the terms which come from the $L^{2}$ norms of the lower order optical quantities. In view of the proof for Proposition \ref{Proposition lower order L2} and \ref{Proposition lower order L2 mu} and the bootstrap assumption on $L^{\infty}$ norms of variations, the contribution to the spacetime error integral from the $L^{2}$ norms of the lower order optical terms is bounded as:

\begin{align}\label{L2 lower optical K1}
\begin{split}
&\int_{-r_{0}}^{t}(-t')^{-2}\left(\sum_{|\alpha'|\leq |\alpha|+1}\delta^{1/2+l}\|\ds Z_{i}^{\alpha'}\psi\|_{L^{2}(\Sigma_{t'}^{\ub})}+\delta^{1/2}\sqrt{E_{\leq|\alpha|+2}(t',\ub)}\right)dt'\\
&\cdot\delta^{1/2}\int_{-r_{0}}^{t}(-t')^{-1}\left\|\left(\Lb\psi_{|\alpha|+2}+\frac{1}{2}\widetilde{\tr}\chib\psi_{|\alpha|+2}\right)\right\|_{L^{2}(\Sigma_{t'}^{\ub})}dt'\\
\lesssim&\int_{-r_{0}}^{t}(-t')^{-2}\left(\sum_{|\alpha'|\leq|\alpha|+1}\delta^{2+2l}\|\ds Z^{\alpha'}_{i}\psi\|^{2}_{L^{2}(\Sigma_{t'}^{\ub})}+\delta^{2} E_{\leq|\alpha|+2}(t',\ub)\right)dt'\\
&+\int_{0}^{\ub}\Fb_{\leq|\alpha|+2}(t,\ub')d\ub'\\
\lesssim&\delta^{2}\int_{-r_{0}}^{t}(-t')^{-2}E_{\leq|\alpha|+2}(t',\ub)dt'+\int_{0}^{\ub}\Fb_{\leq|\alpha|+2}(t,\ub')d\ub'+\delta^{2}K_{\leq|\alpha|+2}(t,\ub)\quad \textrm{for}\quad K_{1}.
\end{split}
\end{align}
and

\begin{align}\label{L2 lower optical K2}
\begin{split}
&\int_{-r_{0}}^{t}(-t')^{-2}\left(\sum_{|\alpha'|\leq |\alpha|+1}\delta^{1/2+l}\|\ds Z_{i}^{\alpha'}\psi\|_{L^{2}(\Sigma_{t'}^{\ub})}+\delta^{1/2}\sqrt{E_{\leq|\alpha|+2}(t',\ub)}\right)dt'\\
&\cdot\delta^{1/2}\int_{-r_{0}}^{t}(-t')^{-3/2}\left\|\Lb\psi_{|\alpha|+2}+L\psi_{|\alpha|+2}\right\|_{L^{2}(\Sigma_{t'}^{\ub}}dt'\\
\lesssim&\int_{-r_{0}}^{t}(-t')^{-2}\left(\sum_{|\alpha'|\leq|\alpha|+1}\delta^{1+2l}\|\ds Z^{\alpha'}_{i}\psi\|^{2}_{L^{2}(\Sigma_{t'}^{\ub})}+\delta E_{\leq|\alpha|+2}(t',\ub)\right)dt'\\
&+\delta\int_{-r_{0}}^{t}(-t')^{-3/2}E_{\leq|\alpha|+2}(t',\ub)dt'\\
\lesssim&\delta\int_{-r_{0}}^{t}(-t')^{-3/2}E_{\leq|\alpha|+2}(t',\ub)dt'+\delta K_{\leq|\alpha|+2}(t,\ub),\quad \textrm{for}\quad K_{0}.
\end{split}
\end{align}

Next we consider the case in which the deformation tensors receive less derivatives such that they can be bounded in $L^{\infty}$.  More specifically, we consider the terms in the sum \eqref{rho tilde n} in which there are at most $\big[\dfrac{|\alpha|+1}{2}\big]$ derivatives hitting the deformation tensor $\leftexp{(Z)}{\tilde{\slashed{\pi}}}$, thus the spatial derivatives on $\chib$ is at most $\big[\dfrac{|\alpha|+1}{2}\big]$ and the spatial derivatives on $\mu$ is at most $\big[\dfrac{|\alpha|+1}{2}\big]+1$, which are bounded in $L^{\infty}(\Sigma_{t}^{\ub})$. 

Let us start with the contribution associated to $K_{1}$. In view of \eqref{L infinity etimates on the deformation tensor of T}, \eqref{Improved estimate for trQ} and \eqref{estimates on the deformation tensor of Rotational R_i in g tilde}, the contribution from the first line in \eqref{sigma1} is bounded by 

\begin{align}\label{sigma1 1 K1}
\int_{0}^{\ub}\Fb_{\leq|\alpha|+2}(t,\ub')d\ub'.
\end{align}
In view of \eqref{L infinity etimates on the deformation tensor of T}, \eqref{Linfty of deformation tensor of Q} and \eqref{estimates on the deformation tensor of Rotational R_i in g tilde}, the contribution from the second line in \eqref{sigma1} is bounded by (up to a constant)

\begin{align}\label{sigma1 2 K1}
\int_{0}^{\ub}\Fb_{\leq|\alpha|+2}(t,\ub')d\ub'+\delta\int_{-r_{0}}^{t}(-t')^{-2}E_{\leq|\alpha|+2}(t',\ub)dt'.  
\end{align}
The contribution from the first term in the third line of \eqref{sigma1} enjoys the same estimate as \eqref{sigma1 2 K1}, while the contribution of the second term is bounded by (up to a constant)

\begin{align}\label{sigma1 3 K1}
\delta\int_{-r_{0}}^{t}(-t')^{-2}E_{\leq|\alpha|+2}(t',\ub)dt'+\int_{0}^{\ub}\Fb_{\leq|\alpha|+2}(t,\ub')d\ub'
\end{align}
in view of \eqref{L infinity etimates on the deformation tensor of T} and \eqref{estimates on the deformation tensor of Rotational R_i in g tilde}. The contribution from the last line in \eqref{sigma1} is bounded by (up to a constant)

\begin{align}\label{sigma1 4 K1}
\delta^{1/2}K_{\leq|\alpha|+2}(t,\ub)+\delta^{-1/2}\int_{0}^{\ub}\Fb_{\leq|\alpha|+2}(t,\ub')d\ub'+\int_{-r_{0}}^{t}(-t')^{-2}\Eb_{\leq|\alpha|+2}(t',\ub)dt'.
\end{align}
Here besides \eqref{L infinity etimates on the deformation tensor of T}, \eqref{Linfty of deformation tensor of Q} and \eqref{estimates on the deformation tensor of Rotational R_i in g tilde}, we also used the inequality $ab\leq \frac{1}{2}\left(\delta^{1/2}a^{2}+\delta^{-1/2}b^{2}\right)$. 

The contribution of the third, fourth and fifth lines in \eqref{sigma2} is bounded in the similar way as the third and fourth lines in \eqref{sigma1}. The second term in the first line of \eqref{sigma2} can be bounded in the similar way as the first line of \eqref{sigma1}. The contributions from the first term in the first line and the second line of \eqref{sigma2} are bounded by (up to a constant)

\begin{align}\label{sigma2 1 K1}
\int_{0}^{\ub}\Fb_{\leq|\alpha|+2}(t,\ub')d\ub'+\delta\int_{-r_{0}}^{t}(-t')^{-2}E_{\leq|\alpha|+2}(t',\ub)dt'.
\end{align}

Since the contributions from \eqref{sigma3} are of lower order compared to the contributions of \eqref{sigma1} and \eqref{sigma2}, we omit the details. Therefore the contributions associated to $K_{1}$ are bounded by (up to a constant)

\begin{align}\label{sigma K1}
\begin{split}
&\delta^{1/2}K_{\leq|\alpha|+2}(t,\ub)+\delta^{-1/2}\int_{0}^{\ub}\Fb_{\leq|\alpha|+2}(t,\ub')d\ub'\\
+&\int_{-r_{0}}^{t}(-t')^{-2}\Eb_{\leq|\alpha|+2}(t',\ub)dt'+\delta\int_{-r_{0}}^{t}(-t')^{-2}E_{\leq|\alpha|+2}(t',\ub)dt'.
\end{split}
\end{align}

Similarly to the contributions associated to $K_{1}$, these contributions are bounded by (up to a constant)

\begin{align}\label{sigma K0}
\delta^{-1/2}\int_{0}^{\ub}\Fb_{\leq|\alpha|+2}(t,\ub')d\ub'+\delta^{1/2}\int_{-r_{0}}^{t}(-t')^{-2}E_{\leq|\alpha|+2}(t',\ub)dt'+\delta^{1/2}K_{\leq|\alpha|+2}(t,\ub).
\end{align}

\subsection{Top Order Optical Estimates}
Now we estimate the contributions from the top order optical terms to the error spacetime integrals. In estimating the top order optical terms, we need to choose the power of $\mu_{m}(t)$ large enough. Therefore from this subsection on, we will use $C$ to denote an absolute positive constant so that one can see the largeness of the power of $\mu_{m}(t)$ more clearly.

The top order optical terms come from the term in which all the commutators hit the deformation tensors in the expression of $\leftexp{(Z_{1})}{\sigma}_{1}$, namely, the term:
\begin{align*}
(Z_{|\alpha|+1}+\leftexp{(Z_{|\alpha|+1})}{\delta})...(Z_{2}+\leftexp{(Z_{2})}{\delta})\leftexp{(Z_{1})}{\sigma}_{1}
\end{align*}
more precisely, in:
\begin{align*}
Z_{|\alpha|+1}...Z_{2}\big(-\frac{1}{2}L(\leftexp{(Z_{1})}{J}_{1,\Lb})-\frac{1}{2}\Lb(\leftexp{(Z_{1})}{J}_{1,L})+\slashed{\text{div}}(\mu\leftexp{(Z_{1})}{\slashed{J}})\big)
\end{align*}
when the operators $L, \Lb, \slashed{\text{div}}$ hit the deformation tensors in the expression of $\leftexp{(Z_{1})}{J}$.

Now we consider the top order variations:
\begin{align*}
R_{i}^{\alpha+1}\psi,\quad R_{i}^{\alpha'}T^{l+1}\psi,\quad QR^{\alpha'}_{i}T^{l}\psi
\end{align*}
where $|\alpha|=\Ntop-1$ and $|\alpha'|+l+1=|\alpha|+1$. 
Then the corresponding principal optical terms are:
\begin{align*}
\widetilde{\rho}_{|\alpha|+2}(R_{i}^{\alpha+1}\psi):=\frac{1}{c}(R_{i}^{\alpha+1}\text{tr}\chib' )\cdot T\psi,\quad \widetilde{\rho}_{|\alpha|+2}(R_{i}^{\alpha'}T^{l+1}\psi):=\frac{1}{c}(R_{i}^{\alpha'}T^{l}\slashed{\Delta}\mu)\cdot T\psi\\
\widetilde{\rho}_{|\alpha|+2}(QR^{\alpha'}_{i}T^{l}\psi):=\frac{t\mu}{c}\big(\slashed{d}R^{\alpha'}_{i}\text{tr}
\chib'\big)\cdot\ds\psi+\frac{t\mu}{c}(\ds R^{\alpha'}_{i}\slashed{\Delta}\mu)\cdot\Lb\psi,\quad\text{if}\quad l=0\\
\widetilde{\rho}_{|\alpha|+2}(QR^{\alpha'}_{i}T^{l}\psi):=\frac{t\mu}{c}\big(\slashed{d}R^{\alpha'}_{i}T^{l-1}
\slashed{\Delta}\mu\big)\cdot\ds\psi+\frac{t\mu}{c}(\ds R^{\alpha'}_{i}T^{l}\slashed{\Delta}\mu)\cdot\Lb\psi,\quad\text{if}\quad l\geq1
\end{align*}
Here we used the structure equation \eqref{Structure Equation T chib}
\begin{align*}
\slashed{\mathcal{L}}_{T}\text{tr}\chib
=\slashed{\Delta}\mu+\mathcal{O}^{\leq1}_{\geq0, 2}
\end{align*}
Now we briefly investigate the behavior of the above terms with respect to $\delta$ and $t$. Note that $R^{\alpha}_{i}Q\psi$ has the same behavior as $R^{\alpha+1}_{i}\psi$ with respect to $\delta$ and $t$, while for their corresponding top order optical terms $\dfrac{t\mu}{c}\big(\ds R^{\alpha}_{i}\text{tr}\chib'\big)
\cdot\ds\psi$ and $\dfrac{1}{c}\big(R^{\alpha+1}_{i}\text{tr}\chib'
\big)\cdot T\psi$, the former behaves better than the latter with respect to $\delta$ and $t$:
\begin{align*}
|\dfrac{t\mu}{c}\big(\ds R^{\alpha}_{i}\text{tr}\chib'\big)
\cdot\ds\psi|\sim |\mu \big(R^{\alpha+1}_{i}\text{tr}\chib'\big)|\delta^{1/2}(-t)^{-1},\quad |\frac{1}{c}\big(R^{\alpha+1}_{i}\text{tr}
\chib'\big)\cdot T\psi|\sim|\big(R^{\alpha+1}_{i}\text{tr}\chib'\big)|\delta^{-1/2}
\end{align*}
 We see that not only the former behaves better with respect to $\delta$, but also has an extra $\mu$, which makes the behavior even better when $\mu$ is small. This means that we only need to estimate the contribution of $\dfrac{1}{c}\big(R^{\alpha+1}_{i}\text{tr}\chib'
\big) \cdot T\psi$. The same analysis applies to the terms involving $\Lb\psi$ as well as the comparison between $R^{\alpha'}_{i}T\psi$ and $QR^{\alpha'}_{i}T\psi$, which correspond to $\dfrac{1}{c}\big(R^{\alpha}_{i}\slashed{\Delta}
\mu\big)\cdot T\psi$ and $\dfrac{t\mu}{c}\big(R^{\alpha'+1}_{i}
\slashed{\Delta}\mu\big)\cdot\ds\psi$. So in the following, we do not need to estimate the contributions corresponding to the variations containing a $Q$.

\subsubsection{Contribution of $K_{0}$}
In this subsection we first estimate the spacetime integral:
\begin{align}\label{spacetime K_0 chib}
\begin{split}
&\int_{W^{t}_{\ub}}\frac{1}{c}|R_{i}^{\alpha+1}\text{tr}\chib'||T\psi||LR^{\alpha+1}_{i}\psi|dt'd\ub' d\mu_{\slashed{g}}\lesssim\\ &\int_{-r_{0}}^{t}\sup_{\Sigma_{t'}^{\ub}}(\mu^{-1}|T\psi|)(-t')\|\mu\ds R_{i}^{\alpha}\text{tr}\chib'\|_{L^{2}(\Sigma_{t'}^{\ub})}\|LR^{\alpha+1}_{i}\psi\|_{L^{2}(\Sigma_{t'}^{\ub})}dt'.
\end{split}
\end{align}
Here we do not include the contribution from $\Lb R^{\alpha+1}_{i}\psi$ in the definition of $K_{0}$, because compared to the spacetime error integral associated to $K_{1}$, this contribution is lower order with respect to the behavior in $t$. We will see how this contribution is bounded when we estimate the spacetime integral for $K_{1}$.

By \eqref{estimate for top order chib}, we have, in view of the monotonicity of $\widetilde{E}_{\leq|\alpha|+2}(t,\ub)$ in $t$:

\begin{align}\label{final top chib}
\begin{split}
(-t)\|\ds R^{\alpha}_{i}\tr\chib\|_{L^{2}(\Sigma_{t}^{\ub})}\leq &C\delta\int_{-r_{0}}^{t}(-t')^{-2}\mathcal{B}_{0,\leq|\alpha|+1}(-r_{0})dt'\\
+&C\delta^{1/2}\mu^{-b_{|\alpha|+2}}_{m}(t)(-t)^{-1}\sqrt{\widetilde{E}_{\leq|\alpha|+2}(t,\ub)}\\
+&C\delta^{1/2}\sqrt{\widetilde{\Eb}_{\leq|\alpha|+2}(t,\ub)}\int_{-r_{0}}^{t}(-t')^{-2}\mu_{m}^{-b_{|\alpha|+2}-1/2}(t')dt'\\
+&C\delta^{1/2}\int_{-r_{0}}^{t}(-t')^{-2}\mu_{m}^{-b_{|\alpha|+2}}(t')\sqrt{\widetilde{E}_{\leq|\alpha|+2}(t',\ub)}dt'\\
+&C\delta^{-1/2}\mu_{m}^{-b_{|\alpha|+2}}(t)(-t)^{-1}\left(\int_{0}^{\ub}\widetilde{\Fb}_{\leq|\alpha|+2}(t,\ub')d\ub'\right)^{1/2}.
\end{split}
\end{align}
Without loss of generality, here we assume that there is a $t_{0}\in[-r_{0},t^{*})$ such that $\mu_{m}(t_{0})=\frac{1}{10}$ and $\mu_{m}(t')\geq \frac{1}{10}$ for $t\leq t_{0}$. If there is no such $t_{0}$ in $[-r_{0},t^{*})$, then $\mu_{m}(t)$ has an absolute positive lower bound for all $[-r_{0},t^{*})$ and it is clear to see that the following argument simplifies and also works in this case. In view of part (1') and part (2) in Lemma \ref{lemma on mu to power a} and the fact $\mu_{m}(t')\geq \frac{1}{10}$ for $t'\in[-r_{0},t_{0}]$, we have

\begin{align*}
\int_{-r_{0}}^{t_{0}}\mu_{m}^{-b_{|\alpha|+2}-1/2}(t')dt'\lesssim&\mu_{m}^{-b_{|\alpha|+2}+1/2}(t_{0})\leq \mu_{m}^{-b_{|\alpha|+2}+1/2}(t),\\
\int_{t_{0}}^{t}\mu_{m}^{-b_{|\alpha|+2}-1/2}(t')dt'\lesssim&\frac{1}{\left(b_{|\alpha|+2}-1/2\right)}\mu_{m}^{-b_{|\alpha|+2}+1/2}(t).
\end{align*}

Therefore the third term in \eqref{final top chib} are bounded by:
\begin{align*}
\delta^{1/2}\mu_{m}^{-b_{|\alpha|+2}+1/2}(t)\sqrt{\widetilde{\Eb}_{\leq |\alpha|+2}(t,\ub)}.
\end{align*}
Substituting this in \eqref{spacetime K_0 chib}, and using the fact that $|T\psi|\lesssim\delta^{-1/2}(-t)^{-1}$, we see that the spacetime integral \eqref{spacetime K_0 chib} is bounded by (up to a constant):

\begin{align}\label{top K0 chib temp 1}
\begin{split}
&\int_{-r_{0}}^{t}(-t')^{-2}\mu^{-b_{|\alpha|+2}-1}_{m}(t')\sqrt{\widetilde{E}_{\leq|\alpha|+2}(t',\ub)}\|LR^{\alpha+1}_{i}\psi\|_{L^{2}(\Sigma_{t'}^{\ub})}dt'\\&+\int_{-r_{0}}^{t}(-t')^{-2}\mu^{-b_{|\alpha|+2}-1/2}_{m}(t')\sqrt{\widetilde{\Eb}_{\leq|\alpha|+2}(t',\ub)}\|LR^{\alpha+1}_{i}\psi\|_{L^{2}(\Sigma_{t'}^{\ub})}dt'\\
&+\int_{-r_{0}}^{t}\delta^{-1}(-t')^{-2}\mu^{-b_{|\alpha|+2}-1}_{m}(t')\sqrt{\int_{0}^{\ub}\widetilde{\Fb}_{\leq|\alpha|+2}(t',\ub')d\ub'}\|LR^{\alpha+1}_{i}\psi\|_{L^{2}(\Sigma_{t'}^{\ub})}dt'\\
&+\int_{-r_{0}}^{t}(-t')^{-1}\mu_{m}^{-b_{|\alpha|+2}-1}(t')\left(\int_{-r_{0}}^{t'}(-t'')^{-2}\sqrt{\widetilde{E}_{\leq|\alpha|+2}(t'',\ub)}dt''\right)\|LR^{\alpha+1}_{i}\psi\|_{L^{2}(\Sigma_{t'}^{\ub})}dt'.
\end{split}
\end{align}
For the factor $\|LR^{\alpha+1}_{i}\psi\|_{L^{2}(\Sigma_{t}^{\ub})}$, we bound it as:
\begin{align*}
\|LR^{\alpha+1}_{i}\psi\|_{L^{2}(\Sigma_{t}^{\ub})}\leq \sqrt{E_{|\alpha|+2}(t,\ub)}\leq \mu_{m}^{-b_{|\alpha|+2}}(t)\sqrt{\widetilde{E}_{\leq|\alpha|+2}(t,\ub)},
\end{align*}
To estimate \eqref{top K0 chib temp 1} we split the integral as $\int_{-r_{0}}^{t_{0}}+\int_{t_{0}}^{t}$, which we call the ``non-shock" and ``shock" parts respectively. In view of the second part of Lemma \ref{lemma on mu to power a}, the ``non-shock" part of \eqref{top K0 chib temp 1} is bounded by (up to a constant)

\begin{align}\label{top K0 chib temp 2 non shock}
\begin{split}
&
\mu^{-2b_{|\alpha|+2}+1/2}_{m}(t)\int_{-r_{0}}^{t_{0}}(-t')^{-2}\widetilde{\Eb}_{\leq|\alpha|+2}(t',\ub)dt'\\
&+
\delta^{-2}\mu_{m}^{-2b_{|\alpha|+2}}(t)\int_{0}^{\ub}\widetilde{\Fb}_{\leq|\alpha|+2}(t,\ub')d\ub'\\
&+\mu_{m}^{-2b_{|\alpha|+2}}(t)\int_{-r_{0}}^{t_{0}}(-t')^{-2}\widetilde{E}_{\leq|\alpha|+2}(t',\ub)dt'
\end{split}
\end{align}
Following the proof of Lemma \ref{lemma on mu to power a}, we have, for any $a>0$

\begin{align}\label{integral mu power}
\int_{t_{0}}^{t}\mu^{-a-1}_{m}(t')(-t')^{-2}dt'\leq \frac{C}{a}\mu_{m}^{-a}(t).
\end{align}
Therefore the ``shock part" of \eqref{top K0 chib temp 1} is bounded by (up to a constant)
\begin{align}\label{top K0 chib temp 2 shock}
\begin{split}
\boxed{\frac{\mu_{m}^{-2b_{|\alpha|+2}}(t)}{2b_{|\alpha|+2}}\widetilde{E}_{\leq|\alpha|+2}(t,\ub)}+\frac{\mu_{m}^{-2b_{|\alpha|+2}+1/2}(t)}{2b_{|\alpha|+2}-1/2}\widetilde{\Eb}_{\leq|\alpha|+2}(t,\ub)+\frac{\delta^{-2}\mu_{m}^{-2b_{|\alpha|+2}}(t)}{2b_{|\alpha|+2}}\int_{0}^{\ub}\widetilde{\Fb}_{\leq|\alpha|+2}(t,\ub')d\ub'.
\end{split}
\end{align}

\begin{remark}\label{a dependence}
The boxed term is from the estimates for the top order term $F_{\alpha}$. In view of \eqref{transport equation for F alpha for chib}, the number of top order terms contributed by the variations is independent of $\delta$ and $|\alpha|$, so is the constant $C$ in the boxed term. Later on in the top order energy estimates we will choose $b_{|\alpha|+2}$ in such a way that $\frac{C}{b_{top}}\leq \frac{1}{10}$. (The purpose of doing this is to make sure that this term can be absorbed by the left hand side the energy inequality.)  Therefore we can choose $b_{top}=\{10, \frac{C}{20}\}$. In particular $b_{top}$ is independent of $\delta$. 
\end{remark}

Next we consider the spacetime integral:
\begin{align}\label{spacetime K0 mu}
&\delta^{2l+2}\int_{W^{t}_{\ub}}\frac{1}{c}|R^{\alpha'}_{i}T^{l}\slashed{\Delta}\mu||T\psi||LR^{\alpha'}T^{l+1}\psi|dt'd\ub'd\mu_{\slashed{g}}\lesssim\\\notag
&\delta^{2l+2}\int_{-r_{0}}^{t}\sup_{\Sigma_{t'}^{\ub}}\big(\mu^{-1}|T\psi|\big)\|\mu R^{\alpha'}_{i}T^{l}\slashed{\Delta}\mu\|_{L^{2}(\Sigma_{t'}^{\ub})}\|L R^{\alpha'}T^{l+1}\psi\|_{L^{2}(\Sigma_{t'}^{\ub})}dt'.
\end{align}
In view of \eqref{estimates for top order mu} and Lemma \ref{lemma on mu to power a} we have 

\begin{align}\label{final top mu}
\begin{split}
\delta^{l+1}\|\mu R^{\alpha'}T^{l}\slashed{\Delta}\mu\|_{L^{2}(\Sigma_{t}^{\ub})}\lesssim&\delta^{l+1}r_{0}(-t)^{-1}\|F'_{\alpha,l}(-r_{0})\|_{L^{2}(\Sigma_{-r_{0}}^{\ub})}\\
+&\delta^{1/2}\mu^{-b_{|\alpha|+2}}_{m}(t)(-t)^{-1}\sqrt{\Et_{\leq|\alpha|+2}(t,\ub)}\\
+&\delta^{3/2}(-t)^{-2}\mu^{-b_{|\alpha|+2}}_{m}(t)\sqrt{\Ebt_{\leq|\alpha|+2}(t,\ub)}\\
+&\delta^{1/2}\mu_{m}^{-b_{|\alpha|+2}}(t)\int_{-r_{0}}^{t}(-t')^{-2}\sqrt{\Et_{\leq|\alpha|+2}(t',\ub)}dt'\\
+&\delta^{1/2}\mu^{-b_{|\alpha|+2}+1/2}_{m}(t)\int_{-r_{0}}^{t}(-t')^{-2}\sqrt{\Ebt_{\leq|\alpha|+2}(t',\ub)}dt'\\
+&\delta^{1/2}\mu_{m}^{-b_{|\alpha|+2}}(t)(-t)^{-1}\left(\int_{0}^{\ub}\widetilde{\Fb}_{\leq|\alpha|+2}(t,\ub')d\ub'\right)^{1/2}.
\end{split}
\end{align}
Substituting this to \eqref{spacetime K0 mu} and using the fact that $|T\psi|\lesssim\delta^{-1/2}(-t)^{-1}$, we see that \eqref{spacetime K0 mu} is bounded by (up to a constant)

\begin{align}\label{top K0 mu temp 1}
\begin{split}
&\int_{-r_{0}}^{t}\mu^{-2b_{|\alpha|+2}-1}_{m}(t')(-t')^{-2}\widetilde{E}_{\leq|\alpha|+2}(t',\ub)dt'\\
&+\int_{-r_{0}}^{t}\mu^{-2b_{|\alpha|+2}-1}_{m}(t')(-t')^{-3}\sqrt{\widetilde{\Eb}_{\leq|\alpha|+2}(t',\ub)}\sqrt{\widetilde{E}_{\leq|\alpha|+2}(t',\ub)}dt'\\
&+\int_{-r_{0}}^{t}\mu_{m}^{-2b_{|\alpha|+2}-1}(t')(-t')^{-1}\int_{-r_{0}}^{t'}(-t'')^{-2}\sqrt{\widetilde{E}_{\leq|\alpha|+2}(t',\ub)}dt''\sqrt{\widetilde{E}_{\leq|\alpha|+2}(t',\ub)}dt'\\
&+\int_{-r_{0}}^{t}\mu_{m}^{-2b_{|\alpha|+2}-1}(t')(-t')^{-2}\left(\int_{0}^{\ub}\widetilde{\Fb}_{\leq|\alpha|+2}(t',\ub')d\ub'\right)^{1/2}\sqrt{\widetilde{E}_{\leq|\alpha|+2}(t',\ub)}dt'.
\end{split}
\end{align}
As before we split the spacetime integral \eqref{top K0 mu temp 1} as the ``non-shock" and ``shock" parts. In view of Lemma \ref{lemma on mu to power a}, the ``non-shock" part is bounded by (up to a constant)

\begin{align}\label{top K0 mu temp 2 nonshock}
\begin{split}
&\mu_{m}^{-2b_{|\alpha|+2}}(t)\int_{-r_{0}}^{t_{0}}(-t')^{-2}\widetilde{E}_{\leq|\alpha|+2}(t',\ub)dt'\\
&+\mu_{m}^{-2b_{|\alpha|+2}}(t)\int_{-r_{0}}^{t_{0}}(-t')^{-3}\widetilde{\Eb}_{\leq|\alpha|+2}(t',\ub)dt'\\
&+\mu_{m}^{-2b_{|\alpha|+2}}(t)\int_{0}^{\ub}\widetilde{\Fb}_{\leq|\alpha|+2}(t',\ub')d\ub'.
\end{split}
\end{align}
The ``shock" part is bounded by (up to a constant)

\begin{align}\label{top K0 mu temp 2 shock}
\frac{\mu_{m}^{-2b_{|\alpha|+2}}(t)}{2b_{|\alpha|+2}}\widetilde{E}_{\leq|\alpha|+2}(t,\ub)+\frac{\mu_{m}^{-2b_{|\alpha|+2}}(t)}{2b_{|\alpha|+2}}\widetilde{\Eb}_{\leq|\alpha|+2}(t,\ub)+\frac{\mu_{m}^{-2b_{|\alpha|+2}}(t)}{2b_{|\alpha|+2}}\int_{0}^{\ub}\widetilde{\Fb}_{\leq|\alpha|+2}(t,\ub')d\ub'.
\end{align}

\subsubsection{Contribution of $K_{1}$}
We now turn to estimate the contributions of top order optical terms associated to $K_{1}$. Let us start with the following absolute value of a spacetime integral:

\begin{align}\label{spacetime K1 chib}
\left|\int_{W^{t}_{\ub}}\frac{1}{c}\frac{2(-t)}{\widetilde{\tr\chib}}(R^{\alpha+1}_{i}\tr\chib')\cdot (T\psi)\cdot(\Lb R^{\alpha+1}_{i}\psi+\frac{1}{2}\widetilde{\tr\chib}R^{\alpha+1}_{i}\psi)dt'd\ub'
d\mu_{\slashed{g}}\right|.
\end{align}

In view of the relation between $\slashed{g}$ and $\widetilde{\slashed{g}}$, the above integral can be written as

\begin{align}\label{spacetime K1 chib temp 1}
\int_{W^{t}_{\ub}}\frac{2(-t)}{\widetilde{\tr\chib}}(R^{\alpha+1}_{i}\tr\chib')\cdot (T\psi)\cdot(\Lb R^{\alpha+1}_{i}\psi+\frac{1}{2}\widetilde{\tr\chib}R^{\alpha+1}_{i}\psi)dt'd\ub'
d\mu_{\widetilde{\slashed{g}}}.
\end{align}
For a smooth function $f$ we have

\begin{align*}
\frac{\partial}{\partial t}\left(\int_{S_{t,\ub}}fd\mu_{\widetilde{\slashed{g}}}\right)=\int_{S_{t,\ub}}\left(\Lb f+\widetilde{\tr\chib}f\right)d\mu_{\widetilde{\slashed{g}}},
\end{align*}
which implies

\begin{align*}
\int_{W^{t}_{\ub}}\big(\Lb f+\widetilde{\text{tr}}\chib f\big)d\mu_{\tilde{\slashed{g}}}d\ub'dt'
=\int_{\Sigma_{t}^{\ub}}f
d\mu_{\tilde{\slashed{g}}}d\ub'
-\int_{\Sigma_{-r_{0}}^{\ub}}f
d\mu_{\tilde{\slashed{g}}}d\ub'.
\end{align*}
This inspires us to write the spacetime integral \eqref{spacetime K1 chib temp 1} as

\begin{align*}
&\int_{W^{t}_{\ub}}(R^{\alpha+1}_{i}\text{tr}\chib')\cdot
\left(\frac{2(-t')}{\widetilde{\tr\chib}}T\psi\right)\cdot(\Lb R^{\alpha+1}_{i}\psi
+\widetilde{\text{tr}}\chib R^{\alpha+1}_{i}\psi)dt'd\ub' d\mu_{\tilde{\slashed{g}}}\\
&-\frac{1}{2}\int_{W^{t}_{\ub}}\widetilde{\text{tr}}
\chib(R^{\alpha+1}_{i}\text{tr}\chib')
\cdot\left(\frac{2(-t')}{\widetilde{\tr\chib}}T\psi\right)\cdot (R^{\alpha+1}_{i}\psi)dt'd\ub' d\mu_{\tilde{\slashed{g}}}\\
&=\int_{W^{t}_{\ub}}\big(\Lb
+\widetilde{\text{tr}}\chib\big)\Big[
(R^{\alpha+1}_{i}\text{tr}\chib')
\cdot\left(\frac{2(-t')}{\widetilde{\tr\chib}}T\psi\right)\cdot(R^{\alpha+1}_{i}\psi)\Big]
dt'd\ub'd\mu_{\tilde{\slashed{g}}}\\&-
\int_{W^{t}_{\ub}}(\Lb+\frac{1}{2}\widetilde{\text{tr}}\chib)
\Big[(R^{\alpha+1}_{i}\text{tr}\chib')
\cdot\left(\frac{2(-t')}{\widetilde{\tr\chib}}T\psi\right)\Big]\cdot (R^{\alpha+1}_{i}\psi)dt'd\ub' d\mu_{\tilde{\slashed{g}}}.
\end{align*}

We conclude that \eqref{spacetime K1 chib temp 1} equals:
\begin{align*}
&\int_{\Sigma_{t}^{\ub}}(R^{\alpha+1}_{i}\text{tr}\chib')
\cdot\left(\frac{2(-t)}{\widetilde{\tr\chib}}T\psi\right)\cdot(R^{\alpha+1}_{i}\psi)d\ub'
d\mu_{\tilde{\slashed{g}}}\\
&-\int_{\Sigma_{-r_{0}}^{\ub}}(R^{\alpha+1}_{i}\text{tr}\chib')
\cdot\left(\frac{2r_{0}}{\widetilde{\tr\chib}}T\psi\right)\cdot(R^{\alpha+1}_{i}\psi)d\ub'
d\mu_{\tilde{\slashed{g}}}\\
&-\int_{W^{t}_{\ub}}(\Lb+\frac{1}{2}\widetilde{\text{tr}}\chib)
\Big[(R^{\alpha+1}_{i}\text{tr}\chib')
\cdot\left(\frac{2(-t')}{\widetilde{\tr\chib}}T\psi\right)\Big]\cdot (R^{\alpha+1}_{i}\psi)dt'd\ub' d\mu_{\tilde{\slashed{g}}}.
\end{align*}
We will be using the following lemma for integration by parts:

\begin{lemma}\label{lemma integration by parts}
Let $f,g$ be arbitrary smooth functions defined on $S_{t,\ub}$ and $X$ be a vectorfield tangent to $S_{t,\ub}$. We have:

\begin{align*}
\int_{S_{t,\ub}}f(Xg)d\mu_{\widetilde{\slashed{g}}}=-\int_{S_{t,\ub}}g(Xf)d\mu_{\widetilde{\slashed{g}}}-\frac{1}{2}\int_{S_{t,\ub}}\tr\leftexp{(X)}{\widetilde{\slashed{\pi}}}d\mu_{\widetilde{\slashed{g}}}.
\end{align*}
\end{lemma}

We first consider the hypersurface integral which, integrating by parts using the above lemma, equals:
\begin{align*}
-H_{0}-H_{1}-H_{2}
\end{align*}
where:

\begin{align*}
H_{0}&=\int_{\Sigma_{t}^{\ub}}(R^{\alpha}_{i}\text{tr}\chib)\cdot\left(\frac{2(-t')}{\widetilde{\tr\chib}}T\psi\right)
\cdot(R^{\alpha+2}_{i}\psi)
d\ub'd\mu_{\tilde{\slashed{g}}}\\
H_{1}&=\int_{\Sigma_{t}^{\ub}}\frac{2(-t')}{\widetilde{\tr\chib}}(R^{\alpha}_{i}\text{tr}\chib')
\cdot(R_{i}T\psi)
\cdot(R^{\alpha+1}_{i}\psi)
d\ub'd\mu_{\tilde{\slashed{g}}}\\
H_{2}&=\int_{\Sigma_{t}^{\ub}}(T\psi)\cdot(R^{\alpha+1}_{i}\psi)
\cdot(R^{\alpha}\text{tr}\chib')\left(2R_{i}\left(\frac{(-t')}{\widetilde{\tr\chib}}\right)+\frac{1}{2}\text{tr}\leftexp{(R_{i})}{\tilde{\slashed{\pi}}}\right)
d\ub'd\mu_{\tilde{\slashed{g}}}.
\end{align*}
Since we bound both $T\psi$ and $R_{i}T\psi$ in $L^{\infty}$ norm, compared to $H_{0}$, $H_{1}$ is a lower order term with respect to the order of derivatives. While for $H_{2}$, we use the estimate:

\begin{align*}
2\left|R_{i}\left(\frac{(-t')}{\widetilde{\tr\chib}}\right)\right|+|\text{tr}\leftexp{(R_{i})}{\tilde{\slashed{\pi}}}|\lesssim\delta
\end{align*}
to see that it is a lower order term with respect to both the behavior of $\delta, (-t')$ and the order of derivatives compared to $H_{0}$. This analysis tells us that we only need to estimate $H_{0}$.

\begin{align}\label{H0 pre}
\begin{split}
|H_{0}|\leq &\int_{\Sigma_{t}^{\ub}}(-t)^{3}|T\psi||R^{\alpha}_{i}\text{tr}\chib'||\ds R^{\alpha+1}_{i}\psi|d\ub'd\mu_{\tilde{\slashed{g}}}
\leq \delta^{-1/2}(-t)^{2}\|R^{\alpha}_{i}\tr\chib'\|_{L^{2}(\Sigma_{t'}^{\ub})}\|\ds R^{\alpha+1}_{i}\psi\|_{L^{2}(\Sigma_{t'}^{\ub})}\\
\lesssim&\delta^{-1/2}(-t)\|R^{\alpha}_{i}\tr\chib'\|_{L^{2}(\Sigma_{t'}^{\ub})}\mu_{m}^{-1/2}(t)\sqrt{\Eb_{\leq|\alpha|+2}(t,\ub)}\\
\lesssim&\delta^{-1/2}(-t)\|R^{\alpha}_{i}\tr\chib'\|_{L^{2}(\Sigma_{t'}^{\ub})}\mu_{m}^{-b_{|\alpha|+2}-1/2}(t)\sqrt{\widetilde{\Eb}_{\leq|\alpha|+2}(t,\ub)}.
\end{split}
\end{align}
Even though Proposition 6.3 gives an $L^{2}$-estimate for $R^{\alpha}_{i}\tr\chib'$, here we give an alternative proof which will be used later. In view of \eqref{Structure Equation Lb trchib} and the relation \eqref{definition for chib prime} we have

\begin{align}\label{Lb trchibprime precise}
\Lb\tr\chib'-\frac{2}{\ub-t}\tr\chib'=e\tr\chib-|\chib'|^{2}_{\slashed{g}}-\tr\alpha':=\rho_{0}.
\end{align}
Applying $R^{\alpha}_{i}$ to this equation we have

\begin{align}\label{diff Lb trchibprime}
\Lb R^{\alpha}_{i}\tr\chib'-\frac{2}{\ub-t}R_{i}^{\alpha}\tr\chib'=R^{\alpha}_{i}\rho_{0}+\sum_{|\beta|\leq|\alpha|}R_{i}^{\alpha-\beta}\leftexp{(R_{i})}{\Zb}R^{\beta-1}_{i}\tr\chib',
\end{align}
which can be rewritten as

\begin{align}\label{diff Lb trchibprime modified}
\Lb\left((t-\ub)^{2}R_{i}^{\alpha}\tr\chib'\right)=(t-\ub)^{2}\left(R^{\alpha}_{i}\rho_{0}+\sum_{|\beta|\leq|\alpha|}R_{i}^{\alpha-\beta}\leftexp{(R_{i})}{\Zb}R^{\beta-1}_{i}\tr\chib'\right):=\rho_{\alpha}.
\end{align}
In view of 

\begin{align*}
\|\leftexp{(R_{i})}{\Zb}\|_{L^{\infty}(\Sigma_{t}^{\ub})}\lesssim\delta(-t)^{-2}
\end{align*}
and Proposition 6.3 we have

\begin{align}\label{rho alpha 4}
\begin{split}
\|(t-\ub)^{2}R_{i}^{\alpha-\beta}\leftexp{(R_{i})}{\Zb}R^{\beta-1}_{i}\tr\chib'\|_{L^{2}(\Sigma_{t}^{\ub})}\lesssim&\delta^{3/2}\int_{-r_{0}}^{t}(-t')^{-3}\mu_{m}^{-1/2}(t')\sqrt{\Eb_{\leq|\alpha|+2}(t',\ub)}dt'\\
\lesssim&\delta^{3/2}\int_{-r_{0}}^{t}(-t')^{-3}\mu_{m}^{-b_{|\alpha|+2}-1/2}(t')\sqrt{\widetilde{\Eb}_{\leq|\alpha|+2}(t',\ub)}dt'\\
\lesssim&\delta^{3/2}(-t)^{-2}\mu_{m}^{-b_{|\alpha|+2}-1/2}(t)\sqrt{\widetilde{\Eb}_{\leq|\alpha|+2}(t,\ub)}.
\end{split}
\end{align}
Also by Proposition 6.1, 6.3, the contribution of the second term in $\rho_{0}$ to the $L^{2}(\Sigma_{t}^{\ub})$-norm of $\rho_{\alpha}$ is also bounded by \eqref{rho alpha 4}. In view of the definition of $e$ the contribution of the first term in $\rho_{0}$ to the $L^{2}(\Sigma_{t}^{\ub})$-norm of $\rho_{\alpha}$ is bounded by

\begin{align}\label{rho alpha 1}
\delta^{1/2}(-t)^{-2}\mu_{m}^{-1/2}(t)\sqrt{\Eb_{\leq|\alpha|+2}(t,\ub)}\lesssim\delta^{1/2}(-t)^{-2}\mu_{m}^{-b_{|\alpha|+2}-1/2}(t)\sqrt{\widetilde{\Eb}_{\leq|\alpha|+2}(t,\ub)}.
\end{align}
  In view of the definition of $\alpha'$, the $L^{2}(\Sigma_{t}^{\ub})$-norm of other contributions of $\rho_{0}$ to $\rho_{\alpha}$ is bounded by

\begin{align}\label{rho alpha 3}
\delta^{1/2}(-t)^{-1}\mu_{m}^{-1/2}(t)\sqrt{\Eb_{\leq|\alpha|+2}(t,\ub)}\lesssim\delta^{1/2}(-t)^{-1}\mu_{m}^{-b_{|\alpha|+2}-1/2}(t)\sqrt{\widetilde{\Eb}_{\leq|\alpha|+2}(t,\ub)}.
\end{align}
Integrating the propagation equation \eqref{diff Lb trchibprime modified} we have

\begin{align}\label{L2 chib precise}
\begin{split}
&(-t)\|R^{\alpha}_{i}\tr\chib'\|_{L^{2}(\Sigma_{t}^{\ub})}\lesssim (-t)^{2}\|R^{\alpha}_{i}\tr\chib'(t)\|_{L^{2}([0,\ub]\times\mathbb{S}^{2})}\\
\lesssim&\int_{-r_{0}}^{t}\|\rho_{\alpha}(t')\|_{L^{2}([0,\ub]\times\mathbb{S}^{2})}dt'\lesssim\int_{-r_{0}}^{t}(-t')^{-1}\|\rho_{\alpha}\|_{L^{2}(\Sigma_{t'}^{\ub})}dt'\\
\lesssim&\delta^{1/2}\int_{-r_{0}}^{t}(-t')^{-2}\mu_{m}^{-b_{|\alpha|+2}-1/2}(t')\sqrt{\widetilde{\Eb}_{\leq|\alpha|+2}(t',\ub)}dt'.
\end{split}
\end{align}
Substituting this in \eqref{H0 pre} $|H_{0}|$ is bounded by

\begin{align}\label{H0 temp 1}
\mu_{m}^{-2b_{|\alpha|+2}-1/2}(t)\sqrt{\widetilde{\Eb}_{\leq|\alpha|+2}(t,\ub)}\int_{-r_{0}}^{t}(-t')^{-2}\mu_{m}^{-b_{|\alpha|+2}-1/2}(t')\sqrt{\widetilde{\Eb}_{\leq|\alpha|+2}(t',\ub)}dt'.
\end{align}
As before, we consider the ``shock part $\int_{-r_{0}}^{t_{0}}$" and ``non-shock part $\int_{t_{0}}^{t}$", which are denoted by $H^{S}_{0}$ and $H^{N}_{0}$, separately. 

For $t'\in[-r_{0},t_{0}]$ we have $\mu^{-1}_{m}(t')\leq 10$. Therefore the time integral in the ``non-shock" part is bounded by

\begin{align*}
&\int_{-r_{0}}^{t_{0}}\mu_{m}^{-b_{|\alpha|+2}-1/2}(t')(-t')^{-2}\sqrt{\widetilde{\Eb}_{\leq|\alpha|+2}(t',\ub)}dt'\\
=&\int_{-r_{0}}^{t_{0}}\mu_{m}^{-1}(t')\mu_{m}^{-b_{|\alpha|+2}+1/2}(t')(-t')^{-2}\sqrt{\widetilde{\Eb}_{\leq|\alpha|+2}(t',\ub)}dt'\\
\lesssim&\int_{-r_{0}}^{t_{0}}\mu_{m}^{-b_{|\alpha|+2}+1/2}(t')(-t')^{-2}\sqrt{\widetilde{\Eb}_{\leq|\alpha|+2}(t',\ub)}dt'\\
\lesssim &\int_{-r_{0}}^{t_{0}}(-t')^{-2}\sqrt{\widetilde{\Eb}_{\leq|\alpha|+2}(t',\ub)}dt'\cdot\mu^{-b_{|\alpha|+2}+1/2}_{m}(t).
\end{align*}
Here in the last step we used Lemma \ref{lemma on mu to power a}. Therefore, using Holder's inequality we have
\begin{align*}
|H^{N}_{0}|&\leq C\int_{-r_{0}}^{t_{0}}(-t')^{-2}\sqrt{\widetilde{\Eb}_{\leq|\alpha|+2}(t',\ub)}dt'\cdot \mu^{-2b_{|\alpha|+2}}_{m}(t)\sqrt{\widetilde{\Eb}_{\leq |\alpha|+2}(t,\ub)}\\
&\leq \epsilon\mu^{-2b_{|\alpha|+2}}_{m}(t)\widetilde{\Eb}_{\leq |\alpha|+2}(t,\ub)+C_{\epsilon}\mu^{-2b_{|\alpha|+2}}_{m}(t)\int_{-r_{0}}^{t_{0}}(-t')^{-2}\widetilde{\Eb}_{\leq|\alpha|+2}(t',\ub)dt'
\end{align*}
For the ``shock part", by the monotonicity of $\widetilde{\Eb}_{\leq|\alpha|+2}(t,\ub)$ we have

\begin{align*}
|H^{S}_{0}|&\lesssim\sqrt{\widetilde{\Eb}_{\leq |\alpha|+2}(t,\ub)}\int_{t_{0}}^{t}\mu^{-b_{|\alpha|+2}-1/2}_{m}(t')(-t')^{-2}dt'\cdot\mu_{m}^{-b_{|\alpha|+2}-1/2}(t)\sqrt{\widetilde{\Eb}_{\leq |\alpha|+2}(t,\ub)}\\
&\lesssim\widetilde{\Eb}_{\leq |\alpha|+2}(t,\ub)\mu_{m}^{-b_{|\alpha|+2}-1/2}(t)\int_{t_{0}}^{t}\mu_{m}^{-b_{|\alpha|+2}-1/2}(t')(-t')^{-2}dt'\\
&\lesssim \frac{1}{\big(b_{|\alpha|+2}-1/2\big)}\mu_{m}^{-b_{|\alpha|+2}-1/2}(t)\widetilde{\Eb}_{\leq |\alpha|+2}(t,\ub)\cdot\mu_{m}^{-b_{|\alpha|+2}+1/2}(t)\\&= \frac{1}{\big(b_{|\alpha|+2}-1/2\big)}\mu_{m}^{-2b_{|\alpha|+2}}(t)\widetilde{\Eb}_{\leq |\alpha|+2}(t,\ub).
\end{align*}
We obtain the following estimate for $|H_{0}|$:
\begin{align}\label{estimates for H0}
\begin{split}
|H_{0}|&\leq  \frac{C}{\big(b_{|\alpha|+2}-1/2\big)}\mu_{m}^{-2b_{|\alpha|+2}}(t)\widetilde{\Eb}_{\leq |\alpha|+2}(t,\ub)\\&+C_{\epsilon}\mu^{-2b_{|\alpha|+2}}_{m}(t)\int_{-r_{0}}^{t_{0}}(-t')^{-2}\widetilde{\Eb}_{\leq|\alpha|+2}(t',\ub)dt'+\epsilon\mu^{-2b_{|\alpha|+2}}_{m}(t)\widetilde{\Eb}_{\leq |\alpha|+2}(t,\ub).
\end{split}
\end{align}
Next we consider the spacetime integral:

\begin{align*}
&\int_{W^{t}_{\ub}}(\Lb+\frac{1}{2}\widetilde{\text{tr}}\chib)
\Big[(R^{\alpha+1}_{i}\text{tr}\chib)\cdot
\left(\frac{2(-t')}{\widetilde{\tr}\chib}T\psi\right)\Big]\cdot(R^{\alpha+1}_{i}\psi)
dt'd\ub'd\mu_{\tilde{\slashed{g}}}\\
=&\int_{W^{t}_{\ub}}\big((\Lb
+\widetilde{\text{tr}}\chib)(R^{\alpha+1}_{i}\text{tr}\chib')\big)
\cdot\left(\frac{2(-t')}{\widetilde{\tr}\chib}T\psi\right)\cdot(R^{\alpha+1}_{i}\psi)dt'
d\ub'd\mu_{\tilde{\slashed{g}}}\\
&+\int_{W^{t}_{\ub}}(R^{\alpha+1}_{i}\text{tr}\chib')
\cdot(\Lb+\frac{1}{2}\widetilde{\text{tr}}\chib)\left(\frac{2(-t')}{\widetilde{\tr}\chib}T\psi\right)\cdot(R^{\alpha+1}_{i}\psi)
dt'd\ub'd\mu_{\tilde{\slashed{g}}}\\
&-\int_{W^{t}_{\ub}}(\widetilde{\text{tr}}\chib)(R^{\alpha+1}_{i}\text{tr}\chib')
\cdot\left(\frac{2(-t')}{\widetilde{\tr}\chib}T\psi\right)\cdot(R^{\alpha+1}_{i}\psi)dt'
d\ub'd\mu_{\tilde{\slashed{g}}}:=I+II+III.
\end{align*}
We start with $III$. Using Lemma \ref{lemma integration by parts} we rewrite $III$ as 

\begin{align*}
&\int_{W^{t}_{\ub}}(\widetilde{\text{tr}}\chib)(R^{\alpha}_{i}\text{tr}\chib')
\cdot\left(\frac{2(-t')}{\widetilde{\tr}\chib}T\psi\right)\cdot(R^{\alpha+2}_{i}\psi)dt'
d\ub'd\mu_{\tilde{\slashed{g}}}\\
+&\int_{W^{t}_{\ub}}(R_{i}\widetilde{\tr}\chib)(R_{i}^{\alpha}\tr\chib')\cdot\left(\frac{2(-t')}{\widetilde{\tr}\chib}T\psi\right)\cdot(R^{\alpha+1}_{i}\psi)dt'
d\ub'd\mu_{\tilde{\slashed{g}}}\\
+&\int_{W^{t}_{\ub}}(\widetilde{\text{tr}}\chib)(R^{\alpha}_{i}\text{tr}\chib')
\cdot R_{i}\left(\frac{2(-t')}{\widetilde{\tr}\chib}T\psi\right)\cdot(R^{\alpha+1}_{i}\psi)dt'
d\ub'd\mu_{\tilde{\slashed{g}}}\\
+&\frac{1}{2}\int_{W^{t}_{\ub}}\tr\leftexp{(R_{i})}{\widetilde{\slashed{\pi}}}(\widetilde{\text{tr}}\chib)(R^{\alpha}_{i}\text{tr}\chib')
\cdot\left(\frac{2(-t')}{\widetilde{\tr}\chib}T\psi\right)\cdot(R^{\alpha+1}_{i}\psi)dt'
d\ub'd\mu_{\tilde{\slashed{g}}}\\
=:&III_{1}+III_{2}+III_{3}+III_{4}.
\end{align*}
Compared to $III_{3}$, $III_{2}$ and $III_{4}$ are lower order due to the factor $R_{i}\widetilde{\tr}\chib+\frac{1}{2}\tr\leftexp{(R_{i})}{\widetilde{\slashed{\pi}}}$. While $III_{3}$ is lower order compared to $III_{1}$ since $\psi$ receives less derivatives. So we only need to estimate $III_{1}$. In view of \ref{L2 chib precise} $III_{1}$ can be bounded as

\begin{align}\label{III1}
\begin{split}
&\delta^{-1/2}\int_{-r_{0}}^{t}(-t')\|R^{\alpha}_{i}\tr\chib'\|_{L^{2}(\Sigma_{t'}^{\ub})}\|\ds R^{\alpha+1}_{i}\psi\|_{L^{2}(\Sigma_{t'}^{\ub})}dt'\\
\lesssim&\int_{-r_{0}}^{t}\left(\int_{-r_{0}}^{t'}(-t'')^{-2}\mu^{-b_{|\alpha|+2}-1/2}_{m}(t'')\sqrt{\widetilde{\Eb}_{\leq|\alpha|+2}(t'',\ub)}dt''\right)\cdot\\
&\mu_{m}^{-b_{|\alpha|+2}-1/2}(t')(-t')^{-1}\sqrt{\widetilde{\Eb}_{\leq|\alpha|+2}(t',\ub)}dt'\\
\lesssim&\int_{-r_{0}}^{t}(-t')^{-2}\mu_{m}^{-2b_{|\alpha|+2}-1}(t')\widetilde{\Eb}_{\leq|\alpha|+2}(t',\ub)dt'\\
\lesssim&\frac{1}{2b_{|\alpha|+2}}\mu_{m}^{-2b_{|\alpha|+2}}(t)\widetilde{\Eb}_{\leq|\alpha|+2}(t,\ub)+\mu_{m}^{-2b_{|\alpha|+2}}(t)\int_{-r_{0}}^{t_{0}}(-t')^{-2}\widetilde{\Eb}_{\leq|\alpha|+2}(t',\ub)dt'.
\end{split}
\end{align}
As before, here we split the time integral into the ``shock" and ``non-shock" parts and use Lemma \ref{lemma on mu to power a}.  

Now let us move to $II$. Note that the factor involving $T\psi$ can be rewritten as

\begin{align*}
\frac{2(-t')}{\widetilde{\tr}\chib}(\Lb+\frac{1}{2}\widetilde{\tr}\chib)(T\psi)-\frac{2(-t')}{(\widetilde{\tr}\chib)^{2}}\Lb(\widetilde{\tr}\chib)T\psi:
=T_{1}+T_{2}.
\end{align*}
By the equation \eqref{Structure Equation Lb chibAB nonsingular}, the contribution of $T_{2}$ to $II$ is lower order with respect to $\delta$ compared to $III$. On the other hand, the equation \eqref{box psi_0 in null frame} implies the pointwise estimate

\begin{align}\label{Improved LbTpsi}
\|(\Lb+\frac{1}{2}\widetilde{\tr}\chib)(T\psi)\|_{L^{\infty}(\Sigma_{t}^{\ub})}\lesssim\delta^{1/2}(-t)^{-3},
\end{align}
which shows that the contribution from $T_{1}$ to $II$ is also lower order with respect to $\delta$ and $t$ compared to $III$.

For $I$, we first note that

\begin{align}\label{Rewriting I}
(\Lb+\widetilde{\text{tr}}\chib)(R^{\alpha+1}_{i}\text{tr}\chib')=R_{i}(\Lb+\widetilde{\text{tr}}\chib)
(R^{\alpha}_{i}\text{tr}\chib')
+\leftexp{(R_{i})}{\Zb}R^{\alpha}_{i}\text{tr}\chib'-R_{i}(\text{tr}\chib')R^{\alpha}_{i}\text{tr}\chib'+\text{l.o.t.}
\end{align}
The lower order term above is lower order compared to the second term above. By the pointwise estimate for $\leftexp{(R_{i})}{\Zb}$ and Proposition 6.1, the contributions of the second and the third term are lower order with respect to $\delta$ and $t$ compared to $III$.

For the contribution of the first term in \eqref{Rewriting I}, we use Lemma \ref{lemma integration by parts} to write it as

\begin{align}\label{I1}
\begin{split}
&-\int_{W^{t}_{\ub}}\big((\Lb
+\widetilde{\text{tr}}\chib)(R^{\alpha}_{i}\text{tr}\chib')\big)
\cdot R_{i}\left(\frac{2(-t')}{\widetilde{\tr}\chib}T\psi\right)\cdot(R^{\alpha+1}_{i}\psi)dt'
d\ub'd\mu_{\tilde{\slashed{g}}}\\
&-\int_{W^{t}_{\ub}}\big((\Lb
+\widetilde{\text{tr}}\chib)(R^{\alpha}_{i}\text{tr}\chib')\big)
\cdot\left(\frac{2(-t')}{\widetilde{\tr}\chib}T\psi\right)\cdot(R^{\alpha+2}_{i}\psi)dt'
d\ub'd\mu_{\tilde{\slashed{g}}}\\
&-\int_{W^{t}_{\ub}}\tr\leftexp{(R_{i})}{\widetilde{\slashed{\pi}}}\big((\Lb
+\widetilde{\text{tr}}\chib)(R^{\alpha}_{i}\text{tr}\chib')\big)
\cdot\left(\frac{2(-t')}{\widetilde{\tr}\chib}T\psi\right)\cdot(R^{\alpha+1}_{i}\psi)dt'
d\ub'd\mu_{\tilde{\slashed{g}}}\\
=&-I_{11}-I_{12}-I_{13}.
\end{split}
\end{align}
Again, due to the pointwise estimates for $R_{i}\tr\chib$ and $\tr\leftexp{(R_{i})}{\widetilde{\slashed{\pi}}}$, $I_{11}$ and $I_{13}$ are lower order with respect to $\delta$ and $t$ compared to $I_{12}$. By the equation \eqref{diff Lb trchibprime},

\begin{align*}
(\Lb
+\widetilde{\text{tr}}\chib)(R^{\alpha}_{i}\text{tr}\chib')=\left(\Lb+\frac{2}{t-\ub}\right)(R^{\alpha}_{i}\tr\chib')+\lot=\frac{\rho_{\alpha}}{(t-\ub)^{2}}+\lot.
\end{align*}
Here $\lot$ is bounded as

\begin{align}\label{I12 temp 1}
\|\lot\|_{L^{2}(\Sigma_{t'}^{\ub})}\lesssim\delta(-t')^{-3}\|R^{\alpha}_{i}\tr\chib'\|_{L^{2}(\Sigma_{t'}^{\ub})},
\end{align}
whose contribution to $I_{12}$ is lower order compared to $III_{1}$. In view of \eqref{rho alpha 4}-\eqref{rho alpha 3},

\begin{align}\label{I12 temp 2}
\left\|\frac{\rho_{\alpha}}{(t-\ub)^{2}}\right\|_{L^{2}(\Sigma_{t'}^{\ub})}\lesssim\delta^{1/2}(-t')^{-3}\mu_{m}^{-b_{|\alpha|+2}-1/2}(t')\sqrt{\widetilde{\Eb}_{\leq|\alpha|+2}(t',\ub)}.
\end{align}
Thus the principal contribution in $I_{12}$ is bounded by

\begin{align}\label{I12 principal}
\begin{split}
&\int_{-r_{0}}^{t}(-t')^{-2}\mu_{m}^{-2b_{|\alpha|+2}-1}(t')\widetilde{\Eb}_{\leq|\alpha|+2}(t',\ub)dt'\\
\lesssim&\frac{1}{2b_{|\alpha|+2}}\mu_{m}^{-2b_{|\alpha|+2}}(t)\widetilde{\Eb}_{\leq|\alpha|+2}(t,\ub)+\mu_{m}^{-2b_{|\alpha|+2}}(t)\int_{-r_{0}}^{t_{0}}(-t')^{-2}\widetilde{\Eb}_{\leq|\alpha|+2}(t',\ub)dt'.
\end{split}
\end{align}
This completes the estimate for the spacetime integral \eqref{spacetime K1 chib temp 1}. \vspace{2mm}

Next we consider the top order optical contribution of the variation $R^{\alpha'}T^{l+1}\psi$, where $|\alpha'|+l+1=|\alpha|+1$, which is the following spacetime integral:
\begin{align}\label{top optical K1 mu}
\delta^{2l+2}\Big|\int_{W^{t}_{\ub}}\frac{2(-t')}{\widetilde{\tr}\chib}(T\psi)\cdot(R^{\alpha'}_{i}T^{l}
\slashed{\Delta}\mu)\cdot((\Lb+\frac{1}{2}\widetilde{\tr}\chib) R^{\alpha'}_{i}
T^{l+1}\psi)dt'd\ub'
d\mu_{\tilde{\slashed{g}}}\Big|
\end{align}

Again, we rewrite the above spacetime integral as:

\begin{align*}
&\delta^{2l+2}\int_{W^{t}_{\ub}}\frac{2(-t')}{\widetilde{\tr}\chib}(T\psi)\cdot(R^{\alpha'}_{i}T^{l}
\slashed{\Delta}\mu)\cdot
\big((\Lb+\widetilde{\text{tr}}\chib)(R^{\alpha'}_{i}T^{l+1}\psi)\big)
dt'd\ub'd\mu_{\tilde{\slashed{g}}}
\\
&-\delta^{2l+2}\int_{W^{t}_{\ub}}\frac{(-t')}{\widetilde{\tr}\chib}(T\psi)\cdot(R^{\alpha'}_{i}T^{l}
\slashed{\Delta}\mu)\cdot
\big(\widetilde{\text{tr}}
\chib(R^{\alpha'}_{i}T^{l+1}\psi)\big)
dt'd\ub'd\mu_{\tilde{\slashed{g}}}
\end{align*}
which is:

\begin{align*}
&\delta^{2l+2}\int_{W^{t}_{\ub}}(\Lb+\widetilde{\text{tr}}\chib)
\Big(\frac{2(-t')}{\widetilde{\tr}\chib}(T\psi)(R^{\alpha'}_{i}T^{l}\slashed{\Delta}\mu)(R^{\alpha'}_{i}T^{l+1}\psi)\Big)dt' d\ub'  d\mu_{\tilde{\slashed{g}}}\\
&-\delta^{2l+2}\int_{W^{t}_{\ub}}\left((\Lb+\frac{1}{2}\widetilde{\tr}\chib)\left(\frac{2(-t')}{\widetilde{\tr}\chib} T\psi\right)\right)(R^{\alpha'}_{i}T^{l}\slashed{\Delta}\mu)(R^{\alpha'}_{i}T^{l+1}\psi)dt'd\ub'
d\mu_{\tilde{\slashed{g}}}\\
&-\delta^{2l+2}\int_{W^{t}_{\ub}}\frac{2(-t')}{\widetilde{\tr}\chib}(T\psi)\big((\Lb+
\widetilde{\text{tr}}\chib)(R^{\alpha'}_{i}T^{l}\slashed{\Delta}\mu)
\big)(R^{\alpha'}_{i}T^{l+1}\psi)dt'd\ub'
d\mu_{\tilde{\slashed{g}}}\\
&+\delta^{2l+2}\int_{W^{t}_{\ub}}\frac{2(-t')}{\widetilde{\tr}\chib}(T\psi)\big(
\widetilde{\text{tr}}\chib(R^{\alpha'}_{i}T^{l}\slashed{\Delta}\mu)
\big)(R^{\alpha'}_{i}T^{l+1}\psi)dt'd\ub'
d\mu_{\tilde{\slashed{g}}}\\
=:&H'+I'+II'+III'.
\end{align*}
As before, the spacetime integral in the first line above can be written as:

\begin{align*}
\delta^{2l+2}\int_{\Sigma_{t}^{\ub}}
\frac{2(-t)}{\widetilde{\tr}\chib}(T\psi)(R^{\alpha'}_{i}T^{l}\slashed{\Delta}\mu)(R^{\alpha'}_{i}T^{l+1}\psi)d\ub'
d\mu_{\tilde{\slashed{g}}}-
\delta^{2l+2}\int_{\Sigma_{-r_{0}}^{\ub}}\frac{2r_{0}}{\widetilde{\tr}\chib}(T\psi)(R^{\alpha'}_{i}T^{l}\slashed{\Delta}\mu)(R^{\alpha'}_{i}T^{l+1}\psi)d\ub'
d\mu_{\tilde{\slashed{g}}}
\end{align*}
The integral on $\Sigma_{t}^{\ub}$ can be written as

\begin{align*}
&-\delta^{2l+2}\int_{\Sigma_{t}^{\ub}}\frac{2(-t)}{\widetilde{\tr}\chib}(T\psi)(R^{\alpha'-1}_{i}T^{l}\slashed{\Delta}
\mu)(R^{\alpha'+1}_{i}T^{l+1}\psi)d\ub'
d\mu_{\tilde{\slashed{g}}}
\\
&-\delta^{2l+2}\int_{\Sigma_{t}^{\ub}}\left(R_{i}\left(\frac{2(-t)}{\widetilde{\tr}\chib}T\psi\right)+\frac{1}{2}\text{tr}\leftexp{(R_{i})}{\tilde{\slashed{\pi}}}\right)(R^{\alpha'-1}_{i}T^{l}\slashed{\Delta}\mu)(R^{\alpha'}_{i}T^{l+1}\psi)d\ub'
d\mu_{\tilde{\slashed{g}}}\\
&:=-H'_{0}-H'_{1}
\end{align*}
In view of the estimates

\begin{align*}
\|\tr\leftexp{(R_{i})}{\widetilde{\slashed{\pi}}}\|_{L^{\infty}(\Sigma_{t}^{\ub})}\lesssim\delta(-t)^{-2},\quad \|R_{i}\tr\chib\|_{L^{\infty}(\Sigma_{t}^{\ub})}\lesssim\delta(-t)^{-3},
\end{align*}
$H'_{1}$ is lower order with respect to $\delta$ and the order of derivatives compared to $H'_{0}$. So we only estimate $H'_{0}$. In view of Proposition 6.4, a preliminary estimate for $H'_{0}$ is given by

\begin{align}\label{H0prime pre}
\begin{split}
|H'_{0}|\lesssim&\delta^{-1/2+l+1}(-t)\|R_{i}^{\alpha'-1}T^{l}\slashed{\Delta}\mu\|_{L^{2}(\Sigma_{t}^{\ub})}\mu_{m}^{-b_{|\alpha|+2}-1/2}(t)\sqrt{\widetilde{\Eb}_{\leq|\alpha|+2}(t,\ub)}\\
\lesssim &\delta\int_{-r_{0}}^{t}(-t')^{-2}\left(\mu_{m}^{-b_{|\alpha|+2}}(t')\sqrt{\widetilde{E}_{\leq|\alpha|+2}(t',\ub)}+(-t')^{-1}\mu_{m}^{-b_{|\alpha|+2}-1/2}(t')\sqrt{\widetilde{\Eb}_{\leq|\alpha|+2}(t',\ub)}\right)dt'\cdot\\
&\mu_{m}^{-b_{|\alpha|+2}-1/2}(t)\sqrt{\widetilde{\Eb}_{\leq|\alpha|+2}(t,\ub)}.
\end{split}
\end{align}
As before we split the time integral as ``shock" and ``non-shock" parts. The contribution from the ``non-shock" part is bounded by

\begin{align}\label{H0prime non shock}
\begin{split}
&\delta\mu^{-b_{|\alpha|+2}+1/2}_{m}(t)\int_{-r_{0}}^{t_{0}}(-t')^{-2}\left(\sqrt{\widetilde{E}_{\leq|\alpha|+2}(t',\ub)}+\sqrt{\widetilde{\Eb}_{\leq|\alpha|+2}(t',\ub)}\right)dt'\cdot\mu_{m}^{-b_{|\alpha|+2}-1/2}(t)\sqrt{\widetilde{\Eb}_{\leq|\alpha|+2}(t,\ub)}\\
\lesssim &\delta\mu_{m}^{-2b_{|\alpha|+2}}(t)\widetilde{\Eb}_{\leq|\alpha|+2}(t,\ub)+\delta\mu_{m}^{-2b_{|\alpha|+2}}(t)\int_{-r_{0}}^{t_{0}}(-t')^{-2}\left(\widetilde{E}_{\leq|\alpha|+2}(t',\ub)+\widetilde{\Eb}_{\leq|\alpha|+2}(t',\ub)\right)dt'.
\end{split}
\end{align}
The contribution from the ``shock" part is bounded by

\begin{align}\label{H0prime shock}
\begin{split}
\frac{\delta\mu_{m}^{-2b_{|\alpha|+2}+1/2}(t)}{b_{|\alpha|+2}-1}\sqrt{\widetilde{E}_{\leq|\alpha|+2}(t,\ub)}\sqrt{\widetilde{\Eb}_{\leq|\alpha|+2}(t,\ub)}+\frac{\delta\mu_{m}^{-2b_{|\alpha|+2}}(t)}{b_{|\alpha|+2}-1/2}\widetilde{\Eb}_{\leq|\alpha|+2}(t,\ub).
\end{split}
\end{align}

Next let us turn to the spacetime integrals $I',II',III'$. We start with $III'$. Using Lemma \ref{lemma integration by parts} $III'$ can be written as

\begin{align}\label{IIIprime pre}
\begin{split}
&-\delta^{2l+2}\int_{W^{t}_{\ub}}\frac{2(-t')}{\widetilde{\tr}\chib}(T\psi)\big(
\widetilde{\text{tr}}\chib(R^{\alpha'-1}_{i}T^{l}\slashed{\Delta}\mu)
\big)(R^{\alpha'+1}_{i}T^{l+1}\psi)dt'd\ub'
d\mu_{\tilde{\slashed{g}}}\\
&-\delta^{2l+2}\int_{W^{t}_{\ub}}\frac{2(-t')}{\widetilde{\tr}\chib}(T\psi)\big(
(R_{i}\widetilde{\text{tr}}\chib)(R^{\alpha'-1}_{i}T^{l}\slashed{\Delta}\mu)
\big)(R^{\alpha'}_{i}T^{l+1}\psi)dt'd\ub'
d\mu_{\tilde{\slashed{g}}}\\
&-\delta^{2l+2}\int_{W^{t}_{\ub}}R_{i}\left(\frac{2(-t')}{\widetilde{\tr}\chib}(T\psi)\right)\big(
\widetilde{\text{tr}}\chib(R^{\alpha'-1}_{i}T^{l}\slashed{\Delta}\mu)
\big)(R^{\alpha'}_{i}T^{l+1}\psi)dt'd\ub'
d\mu_{\tilde{\slashed{g}}}\\
&-\delta^{2l+2}\int_{W^{t}_{\ub}}\left(\tr\leftexp{(R_{i})}{\widetilde{\slashed{\pi}}}\right)\frac{(-t')}{\widetilde{\tr}\chib}(T\psi)\big(
\widetilde{\text{tr}}\chib(R^{\alpha'-1}_{i}T^{l}\slashed{\Delta}\mu)
\big)(R^{\alpha'}_{i}T^{l+1}\psi)dt'd\ub'
d\mu_{\tilde{\slashed{g}}}\\
=:&-III'_{1}-III'_{2}-III'_{3}-III'_{4}.
\end{split}
\end{align}
In view of the pointwise estimates for $\tr\leftexp{(R_{i})}{\widetilde{\slashed{\pi}}}$ and $R_{i}\widetilde{\tr}\chib$, $III'_{2}$ and $III'_{4}$ are lower order with respect to $\delta, t$ and the order of derivative compared to $III'_{1}$. Also the pointwise estimate for $R_{i}\widetilde{\tr}\chib$ implies that $III'_{3}$ is lower order with respect to the order of derivative compared to $III'_{1}$. So here we only estimate $III'_{1}$. Using Proposition 6.4, a preliminary estimate for $III'_{1}$ is given by

\begin{align}\label{IIIprime1 pre}
\begin{split}
|III'_{1}|\lesssim&\delta^{1/2}\int_{-r_{0}}^{t}\delta^{l}\|R^{\alpha'-1}_{i}T^{l}\slashed{\Delta}\mu\|_{L^{2}(\Sigma_{t'}^{\ub})}\mu_{m}^{-b_{|\alpha|+2}-1/2}(t')\sqrt{\widetilde{\Eb}_{\leq|\alpha|+2}(t',\ub)}dt'\\
\lesssim&\delta\int_{-r_{0}}^{t}\int_{-r_{0}}^{t'}(-t'')^{-2}\mu_{m}^{-b_{|\alpha|+2}}(t'')\left(\sqrt{\widetilde{E}_{\leq|\alpha|+2}(t'',\ub)}+\mu^{-1/2}_{m}(t'')(-t'')^{-1}\sqrt{\widetilde{\Eb}_{\leq|\alpha|+2}(t'',\ub)}\right)dt''\\
&\cdot\mu^{-b_{|\alpha|+2}-1/2}_{m}(t')(-t')^{-1}\sqrt{\widetilde{\Eb}_{\leq|\alpha|+2}(t',\ub)}dt'\\
\lesssim&\delta\int_{-r_{0}}^{t}\mu_{m}^{-2b_{|\alpha|+2}-1/2}(t')(-t')^{-1}\sqrt{\widetilde{E}_{\leq|\alpha|+2}(t',\ub)}\sqrt{\widetilde{\Eb}_{\leq|\alpha|+2}(t',\ub)}\int_{-r_{0}}^{t'}(-t'')^{-2}dt''dt'\\
&+\delta\int_{-r_{0}}^{t}\mu_{m}^{-2b_{|\alpha|+2}-1}(t')(-t')^{-1}\widetilde{\Eb}_{\leq|\alpha|+2}(t',\ub)\int_{-r_{0}}^{t'}(-t'')^{-3}dt''dt'.
\end{split}
\end{align}
In the last step we used the second part of Lemma \ref{lemma on mu to power a} and the monotonicity of $\widetilde{E}_{\leq|\alpha|+2}(t,\ub)$ and $\widetilde{\Eb}_{\leq|\alpha|+2}(t,\ub)$. Splitting the time integral as ``non-shock" and ``shock" parts, $III'_{1}$ is bounded by

\begin{align}\label{IIIprime1 temp 1}
\begin{split}
&\delta\int_{-r_{0}}^{t}\mu_{m}^{-2b_{|\alpha|+2}-1/2}(t')(-t')^{-2}\sqrt{\widetilde{E}_{\leq|\alpha|+2}(t',\ub)}\sqrt{\widetilde{\Eb}_{\leq|\alpha|+2}(t',\ub)}dt'\\
&+\delta\int_{-r_{0}}^{t}\mu_{m}^{-2b_{|\alpha|+2}-1}(t')(-t')^{-2}\widetilde{\Eb}_{\leq|\alpha|+2}(t',\ub)dt'\\
\lesssim&\delta\mu_{m}^{-2b_{|\alpha|+2}}(t)\int_{-r_{0}}^{t_{0}}(-t')^{-2}\left(\widetilde{E}_{\leq|\alpha|+2}(t',\ub)+\widetilde{\Eb}_{\leq|\alpha|+2}(t',\ub)\right)dt'\\
&+\frac{\delta\mu_{m}^{-2b_{|\alpha|+2}+1/2}(t)}{2b_{|\alpha|+2}-1/2}\widetilde{E}_{\leq|\alpha|+2}(t,\ub)+\frac{\delta\mu^{-2b_{|\alpha|+2}}_{m}(t)}{2b_{|\alpha|+2}}\widetilde{\Eb}_{\leq|\alpha|+2}(t,\ub).
\end{split}
\end{align}
In view of \eqref{Improved LbTpsi} and the propagation equation for $\tr\chib$, $I'$ is lower order with respect to $t'$ compared to $III'$. Finally we consider the estimate for $II'$. Regarding to the contribution from the factor $(\Lb+\widetilde{\tr}\chib)R^{\alpha'}T^{l}\slashed{\Delta}\mu$, we rewrite systematically

\begin{align*}
R^{\alpha'}_{i}T^{l}\slashed{\Delta}\mu=
\slashed{\Delta}R^{\alpha'}_{i}T^{l}\mu+\slashed{D}^{2} R^{\alpha'-1}_{i}T^{l}\mu+\slashed{D}^{2}R^{\alpha'}_{i}T^{l-1}\mu
+\mathcal{O}^{\leq2}_{0,1}\left(\ds R^{\alpha'-1}_{i}T^{l}\mu+\ds R^{\alpha'}_{i}T^{l-1}\mu\right).
\end{align*}
in view of \eqref{commutator slashed Laplace}. Here we abuse the notation by using $\ds$ and $\slashed{D}^{2}$ to denote $\dfrac{\partial}{\partial\theta^{A}}$ and $\dfrac{\partial^{2}}{\partial\theta^{A}\partial\theta^{B}}$ respectively. Obviously the contribution of the last term above is lower order compared to the other three terms. While the contribution of the third term above is lower order with respect to $\delta$ compared to the first two terms. The contribution of the second term is

\begin{align}\label{IIprime temp 1}
-\delta^{2l+2}\int_{W^{t}_{\ub}}\frac{2(-t')}{\widetilde{\tr}\chib}(T\psi)\left((\Lb+\widetilde{\tr}\chib)(\slashed{D}^{2} R^{\alpha'-1}_{i}T^{l}\mu)\right)(R^{\alpha'}_{i}T^{l+1}\psi)dt'd\ub'd\mu_{\widetilde{\slashed{g}}}.
\end{align}
In particular the contribution from the term involving ``$\widetilde{\tr}\chib$" is lower order compared to $III'_{1}$, which has been estimated. For the contribution from the term involving ``$\Lb$", the propagation equation for $\mu$ implies

\begin{align}\label{commuted Lb mu}
\Lb(\slashed{D}^{2}R^{\alpha'-1}_{i}T^{l}\mu)
=\slashed{D}^{2}R^{\alpha'-1}_{i}T^{l}m+
\slashed{D}^{2}R^{\alpha'-1}_{i}T^{l}(\mu e)+\leftexp{(R_{i})}{\Zb}\ds R_{i}^{\alpha'-1}T^{l}\mu
+\Lambda\slashed{D}^{2}R^{\alpha'-1}T^{l-1}\mu.
\end{align}
The contributions of the last two terms are lower order compared to $III'_{1}$ and so are the terms in the second term for which $e$ and its derivatives can be bounded in $L^{\infty}$. The contributions from the first term is bounded by

\begin{align}\label{IIprime temp 2}
\begin{split}
&\left|\delta^{2l+2}\int_{W^{t}_{\ub}}\frac{2(-t')}{\widetilde{\tr}\chib}(T\psi)(\slashed{D}^{2}R^{\alpha'-1}_{i}T^{l}m)(R^{\alpha'}_{i}T^{l+1}\psi)dt'd\ub'
d\mu_{\widetilde{\slashed{g}}}\right|\\
\lesssim& \delta\int_{-r_{0}}^{t}(-t')^{-2}\mu^{-b_{|\alpha|+2}}_{m}(t')\sqrt{\widetilde{E}_{\leq|\alpha|+2}(t',\ub)}\mu_{m}^{-b_{|\alpha|+2}-1/2}(t')\sqrt{\widetilde{\Eb}_{\leq|\alpha|+1}(t',\ub)}dt',
\end{split}
\end{align}
which enjoys the same estimate as \eqref{IIIprime1 temp 1}. The contribution from the terms in which the derivatives of $e$ are bounded in $L^{2}$ is bounded by (Here we only consider the case that all the derivatives fall on $e$, which is the principal term.)

\begin{align}\label{IIprime temp 3}
\begin{split}
&\left|\delta^{2l+2}\int_{W^{t}_{\ub}}\frac{2(-t')}{\widetilde{\tr}\chib}(T\psi)(\mu\slashed{D}^{2}R^{\alpha'-1}_{i}T^{l}e)(R^{\alpha'}_{i}T^{l+1}\psi)dt'd\ub'
d\mu_{\widetilde{\slashed{g}}}\right|\\
\lesssim& \delta\int_{-r_{0}}^{t}(-t')^{-2}\mu^{-b_{|\alpha|+2}-1/2}_{m}(t')\sqrt{\widetilde{\Eb}_{\leq|\alpha|+2}(t',\ub)}\mu_{m}^{-b_{|\alpha|+2}-1/2}(t')\sqrt{\widetilde{\Eb}_{\leq|\alpha|+1}(t',\ub)}dt',
\end{split}
\end{align}
which also enjoys the same estimate as \eqref{IIIprime1 temp 1}.

Finally we consider the contribution from the factor $(\Lb+\widetilde{\tr}\chib)\left(\slashed{\Delta}
R^{\alpha'}_{i}T^{l}\mu\right)$ to $II'$. Let us denote by $\widetilde{\slashed{\text{div}}}$ the divergence operator on $S_{t,\ub}$ with respect to $\widetilde{\slashed{g}}$. A direct computation implies

\begin{align}\label{commutator div}
\begin{split}
(\Lb+\widetilde{\tr}\chib)\left(\slashed
{\Delta}_{\widetilde{\slashed{g}}}
R^{\alpha'}_{i}T^{l}\mu\right)=(\Lb+\widetilde{\tr}\chib)
\left(\widetilde{\slashed{\text{div}}}\left((\slashed{\nabla}_{\widetilde{\slashed{g}}}R^{\alpha'}_{i}T^{l}\mu)
\right)=\widetilde{\slashed{\text{div}}}(\Lb+\widetilde{\tr}\chib)
\left(\slashed{\nabla}_{\widetilde{\slashed{g}}}R^{\alpha'}_{i}T^{l}
\mu\right)\right).
\end{split}
\end{align}
On the other hand since $\slashed{\Delta}_{\widetilde{\slashed{g}}}=c\slashed{\Delta}$, we have

\begin{align}\label{Lb mu over order}
(\Lb+\widetilde{\tr}\chib)\left(\slashed{\Delta}
R^{\alpha'}_{i}T^{l}\mu\right)=\Lb\left(\frac{1}{c}\right)c\slashed{\Delta}R^{\alpha'}_{i}T^{l}\mu+\frac{1}{c}(\Lb+\widetilde{\tr}\chib)\left(\slashed
{\Delta}_{\widetilde{\slashed{g}}}R^{\alpha'}_{i}T^{l}\mu\right).
\end{align}
The contribution of the first term on the right hand side above to $II'$ is 

\begin{align}\label{IIprime temp 4}
-\delta^{2l+2}\int_{W^{t}_{\ub}}\frac{2c(-t')}{\widetilde{\tr}\chib}\Lb\left(\frac{1}{c}\right)(T\psi)(\slashed{\Delta}R^{\alpha'}_{i}T^{l}\mu)(R^{\alpha'}_{i}T^{l+1}\psi)dt'd\ub'd\mu_{\widetilde{\slashed{g}}},
\end{align}
which is lower order with respect to $\delta$ and $t$ compared to $III'$. The contribution of the second term in \eqref{Lb mu over order} is bounded by

\begin{align}\label{IIprime temp 5}
\begin{split}
&-\delta^{2l+2}\int_{W^{t}_{\ub}}\frac{2(-t')}{c\widetilde{\tr}\chib}(T\psi)\widetilde{\slashed{\text{div}}}
\left((\Lb+\widetilde{\tr}\chib)\left(
\slashed{\nabla}_{\widetilde{\slashed{g}}}
R^{\alpha'}_{i}T^{l}\mu\right)\right)(R^{\alpha'}T^{l+1}\psi)dt'd\ub'd\mu_{\widetilde{\slashed{g}}}\\
=&\delta^{2l+2}\int_{W^{t}_{\ub}}\frac{2(-t')}{c\widetilde{\tr}\chib}(T\psi)
\left((\Lb+\widetilde{\tr}\chib)\left(
\slashed{\nabla}_{\widetilde{\slashed{g}}}
R^{\alpha'}_{i}T^{l}\mu\right)\right)(\slashed{\nabla}_{\widetilde{\slashed{g}}}
R^{\alpha'}T^{l+1}\psi)dt'd\ub'd\mu_{\widetilde{\slashed{g}}}\\
&+\delta^{2l+2}\int_{W^{t}_{\ub}}\slashed{\nabla}
_{\widetilde{\slashed{g}}}\left(\frac{2(-t')}{c\widetilde{\tr}\chib}(T\psi)\right)
\left((\Lb+\widetilde{\tr}\chib)\left(
\slashed{\nabla}_{\widetilde{\slashed{g}}}
R^{\alpha'}_{i}T^{l}\mu\right)\right)(
R^{\alpha'}T^{l+1}\psi)dt'd\ub'd\mu_{\widetilde{\slashed{g}}}
\end{split}
\end{align}
In view of the fact $\slashed{\nabla}_{\widetilde{\slashed{g}}}=c\slashed{\nabla}$, the second term on the right hand side above is similar to \eqref{IIprime temp 1} and therefore is also bounded as \eqref{IIIprime1 temp 1}. The only difference between the first term on the right hand side above and \eqref{IIprime temp 1} is that the variation $\slashed{\nabla}_{\widetilde{\slashed{g}}}R^{\alpha'}_{i}T^{l+1}\psi$
is top order, while the variation in \eqref{IIprime temp 1} is one order less. Therefore similar to \eqref{IIprime temp 2} and \eqref{IIprime temp 3}, this contribution is bounded as

\begin{align}\label{IIprime temp 6}
\begin{split}
&\delta\int_{-r_{0}}^{t}(-t')^{-2}\mu^{-b_{|\alpha|+2}}_{m}(t')\sqrt{\widetilde{E}_{\leq|\alpha|+2}(t',\ub)}\mu_{m}^{-b_{|\alpha|+2}-1/2}(t')\sqrt{\widetilde{\Eb}_{\leq|\alpha|+2}(t',\ub)}dt'\\
+&\delta\int_{-r_{0}}^{t}(-t')^{-2}\mu^{-b_{|\alpha|+2}-1/2}_{m}(t')\sqrt{\widetilde{\Eb}_{\leq|\alpha|+2}(t',\ub)}\mu_{m}^{-b_{|\alpha|+2}-1/2}(t')\sqrt{\widetilde{\Eb}_{\leq|\alpha|+2}(t',\ub)}dt',
\end{split}
\end{align}
which is finally bounded as \eqref{IIIprime1 temp 1}. 

\section{Top Order Energy Estimates}
Now we are ready to complete the top order energy estimates, namely, the energy estimates for the variations of order up to $|\alpha|+2$. As we have pointed out, we allow the top order energies to blow up as shock forms. So in this section, we prove that the modified energies $\widetilde{E}_{\leq|\alpha|+2}(t,\ub)$, $\widetilde{\Eb}_{\leq |\alpha|+2}(t,\ub)$ and $\widetilde{\Fb}_{\leq|\alpha|+2}(t,\ub)$, $\widetilde{F}_{\leq |\alpha|+2}(t,\ub)$ are bounded by initial data. Therefore we obtain a rate for the possible blow up of the top order energies.

\subsection{Estimates associated to $K_{1}$}
We start with the energy inequality for $Z^{\alpha'+1}_{i}\psi$ as we obtained in Section 6. Here $Z_{i}$ is any one of $R_{i}$ $Q$ and $T$.
\begin{align*}
\sum_{|\alpha'|\leq |\alpha|}\delta^{2l'}\Big(\Eb[Z^{\alpha'+1}_{i}\psi](t,\ub)+\Fb[Z^{\alpha'+1}_{i}\psi](t,\ub)+K[Z^{\alpha'+1}_{i}\psi](t,\ub)\Big)\\\leq C\sum_{|\alpha'|\leq|\alpha|}\delta^{2l'}\Eb[Z^{\alpha'+1}_{i}\psi](-r_{0},\ub)+C\sum_{|\alpha'|\leq |\alpha|}\delta^{2l'}\int_{W^{t}_{\ub}}c^{-2}
\widetilde{Q}_{1,|\alpha'|+2}
\end{align*}
where $l'$ is the number of $T$s' appearing in the string of $Z^{\alpha'}_{i}$.
In the spacetime integral $\int_{W^{t}_{\ub}}c^{-2}\widetilde{Q}_{1,\alpha'+2}$ we have the contributions from the deformation tensor of $K_{1}$,
which have been investigated in Section 6. Actually, if we choose $N_{\text{top}}$ to be large enough, then we can bound $\|\slashed{d}Z^{\beta}_{i}\mu\|_{L^{\infty}(\Sigma_{t}^{\ub})}$ in terms of initial data by using the same argument as in Section 4.2 for $|\beta|\leq N_{\infty}+1$.

Another contribution of the spacetime integral $\int_{W^{t}_{\ub}}c^{-2}\widetilde{Q}_{1,|\alpha'|+2}$ comes from $\int_{W^{t}_{\ub}}\dfrac{1}{c}\widetilde{\rho}_{|\alpha'|+2}\cdot \Lb Z^{\alpha'+1}_{i}\psi$, namely, the deformation tensor of commutators, which has been studied intensively in the last section. Among these we first consider the lower order optical contributions, which are bounded by (See \eqref{L2 lower optical K1} and \eqref{sigma K1}):

\begin{align}\label{K1 RHS lower order}
\begin{split}
&\delta^{1/2}K_{\leq|\alpha|+2}(t,\ub)+\delta^{-1/2}\int_{0}^{\ub}\Fb_{\leq|\alpha|+2}(t,\ub')d\ub'\\
&+\int_{-r_{0}}^{t}(-t')^{-2}\Eb(t',\ub)dt'+\delta\int_{-r_{0}}^{t}(-t')^{-2}E_{\leq|\alpha|+2}(t',\ub)dt'\\
\lesssim&\mu^{-2b_{|\alpha|+2}}_{m}(t)\Big(\delta^{1/2}\widetilde{K}_{\leq|\alpha|+2}(t,\ub)+\delta^{-1/2}\int_{0}^{\ub}\widetilde{\Fb}_{\leq|\alpha|+2}(t,\ub')d\ub'\\
&+\int_{-r_{0}}^{t}\widetilde{\Eb}_{\leq|\alpha|+2}(t',\ub)dt'+\delta\int_{-r_{0}}^{t}(-t')^{-2}\widetilde{E}_{\leq|\alpha|+2}(t',\ub)dt'\Big).
\end{split}
\end{align}
Here we define the following non-decreasing quantity in $t$:

\begin{align}\label{non decreasing K}
\widetilde{K}_{\leq|\alpha|+2}(t,\ub):=\sup_{t'\in[-r_{0},t]}\{\mu_{m}^{2b_{|\alpha|+2}}(t')K_{\leq|\alpha|+2}(t',\ub)\}
\end{align}
By \eqref{estimates for H0}, \eqref{III1}, \eqref{I12 principal}, \eqref{H0prime non shock}, \eqref{H0prime shock}, \eqref{IIIprime1 temp 1}, the top order optical contributions associated to $K_{1}$ is bounded by (up to a constant)

\begin{align}\label{top optical K1}
\begin{split}
&\frac{1}{b_{|\alpha|+2}-1/2}\mu_{m}^{-2b_{|\alpha|+2}}(t)\widetilde{\Eb}_{\leq|\alpha|+2}(t,\ub)+\epsilon\mu_{m}^{-2b_{|\alpha|+2}}(t)\widetilde{\Eb}_{\leq|\alpha|+2}(t,\ub)\\
&+C_{\epsilon}\mu^{-2b_{|\alpha|+2}}_{m}(t)\int_{-r_{0}}^{t}(-t')^{-2}\widetilde{\Eb}_{\leq|\alpha|+2}(t',\ub)dt'+\delta\mu_{m}^{-2b_{|\alpha|+2}}(t)\widetilde{\Eb}_{\leq|\alpha|+2}(t,\ub)\\
&+\delta\mu^{-2b_{|\alpha|+2}}_{m}(t)\int_{-r_{0}}^{t}(-t')^{-2}\left(\widetilde{E}_{\leq|\alpha|+2}(t',\ub)+\widetilde{\Eb}_{\leq|\alpha|+2}(t',\ub)\right)dt'\\
&+\frac{\delta\mu^{-2b_{|\alpha|+2}+1/2}_{m}(t)}{b_{|\alpha|+2}-1}\sqrt{\widetilde{E}_{\leq|\alpha|+2}(t,\ub)}\sqrt{\widetilde{\Eb}_{\leq|\alpha|+2}(t,\ub)}\\
&+\frac{\delta\mu^{-2b_{|\alpha|+2}+1/2}_{m}(t)}{2b_{|\alpha|+2}-1/2}\widetilde{E}_{\leq|\alpha|+2}(t,\ub).
\end{split}
\end{align}
Substituting these contributions into the energy inequality, and use the fact that $\mu_{m}(t)\leq 1$, we obtain:

\begin{align*}
&\mu_{m}^{2b_{|\alpha|+2}}(t)\sum_{|\alpha'|\leq |\alpha|}\delta^{2l'}\Big(\Eb[Z^{\alpha'+1}_{i}\psi](t,\ub)+\Fb[Z^{\alpha'+1}_{i}\psi](t,\ub)+K[Z^{\alpha'+1}_{i}\psi](t,\ub)\Big)\\
\lesssim& \mu^{2b_{|\alpha|+2}}_{m}(t)\sum_{|\alpha'|\leq|\alpha|}\delta^{2l'}\Eb[Z^{\alpha'+1}_{i}\psi](-r_{0},\ub)+\frac{1}{b_{|\alpha|+2}-1/2}\widetilde{\Eb}_{\leq|\alpha|+2}(t,\ub)\\
&+\epsilon\widetilde{\Eb}_{\leq|\alpha|+2}(t,\ub)+C_{\epsilon}\int_{-r_{0}}^{t}(-t')^{-2}\widetilde{\Eb}_{\leq|\alpha|+2}(t',\ub)dt'+\delta\widetilde{\Eb}_{\leq|\alpha|+2}(t,\ub)\\
&+\delta\int_{-r_{0}}^{t}(-t')^{-2}\left(\widetilde{E}_{\leq|\alpha|+2}(t',\ub)+\widetilde{\Eb}_{\leq|\alpha|+2}(t',\ub)\right)dt'\\
&+\frac{\delta}{b_{|\alpha|+2}-1}\sqrt{\widetilde{E}_{\leq|\alpha|+2}(t,\ub)}\sqrt{\widetilde{\Eb}_{\leq|\alpha|+2}(t,\ub)}+\frac{\delta}{2b_{|\alpha|+2}-1/2}\widetilde{E}_{\leq|\alpha|+2}(t,\ub).
\end{align*}

Now the right hand side of the above inequality is non-decreasing in $t$, so the above inequality is also valid if we replace ``$t$" by any $t'\in[-r_{0},t]$ on the left hand side:
\begin{align*}
&\mu_{m}^{2b_{|\alpha|+2}}(t')\sum_{|\alpha'|\leq |\alpha|}\Big(\Eb[Z^{\alpha'+1}_{i}\psi](t',\ub)+\Fb[Z^{\alpha'+1}_{i}\psi](t',\ub)+K[Z^{\alpha'+1}_{i}\psi](t',\ub)\Big)\\
\lesssim& \mu^{2b_{|\alpha|+2}}_{m}(t)\sum_{|\alpha'|\leq|\alpha|}\delta^{2l'}\Eb[Z^{\alpha'+1}_{i}\psi](-r_{0},\ub)+\frac{1}{b_{|\alpha|+2}-1/2}\widetilde{\Eb}_{\leq|\alpha|+2}(t,\ub)\\
&+\epsilon\widetilde{\Eb}_{\leq|\alpha|+2}(t,\ub)+C_{\epsilon}\int_{-r_{0}}^{t}(-t')^{-2}\widetilde{\Eb}_{\leq|\alpha|+2}(t',\ub)dt'+\delta\widetilde{\Eb}_{\leq|\alpha|+2}(t,\ub)\\
&+\delta\int_{-r_{0}}^{t}(-t')^{-2}\left(\widetilde{E}_{\leq|\alpha|+2}(t',\ub)+\widetilde{\Eb}_{\leq|\alpha|+2}(t',\ub)\right)dt'\\
&+\frac{\delta}{b_{|\alpha|+2}-1}\sqrt{\widetilde{E}_{\leq|\alpha|+2}(t,\ub)}\sqrt{\widetilde{\Eb}_{\leq|\alpha|+2}(t,\ub)}+\frac{\delta}{2b_{|\alpha|+2}-1/2}\widetilde{E}_{\leq|\alpha|+2}(t,\ub).
\end{align*}
For each term in the sum on the left hand side of above inequality, we keep it on the left hand side and ignore all the other terms. Then taking supremum of the term we kept with respect to $t'\in[-r_{0},t]$. Repeat this process for all the terms on the left hand side, we finally obtain:
\begin{align*}
&\widetilde{\Eb}_{\leq |\alpha|+2}(t,\ub)+\widetilde{\Fb}_{\leq|\alpha|+2}(t,\ub)+\widetilde{K}_{\leq|\alpha|+2}(t,\ub)\\
\lesssim& \mu^{2b_{|\alpha|+2}}_{m}(t)\sum_{|\alpha'|\leq|\alpha|}\delta^{2l'}\Eb[Z^{\alpha'+1}_{i}\psi](-r_{0},\ub)+\boxed{\frac{1}{b_{|\alpha|+2}-1/2}\widetilde{\Eb}_{\leq|\alpha|+2}(t,\ub)}\\
&+\epsilon\widetilde{\Eb}_{\leq|\alpha|+2}(t,\ub)+C_{\epsilon}\int_{-r_{0}}^{t}(-t')^{-2}\widetilde{\Eb}_{\leq|\alpha|+2}(t',\ub)dt'+\delta\widetilde{\Eb}_{\leq|\alpha|+2}(t,\ub)\\
&+\delta\int_{-r_{0}}^{t}(-t')^{-2}\left(\widetilde{E}_{\leq|\alpha|+2}(t',\ub)+\widetilde{\Eb}_{\leq|\alpha|+2}(t',\ub)\right)dt'\\
&+\frac{\delta}{b_{|\alpha|+2}-1}\sqrt{\widetilde{E}_{\leq|\alpha|+2}(t,\ub)}\sqrt{\widetilde{\Eb}_{\leq|\alpha|+2}(t,\ub)}+\frac{\delta}{2b_{|\alpha|+2}-1/2}\widetilde{E}_{\leq|\alpha|+2}(t,\ub).
\end{align*}
The control on the boxed term relies on Remark \ref{a dependence}--since $\frac{C}{b_{|\alpha|+2}}$ is suitably small, the boxed term can be absorbed by the left hand side. So if $\delta$ and $\epsilon$ are appropriately small, we obtain

\begin{align}\label{Eb temp 1}
\begin{split}
&\widetilde{\Eb}_{\leq |\alpha|+2}(t,\ub)+\widetilde{\Fb}_{\leq|\alpha|+2}(t,\ub)+\widetilde{K}_{\leq|\alpha|+2}(t,\ub)\\
\lesssim&\mu^{2b_{|\alpha|+2}}_{m}(t)\sum_{|\alpha'|\leq|\alpha|}\delta^{2l'}\Eb[Z^{\alpha'+1}_{i}\psi](-r_{0},\ub)+\frac{\delta}{b_{|\alpha|+2}-1}\widetilde{E}_{\leq|\alpha|+2}(t,\ub)\\
&+\int_{-r_{0}}^{t}(-t')^{-2}\left(\delta\widetilde{E}_{\leq|\alpha|+2}(t',\ub)+\widetilde{\Eb}_{\leq|\alpha|+2}(t',\ub)\right)dt',
\end{split}
\end{align}
which implies, by Gronwall,

\begin{align}\label{Eb temp 2}
\begin{split}
&\widetilde{\Eb}_{\leq |\alpha|+2}(t,\ub)+\widetilde{\Fb}_{\leq|\alpha|+2}(t,\ub)+\widetilde{K}_{\leq|\alpha|+2}(t,\ub)\\
\lesssim&\mu^{2b_{|\alpha|+2}}_{m}(t)\sum_{|\alpha'|\leq|\alpha|}\delta^{2l'}\Eb[Z^{\alpha'+1}_{i}\psi](-r_{0},\ub)\\
&+\frac{\delta}{b_{|\alpha|+2}-1}\widetilde{E}_{\leq|\alpha|+2}(t,\ub)+\int_{-r_{0}}^{t}(-t')^{-2}\delta\widetilde{E}_{\leq|\alpha|+2}(t',\ub)dt'.
\end{split}
\end{align}

\subsection{Estimates associated to $K_{0}$}
Now we turn to the top order energy estimates for $K_{0}$. We start with the energy identity for $Z^{\alpha'+1}_{i}\psi$, where $Z_{i}$ is any one of $R_{i}$, $Q$ and $T$:
\begin{align*}
&\sum_{|\alpha'|\leq|\alpha|}\delta^{2l'}\left(E[Z^{\alpha'+1}_{i}\psi](t,\ub)+F[Z^{\alpha'+1}_{i}\psi](t,\ub)\right)\\&\lesssim \sum_{|\alpha'|\leq |\alpha|}\delta^{2l'}E[Z^{\alpha'+1}_{i}\psi](-r_{0},\ub)+\sum_{|\alpha'|\leq |\alpha|}\int_{W^{t}_{\ub}}c^{-2}
\widetilde{Q}_{0,|\alpha'|+2}
\end{align*}
Again, $l'$ is the number of $T$s' in the string of $Z^{\alpha'+1}_{i}$.

In the spacetime integral $\int_{W^{t}_{\ub}}c^{-2}\widetilde{Q}_{0,\alpha'+2}$ we have the contributions from the deformation tensor of $K_{0}$, which have been investigated in section 6 and also the contribution of the spacetime integral from $\int_{W^{t}_{\ub}}\dfrac{1}{c}\widetilde{\rho}_{|\alpha'|+2}\cdot LZ^{\alpha'+1}_{i}\psi$, namely, the deformation tensor of commutators, which has been studied intensively in the
last section. We first consider the lower order optical contributions, which are bounded by (See \eqref{L2 lower optical K2}, \eqref{sigma K0} and \eqref{Eb temp 2} and provided that $\delta$ is sufficiently small):

\begin{align}\label{K0 RHS lower order optical}
\begin{split}
&\mu_{m}^{-2b_{|\alpha|+2}}(t)\left(\int_{-r_{0}}^{t}(-t')^{-3/2}\delta^{1/2}\widetilde{E}_{\leq|\alpha|+2}(t',\ub)dt'+\delta^{1/2} \widetilde{K}_{\leq|\alpha|+2}(t,\ub)+\delta^{-1/2}\int_{0}^{\ub}\widetilde{\Fb}_{\leq|\alpha|+2}(t,\ub')d\ub'\right)\\
\lesssim&\Eb_{\leq|\alpha|+2}(-r_{0},\ub)+\mu_{m}^{-2b_{|\alpha|+2}}(t)\left(\int_{-r_{0}}^{t}(-t')^{-3/2}\delta^{1/2}\widetilde{E}_{\leq|\alpha|+2}(t',\ub)dt'+\frac{\delta^{3/2}}{b_{|\alpha|+2}-1}\widetilde{E}_{\leq|\alpha|+2}(t,\ub)\right).
\end{split}
\end{align}
Here we used the following fact: Since the right hand side of \eqref{Eb temp 2} is non-decreasing in $\ub$, $\sup_{\ub'\in[0,\ub]}\{\widetilde{\Fb}_{\leq|\alpha|+2}(t,\ub')\}$ is also bounded by the right hand side of \eqref{Eb temp 2}.

By \eqref{top K0 chib temp 2 non shock}, \eqref{top K0 chib temp 2 shock}, \eqref{top K0 mu temp 2 nonshock}, \eqref{top K0 mu temp 2 shock} and \eqref{Eb temp 2}, the top order optical contributions are bounded by (up to a constant)

\begin{align}\label{K0 RHS top optical}
\begin{split}
&\frac{\mu^{-2b_{|\alpha|+2}}_{m}(t)}{2b_{|\alpha|+2}}\widetilde{E}_{\leq|\alpha|+2}(t,\ub)+\mu_{m}^{-2b_{|\alpha|+2}}(t)\int_{-r_{0}}^{t}(-t')^{-2}\widetilde{E}_{\leq|\alpha|+2}(t',\ub)dt'.
\end{split}
\end{align}
Substituting \eqref{K0 RHS lower order optical} and \eqref{K0 RHS top optical} into the energy inequality we obtain

\begin{align}\label{E temp 1}
\begin{split}
&\mu_{m}^{2b_{|\alpha|+2}}(t)E_{\leq|\alpha|+2}(t,\ub)+\mu_{m}^{2b_{|\alpha|+2}}(t)F_{\leq|\alpha|+2}(t,\ub)\\
\lesssim&E_{\leq|\alpha|+2}(-r_{0},\ub)+\Eb_{\leq|\alpha|+2}(-r_{0},\ub)\\
&+\frac{1}{2b_{|\alpha|+2}}\widetilde{E}_{\leq|\alpha|+2}(t,\ub)+\int_{-r_{0}}^{t}(-t')^{-3/2}\widetilde{E}_{\leq|\alpha|+2}(t',\ub)dt',
\end{split}
\end{align}
which, using the fact that the right hand side above is non-decreasing in $t$, implies

\begin{align}\label{E temp 2}
\begin{split}
&\widetilde{E}_{\leq|\alpha|+2}(t,\ub)+\widetilde{F}_{\leq|\alpha|+2}(t,\ub)\\
\lesssim&E_{\leq|\alpha|+2}(-r_{0},\ub)+\Eb_{\leq|\alpha|+2}(-r_{0},\ub)\\
&+\frac{1}{2b_{|\alpha|+2}}\widetilde{E}_{\leq|\alpha|+2}(t,\ub)+\int_{-r_{0}}^{t}(-t')^{-3/2}\widetilde{E}_{\leq|\alpha|+2}(t',\ub)dt'.
\end{split}
\end{align}
Using Gronwall and taking $b_{|\alpha|+2}$ large enough implies

\begin{align}\label{E temp 3}
\widetilde{E}_{\leq|\alpha|+2}(t,\ub)+\widetilde{F}_{\leq|\alpha|+2}(t,\ub)
\lesssim E_{\leq|\alpha|+2}(-r_{0},\ub)+\Eb_{\leq|\alpha|+2}(-r_{0},\ub).
\end{align}
Substituting this into \eqref{Eb temp 2} we obtain

\begin{align}\label{Eb temp 3}
\begin{split}
&\Ebt_{\leq|\alpha|+2}(t,\ub)+\widetilde{\Fb}_{\leq|\alpha|+2}(t,\ub)+\widetilde{K}_{\leq|\alpha|+2}(t,\ub)\\
\lesssim&\Eb_{\leq|\alpha|+2}(-r_{0},\ub)+\delta E_{\leq|\alpha|+2}(-r_{0},\ub).
\end{split}
\end{align}
This completes the top order energy estimates. If we denote the initial energies by 

\begin{align}\label{initial top order energies}
\mathcal{D}^{\ub}_{|\alpha|+2}:=E_{\leq|\alpha|+2}(-r_{0},\ub)+\delta^{-1}\Eb_{\leq|\alpha|+2}(-r_{0},\ub),
\end{align}
then the top order energy estimates can be summarized as

\begin{align}\label{top order energy estimate final}
\begin{split}
\Ebt_{\leq|\alpha|+2}(t,\ub)+\widetilde{\Fb}_{\leq|\alpha|+2}(t,\ub)+\widetilde{K}_{\leq|\alpha|+2}(t,\ub)\lesssim&\delta\mathcal{D}^{\ub}_{|\alpha|+2}\\
\widetilde{E}_{\leq|\alpha|+2}(t,\ub)+\widetilde{F}_{\leq|\alpha|+2}(t,\ub)\lesssim&\mathcal{D}^{\ub}_{|\alpha|+2}.
\end{split}
\end{align}

\section{Descent Scheme}
In the previous section, we have shown that the modified energies $\widetilde{E}_{\leq|\alpha|+2}(t), \widetilde{\Eb}_{\leq |\alpha|+2}(t)$ for the top order variations are bounded by the initial energies $\mathcal{D}^{\ub}_{|\alpha|+2}$. According to the definition, the modified energies go to zero when $\mu_{m}(t)$ goes to zero. This means the energy estimates obtained in the last section are not sufficient for us to close the argument when shock forms. However, based on those estimates, we shall show in this section, that if the order of derivative decreases, the power of $\mu_{m}(t)$ needed in the definition of modified energies also decreases. The key point is that after several steps, this power could be zero and the energies which do not go to zero as shock forms can be bounded.

\subsection{Next-to-top order error estimates}
We first investigate the estimates associated to $K_{1}$. To improve the energy estimates for the next-to-the-top variations, we consider the spacetime integral (Keep in mind that the top order quantities are of order $|\alpha|+2$):
\begin{align}\label{K1 next to top spacetime}
\begin{split}
&\left|\int_{W^{t}_{\ub}}\frac{1}{c}\frac{2(-t)}{\widetilde{\tr}\chib}(T\psi)\cdot(Z^{\alpha}_{i}\text{tr}\chib')\cdot\left(\Lb Z^{\alpha}_{i}\psi+\frac{1}{2}\widetilde{\tr}\chib Z_{i}^{\alpha}\psi\right)dt'd\ub'd\mu_{\tilde{g}}\right|\\
\lesssim&\delta^{-1/2}\int_{W^{t}_{\ub}}(-t')|Z^{\alpha}_{i}\text{tr}\chib'|\left|\Lb Z^{\alpha}_{i}\psi+\frac{1}{2}\widetilde{\tr}\chib Z_{i}^{\alpha}\psi\right|dt'd\ub'd\mu_{\tilde{g}}\\
\lesssim&\delta^{-1/2}\Big(\int_{W^{t}_{\ub}}|Z^{\alpha}_{i}\text{tr}\chib'|^{2}dt'd\ub'
d\mu_{\tilde{\slashed{g}}}\Big)^{1/2}
\cdot\left(\int_{W^{t}_{\ub}}(-t')^{2}\left(\Lb Z^{\alpha}_{i}\psi+\frac{1}{2}\widetilde{\tr}\chib\psi\right)^{2}dt'd\ub'd\mu_{\tilde{\slashed{g}}}
\right)^{1/2}\\
\lesssim&\delta^{-1/2}\left(\int_{-r_{0}}^{t}\|Z^{\alpha}_{i}\text{tr}\chib'\|^{2}_{L^{2}(\Sigma_{t'}^{\ub})}dt'\right)^{1/2}\cdot
\left(\int_{0}^{\ub}\Fb[Z^{\alpha}_{i}\psi](t,\ub')d\ub'\right)^{1/2}
\end{split}
\end{align}
Throughout this subsection, $Z_{i}$ is either $R_{i}$ or $Q$.

By Proposition \ref{Proposition lower order L2},

\begin{align}\label{trchib descent}
\begin{split}
\|Z^{\alpha}_{i}\text{tr}\chib'\|_{L^{2}(\Sigma_{t}^{\ub})}&\lesssim \delta^{1/2}\int_{-r_{0}}^{t}(-t')^{-3}\mu^{-1/2}_{m}(t')\sqrt{\underline{E}_{|\alpha|+2}(t',\ub)}dt'\\
&\lesssim\delta^{1/2}\int_{-r_{0}}^{t}(-t')^{-3}\mu^{-1/2-b_{|\alpha|+2}}_{m}(t')\sqrt{\widetilde{\Eb}_{\leq|\alpha|+2}(t',\ub)}dt'\\
&\lesssim\delta^{1/2}\sqrt{\widetilde{\Eb}_{\leq |\alpha|+2}(t,\ub)}\int_{-r_{0}}^{t}\mu_{m}^{-b_{|\alpha|+2}-1/2}(t')(-t')^{-3}dt'\\
&\lesssim \delta^{1/2}(-t)^{-1}\sqrt{\widetilde{\Eb}_{\leq|\alpha|+2}(t,\ub)}\int_{-r_{0}}^{t}\mu_{m}^{-b_{|\alpha|+2}-1/2}(t')(-t')^{-2}dt'\\
&\lesssim\delta^{1/2}\mu^{-b_{|\alpha|+2}+1/2}_{m}(t)(-t)^{-1}\sqrt{\widetilde{\Eb}_{\leq |\alpha|+2}(t,\ub)}.
\end{split}
\end{align}
Then by the top order energy estimates obtained in the last section, the integral in the first factor of \eqref{K1 next to top spacetime} is bounded by (up to a constant):

\begin{align*}
&\delta\int_{-r_{0}}^{t}\mu^{-2b_{|\alpha|+2}+1}_{m}(t')(-t')^{-2}\widetilde{\Eb}_{\leq|\alpha|+2}(t',\ub)dt'\\
\lesssim &\delta\widetilde{\Eb}_{\leq |\alpha|+2}(t,\ub)\int_{-r_{0}}^{t}\mu^{-2b_{|\alpha|+2}+1}_{m}(t')(-t')^{-2}dt'\\
\lesssim&\delta^{2}\mu^{-2b_{|\alpha|+2}+2}_{m}(t)\mathcal{D}^{\ub}_{|\alpha|+2}
\end{align*}
On the other hand, the second factor in \eqref{K1 next to top spacetime} is bounded by:

\begin{align*}
&\int_{0}^{\ub}\Fb[Z^{\alpha}_{i}\psi](t,\ub')d\ub'\leq \mu^{-2b_{|\alpha|+1}}_{m}(t)\int_{0}^{\ub}\sup_{t'\in[-r_{0},t]}\{\mu_{m}^{2b_{|\alpha|+1}}(t')\Fb[Z^{\alpha}_{i}\psi](t',\ub')\}d\ub'
\end{align*}
where $b_{|\alpha|+1}=b_{|\alpha|+2}-1$. Therefore \eqref{K1 next to top spacetime} is bounded by (up to a constant):

\begin{align}\label{K1 next to top estimates}
\begin{split}
&\delta^{1/2}\mu^{-2b_{|\alpha|+1}}_{m}(t)\sqrt{\mathcal{D}^{\ub}_{|\alpha|+2}}\sqrt{\int_{0}^{\ub}
\widetilde{\Fb}_{\leq|\alpha|+1}(t,\ub')d\ub'}\\
\lesssim &\delta^{2}\mu^{-2b_{|\alpha|+1}}_{m}(t)\mathcal{D}^{\ub}_{|\alpha|+2}+C\delta^{-1}\mu_{m}^{-2b_{|\alpha|+1}}(t)\int_{0}^{\ub}\widetilde{\Fb}_{\leq|\alpha|+1}(t,\ub')d\ub'.
\end{split}
\end{align}
Next we consider the spacetime integral

\begin{align}\label{K1T next to top spacetime}
\begin{split}
&\delta^{2l'+2}\left|\int_{W^{t}_{\ub}}\frac{1}{c}\frac{2(-t')}{\widetilde{\tr}\chib}(T\psi)\cdot(Z^{\alpha'}_{i}T^{l'}\slashed{\Delta}\mu)\cdot\left(\Lb Z^{\alpha'}_{i}T^{l'+1}\psi+\frac{1}{2}\widetilde{\tr}\chib Z^{\alpha'}_{i}T^{l'+1}\psi\right)dt'd\ub'd\mu_{\tilde{\slashed{g}}}\right|\\
\lesssim &\delta^{-1/2}\left(\int_{-r_{0}}^{t}\delta^{l'+1}\|Z^{\alpha'}_{i}T^{l'}\slashed{\Delta}\mu\|^{2}_{L^{2}(\Sigma_{t'}^{\ub})}\right)^{1/2}
\left(\int_{0}^{\ub}\delta^{l'+1}
\Fb[Z^{\alpha'}_{i}T^{l'+1}\psi](t,\ub')d\ub'\right)^{1/2}
\end{split}
\end{align}
with $|\alpha'|+l'\leq |\alpha|-1$. By Proposition \ref{Proposition lower order L2 mu},

\begin{align}\label{mu descent}
\begin{split}
&\delta^{l'+1}\|Z_{i}^{\alpha'}T^{l'}\slashed{\Delta}\mu\|_{L^{2}(\Sigma_{t}^{\ub})}\\
\lesssim&\delta^{1/2}(-t)^{-1}\int_{-r_{0}}^{t}(-t')^{-2}\left(\sqrt{E_{\leq|\alpha|+2}(t',\ub)}+\mu_{m}^{-1/2}(t')(-t')^{-1}\sqrt{\underline{E}_{\leq|\alpha|+2}(t',\ub)}\right)dt'\\
\lesssim&\delta^{1/2}(-t)^{-1}\int_{-r_{0}}^{t}(-t')^{-2}\left(\mu^{-b_{|\alpha|+2}}_{m}(t')\sqrt{\widetilde{E}_{\leq|\alpha|+2}(t',\ub)}+\mu_{m}^{-b_{|\alpha|+2}-1/2}(t')\sqrt{\widetilde{\Eb}_{\leq|\alpha|+2}(t',\ub)}\right)dt'\\
\lesssim&\delta^{1/2}(-t)^{-1}\left(\mu^{-b_{|\alpha|+2}+1/2}_{m}(t)\sqrt{\widetilde{\Eb}_{\leq|\alpha|+2}(t,\ub)}+\mu_{m}^{-b_{|\alpha|+2}+1}(t)\sqrt{\widetilde{E}_{\leq|\alpha|+2}(t,\ub)}\right).
\end{split}
\end{align}
Then by the top order energy estimates obtained in the last section, the integral in the first factor of \eqref{K1T next to top spacetime} is bounded by ($\mu_{m}(t)\leq1$):

\begin{align*}
&\delta\int_{-r_{0}}^{t}\mu^{-2b_{|\alpha|+2}+1}_{m}(t')(-t')^{-2}\left(\widetilde{E}_{\leq|\alpha|+2}(t',\ub)+\widetilde{\Eb}_{\leq|\alpha|+2}(t',\ub)\right)dt'\\
\lesssim&\delta\left(\widetilde{E}_{\leq|\alpha|+2}(t)+\widetilde{\Eb}_{\leq|\alpha|+2}(t,\ub)\right)\int_{-r_{0}}^{t}\mu^{-2b_{|\alpha|+2}+1}_{m}(t')(-t')^{-2}dt'\\
\lesssim&\delta\mu^{-2b_{|\alpha+2}+2}_{m}(t)\left(\widetilde{E}_{\leq|\alpha|+2}(t,\ub)+\widetilde{\Eb}_{\leq|\alpha|+2}(t,\ub)\right)\\
\lesssim&\delta\mu^{-2b_{|\alpha+2}+2}_{m}(t)\mathcal{D}^{\ub}_{|\alpha|+2}.
\end{align*}
Again, with $b_{|\alpha|+1}=b_{|\alpha|+2}-1$, the spacetime integral \eqref{K1T next to top spacetime} is bounded by (up to a constant):
\begin{equation}\label{K1T next to top estimate}
\begin{split}
&\mu^{-2b_{|\alpha|+1}}_{m}(t)\sqrt{\mathcal{D}^{\ub}_{|\alpha|+2}}\sqrt{\int_{0}^{\ub}\widetilde{\Fb}_{\leq|\alpha|+1}(t,\ub')d\ub'}\\\notag
\lesssim&\delta\mu_{m}^{-2b_{|\alpha|+1}}(t)\mathcal{D}^{\ub}_{|\alpha|+2}+C\delta^{-1}\mu^{-2b_{|\alpha|+1}}_{m}(t)\int_{0}^{\ub}\widetilde{\Fb}_{\leq|\alpha|+1}(t,\ub')d\ub'.
\end{split}
\end{equation}

We proceed to consider the spacetime error integral associated to $K_{0}$. We first consider the spacetime integral:
\begin{align*}
&\int_{W^{t}_{\ub}}|T\psi||Z^{\alpha}_{i}\text{tr}\chib'||L Z^{\alpha}_{i}\psi|dt'd\ub'd\mu_{\tilde{\slashed{g}}}\\
&\lesssim\delta^{-1/2}\int_{-r_{0}}^{t}(-t')^{-1}\|Z^{\alpha}_{i}\text{tr}\chib'\|_{L^{2}(\Sigma_{t'}^{\ub})}\|LZ^{\alpha}_{i}\psi\|_{L^{2}(\Sigma_{t'}^{\ub})}dt'
\end{align*}
Substituting the estimates:

\begin{align*}
\|Z^{\alpha}_{i}\text{tr}\chib'\|_{L^{2}(\Sigma_{t'}^{\ub})}&\lesssim\delta^{1/2}\mu^{-b_{|\alpha|+2}+1/2}_{m}(t')(-t')^{-1}\sqrt{\widetilde{\Eb}_{\leq|\alpha|+2}(t',\ub)}\\&\lesssim\delta\mu^{-b_{|\alpha|+2}+1/2}_{m}(t)(-t')^{-1}\sqrt{\mathcal{D}^{\ub}_{|\alpha|+2}},\\
\|L Z^{\alpha}_{i}\psi\|_{L^{2}(\Sigma_{t'}^{\ub})}&\lesssim\mu^{-b_{|\alpha|+1}}_{m}(t')\sqrt{\widetilde{E}_{\leq|\alpha|+1}(t',\ub)}
\end{align*}
with $b_{|\alpha|+1}=b_{|\alpha|+2}-1$, and using the fact that $\widetilde{E}_{\leq|\alpha|+1}(t)$ are non-decreasing in $t$, we see that the spacetime integral is bounded by ($\mu_{m}(t)\leq1$):
\begin{align}
\begin{split}
&\delta^{1/2}\sqrt{\mathcal{D}^{\ub}_{|\alpha|+2}}\sqrt{\widetilde{E}_{\leq|\alpha|+1}(t,\ub)}\int_{-r_{0}}^{t}\mu^{-2b_{|\alpha|+1}-1/2}_{m}(t')(-t')^{-2}dt'\\
\lesssim&\delta^{1/2}\mu^{-2b_{|\alpha|+1}+1/2}_{m}(t)\sqrt{\mathcal{D}^{\ub}_{|\alpha|+2}}\sqrt{\widetilde{E}_{\leq|\alpha|+1}(t,\ub)}\\
\lesssim&\mu_{m}^{-2b_{|\alpha|+1}}(t)\mathcal{D}^{\ub}_{|\alpha|+2}+\delta\mu^{-2b_{|\alpha|+1}}_{m}(t)\widetilde{E}_{\leq|\alpha|+1}(t,\ub)
\end{split}
\end{align}
Finally, we consider the spacetime integral:

\begin{align}\label{K0T next to top spacetime}
\begin{split}
&\delta^{2l'+2}\int_{W^{t}_{\ub}}|Z^{\alpha'}_{i}T^{l'}\slashed{\Delta}\mu||T\psi||LZ^{\alpha'}_{i}T^{l'+1}\psi|dt'd\ub'd\mu_{\tilde{\slashed{g}}}\\
\lesssim&\delta^{2l'+2-1/2}\int_{-r_{0}}^{t}(-t')^{-1}\|Z^{\alpha'}_{i}T^{l'}\slashed{\Delta}\mu\|_{L^{2}(\Sigma_{t'}^{\ub})}\|LZ^{\alpha'}_{i}T^{l'+1}\psi\|_{L^{2}(\Sigma_{t'}^{\ub})}dt'
\end{split}
\end{align}
for $|\alpha'|+l'\leq |\alpha|-1$.
Again, substituting the estimates ($\mu_{m}(t)\leq1$):
\begin{align*}
\delta^{l'+1}\|Z^{\alpha'}_{i}T^{l'}\slashed{\Delta}\mu\|_{L^{2}(\Sigma_{t'}^{\ub})}\lesssim\delta^{1/2}\mu^{-b_{|\alpha|+2}+1/2}_{m}(t')(-t')^{-1}\left(\sqrt{\widetilde{E}_{\leq|\alpha|+2}(t',\ub)}+\sqrt{\widetilde{\Eb}_{\leq|\alpha|+2}(t',\ub)}\right)
\end{align*}
with $b_{|\alpha|+1}=b_{|\alpha|+2}-1$, the same argument implies that the spacetime integral is bounded by:
\begin{align}\label{K0T next to top estimates}
\begin{split}
&\sqrt{\mathcal{D}^{\ub}_{|\alpha|+2}}\sqrt{\widetilde{E}_{\leq|\alpha|+1}(t,\ub)}\int_{-r_{0}}^{t}\mu^{-2b_{|\alpha|+1}-1/2}_{m}(t')(-t')^{-2}dt'\\
\lesssim&\mu^{-2b_{|\alpha|+1}+1/2}_{m}(t)\sqrt{\mathcal{D}^{\ub}_{|\alpha|+2,l'}}\sqrt{\widetilde{E}_{\leq|\alpha|+1}(t,\ub)}\\
\leq&C_{\epsilon}\mu^{-2b_{|\alpha|+1}}_{m}(t)\mathcal{D}^{\ub}_{|\alpha|+2}+\epsilon\widetilde{E}_{\leq|\alpha|+1}(t,\ub).
\end{split}
\end{align}
Here $\epsilon$ is a small absolute positive constant.

\subsection{Energy estimates-next to top order}
Throughout this subsection $Z_{i}$ could be $R_{i}, Q$ and $T$. Now we consider the other contributions from the spacetime error integrals associated to $K_{1}$. For the variations $Z^{\alpha'}_{i}\psi$ where $|\alpha'|\leq |\alpha|$, the contributions similar to \eqref{sigma K1} are bounded by (up to a constant)

\begin{align}\label{sigma K1 next to top}
\begin{split}
&\delta^{1/2}K_{\leq|\alpha|+1}(t,\ub)+\delta^{-1/2}\int_{0}^{\ub}\Fb_{\leq|\alpha|+1}(t,\ub')d\ub'\\
&+\int_{-r_{0}}^{t}(-t')^{-2}\Eb_{\leq|\alpha|+1}(t',\ub)dt'+\delta\int_{-r_{0}}^{t}(-t')^{-2}E_{\leq|\alpha|+1}(t',\ub)dt'\\
\lesssim&\delta^{1/2}\mu_{m}^{-2b_{|\alpha|+1}}(t)\widetilde{K}_{\leq|\alpha|+1}(t,\ub)+\delta^{-1/2}\mu_{m}^{-2b_{|\alpha|+1}}(t)\int_{0}^{\ub}\widetilde{\Fb}_{\leq|\alpha|+1}(t,\ub')d\ub'\\
&+\mu_{m}^{-2b_{|\alpha|+1}}(t)\int_{-r_{0}}^{t}(-t')^{-2}\Ebt_{\leq|\alpha|+1}(t',\ub)dt'+\delta\mu_{m}^{-2b_{|\alpha|+1}}(t)\int_{-r_{0}}^{t}(-t')^{-2}\Et_{\leq|\alpha|+1}(t',\ub)dt'.
\end{split}
\end{align}
In view of \eqref{K1 next to top estimates}, \eqref{K1T next to top estimate}, \eqref{sigma K1 next to top} and multiplying $\mu_{m}^{2b_{|\alpha|+1}}(t)$ on both sides of the energy inequality associated to $K_{1}$ for $Z^{\alpha'}_{i}\psi$ with $|\alpha'|\leq |\alpha|$ gives us

\begin{align*}
&\sum_{|\alpha'|\leq |\alpha|}\mu^{2b_{|\alpha|+1}}_{m}(t)\delta^{2l'}\left(\Eb[Z^{\alpha'}_{i}\psi](t,\ub)+\Fb[Z^{\alpha'}_{i}\psi](t,\ub)+K[Z^{\alpha'}_{i}\psi](t,\ub)\right)\\
\lesssim&\sum_{|\alpha'|\leq |\alpha|}\delta^{2l'}\Eb[Z^{\alpha'}_{i}\psi](-r_{0},\ub)+\delta^{-1}\int_{0}^{\ub}\widetilde{\Fb}_{\leq|\alpha|+1}(t,\ub')d\ub'+\int_{-r_{0}}^{t}(-t')^{-2}\Ebt(t',\ub)dt'\\
&+\delta\int_{-r_{0}}^{t}(-t')^{-2}\widetilde{E}_{\leq|\alpha|+1}(t',\ub)dt'
+\delta^{1/2}\widetilde{K}_{\leq|\alpha|+1}(t,\ub)+\delta\mathcal{D}^{\ub}_{|\alpha|+2}
\end{align*}
Arguing as in the previous section, this inequality holds $t$ is replaced by $t'\in[-r_{0},t]$ on the left hand side. Taking the supremum  with respect to $t'\in[-r_{0},t]$ we obtain

\begin{align*}
&\widetilde{\Eb}_{\leq|\alpha|+1}(t,\ub)+\widetilde{\Fb}_{\leq|\alpha|+1}(t,\ub)+\widetilde{K}_{\leq|\alpha|+1}(t,\ub)\\
&\lesssim\delta\mathcal{D}^{\ub}_{|\alpha|+2}
+\delta^{-1}\int_{0}^{\ub}\widetilde{\Fb}_{\leq|\alpha|+1}(t,\ub')d\ub'+\int_{-r_{0}}^{t}(-t')^{-2}\Ebt_{\leq|\alpha|+2}(t',\ub)dt'\\
&+\delta\int_{-r_{0}}^{t}(-t')^{-2}\widetilde{E}_{\leq|\alpha|+1}(t',\ub)dt'+\delta^{1/2}\widetilde{K}_{\leq|\alpha|+1}(t,\ub).
\end{align*}
Choosing $\delta$ sufficiently small and using Gronwall implies 

\begin{align}\label{Next to top K1}
\widetilde{\Eb}_{\leq|\alpha|+1}(t,\ub)+\widetilde{\Fb}_{\leq|\alpha|+1}(t,\ub)+\widetilde{K}_{\leq|\alpha|+1}(t,\ub)\lesssim\delta\mathcal{D}^{\ub}_{|\alpha|+2}+\delta\int_{-r_{0}}^{t}(-t')^{-2}\widetilde{E}_{\leq|\alpha|+1}(t',\ub)dt'.
\end{align}

Next we consider the energy estimates associated to $K_{0}$. We start with the variation $Z^{\alpha'}_{i}\psi$ with $|\alpha'|\leq |\alpha|$. The contributions similar to \eqref{sigma K0} is bounded by (up to a constant)

\begin{align}\label{Next to top K0 pre}
\begin{split}
&\delta^{-1/2}\int_{0}^{\ub}\Fb_{\leq|\alpha|+1}(t,\ub')d\ub'+\delta^{1/2}\int_{-r_{0}}^{t}(-t')^{-2}E_{\leq|\alpha|+1}(t',\ub)dt'+\delta^{1/2}K_{\leq|\alpha|+1}(t,\ub)\\
\lesssim&\delta^{-1/2}\mu_{m}^{-2b_{|\alpha|+1}}(t)\int_{0}^{\ub}\widetilde{\Fb}_{\leq|\alpha|+1}(t,\ub')d\ub'+\delta^{1/2}\mu_{m}^{-2b_{|\alpha|+1}}(t)\int_{-r_{0}}^{t}(-t')^{-2}\Et_{\leq|\alpha|+1}(t',\ub)dt'\\
&+\delta^{1/2}\mu^{-2b_{|\alpha|+1}}_{m}(t)\widetilde{K}_{\leq|\alpha|+1}(t,\ub)\\
\lesssim&\delta^{3/2}\mu_{m}^{-2b_{|\alpha|+1}}(t)\mathcal{D}^{\ub}_{|\alpha|+2}+\delta^{3/2}\mu_{m}^{-2b_{|\alpha|+1}}(t)\int_{-r_{0}}^{t}(-t')^{-2}\Et_{\leq|\alpha|+1}(t',\ub)dt'.
\end{split}
\end{align}
Here in the last step we used \eqref{Next to top K1}. \eqref{Next to top K0 pre} together with \eqref{K0T next to top estimates} gives the following estimate

\begin{align*}
&\sum_{|\alpha'|\leq |\alpha|}\mu_{m}^{2b_{|\alpha|+1}}(t)\delta^{2l'}\left(E[Z^{\alpha'}_{i}\psi](t,\ub)+F[Z^{\alpha'}_{i}\psi](t,\ub)\right)\\
\lesssim&\mathcal{D}^{\ub}_{|\alpha|+2}+\epsilon\widetilde{E}_{\leq|\alpha|+1}(t,\ub)+\delta^{1/2}\int_{-r_{0}}^{t}(-t')^{-2}\widetilde{E}_{\leq|\alpha|+1}(t',\ub)dt',
\end{align*}
which implies

\begin{align*}
&\widetilde{E}_{\leq|\alpha|+1}(t,\ub)+\widetilde{F}_{\leq|\alpha|+1}(t,\ub)\\
\lesssim&\mathcal{D}^{\ub}_{|\alpha|+2}
+\epsilon\widetilde{E}_{\leq|\alpha|+1}(t,\ub)+\delta^{1/2}\int_{-r_{0}}^{t}(-t')^{-2}\widetilde{E}_{\leq|\alpha|+1}(t',\ub)dt'
\end{align*}
Choosing $\epsilon$ sufficiently small and using Gronwall, we finally have:

\begin{align}\label{Next to top K0}
\widetilde{E}_{\leq|\alpha|+1}(t,\ub)+\widetilde{F}_{\leq|\alpha|+1}(t,\ub)\lesssim\mathcal{D}^{\ub}_{|\alpha|+2}
\end{align}
Now substituting \eqref{Next to top K0} to the right hand side of \eqref{Next to top K1}, we have:

\begin{align*}
\widetilde{\Eb}_{\leq|\alpha|+1}(t,\ub)+\widetilde{\Fb}_{\leq|\alpha|+1}(t,\ub)+\widetilde{K}_{\leq|\alpha|+1}(t,\ub)\lesssim \delta\mathcal{D}^{\ub}_{|\alpha|+2}
\end{align*}
Summarizing, we have:

\begin{align}\label{energy estimates next to top}
\begin{split}
\widetilde{\Eb}_{\leq|\alpha|+1}(t,\ub)+\widetilde{\Fb}_{\leq|\alpha|+1}(t,\ub)+\widetilde{K}_{\leq|\alpha|+1}(t,\ub)\lesssim&\delta\mathcal{D}^{\ub}_{|\alpha|+2}\\
\widetilde{E}_{\leq|\alpha|+1}(t,\ub)+\widetilde{F}_{\leq|\alpha|+1}(t,\ub)\lesssim
&\mathcal{D}^{\ub}_{|\alpha|+2}.
\end{split}
\end{align}

\subsection{Descent scheme}
We proceed in this way taking at the $n$th step:
\begin{align*}
b_{|\alpha|+2-n}=b_{|\alpha|+2}-n,\quad b_{|\alpha|+1-n}=b_{|\alpha|+2}-n-1
\end{align*}
in the role of $b_{|\alpha|+2}$ and $b_{|\alpha|+1}$ respectively, the argument beginning in the paragraph containing \eqref{K1 next to top spacetime} and concluding with \eqref{energy estimates next to top} being step $0$. The $n$th step is exactly the same as the $0$th step as above, as long as $b_{|\alpha|+1-n}>0$. If we choose 

\begin{align}\label{b choice}
b_{|\alpha|+2}=[b_{|\alpha|+2}]+\frac{3}{4}
\end{align}
where $[b_{|\alpha|+2}]$ is the integer part of $b_{|\alpha|+2}$, then $b_{|\alpha|+1-n}>0$ is equivalent to $n\leq[b_{|\alpha|+2}]-1$. For each of such $n$, we need to estimate the integrals:
\begin{align*}
\int_{-r_{0}}^{t}\mu_{m}^{-b_{|\alpha|+2-n}-1/2}(t')(-t')^{-2}dt',\quad \int_{-r_{0}}^{t}\mu_{m}^{-2b_{|\alpha|+2-n}+1}(t')(-t')^{-2}dt'
\end{align*}
As in the previous sections, we split the interval $[-r_{0},t]$ into two parts: $t'\in[-r_{0},t_{0}]$ and $t'\in[t_{0},t]$ where $\mu_{m}(t_{0})=\frac{1}{10}$. If $t'\in[-r_{0},t_{0}]$, we have

\begin{align}\label{mu power descent 1}
\begin{split}
\int_{-r_{0}}^{t_{0}}\mu^{-b_{|\alpha|+2-n}-1/2}_{m}(t')(-t')^{-2}dt'&\lesssim\int_{-r_{0}}^{t_{0}}\mu^{-b_{|\alpha|+2-n}+1/2}_{m}(t')(-t')^{-2}dt'\lesssim\mu_{m}^{-b_{|\alpha|+2-n}+1/2}(t)\\
\int_{-r_{0}}^{t_{0}}\mu^{-2b_{|\alpha|+2-n}+1}_{m}(t')(-t')^{-2}dt'&\lesssim
\int_{-r_{0}}^{t_{0}}\mu_{m}^{-2b_{|\alpha|+2-n}+2}(t')(-t')^{-2}dt'\lesssim
\mu_{m}^{-2b_{|\alpha|+2-n}+2}(t).
\end{split}
\end{align}
Here we used the fact that $\mu_{m}(t')\geq\dfrac{1}{10}$ for $t'\in[-r_{0},t_{0}]$ and the second part of Lemma \ref{lemma on mu to power a}. For $t'\in[t_{0},t]$, since 

\begin{align*}
b_{|\alpha|+2-n}=[b_{|\alpha|+2-n}]+\frac{3}{4}\geq 1+\frac{3}{4}=\frac{7}{4},
\end{align*}
which implies

\begin{align*}
\eta_{m}-\frac{1}{2(b_{|\alpha|+2-n}-1/2)}\geq c_{0}>0
\end{align*}
for some absolute constant $c_{0}$. Here $\eta_{m}$ is defined by \eqref{min Lbmu initial}. Therefore the proof of Lemma \ref{lemma on mu to power a} goes through and we have

\begin{align}\label{mu power descent 2}
\begin{split}
\int_{t_{0}}^{t}\mu^{-b_{|\alpha|+2-n}-1/2}_{m}(t')(-t')^{-2}dt'&\lesssim\mu^{-b_{|\alpha|+2-n}+1/2}_{m}(t)\\
\int_{t_{0}}^{t}\mu^{-2b_{|\alpha|+2-n}+1}_{m}(t')(-t')^{-2}dt'&\lesssim\mu^{-2b_{|\alpha|+2-n}+2}_{m}(t).
\end{split}
\end{align}
So indeed, we can repeat the process of $0$th for $n=1,...,[b_{|\alpha|+2}]-1$. Therefore we have the following estimates:

\begin{align}\label{descent energy estimates}
\begin{split}
\widetilde{E}_{\leq|\alpha|+1-n}(t,\ub)+\widetilde{F}_{\leq|\alpha|+1-n}(t,\ub)&\lesssim\mathcal{D}^{\ub}_{|\alpha|+2}\\
\widetilde{\Eb}_{\leq|\alpha|+1-n}(t,\ub)+\widetilde{\Fb}_{\leq|\alpha|+1-n}(t,\ub)+\widetilde{K}_{\leq |\alpha|+1-n}(t,\ub)&\lesssim\delta\mathcal{D}^{\ub}_{|\alpha|+2}.
\end{split}
\end{align}
We now make the final step $n=[b_{|\alpha|+2-n}]$. In this case we have $b_{|\alpha|+2-n}=\frac{3}{4}$. Using the same process as in \eqref{trchib descent} and \eqref{mu descent}, the optical terms are bounded by:

\begin{align}\label{trchib mu descent last step}
\begin{split}
\|Z^{\alpha'}\text{tr}\chib'\|_{L^{2}(\Sigma_{t}^{\ub})}&\lesssim\delta\mu^{-1/4}_{m}(t)(-t)^{-1}\sqrt{\mathcal{D}^{\ub}_{|\alpha|+2}}\quad\text{with}\quad |\alpha'|+2\leq |\alpha|+1-[b_{|\alpha|+2}]\\
\|Z^{\alpha'}_{i}T^{l'}\slashed{\Delta}\mu\|_{L^{2}(\Sigma_{t}^{\ub})}&\lesssim \delta^{1/2}\mu^{-1/4}_{m}(t)(-t)^{-1}\sqrt{\mathcal{D}^{\ub}_{|\alpha|+2}}\quad\text{with}\quad |\alpha'|+l'+2\leq |\alpha|+1-[b_{|\alpha|+2}]
\end{split}
\end{align}
with $Z_{i}=R_{i}$ or $Q$.
As before, in order to bound the corresponding integrals:
\begin{align*}
\int_{-r_{0}}^{t}\|Z^{\alpha'}_{i}\text{tr}\chib'\|^{2}_{L^{2}(\Sigma_{t'}^{\ub})}dt',\quad\int_{-r_{0}}^{t}\|Z^{\alpha'}_{i}T^{l'}\slashed{\Delta}\mu\|^{2}_{L^{2}(\Sigma_{t'}^{\ub})}dt'
\end{align*}
we need to consider the integral:
\begin{align*}
\int_{-r_{0}}^{t}\mu^{-1/2}_{m}(t')(-t')^{-2}dt'\leq \int_{-r_{0}}^{t_{0}}\mu^{-1/2}_{m}(t')(-t')^{-2}dt'+\int_{t_{0}}^{t}\mu^{-1/2}_{m}(t')(-t')^{-2}dt'\quad\text{with}\quad \mu_{m}(t_{0})=\frac{1}{10}
\end{align*}
For the ``non-shock part $\int_{-r_{0}}^{t_{0}}$", since $\mu_{m}(t_{0})\geq\dfrac{1}{10}$,
\begin{align*}
\int_{-r_{0}}^{t_{0}}\mu_{m}^{-1/2}(t')(-t')^{-2}dt'\lesssim 1.
\end{align*}
Let $s$ be such that $t_{0}\leq t'\leq t<s<t^{*}$. Following the same arguments as in \eqref{mu m lower bound temp 1} and \eqref{mu m upper bound temp 1}, we have

\begin{align}\label{mu m lower bound temp 2}
\mu_{m}(t)\geq \mu_{m}(s)+\left(\eta_{m}-O(\delta M^{4})\right)\left(\frac{1}{t}-\frac{1}{s}\right),
\end{align}
and 

\begin{align}\label{mu m upper bound temp 2}
\mu_{m}(t)\leq \mu_{m}(s)+\left(\eta_{m}+O(\delta M^{4})\right)\left(\frac{1}{t}-\frac{1}{s}\right).
\end{align}
If $\delta$ is sufficiently small, a similar argument as deriving \eqref{la temp 1} implies

\begin{align}\label{mu m descent final}
\int_{t_{0}}^{t}\mu_{m}^{-1/2}(t')(-t')^{-2}dt'\lesssim\mu_{m}^{1/2}(t)\lesssim 1.
\end{align}
So we have the following bounds:
\begin{align*}
\Big(\int_{-r_{0}}^{t}\|Z^{\alpha'}\text{tr}\chib'\|^{2}_{L^{2}(\Sigma_{t'}^{\ub})}dt'\Big)^{1/2}&\lesssim\delta\sqrt{\mathcal{D}^{\ub}_{|\alpha|+2}}\quad\text{with}\quad |\alpha'|+2\leq |\alpha|+1-[b_{|\alpha|+2}],\\
\Big(\int_{-r_{0}}^{t}\|Z^{\alpha'}_{i}T^{l'}\slashed{\Delta}\mu\|_{L^{2}(\Sigma_{t'}^{\ub})}dt'\Big)^{1/2}&\lesssim\delta^{1/2}\sqrt{\mathcal{D}^{\ub}_{|\alpha|+2}}\quad\text{with}\quad |\alpha'|+l'+2\leq |\alpha|+1-[b_{|\alpha|+2}].
\end{align*}
Therefore we can set:
\begin{align*}
b_{|\alpha|+1-n}=b_{|\alpha|+1-[b_{|\alpha|+2}]}=0
\end{align*}
in this step. Then we can proceed exactly the same as in the preceding steps. We thus arrive at the estimates:
\begin{align}\label{descent estimates final}
\begin{split}
\widetilde{E}_{\leq|\alpha|+1-[b_{|\alpha|+2}]}(t,\ub)+\widetilde{F}_{\leq|\alpha|+1-[b_{|\alpha|+2}]}(t,\ub)&\lesssim\mathcal{D}^{\ub}_{|\alpha|+2}\\
\widetilde{\Eb}_{\leq|\alpha|+1-[b_{|\alpha|+2}]}(t,\ub)+\widetilde{\Fb}_{\leq|\alpha|+1-[b_{|\alpha|+2}]}(t,\ub)+\widetilde{K}_{\leq |\alpha|+1-[b_{|\alpha|+2}]}(t,\ub)&\lesssim\delta\mathcal{D}^{\ub}_{|\alpha|+2}.
\end{split}
\end{align}
These are the desired estimates, because from the definitions:
\begin{align}\label{definitions for finite energy}
\begin{split}
\widetilde{\Eb}_{\leq|\alpha|+1-[b_{|\alpha|+2}]}(t,\ub):&=\sup_{t'\in[-r_{0},t]}\{\underline{E}_{\leq|\alpha|+1-[b_{|\alpha|+2}]}(t',\ub)\}\\
\widetilde{E}_{\leq|\alpha|+1-[b_{|\alpha|+2}]}(t,\ub):&=\sup_{t'\in[-r_{0},t]}\{E_{\leq|\alpha|+1-[b_{|\alpha|+2}]}(t',\ub)\}\\
\widetilde{\Fb}_{\leq|\alpha|+1-[b_{|\alpha|+2}]}(t,\ub):&=\sup_{t'\in[-r_{0},t]}\{\Fb_{\leq|\alpha|+1-[b_{|\alpha|+2}]}(t',\ub)\}\\
\widetilde{F}_{\leq|\alpha|+1-[b_{|\alpha|+2}]}(t,\ub):&=\sup_{t'\in[-r_{0},t]}\{F_{\leq|\alpha|+1-[b_{|\alpha|+2}]}(t',\ub)\}\\
\widetilde{K}_{\leq|\alpha|+1-[b_{|\alpha|+2}]}(t,\ub):&=\sup_{t'\in[-r_{0},t]}\{K_{\leq|\alpha|+1-[b_{|\alpha|+2}]}(t',\ub)\}
\end{split}
\end{align}
the weight $\mu_{m}(t')$ has been eliminated.
\section{Completion of Proof}
\subsection{Proof of Theorem 3.1}
Let us define:
\begin{align*}
\mathcal{S}_{2}[\phi]:=\int_{S_{t,\ub}}\Big(|\phi|^{2}+|R_{i_{1}}\phi|^{2}+|R_{i_{1}}R_{i_{2}}\phi|^{2}\Big)d\mu_{\slashed{g}}
\end{align*}
And also let us denote by $\mathcal{S}_{n}(t,\ub)$ the integral on $S_{t,\ub}$ (with respect to $d\mu_{\slashed{g}}$) of the sum of the square of all the variation $\psi=\delta^{l'}Z^{\alpha'}_{i}\psi_{\gamma}$ up to order $|\alpha|+1-[b_{|\alpha|+2}]$, where $l'$ is the number of $T$'s in the string of $Z^{\alpha'}_{i}$ and $\gamma=0,1,2,3$. Then by \eqref{Calculus Inequality} we have:
\begin{align*}
\mathcal{S}_{|\alpha|-[b_{|\alpha|+2}]}(t,\ub)\lesssim \delta\left( E_{|\alpha|+1-[b_{|\alpha|+2}]}(t,\ub)+\Eb_{|\alpha|+1-[b_{|\alpha|+2}]}(t,\ub)\right) \quad\text{for all}\quad (t,\ub)\in[-r_{0},t^{*})\times [0,\delta].
\end{align*}
Hence, in view of \eqref{descent estimates final} and \eqref{definitions for finite energy},
\begin{align}\label{surface integral bounded by initial data}
\mathcal{S}_{|\alpha|-[b_{|\alpha|+2}]}(t,\ub)\lesssim\delta\mathcal{D}^{\ub}_{|\alpha|+2} \quad\text{for all}\quad (t,\ub)\in[-r_{0},t^{*})\times [0,\delta].
\end{align}
Then for any variations $\psi$ of order up to $|\alpha|-2-[b_{|\alpha|+2}]$ we have:
\begin{align}\label{transaction of surface integral}
\mathcal{S}_{2}[\psi]\leq \mathcal{S}_{|\alpha|-[b_{|\alpha|+2}]}(t,\ub)
\end{align}
Then by the isoperimetric inequality in \eqref{isoperimetric}, \eqref{surface integral bounded by initial data} and \eqref{transaction of surface integral}, we have:
\begin{align}\label{Linfty on surface}
\delta^{l'}\sup_{S_{t,\ub}}|Z^{\alpha'}_{i}\psi_{\alpha}|=\sup_{S_{t,\ub}}|\psi|\lesssim\delta^{1/2}\sqrt{\mathcal{D}^{\ub}_{|\alpha|+2}}\leq C_{0}\delta^{1/2}
\end{align}
where $C_{0}$ depends on the initial energy $\mathcal{D}^{\ub}_{|\alpha|+2}$, the constant in the isoperimetric inequality and the constant in \eqref{Calculus Inequality} as well as the constants in \eqref{descent estimates final}, which are absolute constants.
If we choose $|\alpha|$ large enough such that
\begin{align*}
\Big[\frac{|\alpha|+1}{2}\Big]+3\leq |\alpha|-2-[b_{|\alpha|+2}]
\end{align*}
then \eqref{Linfty on surface} recovers the bootstrap assumption (B.1) for $(t,\ub)\in[-r_{0},t^{*})\times[0,\delta]$. 

\bigskip

To complete the proof of Theorem 3.1, it remains to show that the smooth solution exists for $t\in[-r_{0},s^{*})$, i.e. $t^{*}=s^{*}$. More precisely, we will prove that either $\mu_{m}(t^{*})=0$ if shock forms before $t=-1$ or otherwise $t^{*}=-1$.

If $t^{*}<s^{*}$, then $\mu$ would be positive on $\Sigma_{t^{*}}^{\delta}$. In particular $\mu$ has a positive lower bound on $\Sigma_{t^{*}}^{\delta}$. Therefore by Remark \ref{geometric meaning of mu}, the Jacobian $\triangle$ of the transformation from optical coordinates to rectangular coordinates has a positive lower bound on $\Sigma_{t^{*}}^{\delta}$. This implies that the inverse transformation from rectangular coordinates to optical coordinates is regular. On the other hand, in the course of recovering bootstrap assumption we have proved that all the derivatives of the first order variations $\psi_{\alpha}$ extend smoothly in optical coordinates to $\Sigma_{t^{*}}^{\delta}$. Since the inverse transformation is regular, $\psi_{\alpha}$ also extend smoothly to $\Sigma_{t^{*}}^{\delta}$ in rectangular coordinates. Once $\psi_{\alpha}$ extend to functions of rectangular coordinates on $\Sigma_{t^{*}}^{\delta}$ belonging to some Sobolev space $H^{3}$, then the standard local existence theorem (which is stated and proved in rectangular coordinates) applies and we obtain an extension of the solution to a development containing an extension of all null hypersurface $\Cb_{\ub}$ for $\ub\in[0,\delta]$, up to a value $t_{1}$ of $t$ for some $t_{1}>t^{*}$, 
which contradicts with the definition of $t^{*}$ and therefore $t^{*}=s^{*}$. This completes the proof of Theorem 3.1. 
\subsection{Data leading to shock formation}
Finally, let us identify a class of initial data constructed in Lemma \ref{lemma constraint} which guarantee shock formation. As we have seen in Remark \ref{geometric meaning of mu}, to let shock form before $t=-1$, we need to let $\mu$ vanish before $t=-1$. By Proposition \ref{proposition on expansion for mu} we have

\begin{align}\label{expan mu temp 1}
\mu(t,\ub,\theta)=\mu(-r_{0},\ub,\theta)-\left(\frac{1}{t}+\frac{1}{r_{0}}\right)r_{0}^{2}\Lb\mu(-r_{0},\ub,\theta)+O\left(\frac{\delta}{t^{2}}\right).
\end{align}
Note that since we already recovered all the bootstrap assumptions, the large parameter $M$ in Proposition \ref{proposition on expansion for mu} goes away here. In view of the fact $\mu(-r_{0},\ub,\theta)=1+O\left(\dfrac{\delta}{r_{0}^{2}}\right)$, the propagation equation for $\mu$ as well as the pointwise estimates for $\psi$ and its derivatives, we rewrite \eqref{expan mu temp 1} as

\begin{align}\label{expan mu temp 2}
\mu(t,\ub,\theta)=1+3r_{0}^{2}\left(\frac{1}{|t|}-\frac{1}{r_{0}}\right)\left(G''(0)\phi_{1}\left(\frac{r-r_{0}}{\delta},\ub,\theta\right)\partial_{s}\phi_{1}\left(\frac{r-r_{0}}{\delta},\ub,\theta\right)\right)+O\left(\frac{\delta}{t^{2}}\right).
\end{align}
Here $\partial_{s}$ is the partial derivative with respect to the first argument of the function $\phi_{1}(s,\theta)$. Therefore if \eqref{shock condition} holds, $\mu(t,\ub,\theta)$ becomes zero before $t=-1$. This completes the proof of the main theorem of the paper.

\section*{Acknowledgment}
This work was supported by NSF grant DMS-1253149 to The University of Michigan and in its initial phase by ERC Advanced Grant 246574 ``Partial Differential Equations of Classical Physics".
\bibliographystyle{plain}
\bibliography{Shockbib}

\end{document}